\documentclass[reqno]{amsart}
\usepackage{amssymb}
\usepackage{amsmath}
\usepackage{graphicx}
\usepackage{tensor}
\newcommand{\be}{\begin{equation}}
\newcommand{\ee}{\end{equation}}
\newcommand{\bea}{\begin{eqnarray*}}
\newcommand{\eea}{\end{eqnarray*}}
\newcommand{\ba}{\begin{array}}
\newcommand{\ea}{\end{array}}
\newcommand{\bi}{\begin{itemize}}
\newcommand{\ei}{\end{itemize}}
\newcommand{\bc}{\begin{center}}
\newcommand{\ec}{\end{center}}
\newcommand{\bfr}{\begin{flushright}}
\newcommand{\efr}{\end{flushright}}

\newcommand{\ds}{\displaystyle}

\numberwithin{equation}{section}

\newtheorem{theorem}{Theorem}[section]

\newtheorem{defi}[theorem]{Definition}

\newtheorem{ex}[theorem]{Example}

\newtheorem{lemma}[theorem]{Lemma}
\newtheorem{prob}{Problem}[section]

\newtheorem{fantoma}[theorem]{\phantom{i}}

\usepackage{graphicx}
\usepackage{latexsym}
\usepackage[colorlinks]{hyperref}

 \newtheorem{thm}{Theorem}[section]
 
 \newtheorem{lem}[thm]{Lemma}

 \newtheorem{defn}[thm]{Definition}

 \newtheorem{rem}[thm]{Remark}
 \numberwithin{equation}{section}

\newcommand{\rd}{\color{red}}
\renewcommand{\rd}{}

\newcommand{\mg}{\color{magenta}}
\renewcommand{\mg}{}
\newcommand{\mn}{\color{magenta}}
\renewcommand{\mn}{}

\newcommand{\comment}[1]{}
\allowdisplaybreaks

\begin{document}

\title[Layer potentials for the Stokes system with $L_{\infty }$ coefficient tensor]
{Variational approach for layer potentials of the Stokes system with $L_{\infty }$ symmetrically elliptic coefficient tensor and applications to Stokes and Navier-Stokes boundary problems}

\author{Mirela Kohr}
\address{Faculty of Mathematics and Computer Science, Babe\c
s-Bolyai University, 1 M. Kog\u al\-niceanu Str., 400084
Cluj-Napoca, Romania}
\email{mkohr@math.ubbcluj.ro}

\author
{Sergey E. Mikhailov}
\address{Department of Mathematics, Brunel University London,
             Uxbridge, UB8 3PH, United Kingdom}
\email{sergey.mikhailov@brunel.ac.uk}

\author{Wolfgang L. Wendland}
\address{Institut f\"ur Angewandte Analysis und Numerische Simulation, Universit\"at Stuttgart,
Pfaffenwaldring, 57, 70569 Stuttgart, Germany}
\email{wendland@mathematik.uni-stuttgart.de}



\begin{abstract}
The first aim of this paper is to develop a layer potential theory in $L_2$-based weighted Sobolev spaces on Lipschitz bounded  and exterior domains of ${\mathbb R}^n$, $n\geq 3$, for the anisotropic Stokes system with {$L_{\infty }$ viscosity coefficient tensor satisfying {an} ellipticity condition for symmetric matrices}.
To do this, we explore equivalent mixed variational formulations and prove the well-posedness of some transmission problems for the anisotropic Stokes system in Lipschitz domains of ${\mathbb R}^n$, with the given data in $L_2$-based weighted Sobolev spaces.
These results are used to define the Newtonian and layer potentials and to obtain their properties.
Then we analyze well-posedness of the exterior Dirichlet, Neumann and mixed problems for the Stokes system with $L_{\infty }$ symmetrically elliptic coefficient tensor. {Solutions} of some of these problems are also represented in terms of the anisotropic Stokes Newtonian and layer potentials.
Finally, we prove the existence of a weak solution for a transmission problem in complementary Lipschitz domains in ${\mathbb R}^3$ for the anisotropic Navier-Stokes system with general data in $L_2$-based weighted Sobolev spaces. The analysis relies on an existence result for a {Dirichlet} problem for the anisotropic Navier-Stokes system in a family of {bounded} domains, and {on} the Leray-Schauder fixed point theorem.
\end{abstract}


\maketitle


\section{Introduction}
\label{Intro}
\setcounter{equation}{0}
The layer potential methods play a fundamental role in the analysis of elliptic boundary value problems (see, e.g., \cite{Co,Co-We,D-M,H-W,Maz'ya-Rossmann,M-W}). 
Fabes, Kenig and Verchota \cite{Fa-Ke-Ve} obtained mapping properties of layer potential operators for the constant coefficient Stokes system in $L_p$ spaces by using a technique of harmonic analysis. Further extension of these results to $L_p$, Sobolev, Bessel potential, and Besov spaces has been obtained by Mitrea and Wright \cite{M-W}. Moreover, Mitrea and Wright have used layer potential methods to obtain well-posedness results for the main boundary value problems for the standard Stokes system with constant coefficients in
arbitrary
Lipschitz domains in ${\mathbb R}^n$. The authors in \cite{K-L-M-W} obtained mapping properties of the constant-coefficient Stokes and Brinkman layer potential operators in standard and weighted Sobolev spaces in ${\mathbb R}^3$ by exploiting results of singular integral operators. Medkov\'{a} \cite{Med-CVEE-16} has used an integral equation method in the study of bounded solutions of the Dirichlet problem for the Stokes resolvent system in the Euclidean setting.
The authors in \cite{K-L-W} combined a layer potential approach with a fixed point theorem to show an existence result for a nonlinear Neumann-transmission problem for the constant-coefficient Stokes and Brinkman systems in $L_p$, Sobolev, and Besov spaces (see also \cite{K-L-W1}).

Choi and Lee \cite{Choi-Lee} have studied the Dirichlet problem for the stationary Stokes system with irregular coefficients. They have proved the unique solvability of the problem in Sobolev spaces on a Lipschitz domain in ${\mathbb R}^n$, $n\geq 3$, with a small Lipschitz constant, by assuming that the coefficients have vanishing mean oscillations (VMO) with respect to all variables. Existence and pointwise bounds of the fundamental solution for the stationary Stokes system with measurable coefficients in ${\mathbb R}^n$ ($n\geq 3$) have been obtained by Choi and Yang in \cite{Choi-Yang} under the assumption of local H\"{o}lder continuity of weak solutions of the Stokes system. They also discussed the existence and pointwise bounds of the Green function for the Stokes system with measurable coefficients on unbounded domains where the divergence equation is solvable, particularly on the half-space.
The solvability in Sobolev spaces of the conormal derivative problem for the stationary Stokes system with non-smooth coefficients on bounded Reifenberg flat domains have been proved by Choi, Dong and Kim in \cite{Choi-Dong-Kim} {(see also \cite{Choi-Dong-Kim-JMFM})}. 

The methods of layer potential theory play also a significant role in the study of elliptic boundary value problems with variable coefficients.
Mitrea and Taylor \cite{M-T} have obtained well-posedness results for the Dirichlet problem for the smooth coefficient Stokes system in $L_p$-spaces on arbitrary Lipschitz domains in a compact Riemannian manifold, and extended the well-posedness results in \cite{Fa-Ke-Ve} from the Euclidean setting to the compact Riemannian setting. Dindo\u{s} and Mitrea \cite{D-M} have used the mapping properties of Stokes layer potentials in Sobolev and Besov spaces to show well-posedness results for Poisson problems for the smooth coefficient Stokes and Navier-Stokes systems with Dirichlet boundary condition on $C^1$ and Lipschitz domains in compact Riemannian manifolds. 
Well-posedness results for transmission problems for the smooth coefficient Navier-Stokes and Darcy-Forchheimer-Brinkman systems in Lipschitz domains on compact Riemannian manifolds have been obtained in \cite{K-M-W}.

An alternative approach, which reduces various boundary value problems for variable-coefficient elliptic partial differential equations to {\em boundary-domain integral equations} (BDIEs), by means of explicit parametrix-based integral potentials, was explored e.g., in \cite{CMN-1,CMN-2,Ch-Mi-Na,Ch-Mi-Na-2,Ch-Mi-Na-3,Mikh-18}.
Equivalence of BDIEs to the boundary problems and invertibility of BDIE operators in $L_2$ and $L_p$-based Sobolev spaces have been analyzed in these works.
Localized BDIEs based on a harmonic parametrix for divergence-form elliptic PDEs with variable matrix coefficients have been {also developed, see \cite{CMN2017}
and the references therein.}

Amrouche, Girault and Giroire \cite{{AGG1997}} used a variational approach in the analysis of the exterior Dirichlet and Neumann problems for the $n$-dimensional Laplace operator in weighted Sobolev spaces. Mazzucato and Nistor \cite{Ma-Ni} obtained well-posedness and regularity results for the elasticity equations with mixed conditions on polyhedral domains. Hofmann, Mitrea and Morris \cite{H-M-M} considered layer potentials in $L_p$ spaces for elliptic operators of the form $L=-{\rm{div}}(A\nabla u)$ that act in the upper half-space ${\mathbb R}_{+}^{n+1}:=\{(x,t):x\in {\mathbb R}^n,\, t\in {\mathbb R}_+\}$, $n\geq 2$, or in more general Lipschitz graph domains, where $A$ is an $(n+1)\times (n+1)$ type matrix of $L_{\infty }$ complex, $t$-independent coefficients satisfying a uniform ellipticity condition, and solutions of the equation $L u =0$ satisfying De Giorgi-Nash-Moser type interior estimates. They developed a Calder\'{o}n-Zygmund type theory associated with the layer potentials, and obtained well-posedness results for related boundary problems in $L_p$ and endpoint spaces. Brewster et al. in \cite{B-M-M-M} have used a variational approach to obtain well-posedness results for Dirichlet, Neumann and mixed boundary problems for higher order divergence-form elliptic equations with $L_{\infty }$ coefficients in locally $(\epsilon,\delta )$-domains and in Besov and Bessel potential spaces (see also \cite{H-J-K-R}). Barton \cite{Barton} has used the Lax-Milgram Lemma to construct layer potentials for {strongly} elliptic differential operators in Banach spaces, and generalized many properties of layer potentials for the harmonic equation. Barton and Mayboroda \cite{Barton-Mayboroda} developed layer potentials for second order divergence elliptic operators with bounded measurable coefficients that are independent of the $(n+1)$st coordinate, and well-posedness results for related boundary problems with data in Besov spaces.

Girault and Sequeira \cite{Gi-Se} obtained well-posedness of the exterior Dirichlet problem for the constant coefficient Stokes system in weighted Sobolev spaces on exterior Lipschitz domains in ${\mathbb R}^n$ for $n\in\{2,3\}$, by applying a mixed variational formulation.
{Angot \cite{Angot-2} analyzed some Stokes/Brinkman transmission problems with a scalar viscosity coefficient on bounded domains.}
Sayas and Selgas in \cite{Sa-Se} developed a variational approach for the constant-coefficient Stokes layer potentials on Lipschitz boundaries, by using the technique of N\'{e}d\'{e}lec \cite{Ne}.
{The book by Sayas, Brown and Hassell \cite{Sayas-book} gives a comprehensive presentation
of the basic variational theory for elliptic PDEs in Lipschitz domains.}
B\u{a}cu\c{t}\u{a}, Hassell and Hsiao \cite{B-H} developed a variational approach for the constant-coefficient Brinkman single layer potential and used it to analyze the corresponding time dependent exterior Dirichlet problem in ${\mathbb R}^n$, $n\!=\!2,3$. Alliot and Amrouche \cite{Al-Am} have used a variational approach to obtain weak solutions for the exterior Stokes problem in weighted Sobolev spaces (see also \cite{Amrouche-1}).

In \cite{K-M-W-1},  we obtained the well-posedness results for the {\it isotropic} Stokes system with a non-smooth scalar viscosity coefficient $\mu \in L_{\infty }({\mathbb R}^3)$ (see also \cite{K-W,K-W1,K-W2} for the Stokes and Navier-Stokes systems with nonsmooth coefficients in compact Riemannian {manifolds}).
In \cite{K-M-W-2} we analysed transmission problems in weighted Sobolev spaces for {\it anisotropic} Stokes and Navier-Stokes systems with an $L_\infty $ strongly elliptic coefficient tensor, in the pseudostress setting.

In this paper we proceed with the study of transmission and exterior boundary value problems for the anisotropic Stokes system. However, unlike  \cite{K-M-W-2}, we consider the $L_{\infty }$ viscosity coefficient tensor satisfying a strong ellipticity condition only with respect to all {\it symmetric} matrices in ${\mathbb R}^{n\times n}$, see \eqref{mu}.
Our first aim is to develop a layer potential theory in $L_2$-based weighted Sobolev spaces for such Stokes systems.
To do this, we explore equivalent mixed variational formulations and prove the well-posedness of some transmission problems for the anisotropic Stokes system in Lipschitz domains of ${\mathbb R}^n$, with the given data in $L_2$-based weighted Sobolev spaces.
These results are used to define the Newtonian and layer potentials and to obtain their properties.
Then we analyze the well-posedness of exterior Dirichlet, Neumann and mixed problems for the Stokes system with $L_{\infty }$ symmetrically elliptic coefficient tensor. The solutions of some of these problems are also represented in terms of the anisotropic Stokes Newtonian and layer potentials.
Finally, we prove the existence of a weak solution ${\bf u}\in {\mathcal H}^1({\mathbb R}^3)^3$ 
of a transmission problem in complementary Lipschitz domains in ${\mathbb R}^3$ for the anisotropic Navier-Stokes system with general data in $L_2$-based weighted Sobolev spaces (see formula \eqref{weight-1} for the definition of the weighted space ${\mathcal H}^1({\mathbb R}^3)^3$).
The analysis relies on an existence result for a Dirichlet-transmission problem for the anisotropic Navier-Stokes system {in a family of bounded Lipschitz domains in} $\mathbb R^3$ and on the Leray-Schauder fixed point theorem.
We show also the existence of a pressure field $\pi $, which is locally in $L_2({\mathbb R}^3)$, such that the couple $({\bf u},\pi )$ solves the Navier-Stokes transmission problem in the sense of distributions.
Note that the global membership of the function $\pi $ in $L_2({\mathbb R}^3)$
would require more regularity of the velocity field ${\bf u}$, of the viscosity coefficient tensor, and of the geometry of the domains (see, e.g., \cite[Section 2]{A-A}).

{The boundary value problems for the anisotropic Stokes and Navier-Stokes systems with $L_\infty $ coefficients analyzed in this paper can describe physical, engineering, or industrial processes {related to the {flow} of immiscible fluids}, {or the flow of nonhomogeneous fluids with density dependent viscosity} (cf., e.g., \cite{Choi-Dong-Kim}).}

\subsection{The anisotropic Stokes system with $L_\infty $ {symmetrically} elliptic coefficient tensor}

{All along the paper we use the Einstein summation convention for repeated indices from $1$ to $n$, and the standard notation $\partial _\alpha $ for the first order partial derivative $\ds\frac{\partial }{\partial x_\alpha }$, $\alpha =1,\ldots ,n$.}

{Let $\boldsymbol{\mathbb L}$ be a second order differential operator in the divergence form
{in an open set $\Omega\subseteq\mathbb R^n$, $n\ge 3$,}
\begin{equation}
\label{Stokes-0-0}
\begin{array}{lll}
{\boldsymbol{\mathbb L}{\bf u}=\mathrm{div}\left(\mathbb{AE}(\mathbf u)\right)\quad  \Longleftrightarrow\quad }
(\boldsymbol{\mathbb L}{\bf u})_i:=\partial _\alpha\left(a_{ij}^{\alpha \beta }E_{j\beta }({\bf u})\right),\ \ i=1,\ldots ,n,
\end{array}
\end{equation}
where ${\bf u}=(u_1,\ldots ,u_n)^\top$, and
{${\mathbb E}({\bf u})=\left(E_{j\beta }({\bf u})\right)_{1\leq j,\beta\leq n}$
is the symmetric part of the gradient
$\nabla {\bf u}$. 
Therefore, the components of the tensor field ${\mathbb E}({\bf u})$ are defined by 
$E_{j\beta }({\bf u}):=\frac{1}{2}(\partial_j u_\beta +\partial _\beta u_j)$.

The {viscosity coefficient tensor $\mathbb A$ in} the operator ${\mathbb L}$ consists of {$n\times n$ matrix-valued functions $A^{\alpha \beta }=A^{\alpha \beta }(x)$ with {essentially bounded}, real-valued entries}, i.e.,
\begin{align}
\label{Stokes-1}
{\mathbb A}=\left(A^{\alpha \beta }\right)_{1\leq \alpha,\beta \leq n}
=\left(a_{ij}^{\alpha\beta}\right)_{1\leq \alpha,\beta,i,j\leq n};\quad
\ a_{ij}^{\alpha \beta }\in L_{\infty }(\Omega),\ 1\leq \alpha,\beta,i,j \leq n,
\end{align}
satisfying the {symmetry} conditions
\begin{align}
\label{Stokes-sym}
a_{ij}^{\alpha \beta }(x)=a_{\alpha j}^{i\beta }(x)=a_{i\beta }^{\alpha j}(x),\ \ x\in \Omega
\end{align}
(cf. \cite[Eq. (3.2)]{Oleinik}, \cite[Eqs. (6), (7)]{duffy}).
Note that the {symmetry conditions \eqref{Stokes-sym} {\it do not imply} the symmetry
$a_{ij}^{\alpha \beta }(x)=a^{\beta\alpha}_{ji}(x)$, which will be generally not assumed in the paper}.
In addition, we assume that the coefficients satisfy the following {\it ellipticity} condition, which asserts that there exists a constant $c_{\mathbb A} >0$ such that for almost all $x\in\Omega$,
\begin{align}
\label{mu}
{a_{ij}^{\alpha \beta }(x)\xi _{i\alpha }\xi _{j\beta }\geq c_{\mathbb A}^{-1}|\boldsymbol\xi|^2}\,
\ \ \forall\ \boldsymbol\xi =(\xi _{i\alpha })_{i,\alpha =1,\ldots ,n}\in {\mathbb R}^{n\times n}\ \mbox{ with } \boldsymbol\xi=\boldsymbol\xi^\top \mbox{ and }
\sum_{i=1}^n\xi _{ii}=0,
\end{align}
where $|\boldsymbol\xi |^2=\xi _{i\alpha }\xi _{i\alpha }$.
Note that the ellipticity { condition} \eqref{mu} is assumed only for all {\it symmetric} matrices $\boldsymbol\xi =(\xi _{i\alpha })_{i,\alpha =1,\ldots ,n}\in {\mathbb R}^{n\times n}$, cf. \cite[Eqs. (3.1), (3.3)]{Oleinik}, having zero matrix trace, $\sum_{i=1}^n\xi _{ii}=0$.

In view of \eqref{Stokes-1}, ${\mathbb A}$ is endowed with the norm
\begin{align}\label{A-norm}
\|{\mathbb A}\|_{L_\infty (\Omega)}:=\max_{i,j,\alpha ,\beta \in\{1,\ldots, n\}}
\left\{
\|a_{ij}^{\alpha \beta }\|_{L_\infty (\Omega)}
\right\}.
\end{align}

The symmetry conditions \eqref{Stokes-sym} allow us to express the operator $\boldsymbol{\mathbb L}$ in the equivalent forms
\begin{align}
\label{L-oper-global}
&(\boldsymbol{\mathbb L}{\bf u})_i=\partial _\alpha\left(a_{ij}^{\alpha \beta }E_{j\beta }({\bf u})\right)=\partial _\alpha\left(a_{ij}^{\alpha \beta }\partial _\beta u_j\right),\ \ i=1,\ldots ,n,\\
\label{Stokes-0}
&{\boldsymbol{\mathbb L}{\bf u}= 
\partial _\alpha\left(A^{\alpha \beta }\partial _\beta {\bf u}\right)}.
\end{align}

Note that the {first} equality in \eqref{L-oper-global} has not been encountered in \cite{K-M-W-2}, where the coefficients of the forth order tensor ${\mathbb A}$ have been assumed to satisfy the strong ellipticity condition similar to the second condition in \eqref{mu} but for all (not only symmetric) matrices $\boldsymbol\xi =(\xi _{i\alpha })_{i,\alpha =1,\ldots ,n}\in {\mathbb R}^{n\times n}$ (see \cite[Eqs. (2)-(3)]{K-M-W-2}). The more restrictive ellipticity condition in \cite{K-M-W-2}  allowed to explore there the associated non-symmetric pseudostress setting.
In this paper we require the symmetry conditions \eqref{Stokes-sym} and the ellipticity condition \eqref{mu} only for symmetric matrices $\boldsymbol\xi =(\xi _{i\alpha })_{i,\alpha =1,\ldots ,n}\in {\mathbb R}^{n\times n}$, and develop our results in the symmetric stress setting.
{This approach allows us to obtain properties of layer potentials for the Stokes system with $L_\infty $ variable coefficients generalizing well known results for constant coefficients.}

Let ${\bf u}$ be an unknown vector field, $\pi $ be an unknown scalar field, and ${\bf f}$ be a given vector field defined in $\Omega \subseteq {\mathbb R}^n$. 
Then the equations
\begin{equation}
\label{Stokes}
\boldsymbol{\mathcal L}({\bf u},\pi ):=\boldsymbol{\mathbb L}{\bf u}-\nabla \pi={\bf f},\
{\rm{div}}\ {\bf u}=g \mbox{ in } \Omega
\end{equation}
determine the Stokes system which describes viscous compressible fluid flows with variable anisotropic viscosity coefficient tensor
$\mathbb A$
depending on the physical properties of the fluid, such as, e.g., the given fluid temperature (cf. \cite{duffy}, \cite{Ni-Be}). If $g=0$ then the fluid is incompressible.

{According to \eqref{L-oper-global} and \eqref{Stokes-0}, the Stokes operator $\boldsymbol{\mathcal L}$ can be written in any of the equivalent forms}
\begin{align}
\label{Stokes-new}
{\boldsymbol{\mathcal L}({\bf u},\pi )=\partial _\alpha\left(A^{\alpha \beta }\partial _\beta {\bf u}\right)-\nabla \pi ,\ \ \left(\boldsymbol{\mathcal L}({\bf u},\pi )\right)_i=\partial _\alpha\left(a_{ij}^{\alpha \beta }E_{j\beta }({\bf u})\right)-\partial _i\pi ,\ \ i=1,\ldots ,n\,.}
\end{align}

The anisotropic Stokes system \eqref{Stokes} is Agmon-Douglis-Nirenberg elliptic (see, e.g., \cite[Definition 6.2.3]{H-W}, and Lemma \ref{ADN-system}). 

\subsection{{\bf Isotropic case}}
For the {\it isotropic case}, the viscosity tensor ${\mathbb A}$ in
\eqref{Stokes-1}
has the form
{(cf., e.g., Appendix III, Part I, Section 1 in \cite{Temam}),}
\begin{align}
\label{isotropic}
{a_{ij}^{\alpha \beta}{(x)}={\lambda{(x)} \delta _{i\alpha }\delta _{j\beta }}+\mu{(x)} \left(\delta_{\alpha j}\delta _{\beta i}+\delta_{\alpha \beta }\delta _{ij}\right),\ 1\leq i,j, \alpha ,\beta \leq n}
\end{align}
where {$\lambda,\mu\in L_{\infty }(\Omega)$ and}
\begin{align}
\label{mu-0}
c_\mu^{-1}\leq\mu(x) \leq {c_\mu} \mbox{ for a.e. } x\in\Omega
\end{align}
with a constant $c_\mu>0$.
Then
$$
a_{ij}^{\alpha\beta}(x)\xi_{i\alpha}\xi_{j\beta}
=\lambda(x)(\xi_{ii})^2+2\mu(x)\xi_{i\alpha}\xi_{i\alpha}
=2\mu(x)\xi_{i\alpha}\xi_{i\alpha}=2\mu(x)|\boldsymbol\xi|^2\ge c_\mu^{-1}|\boldsymbol\xi|^2\quad
\mbox{for a.e. } x\in \Omega,
$$
for any symmetric matrix
$\boldsymbol\xi =(\xi _{i\alpha })_{1\leq i,\alpha \leq n}\in {\mathbb R}^{n\times n}$ such that
$\xi_{ii}=\sum_{i=1}^n\xi_{ii}=0$.
Therefore, the symmetric ellipticity condition in \eqref{mu} is satisfied as well, and hence our results are also applicable to the {\it Stokes system in the isotropic case}. If $\mu >0$ is a constant and $g=0$, then \eqref{Stokes} reduces to the well known isotropic incompressible Stokes system with constant viscosity $\mu$.

\section{Functional setting and preliminary results}
\label{preliminaries}
\setcounter{equation}{0}

Let $\Omega _{+}$ be a bounded Lipschitz domain in ${\mathbb R}^n$, i.e., an open connected set whose boundary $\partial {\Omega }$ is locally the graph of a Lipschitz function and is connected.
We further assume that $n\geq 3$ unless explicitly stated otherwise.
Sometimes we will write just $\Omega$ instead of $\Omega_+$.
Let $\Omega _{-}:={\mathbb R}^n\setminus \overline{{\Omega }}_{+}$ be the corresponding exterior Lipschitz domain. Let $\mathring E_\pm$ denote the operator of extension {of functions} by zero outside $\Omega_\pm$.

\subsection{\bf $L_2$-based Sobolev spaces}
Given a Banach space ${\mathcal X}$, its topological dual is denoted by ${\mathcal X}'$, and the notation $\langle \cdot ,\cdot \rangle _X$ means the duality pairing of two dual spaces defined on a set $X\subseteq {\mathbb R}^n$.

Let ${\mathcal S}({\mathbb R}^n)$ be the Schwartz space of rapidly decaying smooth functions on ${\mathbb R}^n$, and let ${\mathcal S}'({\mathbb R}^n)$ be the corresponding topological dual space, i.e., the space of tempered distributions. Let ${\mathcal F}$ and ${\mathcal F}^{-1}$ denote the Fourier transform and its inverse defined on ${\mathcal S}({\mathbb R}^n)$ 
and generalized to the space ${\mathcal S}'({\mathbb R}^n)$. Let $H^1({\mathbb R}^n)$ and $H^1({\mathbb R}^n)^n$ denote the $L_2$-based Sobolev (Bessel potential) spaces
\begin{align}
\label{bessel-potential}
&\!H^1({\mathbb R}^n):=\left\{f\in {\mathcal S}'({\mathbb R}^n):\|f\|_{H^1({\mathbb R}^n)}=\big\|{\mathcal F}^{-1}[(1+|\xi |^2)^{\frac{1}{2}}{\mathcal F}f]\big\|_{L_2({\mathbb R}^n)}<\infty \right\}\,,
\\
&H^1({\mathbb R}^n)^n:=\{f=(f_1,\ldots ,f_n):f_j\in H^1({\mathbb R}^n),\ j=1,\ldots ,n\}\,.
\label{bessel-potential2}
\end{align}
The dual of $H^1({\mathbb R}^n)$ is the space $H^{-1}({\mathbb R}^n)$.

Let {$\Omega '$ be either a bounded Lipschitz domain, or the exterior of a bounded Lipschitz domain in ${\mathbb R}^n$.}
Let ${\mathcal D}(\Omega '):=C^{\infty }_{0}(\Omega ')$ denote the space of infinitely differentiable functions with compact support in $\Omega '$, equipped with the inductive limit topology.
Let ${\mathcal D}'(\Omega ')$ denote the corresponding space of distributions on $\Omega '$, i.e., the dual of the space ${\mathcal D}(\Omega ')$. Let $L_2(\Omega ')$ denote the Lebesgue space of (equivalence classes of) measurable, square-integrable functions on $\Omega '$, and $L_{\infty }(\Omega ')$ denotes the space of (equivalence classes of) essentially bounded measurable functions on $\Omega '$.
Let also
\begin{align}
\label{spaces-Sobolev-inverse}
&H^1({\Omega '}):=\{f\in {\mathcal D}'(\Omega '):\exists \, F\in H^1({\mathbb R}^n)
\mbox{ such that } F_{|{\Omega '}}=f\}\,,
\end{align}
where $_{|\Omega'}=r_{_{\Omega'}}$ denotes the {operator of restriction to} $\Omega '$.

The space $\widetilde{H}^1({\Omega '})$ is the closure of ${\mathcal D}(\Omega ')$ in $H^1({\mathbb R}^n)$. It can be also characterized as
\begin{align}\label{2.5}
&\widetilde{H}^1({\Omega '}):=\left\{\widetilde{f}\in H^1({\mathbb R}^n):{\rm{supp}}\,
\widetilde{f}\subseteq \overline{\Omega '}\right\},
\end{align}
where ${\rm{supp}}f:=\overline{\{x\in {\mathbb R}^n:f(x)\neq 0\}}$ (see, e.g., \cite[Theorem 3.33]{Lean}).

Similar to definition \eqref{bessel-potential2}, $H^1({\Omega '})^n$ and $\widetilde{H}^1({\Omega '})^n$ are the spaces of vector-valued functions whose components belong to the scalar spaces $H^1({\Omega '})$ and $\widetilde{H}^1({\Omega '})$, respectively.

The $L_2$-based Sobolev space $W_2^1(\Omega ')$ is defined as $W_2^1(\Omega '):=\left\{f\in L_2(\Omega '):\nabla f\in L_2(\Omega ')^n\right\}$ and formula
\begin{align}
\label{Sobolev}
\|f\|_{W_2^1(\Omega ')}=\|f\|_{L_2(\Omega ')}+\|\nabla f\|_{L_2(\Omega ')^n}
\end{align}
defines a norm on this space.

The above assumption about $\Omega '$ implies that $H^1(\Omega ')=W_2^1(\Omega ')$ with equivalent norms (cf., e.g., \cite[Theorem 3.18]{Lean} and \cite[Theorem 2.4.1]{M-W}).

The space $\widetilde{H}^1({\Omega '})$ can be identified with the space $\mathring{H}^1(\Omega ')$, which is defined as the closure of ${\mathcal D}(\Omega ')$ in ${H}^1({\Omega '})$ (cf., e.g., \cite[Theorem 3.33]{Lean}). In addition, the assumption that {the set $\Omega'$ has a compact Lipschitz boundary} implies that $\mathring{H}^1(\Omega ')$ can be equivalently described as the space of all functions in ${H}^1(\Omega ')$ with null traces on the boundary of $\Omega '$} (cf., e.g., \cite[Proposition 3.3]{Lean}).

On the other hand, since the space ${\mathcal D}(\overline{\Omega '})$ (the restriction of functions from 
${\mathcal D}({\mathbb R}^n)$ onto $\Omega '$) is dense in ${H}^{1}({\Omega '})$, the dual of ${H}^{1}({\Omega '})$, {denoted by $\widetilde{H}^{-1}(\Omega ')$}, is a space of distributions. Moreover, {the assumption that $\Omega '$ is a bounded Lipschitz domain, or the exterior of a bounded Lipschitz domain in ${\mathbb R}^n$} yields that the following spaces can be isomorphically identified (cf., e.g., \cite[(1.9)]{Fa-Me-Mi}, \cite[Theorem 3.14]{Lean})
\begin{equation}
\label{duality-spaces} \left({H}^{1}({\Omega '})\right)'=\widetilde{H}^{-1}({\Omega '}),\ \ {H}^{-1}({\Omega '})=\left(\widetilde{H}^{1}({\Omega '})\right)'.
\end{equation}

The boundary Sobolev space $H^s(\partial \Omega )$, $0<s<1$, can be defined by
\begin{align*}
H^s(\partial \Omega)=\left\{f\in L_2(\partial \Omega): \int _{\partial \Omega}\int_{\partial \Omega }\frac{|f({\bf x})-f({\bf y})|^2}{|{\bf x}-{\bf y}|^{n-1+2s}}d\sigma _{{\bf x}}d\sigma _{{\bf y}}<\infty \right\}\,,
\end{align*}
where $\sigma _{\bf y}$ is the surface measure on $\partial \Omega $ (see, e.g., \cite[Proposition 2.5.1]{M-W}). The dual of $H^s(\partial \Omega )$ is the space $H^{-s}(\partial \Omega )$, 
and we set $H^0(\partial \Omega)\!=\!L_2(\partial \Omega )$.
Let $H^{s}(\partial \Omega)^n$ denote the space of vector-valued functions whose components belong to $H^{s}(\partial \Omega)$. The dual of $H^{s}(\partial \Omega)^n$ is the space $H^{-s}(\partial \Omega)^n$.

All $L_2$-based Sobolev spaces mentioned above are Hilbert spaces.
The following well-known trace theorem holds true
(see \cite{Co}, \cite[Proposition 3.3]{J-K1}, \cite[Lemma 2.6]{Mikh}, \cite[Theorem 2.5.2]{M-W}).
\begin{thm}
\label{trace-operator1} Let ${\Omega }\!:=\!\Omega _{+}$ be a bounded Lipschitz domain of ${\mathbb R}^n$  with connected boundary $\partial \Omega $,
and let $\Omega _{-}\!:=\!{\mathbb R}^n\setminus \overline{\Omega }$ be the corresponding
exterior domain.
Then there exist linear bounded trace operators $\gamma_{\pm }\!:\!H^1({\Omega }_{\pm })\!\to \!H^{\frac{1}{2}}({\partial\Omega })$ such that $\gamma_{\pm }f\!=\!f_{|{{\partial\Omega }}}$ for any $f\!\in \!C^{\infty }(\overline{\Omega }_{\pm })$.
{The operators $\gamma _{\pm }$ are surjective and have $($non-unique$)$ linear and bounded right inverse operators} $\gamma ^{-1}_{\pm }\!:\!H^{\frac{1}{2}}({\partial\Omega })\!\to \!H^1({\Omega }_{\pm }).$ The trace operator $\gamma :{H}^1({\mathbb R}^n)\to H^{\frac{1}{2}}(\partial \Omega )$ can also be considered and is linear and bounded\footnote{The trace operators defined on Sobolev spaces of vector fields on ${\Omega }_{\pm }$ or ${\mathbb R}^n$ are also denoted by $\gamma_{\pm }$ and $\gamma $, respectively.}.
\end{thm}

Note that any function $u\in H^1_{{\rm loc}}({\mathbb R}^n)$ has the jump
\begin{equation}
\label{jump-notation}
\left[{\gamma }(u)\right]:=\gamma _+(u)-\gamma _{-}(u)
\end{equation}
equal to zero across $\partial \Omega $. 

Other properties of Sobolev spaces can be found in \cite{Agr-1,J-K1,Lean,M-W,Triebel}.

\subsection{\bf  Weighted Sobolev spaces}
\label{S2.2}
Let $|{\bf x}|=(x_1^2+\cdots +x_n^2)^{\frac{1}{2}}$ denote the Euclidean distance of a point ${\bf x}=(x_1,\ldots ,x_n)\in {\mathbb R}^n$ to the origin of ${\mathbb R}^n$. Let $\rho $ be the weight function
\begin{align}
\label{rho}
\rho ({\bf x})=(1+|{\bf x}|^2)^{\frac{1}{2}}\,.
\end{align}

\subsubsection{\bf Weighted Sobolev spaces on ${\mathbb R}^n$}
\label{2.2.1}

The weighted Lebesgue space $L_2(\rho^{-1};{\mathbb R}^n)$ { defined} by
\begin{align}
\label{Lp-weight}
L_2(\rho^{-1};{\mathbb R}^n):=\left\{f\in {{\mathcal D}'({\mathbb R}^n)}: {\rho }^{-1} f\in L_{2}({\mathbb R}^n)\right\},
\end{align}
is a Hilbert space with respect to the inner product and the associated norm
\begin{align}
\label{Lp-weight2}
&(f,g)_{L_2(\rho^{-1};{\mathbb R}^n)} :=\int_{\mathbb R^n}fg{\rho }^{-2}dx,\ \|f\|^2_{L_2(\rho^{-1} ;{\mathbb R}^n)}:= (f,f)_{L_2(\rho^{-1};{\mathbb R}^n)}.
\end{align}

We also consider the weighted Sobolev space
\begin{align}
\label{weight-1}
{\mathcal H}^1({\mathbb R}^n):=
\left\{f\in {\mathcal D}'({\mathbb R}^n):{\rho }^{-1}f\in L_2({\mathbb R}^n),\ \nabla f\in L_2({\mathbb R}^n)^n\right\},\  n\geq 3
\end{align}
(cf. \cite[Definition 1.1]{Al-Am}, \cite[Theorem I.1]{Ha}), { which} is also a Hilbert space with the norm defined by
\begin{align}
\label{weight-2p}
\|f\|_{{\mathcal H}^1({\mathbb R}^n)}^2:=
\left\|{\rho }^{-1}f\right\|_{L_2({\mathbb R}^n)}^2+\|\nabla f\|_{L_2({\mathbb R}^n)^n}^2\,.
\end{align}
The space ${\mathcal D}({\mathbb R}^n)$ is dense in ${\mathcal H}^1({\mathbb R}^n)$ (cf., e.g., {\rd \cite[p. 727]{Al-Am-1}, and \cite[\mn Th\'{e}or\`{e}me I.1]{Ha} and \cite[Proposition 2.1]{Sa-Se} in the case $n=3$}), and, thus, the dual ${\mathcal H}^{-1}({\mathbb R}^n)$ of ${\mathcal H}^1({\mathbb R}^n)$ is a space of distributions.
Let us consider the semi-norm
\begin{align}
\label{seminorm-R3}
|f|_{{\mathcal H}^1({\mathbb R}^n)}:=\|\nabla f\|_{L_2({\mathbb R}^n)^n}.
\end{align}
This semi-norm is a norm on the space ${\mathcal H}^1({\mathbb R}^n)$
and is equivalent to the norm $\|\cdot \|_{{\mathcal H}^1({\mathbb R}^n)}$, given by \eqref{weight-2p}
(cf., e.g., \cite[Theorem 1.1]{Al-Am-1}).

In view of Lemma 2.5 of Kozono and Sohr \cite{Kozono-Shor}, the divergence operator ${\rm{div}}:{\rd {\mathcal H}^1({\mathbb R}^n)^n}\to L_2({\mathbb R}^n)$
is surjective and has a bounded right inverse.
Moreover, Remark 3.8(i) in \cite{Al-Am-1} and Proposition 2.4(i) in \cite{Kozono-Shor} imply that for $n\geq 3$, the weighted Sobolev space ${\mathcal H}^1({\mathbb R}^n)$ 
can be also characterized as
\begin{align}
\label{weight-Lp}
{\rd {\mathcal H}^1({\mathbb R}^n) 
=\left\{u\in L_{\frac{2n}{n-2}}({\mathbb R}^n):\nabla u\in L_2({\mathbb R}^n)^n\right\},}
\end{align}
with equivalent norms,
and the map ${\rm{div}}:{\mathcal H}^1({\mathbb R}^n{\mn)^n}\to L_2({\mathbb R}^n)$ is surjective and has a bounded right inverse. 
Thus, there is a constant $c=c(n)>0$ such that for any $g\in L_2({\mathbb R}^n)$, there exists {\rd ${\bf u}\in {\mathcal H}^1({\mathbb R}^n)^n$} satisfying
\begin{align}
\label{surj-div-p}
{\rm{div}}\, {\bf u}=g \mbox{ and } \|{\bf u}\|_{{\mathcal H}^1({\mathbb R}^n\mn)^n}\leq c\|g\|_{L_2({\mathbb R}^n)}.
\end{align}

\subsubsection{\bf Weighted Sobolev spaces on exterior Lipschitz domains}
The weighted Sobolev space ${\mathcal H}^1(\Omega _{-})$ can be defined as in \eqref{weight-1} with $\Omega _{-}$ in place of ${\mathbb R}^n$. Therefore,
\begin{align}
\label{ext-p-weight}
{\mathcal H}^1(\Omega _-):=
\left\{v\in {\mathcal D}'(\Omega _{-}):{\rho }^{-1}v\in L_2(\Omega _{-}),\ \nabla v\in L_2(\Omega _{-})^n\right\},\ n\geq 3\,,
\end{align}
and is a Hilbert space with a norm given by \eqref{weight-2p} with $\Omega _-$ in place of ${\mathbb R}^n$ (see, e.g., \cite[Definition 1.1]{Al-Am}). {The space {\rd $\widetilde{\mathcal H}^{-1}(\Omega _{-})$} is the dual of the space ${\mathcal H}^1(\Omega _{-})$.}

Next we mention some useful properties of these spaces. 
First, note that the space ${\mathcal D}(\overline{\Omega }_{-})$ is dense in ${\mathcal H}^1(\Omega _-)$. 
Moreover, the functions of ${\mathcal H}^1(\Omega _-)$ belong to $H^1(D)$ for any bounded domain $D$ contained in $\Omega _-$ (see also \cite{Al-Am}).
Since $H^1(\Omega _-)\subset {\mathcal H}^1(\Omega _-)$, the statement of Theorem \ref{trace-operator1} extends also to the weighted Sobolev space ${\mathcal H}^1({\Omega }_{-})$. Therefore, there exists a bounded linear and surjective {exterior trace operator}
\begin{align}
\label{ext-trace}
\gamma_{-}:{\mathcal H}^1({\Omega }_{-})\to H^{\frac{1}{2}}({\partial\Omega })
\end{align}
which has a (non-unique) bounded linear right inverse $\gamma^{-1}_-:H^{\frac{1}{2}}({\partial\Omega })\!\to \! \mathcal{H}^1({\Omega }_-)$ (see \cite[Lemma 2.2]{K-L-M-W}, \cite[Theorem 2.3, Lemma 2.6]{Mikh}, {\cite[p. 69]{Sa-Se}}).
{The trace operator $\gamma :{\mathcal H}^1({\mathbb R}^n)\to H^{\frac{1}{2}}({\partial\Omega })$ defined by $\gamma (u)=\gamma_{+}(u_+)=\gamma_{-}(u_-)$ for any $u\in {\mathcal H}^1({\mathbb R}^n)$, where $u_\pm :=u|_{_{\Omega _\pm }}$,
is bounded linear and surjective as well (cf., e.g., \cite[Theorem 2.3, Lemma 2.6]{Mikh}, \cite[(2.2)]{B-H}).}

Let us now consider the space $\mathring{\mathcal H}^1(\Omega _{-})$ as the closure of the space ${\mathcal D}({\Omega }_{-})$ with respect to the norm $\|\cdot \|_{{\mathcal H}^1(\Omega _{-})}$
defined as in \eqref{weight-2p}, with $\Omega _{-}$ in place of ${\mathbb R}^n$ (cf., e.g., \cite{AGG1997}, \cite[Definition 1.1]{Al-Am}, \cite[Ch.1, Theorem 2.1]{Giroire1987}). This is a Hilbert space that can be also characterized as
\begin{equation}
\label{property}
\mathring{\mathcal H}^1(\Omega _{-})=\big\{v\in {{\mathcal H}^1}(\Omega _{-}):\gamma_{-}v=0 \mbox{ on } \partial \Omega \big\}
\end{equation}
(see \cite[(1.2)]{AGG1997}, \cite[Theorem 4.2]{B-M-M-M}).
{The space ${\mathcal D}(\Omega _-)$ is dense in $\mathring{\mathcal H}^1(\Omega _{-})$. Hence, the dual of $\mathring{\mathcal H}^1(\Omega _{-})$ denoted by ${\mathcal H}^{-1}(\Omega _{-})$ is a subspace of ${\mathcal D}'(\Omega _-)$.} {In addition, the semi-norm
\begin{align}
\label{seminorm}
|f|_{{\mathcal H}^1(\Omega _{-})}:=\|\nabla f\|_{L_2(\Omega _{-})^n}
\end{align}
is a norm on $\mathring{\mathcal H}^1(\Omega _{-})$ that is equivalent to the full norm $\|\cdot \|_{{\mathcal H}^1(\Omega _{-})}$ given by \eqref{weight-2p} with $\Omega _-$ in place of ${\mathbb R}^n$ (cf., e.g., \cite[Theorem 1.2]{AGG1997}, \cite[Theorem 1.2 (ii)]{Al-Am}).}

We need also the space $\widetilde{\mathcal H}^1(\Omega _{-})\!\subset \!{\mathcal H}^1(\mathbb R^n)$, defined as the closure of ${\mathcal D}({\Omega }_{-})$ in ${\mathcal H}^1(\mathbb R^n)$. This space can be also characterized as (see, e.g., \cite[(2.9)]{B-M-M-M})
\begin{align}
\widetilde{\mathcal H}^1(\Omega _{-})=\left\{u\in {\mathcal H}^1(\mathbb R^n):{\rm{supp}}\, {u}\subseteq \overline{\Omega }_{-}\right\},
\end{align}
and can be identified isomorphically with $\mathring{\mathcal H}^1(\Omega _{-})$ via the operator $\mathring E_{-}$ of extension by zero outside $\Omega _{-}$.

{By ${\mathcal H}^{\pm 1}({\mathbb R}^n)^n$ and ${\mathcal H}^{\pm 1}(\Omega _{-})^n$ we denote the spaces of vector-valued functions or distributions whose components belong to the spaces ${\mathcal H}^1({\mathbb R}^n)$ and ${\mathcal H}^1(\Omega _{-})$, respectively.}

\begin{rem}
\label{bounded-weight}
{The weighted Sobolev space ${\mathcal H}^{1}(\Omega _+)$ can be defined as in formula \eqref{weight-1} with $\Omega _{+}$ in place of ${\mathbb R}^n$. The dual of the space ${\mathcal H}^{1}(\Omega _+)$ is denoted by $\widetilde{\mathcal H}^{-1}(\Omega _+)$.
Let also $\mathring{\mathcal H}^{1}(\Omega _+)$ be the weighted space defined as the closure of the space ${\mathcal D}(\Omega _+)$ in ${\mathcal H}^{1}(\Omega _+)$, and let ${\mathcal H}^{-1}(\Omega _+)$ be its dual.
Since $\Omega _+$ is a bounded Lipschitz domain, we have that ${{\mathcal H}^{1}(\Omega _+)=H^{ 1}(\Omega _+)}$ and
${{\mathcal H}^{-1}(\Omega _+)=H^{-1}(\Omega _+)}$
$($with equivalent norms$)$}.
\end{rem}

\subsubsection{\bf Weighted Sobolev spaces on ${\mathbb R}^n\setminus \partial \Omega $}
We also consider the weighted space
\begin{align}
\label{Omega-pm}
{\mathcal H}^1({\mathbb R}^n\setminus \partial \Omega ):=
\left\{{ f\in L_{2}(\rho^{-1};{\mathbb R}^n) :
\nabla f\in L_2(\Omega_\pm)^n}
\right\},\quad { n\geq 3}\,.
\end{align}
This is a Hilbert space with the norm { defined by}
\begin{align}
\label{standard-weight-p}
\|{f}\|^2_{\mathcal H^1(\mathbb R^n\setminus\partial\Omega)}
=\|\rho ^{-1}{f}\|^2_{L_2({\mathbb R}^n)} +\|\nabla {f}\|^2_{L_2(\Omega_-)^{n}}+\|\nabla {f}\|^2_{L_2(\Omega_+)^{n}}\,,
\end{align}
which is equivalent to the norm $(\|f\|^2_{H^1(\Omega_+)}+\|f\|^2_{\mathcal H^1(\Omega_-)})^{1/2}$ on $\mathcal H^1(\mathbb R^n\setminus\partial\Omega)$.

{Note that \mbox{$f|_{\Omega_+}\in {H}^1(\Omega _+)$} and \mbox{$f|_{\Omega_-}\in {\mathcal H}^1(\Omega _{-})$}, whenever \mbox{$f\in \mathcal H^1(\mathbb R^n\setminus\partial\Omega)$}, and $f$ could have a {\it jump across $\partial \Omega $} denoted by $\left[{\gamma }(f)\right]:={\gamma }_{+}(f)-{\gamma }_{-}(f)={\gamma }_{+}(f_+)-{\gamma }_{-}(f_-)$, where $f_\pm :=f|_{\Omega _\pm }$.
However, if ${\bf f}\in {\mathcal H}^1({\mathbb R}^n\setminus \partial \Omega )$ and $\left[{\gamma }(f)\right]=0$ then ${\bf f}\in {\mathcal H}^1({\mathbb R}^n)$, and conversely, if ${\bf f}\in {\mathcal H}^1({\mathbb R}^n)$, then $\left[{\gamma }f)\right]=0$ (see {Lemma B1} and \cite[Theorem 5.13]{B-M-M-M}).}

We conclude the review of $L_2$-based weighted Sobolev spaces with the
asymptotic behaviour of functions from the spaces ${\mathcal H}^1({\mathbb
R}^n)$ and ${\mathcal H}^1(\Omega _-)$. Let $n\geq 3$ and $S^{n-1}$ be the
unit sphere in ${\mathbb R}^n$. Then any $u$ in ${\mathcal H}^1({\mathbb
R}^n)$ or in ${\mathcal H}^1(\Omega _{-})$ vanishes at infinity in the sense
of Leray, i.e.,
\begin{align}
\lim _{r\to \infty }\int _{S^{n-1}}|u(r{\bf y})|d\sigma _{\bf y}=0,
\end{align}
{\mn which can be deduced, e.g., from \cite[Lemma 1.1(i)]{Al-Am-1} for $n\geq
3$ (cf. also \cite[Lemma 2.1, Remark 2.4]{Amrouche-3} for $n=3$).}

\subsection{The conormal derivative for the Stokes system with $L_{\infty}$ viscosity coefficient tensor}
As above, $\boldsymbol{\mathbb L}$ is the divergence form second-order elliptic differential operator given by \eqref{Stokes-0}, and the coefficients $A^{\alpha \beta }$ of the anisotropic tensor ${\mathbb A}=\left(A^{\alpha \beta }\right)_{1\leq \alpha ,\beta \leq n}$ are $n\times n$ matrix-valued functions in $L_{\infty }({\mathbb R}^n)^{n\times n}$, with bounded measurable, real-valued entries $a_{ij}^{\alpha \beta }$, satisfying the symmetry and ellipticity conditions \eqref{Stokes-sym} and \eqref{mu}. Moreover, {${\boldsymbol{\mathcal L}}$ is the Stokes operator given by \eqref{Stokes-new}}.
Let $\boldsymbol\nu =(\nu _1,\ldots ,\nu _n)^\top$ denote the outward unit normal to $\Omega_{ +}$, which is defined a.e. on $\partial {\Omega }$.

In the special case when $({\bf u},\pi)\!\in \!C^1(\overline \Omega_{\pm})^n\!\times \!C^0(\overline\Omega_{\pm})$,
the {\em classical} interior and exterior conormal derivatives (i.e., the {\it boundary traction fields}) for the Stokes system
\begin{equation}
\label{Stokes-new-1}
\boldsymbol{\mathcal L}({\bf u},\pi )=\boldsymbol{\mathbb L}{\bf u}-\nabla \pi={\bf f},
\ \
{\rm{div}}\, {\bf u}=g\  \mbox{ in } \Omega _\pm ,
\end{equation}
where ${\bf f}\in L_2(\Omega _\pm )^n$, $g\in L_2(\Omega _\pm )$, are defined by the formula
\begin{align}
\label{2.37-}
{{\bf t}^{{\rm{c}}\pm}({\bf u},\pi ):=-{\gamma_\pm}{\pi}\, \boldsymbol\nu +\mathrm{T}^{{\rm{c}}\pm }{\bf u},}
\end{align}
where $\mathrm{T}^{{\rm{c}}\pm }{\bf u}$ are the conormal derivatives of ${\bf u}$ on $\partial \Omega $ associated with the operator ${\mathbb L}$ and defined by
\begin{align}
{\mathrm{T}^{{\rm{c}}\pm }{\bf u}:=\gamma _\pm (A^{\alpha \beta }\partial _\beta {\bf u})\nu _\alpha }
\end{align}
(cf., e.g., \cite{Choi-Dong-Kim-JMFM}). In view of \eqref{Stokes-sym}, we obtain that\footnote{Here and in the sequel, the notation $\pm $ applies
to the conormal derivatives from $\Omega _\pm $, respectively.}
\begin{align}\label{E2.28}
\left(\mathrm{T}^{{\rm{c}}\pm }{\bf u}\right)_i=\gamma _\pm \big(a_{ij}^{\alpha \beta }\partial _\beta u_j\big)\nu _\alpha
=a_{ij}^{\alpha \beta }E_{j\beta }({\bf u})\nu _\alpha \,,
\end{align}
where $E_{j\beta }({\bf u}):=\frac{1}{2}\left(\partial _j u_\beta +\partial _\beta u_j\right)$.%
\footnote{Note that another type of conormal derivative, where $E_{j\beta }({\bf u})$ is replaced by its deviator,
$D_{j\beta }({\bf u})=E_{j\beta }({\bf u})-\frac{1}{n}\delta_{j\beta}E_{mm}({\bf u})$ in the formulas like
\eqref{E2.28} and further on, has been considered in \cite{FP-Mik2019} for the isotropic case.
Both types of conormal derivatives coincide for incompressible fluids.}

{Note that for the isotropic case \eqref{isotropic},
the classical conormal derivatives ${\bf t}^{{\rm{c}}\pm}({\bf u},\pi )$ reduce to the well known formulas in the isotropic compressible case {(cf., e.g., Appendix III, Part I, Section 1 in \cite{Temam}),}
\begin{align}
\label{2.37-partiucular}
\left({\bf t}^{{\rm{c}}\pm}({\bf u},\pi )\right)_i=-{\gamma_\pm}{\pi}\, \nu_i
+\lambda ({\rm{div}}\, {\bf u})\nu_i+2\mu E_{i\alpha }({\bf u})\nu _\alpha ,\ \ i=1,\ldots ,n.
\end{align}

For the classical conormal derivatives defined by \eqref{2.37-}-\eqref{E2.28},} the {\it first Green formula}
\begin{align}
\label{special-v}
{{\pm}\left\langle {\bf t}^{{\rm{c}}\pm }({\bf u},\pi ),
\boldsymbol\varphi \right\rangle _{_{\!\partial\Omega  }}}
&={\left\langle a_{ij}^{\alpha \beta }E_{j\beta }({\bf u}),E_{i\alpha }(\boldsymbol\varphi )\right\rangle _{\Omega_{\pm}}-\langle \pi,{\rm{div}}\, \boldsymbol\varphi \rangle _{\Omega_{\pm}}
+\left\langle {\mathcal L}({\bf u},\pi ),\boldsymbol\varphi\right\rangle _{{\Omega_{\pm}}}}\quad \forall \ \boldsymbol\varphi \in {\mathcal D}({\mathbb R}^n)^n
\end{align}
holds and suggests the following definition of the {\it generalized conormal derivative} for the Stokes system with $L_{\infty}$ {viscosity coefficient tensor} in the setting of weighted Sobolev spaces (cf., e.g., \cite[Lemma 3.2]{Co}, \cite[Lemma 2.9]{K-L-M-W}, 
\cite[Definition 3.1, Theorem 3.2]{Mikh}, \cite[Theorem 10.4.1]{M-W}, see also \cite[Definition 2.4]{K-M-W-2}).

\begin{defn}
\label{conormal-derivative-var-Brinkman}
Let conditions {\eqref{Stokes-1} and \eqref{Stokes-sym}} hold.
For any $({\bf u}_{\pm},\pi_{\pm} ,{\tilde{\bf f}}_{\pm})\in {\mathcal H}^1({\Omega_{\pm}})^n\times L_2({\Omega_{\pm}})\times \widetilde{\mathcal H}^{-1}({\Omega_{\pm}})^n$,
the formal conormal derivatives ${\bf t}^{\pm}({\bf u}_{\pm},\pi_{\pm} ;{\tilde{\bf f}}_{\pm})\in H^{-\frac{1}{2}}(\partial\Omega )^n$ are defined in the weak form by
\begin{align}
\label{conormal-derivative-var-Brinkman-3}
\!\!\!\!\!{\pm}\left\langle {\bf t}^{\pm}({\bf u}_{\pm},\pi_{\pm} ;{\tilde{\bf f}}_{\pm}),{\boldsymbol\Phi }\right\rangle _{_{\!\partial\Omega  }}\!\!
:=\!\left\langle a_{ij}^{\alpha \beta }E_{j\beta }({\bf u_\pm}),
E_{i\alpha }(\gamma _\pm^{-1}\boldsymbol\Phi)\right\rangle _{\Omega_{\pm}}\!
-\!\left\langle {\pi_{\pm}},{\rm{div}}(\gamma^{-1}_{\pm}{\boldsymbol\Phi})\right\rangle _{\Omega_{\pm}}\!
+\!\left\langle {\tilde{\bf f}}_{\pm},\gamma^{-1}_{\pm}{\boldsymbol\Phi}\right\rangle _{{\Omega_{\pm}}},
\end{align}
for any $\boldsymbol\Phi\!\in\!H^{\frac{1}{2}}(\partial\Omega )^n$, where $\gamma^{-1}_{\pm}:H^{\frac{1}{2}}(\partial\Omega )^n\to {\mathcal H}^{1}({\Omega_{\pm}})^n$ are $($non-unique$)$ bounded right inverses to the trace operators
$\gamma_{\pm}:{\mathcal H}^{1}({\Omega_{\pm}})^n\to H^{\frac{1}{2}}(\partial\Omega )^n$.

Moreover, if $({\bf u}_{\pm},\pi_{\pm} ,{\tilde{\bf f}}_{\pm})\!\in \!{\pmb{\mathcal H}}^1({\Omega_{\pm}},{\boldsymbol{\mathcal L}})$, where
\begin{align}
\label{conormal-derivative-var-Brinkman-1}
{\pmb{\mathcal H}}^1({\Omega_{\pm}},\boldsymbol{\mathcal L})
:=\Big\{&{({\bf v}_{\pm},q_{\pm},{\tilde{\boldsymbol\phi}}_{\pm})}\in {\mathcal H}^1({\Omega_{\pm}})^n\times L_2({\Omega_{\pm}})\times \widetilde{\mathcal H}^{-1}({\Omega_{\pm}})^n: 
\boldsymbol{\mathcal L}({\bf v}_{\pm},q_{\pm})={\tilde{\boldsymbol\phi}}_{\pm}|_{\Omega_{\pm}}
\mbox{ in } {\Omega_{\pm}}\Big\},
\end{align}
then relations \eqref{conormal-derivative-var-Brinkman-3} define the generalized conormal derivatives
${\bf t}^{\pm}({\bf u}_{\pm},\pi_{\pm} ;{\tilde{\bf f}}_{\pm})\in H^{-\frac{1}{2}}(\partial\Omega )^n$.
\end{defn}

{ Some properties of the conormal derivatives are presented in the following assertion (cf. \cite{Co}, \cite[Theorem 5.3]{Mikh-3}, \cite[Lemma 2.9]{K-L-M-W}, \cite[Theorem 10.4.1]{M-W}).}
\begin{lem}
\label{lem-add1}
Let  conditions \eqref{Stokes-1} and \eqref{Stokes-sym} hold.
\begin{itemize}
\item[$(i)$] The formal conormal derivative operators
${\bf t}^{\pm }:{\mathcal H}^1({\Omega_{\pm}})^n\times L_2({\Omega_{\pm}})\times \widetilde{\mathcal H}^{-1}({\Omega_{\pm}})^n\to H^{-\frac{1}{2}}(\partial \Omega )^n$
are linear and bounded.
\item[$(ii)$]
The generalized conormal derivative operators ${\bf t}^{\pm }:\pmb{\mathcal H}^1(\Omega _{\pm },\boldsymbol{\mathcal L})\to H^{-\frac{1}{2}}(\partial \Omega )^n$ {with ${\boldsymbol{\mathcal L}}$ given by \eqref{Stokes}}, are linear and bounded, and {do not depend on the choice of the right inverse operators} $\gamma _{\pm }^{-1}$ in \eqref{conormal-derivative-var-Brinkman-3}.
In addition, for all ${\bf w}_\pm \!\in \!{\mathcal H}^{1}({\Omega_{\pm}})^n$ and $({\bf u}_{\pm},\pi_{\pm},{\tilde{\bf f}}_{\pm})\in {\pmb{\mathcal H}}^1({\Omega_{\pm}},{\boldsymbol{\mathcal L}})$,
the following Green formula holds
\begin{align}
\label{Green-particular-p}
{\pm}\big\langle {\bf t}^{\pm}({\bf u}_{\pm},\pi_{\pm};{\tilde{\bf f}}_{\pm}),\gamma_{\pm}{\bf w}_\pm\big\rangle _{_{\partial\Omega  }}
&={\left\langle a_{ij}^{\alpha \beta }E_{j\beta }({\bf u}_{\pm}),E_{i\alpha }({\bf w}_\pm)\right\rangle _{\Omega_{\pm}}-\langle {\pi_{\pm}},{\rm{div}}\, {\bf w}_\pm \rangle _{\Omega_{\pm}}+\langle {\tilde{\bf f}}_{\pm},{\bf w}_\pm \rangle _{{\Omega_{\pm}}}}\,.
\end{align}
\end{itemize}

\end{lem}
\begin{proof}
We use similar arguments to those { in} \cite[Lemma 2.2]{K-L-W1} (see also
\cite[Definition 3.1, Theorem 3.2]{Mikh}, \cite{Mikh-3}, 
\cite[Theorem 10.4.1]{M-W}).
First, we note that for $({\bf u}_{\pm},\pi_{\pm} ;{\tilde{\bf f}}_{\pm})\in {\mathcal H}^1({\Omega_{\pm}})^n\times L_2({\Omega_{\pm}})\times \widetilde{\mathcal H}^{-1}({\Omega_{\pm}})^n$, the right hand side in \eqref{conormal-derivative-var-Brinkman-3} defines a bounded linear functional acting on $\boldsymbol\Phi\in H^{\frac{1}{2}}(\partial\Omega )^n$, and, hence, the left hand side determines the formal conormal derivatives ${\bf t}^{\pm}({\bf u}_{\pm},\pi_{\pm} ;{\tilde{\bf f}}_{\pm})$ in $H^{-\frac{1}{2}}(\partial\Omega )^n$
and the formal conormal derivative operators ${\bf t}^{\pm }:{\mathcal H}^1({\Omega_{\pm}})^n\times L_2({\Omega_{\pm}})\times
\widetilde{\mathcal H}^{-1}({\Omega_{\pm}})^n \to H^{-\frac{1}{2}}(\partial\Omega )^n$
given by \eqref{conormal-derivative-var-Brinkman-3} are bounded. Therefore, the generalized conormal derivative operators ${\bf t}^{\pm }:\pmb{\mathcal H}^1(\Omega _{\pm },\boldsymbol{\mathcal L}) \to H^{-\frac{1}{2}}(\partial\Omega )^n$
are bounded as well.

Further, the property that the generalized conormal derivative operators
${\bf t}^{\pm }:\pmb{\mathcal H}^1(\Omega _{\pm },\boldsymbol{\mathcal L}) \to H^{-\frac{1}{2}}(\partial\Omega )^n$
defined by \eqref{conormal-derivative-var-Brinkman-3} are invariant with respect to the choice of a right inverse of the trace operator $\gamma _{\pm }:{\mathcal H}^{1}(\Omega _\pm )^n\to H^{\frac{1}{2}}(\partial \Omega )^n$ can be obtained with an argument similar to that for Theorem 3.2 in \cite{Mikh}.

Now let $({\bf u}_\pm ,\pi _\pm ,\tilde{\bf f}_\pm )\in \pmb{\mathcal H}^1(\Omega _{\pm },\boldsymbol{\mathcal L})$. According to formula \eqref{conormal-derivative-var-Brinkman-3}, we deduce the following equality
\begin{align}
\label{2.34}
\pm \left\langle {\bf t}^{\pm}({\bf u}_{\pm},\pi_{\pm} ;{\tilde{\bf f}}_{\pm}),\gamma_{\pm }{\bf w}_\pm \right\rangle _{_{\partial\Omega  }}&=
\left\langle A^{\alpha \beta }\partial _\beta ({\bf u}_{\pm}),\partial _\alpha \left(\gamma^{-1}_\pm (\gamma_{\pm }{\bf w}_\pm )\right)\right\rangle _{\Omega_{\pm}}\\
&\hspace{1em}-\left\langle {\pi_{\pm}},{\rm{div}}\left(\gamma^{-1}_\pm (\gamma_{\pm }{\bf w}_\pm )\right)\right\rangle _{\Omega_{\pm}}+\left\langle {\tilde{\bf f}}_{\pm},\gamma^{-1}_\pm (\gamma_{\pm }{\bf w}_\pm )\right\rangle _{{\Omega_{\pm}}}\nonumber\\
&=\left\langle A^{\alpha \beta }\partial _\beta ({\bf u}_\pm ),\partial _\alpha ({\bf w}_\pm )\right\rangle _{\Omega _\pm }-\big\langle \pi _\pm ,{\rm{div}\,}{\bf w}_\pm \big\rangle _{\Omega _\pm } +\left\langle \tilde{\bf f}_\pm ,{\bf w}_\pm \right\rangle _{\Omega _\pm }\nonumber\\
&\hspace{1em}+\left\langle A^{\alpha \beta }\partial _\beta ({\bf u}_\pm ),\partial _\alpha \left(\gamma^{-1}_\pm (\gamma_{\pm }{\bf w}_\pm )-{\bf w}_\pm \right)\right\rangle _{\Omega _\pm }\nonumber\\
&\hspace{1em}-\big\langle \pi _\pm ,{\rm{div}}\left(\gamma^{-1}_\pm (\gamma_{\pm }{\bf w}_\pm )-{\bf w}_\pm \right)\big\rangle _{\Omega _\pm }+\big\langle\tilde{\bf f}_\pm ,\gamma^{-1}_\pm \left(\gamma_{\pm }{\bf w}_\pm \right)-{\bf w}_\pm \big\rangle _{\Omega _\pm },\nonumber
\end{align}
for all ${\bf w}\in {\mathcal H}^{1}(\Omega _\pm )^n$. According to the property \eqref{property} and the equality $\gamma_\pm \left(\gamma^{-1}_\pm (\gamma_{\pm }{\bf w}_\pm )-{\bf w}_\pm \right)={\bf 0}$ on $\partial \Omega $, as well as the following equivalent description of the space $\mathring{\mathcal H}^{1}(\Omega _\pm )^n$,
\begin{align}
\label{null-trace}
\mathring{\mathcal H}^{1}(\Omega _\pm )^n=\left\{{\bf v}_\pm \in {\mathcal H}^{1}(\Omega _\pm )^n:\gamma_{\pm }{\bf v}_\pm ={\bf 0} \mbox{ on } \partial \Omega \right\}
\end{align}
(cf., e.g., \cite[(1.2)]{A-A}),
we obtain the membership relation
\begin{align}
\label{weight-trace}
\gamma^{-1}_\pm \left(\gamma_{\pm }{\bf w}_\pm \right)-{\bf w}_\pm \in \mathring{\mathcal H}^{1}(\Omega _\pm )^n\,.
\end{align}
Therefore, the Green formula \eqref{Green-particular-p} will follow from formula \eqref{2.34} if we show { that}
\begin{align}
\label{add-Green}
\left\langle A^{\alpha \beta }\partial _\beta ({\bf u}_\pm ),\partial _\alpha ({\bf v}_\pm )\right\rangle _{\partial \Omega }-\left\langle \pi _\pm ,{\rm{div}}\, {\bf v}_\pm \right\rangle _{\Omega _\pm }+\big\langle {\tilde{\bf f}}_\pm ,{\bf v}_\pm \big\rangle _{\Omega _\pm }=0\quad \forall \, {\bf v}_\pm \in \mathring{\mathcal H}^{1}({\Omega }_\pm )^n.
\end{align}
Since the space ${\mathcal D}(\Omega _\pm )^n$ is dense in $\mathring{\mathcal H}^{1}({\Omega }_\pm )^n$, we need to show identity \eqref{add-Green} only for the test functions ${\bf v}_\pm $ in ${\mathcal D}({\Omega }_\pm )^n$. Indeed, the membership of $({\bf u}_\pm ,\pi _\pm ,\tilde{\bf f}_\pm )$ in $\pmb{\mathcal H}^1(\Omega _{\pm },\boldsymbol{\mathcal L})$ implies the equality ${\mathcal L}({\bf u}_\pm ,\pi _\pm )\!=\!\tilde{\bf f}_{\pm }|_{\Omega _\pm }$ in the sense of distributions, and accordingly identity \eqref{add-Green} holds for any ${\bf v}_\pm \!\in \!{\mathcal D}(\Omega _\pm )^n$.

Finally, we note that {conditions \eqref{Stokes-sym} show that the second equality in \eqref{Green-particular-p} holds as well.}
\end{proof}
In the sequel we use the simplified notation ${\bf t}^{\pm}({\bf u}_{\pm},\pi_{\pm})$ for ${\bf t}^{\pm}({\bf u}_{\pm},\pi_{\pm};{\bf 0})$.

Let $\mathring{E}_\pm $ denote the operator of extension by zero outside $\Omega _\pm $. Thus, for a function $v_\pm $ from $\Omega _\pm $ to ${\mathbb{R}}$,
\begin{align}\label{ringE}
\mathring{E}_\pm (v_\pm )(x):=\left\{
\begin{array}{ll}
v_\pm (x)  & {\mathrm{if}}\ x\in \Omega _\pm \,,
\\
0  & {\mathrm{if}}\ x\in {\mathbb{R}}^{n}\setminus\Omega _\pm \,.
\end{array}
\right.
\end{align}
Let $\gamma $ be the trace operator from ${\mathcal H}^1({\mathbb R}^n)^n$ to $H^{\frac{1}{2}}(\partial \Omega )^n$.
For any
$({\bf u}_{\pm},\pi_{\pm},{\tilde{\bf f}}_{\pm})\in {\mathcal H}^1({\Omega_{\pm}})^n\times L_2({\Omega_{\pm}})\times
\widetilde{\mathcal H}^{-1}({\Omega_{\pm}})^n$,
let
\begin{align}
\label{u-pi-f}
&{\bf u}:=\mathring E_+{\bf u}_+ + \mathring E_-{\bf u}_-,\
\pi:=\mathring E_+\pi_+ + \mathring E_-\pi_-,\
\tilde{\bf f}:={\tilde{\bf f}}_+ + {\tilde{\bf f}}_-\,,
\end{align}
and the jump of the corresponding {formal or generalized} conormal derivatives is denoted by
\begin{align}
\label{jt}
&[{\bf t}({\bf u},\pi;\tilde{\bf f})]:=
\!{\bf t}^{+}({\bf u}_+,\pi_+;\tilde{\bf f}_+)\!-\!{\bf t}^{-}({\bf u}_-,\pi_-;\tilde{\bf f}_-).
\end{align}
{\rd Note that the inclusions ${\tilde{\bf f}}_{\pm}\in \widetilde{\mathcal H}^{-1}({\Omega_{\pm}})^n \subset {\mathcal H}^{-1}({\mathbb R}^n)^n$ imply that $\tilde{\bf f}={\tilde{\bf f}}_+ +{\tilde{\bf f}}_-$ belongs to the space ${\mathcal H}^{-1}({\mathbb R}^n)^n$.} 
In the special case $\tilde{\bf f}={\bf 0}$, we use the notation 
\begin{align}
\label{jt0}
&[{\bf t}({\bf u},\pi)]:=[{\bf t}({\bf u},\pi;{\bf 0})]
=\!{\bf t}^{+}({\bf u}_+,\pi_+)\!-\!{\bf t}^{-}({\bf u}_-,\pi_-).
\end{align}
Then Lemma \ref{lem-add1} implies the following assertion.
\begin{lem}
\label{lemma-add-new-1}
Let conditions \eqref{Stokes-1} and \eqref{Stokes-sym} hold.
For $({\bf u}_{\pm},\pi_{\pm},{\tilde{\bf f}}_{\pm})\in {\pmb{\mathcal H}}^1({\Omega_{\pm}},{\boldsymbol{\mathcal L}})$ given, let $({\bf u},\pi ,\tilde{\bf f})$ be defined as in \eqref{u-pi-f}.
Then the following identity holds for {any} ${\bf w}\!\in \!{\mathcal H}^{1}(\mathbb R^n)^n$,
\begin{align}
\label{Green-particular}
\!\!\!\!\!\big\langle [{\bf t}({\bf u},\pi;\tilde{\bf f})],\gamma{\bf w}\big\rangle _{_{\partial\Omega  }}
&\!=\!\left\langle a_{ij}^{\alpha \beta }E_{j\beta }({\bf u}_+),E_{i\alpha }({\bf w})\right\rangle _{\Omega_+}\!
{\mn+\left\langle a_{ij}^{\alpha \beta }E_{j\beta }({\bf u}_-),E_{i\alpha }({\bf w})\right\rangle _{\Omega_-}}\!
-\!\langle {\pi},{\rm{div}}\, {\bf w} \rangle_{\mathbb R^n}\!+\!\langle \tilde{\bf f},{\bf w} \rangle_{\mathbb R^n}.
\end{align}
\end{lem}
\begin{proof}
Note that $\gamma_+{\bf w}=\gamma_-{\bf w}=\gamma{\bf w}$ for any function ${\bf w}\in {\mathcal H}^{1}(\mathbb R^n)^n$. Then {formula \eqref{Green-particular-p} implies} the desired result.
\end{proof}
The following { assertion is immediately implied by} Lemma \ref{lemma-add-new-1} {\rd and the symmetry conditions \eqref{Stokes-sym}}.
\begin{lem}
\label{lemma-add-new}
Let  conditions \eqref{Stokes-1} and \eqref{Stokes-sym} hold.
Let the pair $({\bf u},\pi )$ in ${\mathcal H}^1({\mathbb R}^n\setminus \partial \Omega )^n\times L_2({\mathbb R}^n)$ be such that {$\boldsymbol{\mathcal L}({\bf u},\pi)\in L_2({\mathbb R}^n\setminus \partial \Omega )^n$} and ${\rm{div}}\, {\bf u}=0$ in ${\mathbb R}^n\setminus \partial \Omega $.
Let 
${\bf u}_\pm:=r_{\Omega_\pm}{\bf u}$,
$\pi_\pm:=r_{\Omega_\pm}\pi$,
$\tilde{\bf f}_\pm :=\mathring{E}_\pm r_{\Omega_{\pm}}\boldsymbol{\mathcal L}({\bf u},\pi)$,
and {$[{\bf t}({\bf u},\pi;{\bf f})]:={\bf t}^{+}({\bf u}_+,\pi_+;\tilde{\bf f}_+)\!-\!{\bf t}^{-}({\bf u}_-,\pi_-;\tilde{\bf f}_-)$.}
Then for all ${\bf w}\in {\mathcal H}^{1}(\mathbb R^n)^n$, the following formula holds
\begin{multline}
\label{jump-conormal-derivative-1}
\big\langle [{\bf t}({\bf u},\pi ;{\bf f})],\gamma{\bf w}\big\rangle _{_{\partial\Omega }}=
{\rd \left\langle A^{\alpha \beta }\partial _\beta{\bf u}_+,\partial _\alpha{\bf w}\right\rangle _{\Omega_+}
+\left\langle A^{\alpha \beta }\partial _\beta{\bf u}_-,\partial_\alpha{\bf w}\right\rangle _{\Omega_-}}
-\langle {\pi},{\rm{div}}\, {\bf w}\rangle_{\mathbb R^n}+\left\langle \boldsymbol{\mathcal L}({\bf u},\pi ),{\bf w}\right\rangle _{{\mathbb R}^n\setminus \partial \Omega}\\
=\left\langle a_{ij}^{\alpha \beta }E_{j\beta }({\bf u}_+),E_{i\alpha }({\bf w})\right\rangle _{\Omega_+}
+\left\langle a_{ij}^{\alpha \beta }E_{j\beta }({\bf u}_-),E_{i\alpha }({\bf w})\right\rangle _{\Omega_-}
\\
-\langle {\pi},{\rm{div}}\, {\bf w}\rangle_{\mathbb R^n}
+\left\langle \boldsymbol{\mathcal L}({\bf u},\pi ),{\bf w}\right\rangle _{{\mathbb R}^n\setminus \partial \Omega}\,.
\end{multline}
\end{lem}

\subsection{Conormal derivative of the adjoint Stokes system}
\label{adj-conormal}
Let $\boldsymbol{\mathbb L}$ be the divergence type elliptic operator given by \eqref{Stokes-0}. Then the formally adjoint $\boldsymbol{\mathbb L}^*$ of {the operator} $\boldsymbol{\mathbb L}$ is defined by
\begin{equation}
\label{Stokes-tr}
\boldsymbol{\mathbb L}^*{\bf u}
=\partial _\alpha\left(A^{*\alpha \beta}\partial_\beta{\bf u}\right)
:=\partial _\alpha\left(\left(A^{\beta \alpha }\right)^\top \partial _\beta {\bf u}\right),
\end{equation}
where ${A^{*\alpha \beta}=}\left(A^{\beta \alpha }\right)^\top$ is the transpose of the matrix $A^{\beta \alpha }$ for all $\alpha ,\beta =1,\ldots ,n$,
i.e.,
\begin{align}\label{2.45}
{A^{*\alpha \beta}
=\left(A^{\beta \alpha }\right)^\top
=\left(a^{*\alpha\beta}_{ij}\right)_{1\leq i,j\leq n}}
=\left(a^{\beta\alpha}_{ji}\right)_{1\leq i,j\leq n}.
\end{align}
Note that the coefficients of the operator $\boldsymbol{\mathbb L}^*$ {belong to $L_\infty (\Omega)^{n\times n}$ (cf. \eqref{Stokes-1})} and satisfy the 
ellipticity condition \eqref{mu} with the same constant $c_{\mathbb A}$.

If a pair $({\bf v},q)\in C^1(\overline \Omega_{\pm})^n\!\times \!C^0(\overline\Omega_{\pm})$ satisfies the formally adjoint Stokes system
\begin{equation}
\label{Stokes-new-adjoint}
\boldsymbol{\mathbb L}^*{\bf v}-\nabla q={\bf f}_*\ \
\mbox{ in } \Omega _\pm ,
\end{equation}
where ${\bf f}_*\in L_2(\Omega _\pm )^n$, then the corresponding classical conormal derivative is defined by
\begin{align}
\label{2.37-adj}
{\bf t}^{{\rm{c}}*\pm}({\bf v},q):=-{\gamma_\pm}q\boldsymbol\nu +\mathrm{T}^{{\rm{c}}*\pm }{\bf v},
\quad \mathrm{T}^{{\rm{c}}*\pm }{\bf v}
:=\gamma _\pm \left(\left(A^{\beta \alpha }\right)^\top \partial _\beta {\bf v}\right)\nu _\alpha .
\end{align}
If $({\bf v}_{\pm},q_{\pm},\tilde{\bf f}_{*\pm})\!\in \!{\mathcal H}^1({\Omega_{\pm}})^n\!\times \! L_2({\Omega_{\pm}})\!\times \!\widetilde{\mathcal H}^{-1}({\Omega_{\pm}})^n$ satisfies the following system (in distributional sense)
\begin{equation}
\label{Stokes-new-adj}
\boldsymbol{\mathcal L}_*({\bf v}_{\pm},q_{\pm})
:=\boldsymbol{\mathbb L}^*{\bf v}_\pm -\nabla q_\pm
={\tilde{\bf f}_{*\pm}}|_{\Omega _\pm}\ \
\mbox{ in } \Omega _\pm ,
\end{equation}
then we define the corresponding {\it {formal} conormal derivative}
${{\bf t}^{*\pm}}({\bf v}_{\pm},q_{\pm} ;\tilde{\bf f}_{*\pm})\!\in \! H^{-\frac{1}{2}}(\partial\Omega )^n$
by setting
\begin{align}
\label{conormal-derivative-var-Brinkman-3-adj}
\!\!\!\!\!{\pm}\left\langle {{\bf t}^{*\pm}}({\bf v}_{\pm},q_{\pm} ;\tilde{\bf f}_{*\pm}),
{\boldsymbol\Phi }\right\rangle _{_{\!\partial\Omega  }}&\!\!
:=\!\!{\left\langle (A^{\beta \alpha })^\top \partial _\alpha {\bf v}_{\pm},\partial _\beta (\gamma^{-1}_{\pm}{\boldsymbol\Phi})\right\rangle _{\Omega_{\pm}}}\!
-\!\left\langle {q_{\pm}},{\rm{div}}(\gamma^{-1}_{\pm}{\boldsymbol\Phi})\right\rangle _{\Omega_{\pm}}\!
+\!\left\langle \tilde{\bf f}_{*\pm},\gamma^{-1}_{\pm}{\boldsymbol\Phi}\right\rangle _{{\Omega_{\pm}}}\nonumber\\
&\!\!
=\!\left\langle A^{\alpha \beta }\partial _\beta (\gamma^{-1}_{\pm}{\boldsymbol\Phi}),
\partial_\alpha{\bf v}_{\pm}\right\rangle _{\Omega_{\pm}}\!
-\!\left\langle {q_{\pm}},{\rm{div}}(\gamma^{-1}_{\pm}{\boldsymbol\Phi})\right\rangle _{\Omega_{\pm}}\!
+\!\left\langle \tilde{\bf f}_{*\pm},\gamma^{-1}_{\pm}{\boldsymbol\Phi}\right\rangle _{{\Omega_{\pm}}},
\end{align}
for any $\boldsymbol\Phi\!\in\!H^{\frac{1}{2}}(\partial\Omega )^n$.
In addition, {if
$({\bf v}_{\pm},q_{\pm},\tilde{\bf f}_{*\pm})\!\in {\pmb{\mathcal H}}^1({\Omega_{\pm}},\boldsymbol{\mathcal L}_*)$
then ${\bf t}^{*\pm}({\bf v}_{\pm},q_{\pm} ;\tilde{\bf f}_{*\pm})$ becomes the generalized conormal derivative and}
an argument similar to that for \eqref{Green-particular-p} along with {relations \eqref{2.45}
imply} the Green formula
\begin{align}
\label{Green-particular-p-adj}
{\pm}\big\langle {{\bf t}^{*\pm}}({\bf v}_{\pm},q_{\pm};\tilde{\bf f}_{*\pm}),\gamma_{\pm}{\bf w}_\pm\big\rangle _{_{\partial\Omega  }}
&={\left\langle a_{ij}^{\alpha \beta }E_{j\beta }({\bf w}_{\pm}),E_{i\alpha }({\bf v}_\pm)\right\rangle _{\Omega_{\pm}}-\langle {q_{\pm}},{\rm{div}}\, {\bf w}_\pm \rangle _{\Omega_{\pm}}+\langle \tilde{\bf f}_{*\pm},{\bf w}_\pm \rangle _{{\Omega_{\pm}}}}\,,
\end{align}
and the following variant of Lemma~\ref{lemma-add-new}.
\begin{lem}
\label{lemma-add-new-adj}
Let  conditions \eqref{Stokes-1} and \eqref{Stokes-sym} hold.
Let the pair $({\bf v},q)$ in ${\mathcal H}^1({\mathbb R}^n\setminus \partial \Omega )^n\times L_2({\mathbb R}^n)$ be such that ${\boldsymbol{\mathcal L}_*({\bf v},q)}\in L_2({\mathbb R}^n\setminus \partial \Omega )^n$
in ${\mathbb R}^n\setminus \partial \Omega $. 
Let {${\bf v}_\pm:=r_{\Omega_\pm}{\bf v}$, $q_\pm:=r_{\Omega_\pm}q$,
$\tilde{\bf f}_{*\pm} :=\mathring{E}_\pm r_{\Omega_\pm}\boldsymbol{\mathcal L}_*({\bf v},q)$},
and
$[{\bf t}^*({\bf v},q;{\bf f}_*)]:={\bf t}^{*+}({\bf v}_+,q_+;\tilde{\bf f}_{*+})\!
-\!{\bf t}^{*-}({\bf v}_-,q_-;\tilde{\bf f}_{*-})$.
Then for any ${\bf w}\in {\mathcal H}^{1}(\mathbb R^n)^n$,
\begin{multline}
\label{jump-conormal-derivative-1-adj}
\big\langle [{\bf t}^*({\bf v},q;{\bf f}_*)],\gamma{\bf w}\big\rangle_{_{\partial\Omega }}
=\!\left\langle a_{ij}^{\alpha \beta }E_{j\beta }({\bf w}),E_{i\alpha }({\bf v}_+)\right\rangle _{\Omega_+}\!\!
+\!\left\langle a_{ij}^{\alpha \beta }E_{j\beta }({\bf w}),E_{i\alpha }({\bf v}_-)\right\rangle _{\Omega_-}\!
\\
-\!\langle {q},{\rm{div}}\, {\bf w}\rangle_{\mathbb R^n}\!
+\!\left\langle {\boldsymbol{\mathcal L}_*({\bf v},q )},{\bf w}\right\rangle_{{\mathbb R}^n\setminus\partial \Omega}\,.  
\end{multline}
\end{lem}
\subsection{Abstract mixed variational formulations}
\label{B-B-theory}
A major role in our analysis of mixed variational formulations is played by the following well-posedness result {\rd by Babu\u{s}ka \cite{Babuska} and Brezzi \cite[Theorem 1.1]{Brezzi}
(see also \cite[Theorem 2.34 {\mn and} Remark 2.35(i)]{Ern-Gu} and \cite{Brezzi-Fortin})}. 
\begin{thm}
\label{B-B}
Let $X$ and ${\mathcal M}$ be two real Hilbert spaces. Let $a(\cdot ,\cdot):X\times X\to {\mathbb R}$ and $b(\cdot ,\cdot):X\times {\mathcal M}\to {\mathbb R}$ be bounded bilinear forms. Let $f\in X'$ and $g\in {\mathcal M}'$. Let $V$ be the subspace of $X$ defined by
\begin{align}
\label{V}
V:=\left\{v\in X: b(v,q)=0\quad \forall \, q\in {\mathcal M}\right\}.
\end{align}
Assume that $a(\cdot ,\cdot ):V\times V\to {\mathbb R}$ is coercive, which means that there exists a constant $C_a>0$ such that
\begin{align}
\label{coercive}
a(w,w)\geq C_a^{-1}\|w\|_X^2\quad \forall \, w\in V,
\end{align}
and that $b(\cdot ,\cdot)\!:\!X\!\times \!{\mathcal M}\!\to \!{\mathbb R}$ satisfies
the {Babu\u{s}ka-Brezzi} condition
\begin{align}
\label{inf-sup-sm}
&\inf _{q\in {\mathcal M}\setminus \{0\}}\sup_{v\in X\setminus \{0\}}\frac{b(v,q)}{\|v\|_X\|q\|_{\mathcal M}}\geq C_b^{-1} \,,
\end{align}
with some constant $C_b>0$. Then the mixed variational {formulation}
\begin{equation}
\label{mixed-variational}
\left\{\begin{array}{ll}
a(u,v)+b(v,p)\hspace{-0.5em}&=f(v) \quad \forall \, v\in X,\\
b(u,q)&=g(q) \quad \forall \, q\in {\mathcal M},
\end{array}
\right.
\end{equation}
has a unique solution $(u,p)\in X\times {\mathcal M}$ and
\begin{align}
\label{mixed-Cu}
&\|u\|_{X}\leq C_a\|f\|_{X'}+C_b(1+\|a\|C_a)\|g\|_{\mathcal M'},\\
\label{mixed-Cp}
&\|p\|_{{\mathcal M}}\leq C_b(1+\|a\|C_a)\|f\|_{X'}+{\|a\|}C_b^{2}(1+\|a\|C_a)\|g\|_{\mathcal M'},
\end{align}
{where $\|a\|$ is the norm of the bilinear form $a(\cdot ,\cdot )$.}
\end{thm}

We need also the following {\mn extension in \cite[Theorem 4.2]{Amrouche-Seloula2013} (cf. also \cite[Lemma A.40]{Ern-Gu})} of the Babu\v{s}ka-Brezzi result.
\begin{lem}
\label{surj-inj-inf-sup}
Let $X$ and ${\mathcal M}$ be reflexive Banach spaces. Let $b(\cdot ,\cdot):X\times {\mathcal M}\to {\mathbb R}$ be a bounded bilinear form.
Let $B:X\to {\mathcal M}'$ and $B^{*}:{\mathcal M}\to X'$ be the linear bounded {operator and its transpose operator} defined by
\begin{align}
\label{B}
&\langle Bv,q\rangle =b(v,q),\ \langle v,B^*q\rangle =\langle Bv,q\rangle \quad \forall \, v\in X,\ \forall \, q\in {\mathcal M},
\end{align}
where $\langle \cdot ,\cdot \rangle :=\!_{X'}\langle \cdot ,\cdot \rangle _X$ denotes the duality pairing between the dual spaces $X'$ and $X$. The duality pairing between the spaces ${\mathcal M}'$ and ${\mathcal M}$ is also denoted by $\langle \cdot ,\cdot \rangle $. Let $V:={\rm{Ker}}\, B$ and $V^\perp:=\{g\in X':\langle g,v\rangle =0\ \forall \, v\in V\}$. Then the following assertions are equivalent:
\begin{itemize}
\item[$(i)$]
There exists a constant $C_b >0$ such that $b(\cdot ,\cdot)$ satisfies the inf-sup condition \eqref{inf-sup-sm}.
\item[$(ii)$]
The operator $B:{X/V}\to {\mathcal M}'$ is an isomorphism and
\begin{equation}
\label{bounded-below-1}
\|Bw\|_{{\mathcal M}'}\geq C_b^{-1} \|w\|_{X/V}\ \ \forall \, w\in X/V.
\end{equation}
\item[$(iii)$]
The operator $B^{*}:{\mathcal M}\to V^\perp$ is an isomorphism and
\begin{equation}
\label{bounded-below}
\|B^{*}q\|_{X'}\geq C_b^{-1} \|q\|_{\mathcal M}\ \ \forall \, q\in {\mathcal M}.
\end{equation}
\end{itemize}
\end{lem}

\section{Variational volume and layer potentials for the anisotropic Stokes system with $L_{\infty }$ coefficient tensor in $L_2$-based Sobolev spaces}
\label{N-S-D}

As in the previous sections, $\Omega _+ \subset {\mathbb R}^n$ ($n\geq 3)$ is a bounded Lipschitz domain with connected boundary $\partial \Omega $, and $\Omega _{-}\!:=\!{\mathbb R}^n\setminus \overline{\Omega_+}$. {Recall that $\boldsymbol{\mathcal L}$ is the Stokes operator defined in \eqref{Stokes-new}}. 
In this section we define the Newtonian and layer potentials for the Stokes system \eqref{Stokes} by means of a variational approach. 

\subsection{Bilinear forms and weak solutions for the Stokes system with $L_{\infty }$ coefficients in ${\mathbb R}^n$.}
Let ${\mathbb A}$ {satisfy} conditions \eqref{Stokes-1}-\eqref{mu}
and
$a_{{\mathbb A};{\mathbb R}^n}:{\mathcal H}^1({\mathbb R}^n)^n\times {\mathcal H}^{1}({\mathbb R}^n)^n\!\to \!{\mathbb R}$, $b_{{\mathbb R}^n}:{\mathcal H}^1({\mathbb R}^n)^n\times L_{2}({\mathbb R}^n)\!\to \!{\mathbb R}$ be the bilinear forms given by
\begin{align}
\label{a-v}
&a_{{\mathbb A};{\mathbb R}^n}({\bf u},{\bf v})
:=\left\langle A^{\alpha \beta }\partial_\beta{\bf u},\partial_\alpha{\bf v}\right\rangle _{{\mathbb R}^n}
={\left\langle a_{ij}^{\alpha \beta }E_{j\beta }({\bf u}),E_{i\alpha }({\bf v})\right\rangle _{{\mathbb R}^n}}\quad \forall \ {\bf u}\in {\mathcal H}^1({\mathbb R}^n)^n,\, {\bf v}\in {\mathcal H}^{1}({\mathbb R}^n)^n\,,\\
\label{b-v}
&b_{{\mathbb R}^n}({\bf v},q):=-\langle {\rm{div}}\, {\bf v},q\rangle _{{\mathbb R}^n}\quad \forall \ {\bf v}\in {\mathcal H}^1({\mathbb R}^n)^n\quad \forall \, q\in L_{2}({\mathbb R}^n)\,.
\end{align}

Let us denote
$$
{\mathcal H}_{\rm{div}}^1(\mathbb R^n)^n:=\left\{{\bf w}\in {\mathcal H}^1(\mathbb R^n)^n:{\rm{div}}\, {\bf w}=0 \mbox{ in }{\rd{\mathbb R}^n}\right\}.
$$
The subspace ${\mathcal H}^1_{\rm div}(\mathbb R^n)^n$ of ${\mathcal H}^1(\mathbb R^n)^n$ has also the characterization
\begin{align}\label{E3.4}
{\mathcal H}^1_{\rm div}(\mathbb R^n)^n
=\left\{{\bf w}\in {\mathcal H}^1({\mathbb R}^n)^n: b_{{\mathbb R}^n}({\bf w},q)\!=\!0\quad \forall \, q\in L_2({\mathbb R}^n) \right\}.
\end{align}

{An important} role in the forthcoming analysis is played by the following well-posedness result (see also \cite[Lemma 4.1]{K-M-W-1} for $p=2$, and \cite[Lemma 3.1]{K-M-W-2}).

\begin{lem}
\label{lemma-a47-1-Stokes}
Let conditions \eqref{Stokes-1}-\eqref{mu} hold on $\mathbb R^n$. Let $a_{{\mathbb A};{\mathbb R}^n}$ and $b_{{\mathbb R}^n}$ be the bilinear forms defined in \eqref{a-v} and \eqref{b-v}, respectively.
Then for all given data ${\bf F} \in {\mathcal H}^{-1}({\mathbb R}^n)^n$ and $\eta \in L_2({\mathbb R}^n)$, the mixed variational formulation
\begin{align}
\label{transmission-S-variational-dl-3-equiv-0-2}
\left\{\begin{array}{ll}
a_{{\mathbb A};{\mathbb R}^n}({\bf u},{\bf v})+b_{{\mathbb R}^n}({\bf v},\pi )
=\langle{\bf F},{\bf v}\rangle_{{\mathbb R}^n}\quad \forall \, {\bf v}\in {\mathcal H}^{1}({\mathbb R}^n)^n,\\
b_{{\mathbb R}^n}({\bf u},q)=\langle \eta ,q\rangle _{{\mathbb R}^n}\quad \forall \, q\in L_{2}({\mathbb R}^n)
\end{array}
\right.
\end{align}
is well-posed. Therefore, \eqref{transmission-S-variational-dl-3-equiv-0-2} has a unique solution $({\bf u},\pi )\in {{\mathcal H}^1({\mathbb R}^n)^n}\times L_2({\mathbb R}^n)$ and there exists a constant
$C=C(c_{\mathbb A},n)>0$ such that
\begin{align}
\label{estimate-1-wp-S-2}
\|{\bf u}\|_{{\mathcal H}^1({\mathbb R}^n)^n}+\|\pi \|_{L_2({\mathbb R}^n)}\leq
C\left\{\|{\bf F} \|_{{\mathcal H}^{-1}({\mathbb R}^n)^n}+\|\eta \|_{L_2({\mathbb R}^n)}\right\}.
\end{align}
\end{lem}
\begin{proof}
We intend to use Theorem \ref{B-B}, which requires the coercivity of the bilinear form $a_{{\mathbb A};{\mathbb R}^n}(\cdot ,\cdot )$ from ${\mathcal H}^{1}({\mathbb R}^n)^n\times {\mathcal H}^{1}({\mathbb R}^n)^n$ to ${\mathbb R}$.
Indeed, the following Korn type inequality for functions in $\mathcal H^1({\mathbb R}^n)^n$ holds,
\begin{align}
\label{Korn3-R3}
\|\nabla {\bf w}\|^2_{L_2({\mathbb R}^n)^{n\times n}}\leq 2\|\mathbb E ({\bf w})\|^2_{L_2({\mathbb R}^n)^{n\times n}}
\end{align}
(cf. \cite[{Eq.}(2.2)]{Sa-Se} for $n=3$, see also 
{\rd \cite[Theorem 10.1]{Lean}}).
Note that if ${\bf w}\!\in \!{\mathcal H}^1_{\rm div}({\mathbb R}^n)^n$, then $\sum_{i=1}^nE_{ii}({\bf w})=0$.
Then the ellipticity condition \eqref{mu}, inequality \eqref{Korn3-R3} and  equivalence of the semi-norm $\|\nabla (\cdot )\|_{L_2({\mathbb R}^n)^{n\times n}}$
to the norm $\|\cdot \|_{\mathcal H^1({\mathbb R}^n)^n}$ in $\mathcal H^1({\mathbb R}^n)^n$ (see {Section~\ref{2.2.1}})
imply that there exists a constant $c_1=c_1(n)>0$ such that
\begin{align}
\label{a-1-v2-S}
\!\!\!\!a_{{\mathbb A};{\mathbb R}^n}({\bf w},{\bf w})
\!\geq c_{\mathbb A}^{-1}\|{\mathbb E}({\bf w})\|_{L_2({\mathbb R}^n)^{n\times n}}^2
\!\geq\frac{1}{2}c_{\mathbb A}^{-1}\|\nabla {\bf w}\|_{L_2({\mathbb R}^n)^{n\times n}}^2
\!\geq\frac{1}{2}c_{\mathbb A}^{-1}c_1\|{\bf w}\|_{{\mathcal H}^1({\mathbb R}^n)^n}^2\
\forall \, {\bf w}\!\in \!\mathcal H^1_{\rm div}({\mathbb R}^n)^n.
\end{align}
Inequality \eqref{a-1-v2-S} shows that the bilinear form
$a_{\mathbb A;\mathbb R^n}(\cdot ,\cdot ):\mathcal H^1_{\rm div}({\mathbb R}^n)^n\times \mathcal H^1_{\rm div}({\mathbb R}^n)^n\to {\mathbb R}$ is coercive.

The continuity of the operator $\nabla :{\mathcal H}^1({\mathbb R}^n)^n\to L_2({\mathbb R}^n)^{n\times n}$
{and} the H\"{o}lder inequality imply that for the constant
$\mathcal C=n^4\|{\mathbb A}\|_{L_\infty ({\mathbb R}^n)}$,
\begin{align}
\label{a-1-v}
|a_{{\mathbb A};{\mathbb R}^n}({\bf u},{\bf v})|
\leq {\mathcal C}\|{\bf u}\|_{{\mathcal H}^1({\mathbb R}^n)^n}\|{\bf v}\|_{{\mathcal H}^{1}({\mathbb R}^n)^n}\quad \forall\, {\bf u},{\bf v}\in \mathcal H^{1}({\mathbb R}^n)^n.
\end{align}
Thus, the bilinear form $a_{{\mathbb A};{\mathbb R}^n}(\cdot ,\cdot ):{\mathcal H}^1({\mathbb R}^n)^n\times {\mathcal H}^{1}({\mathbb R}^n)^n\to {\mathbb R}$ is bounded. Moreover, the boundedness of the divergence operator
${\rm{div}}:{\mathcal H}^1({\mathbb R}^n)^n\to L_2({\mathbb R}^n)$
implies that the bilinear form $b:{\mathcal H}^1({\mathbb R}^n)^n\times L_2({\mathbb R}^n)\to {\mathbb R}$ is bounded as well.

The isomorphism property of the divergence operator
\begin{align}
\label{div-S}
-{\rm{div}}:{\mathcal H}^1({\mathbb R}^n)^n/{\mathcal H}^1_{\rm div}(\mathbb R^n)^n\to L_2(\mathbb R^n)
\end{align}
(cf. \cite[Proposition 2.1]{Al-Am-1}, \cite[Lemma 2.5]{Kozono-Shor}) implies that
there exists a constant $c_0>0$ such that for any $q\in L_2({\mathbb R}^n)$ there exists ${\bf v}\in {\mathcal H}^1({\mathbb R}^n)^n$ satisfying the equation $-{\rm{div}}\, {\bf v}=q$ and the inequality
$\|\mathbf v\|_{{\mathcal H}^1({\mathbb R}^n)^n}\leq {c_0}\|q\|_{L_2({{\mathbb R}^n})}$. Therefore, the following inequality holds {for such $\bf v$,}
\begin{align*}
b_{\mathbb R^n}({\bf v},q)=-\left\langle {\rm{div}}\, {\bf v},q\right\rangle _{{\mathbb R}^n}
=\langle q,q\rangle _{{\mathbb R}^n}=\|q\|_{L_2({\mathbb R}^n)}^2
\geq c_0^{-1}\|{\bf v}\|_{{\mathcal H}^1({\mathbb R}^n)^n}\|q\|_{L_2({\mathbb R}^n)}.
\end{align*}
{This implies} that the bounded bilinear form $b_{{\mathbb R}^n}:{\mathcal H}^1({\mathbb R}^n)^n\times L_2({\mathbb R}^n)\to {\mathbb R}$ satisfies the inf-sup condition
\begin{align}
\label{inf-sup}
\inf _{q\in L_2({\mathbb R}^n)\setminus \{0\}}\sup _{{\bf w}\in {\mathcal H}^1({\mathbb R}^n)^n\setminus \{\bf 0\}}\frac{b_{{\mathbb R}^n}({\bf w},q)}{\|{\bf w}\|_{{\mathcal H}^1({\mathbb R}^n)^n}\|q\|_{L_2({\mathbb R}^n)}}\geq c_0^{-1}
\end{align}
(see also Lemma \ref{surj-inj-inf-sup}(ii), and \cite[Proposition 2.4]{Sa-Se} for $n=2,3$).
{Then Theorem \ref{B-B} with $X\!=\!{\mathcal H}^1({\mathbb R}^n)^n$, ${\mathcal M}\!=\!L_2({\mathbb R}^n)$, and $V\!\!={\mathcal H}^1_{\rm{div}}(\mathbb R^n)^n$ implies} that problem \eqref{transmission-S-variational-dl-3-equiv-0-2} is well-posed, as asserted.
\end{proof}

\subsection{Volume potential operators for the anisotropic Stokes system with $L_{\infty }$ coefficient tensor}
\label{layer-potentials}
Recall that $\boldsymbol{\mathcal L}$ is the anisotropic Stokes operator defined in
{\eqref{Stokes-new}.}

\begin{thm}
\label{Brinkman-problem-p}
Let conditions \eqref{Stokes-1}-\eqref{mu} hold in $\mathbb R^n$.
Then for each ${\bf f} \in {\mathcal H}^{-1}({\mathbb R}^n)^n$ and $g\in L_{2}(\mathbb R^n)$
the anisotropic Stokes system 
\begin{eqnarray}
\label{Newtonian-S-p}
\boldsymbol{\mathcal L}({\bf u},\pi)={\bf f}, \quad
{\rm{div}}\, {\bf u}=g & \mbox{ in } {\mathbb R}^n\,,
\end{eqnarray}
is well-{posed}, which means that \eqref{Newtonian-S-p} has a unique solution
$({\bf u},\pi)\!\in \!{\mathcal H}^1({\mathbb R}^n)^n\!\times \!L_2({\mathbb R}^n)$ and there exists a constant
$C\!=\!C(c_{\mathbb A},n)\!>\!0$ such that
\begin{align}
\label{estimate-1-wp-B1}
\|{\bf u}\|_{{\mathcal H}^1({\mathbb R}^n)^n}+ \|\pi\|_{L_2({\mathbb R}^n)}
\leq C\left(\|{\bf f} \|_{{\mathcal H}^{-1}({\mathbb R}^n)^n}+\|g\|_{L_{2}(\mathbb R^n)}\right)\,.
\end{align}
\end{thm}
\begin{proof}
The dense embedding of the space ${\mathcal D}({\mathbb R}^n)^n$ in ${\mathcal H}^{1}({\mathbb R}^n)^n$ shows that system \eqref{Newtonian-S-p} has the equivalent mixed variational formulation \eqref{transmission-S-variational-dl-3-equiv-0-2}, with ${\bf F} =-{\bf f}$ and $\eta =-g$.
Then the well-posedness of the Stokes system \eqref{Newtonian-S-p} follows from Lemma \ref{lemma-a47-1-Stokes}.
\end{proof}
Theorem \ref{Brinkman-problem-p} {allows us to define
the volume potential operators for the Stokes system with $L_\infty $ coefficients and
obtain their continuity as follows.}
\begin{defn}
\label{Newtonian-B-var-variable-1}
Let conditions \eqref{Stokes-1}-\eqref{mu} hold.
\begin{enumerate}
\item[(i)]
The Newtonian velocity and pressure potential operators
\begin{align}
\label{Newtonian-S-var-2}
\boldsymbol{\mathcal N}_{{{\mathbb R}^n}}:{\mathcal H}^{-1}({\mathbb R}^n)^n\to {\mathcal H}^1({\mathbb R}^n)^n,\
{\mathcal Q}_{{{\mathbb R}^n}}:{\mathcal H}^{-1}({\mathbb R}^n)^n\to L_2({\mathbb R}^n),
\end{align}
are defined as
\begin{align}
\label{Newtonian-S-var}
\boldsymbol{\mathcal N}_{\mathbb R^n}{\bf f}:={\bf u}_{\bf f},\
\mathcal Q_{\mathbb R^n}{\bf f}:=\pi _{\bf f}
\quad \forall\ {\bf f}\in {\mathcal H}^{-1}({\mathbb R}^n)^n,
\end{align}
where $({\bf u}_{{\bf f} },\pi _{{\bf f}})\in {\mathcal H}^1({\mathbb R}^n)^n\times L_2({\mathbb R}^n)$ is the unique solution of problem \eqref{Newtonian-S-p} with ${\bf f}\in {\mathcal H}^{-1}({\mathbb R}^n)^n$ and $g=0$.
\item[(ii)]
The velocity and pressure compressibility  potential operators
\begin{align}
\label{Newtonian-S-var-2c}
\boldsymbol{\mathcal G}_{\mathbb R^n}:L_2({\mathbb R}^n)\to {\mathcal H}^1({\mathbb R}^n)^n,\
{\mathcal G}^0_{\mathbb R^n}:L_2({\mathbb R}^n)\to L_2({\mathbb R}^n),
\end{align}
are defined as
\begin{align}
\label{Newtonian-S-varc}
\boldsymbol{\mathcal G}_{\mathbb R^n}{g}:={\bf u}_{g},\
\mathcal G^0_{\mathbb R^n}{g}:=\pi _{g}
\quad \forall\ g\in L_2(\mathbb R^n),
\end{align}
where $({\bf u}_g,\pi_g)\in {\mathcal H}^1({\mathbb R}^n)^n\times L_2({\mathbb R}^n)$ is the unique solution of problem \eqref{Newtonian-S-p} with $g\in L_2({\mathbb R}^n)$ and ${\bf f}={\bf 0}$.
\end{enumerate}
\end{defn}
\begin{lem}
\label{Newtonian-B-var-1}
Operators \eqref{Newtonian-S-var-2} and \eqref{Newtonian-S-var-2c}
are linear and continuous and for any ${\bf f}\in{\mathcal H}^{-1}(\mathbb R^n)^n$ and $g\in L_2(\mathbb R^n)^n$,
\begin{align*}
&\boldsymbol{\mathcal L}(\boldsymbol{\mathcal N}_{\mathbb R^n}{\bf f},\mathcal Q_{\mathbb R^n}{\bf f})={\bf f},\quad
{\rm{div}}\, \boldsymbol{\mathcal N}_{\mathbb R^n}{\bf f}=0 \quad \mbox{in } {\mathbb R}^n,\\
&\boldsymbol{\mathcal L}(\boldsymbol{\mathcal G}_{\mathbb R^n}g,\mathcal G^0_{\mathbb R^n}g)={\bf 0},\quad
{\rm{div}}\, \boldsymbol{\mathcal G}_{\mathbb R^n}g=g \quad \mbox{in } {\mathbb R}^n.
\end{align*}
\end{lem}
\subsection{The single layer potential operator for the Stokes system with $L_{\infty }$ coefficients}

Recall that $\Omega _{+}\subset {\mathbb R}^n$ $($$n\geq 3$$)$ is a bounded Lipschitz domain with connected boundary $\partial \Omega $, {$\Omega _{-}:={\mathbb R}^n\setminus\overline{\Omega_+}$,
the notation $[ \cdot ]$ is used for jumps}
(see formulas \eqref{jump-notation} and \eqref{jt}{-\eqref{jt0}}),
{and $\boldsymbol{\mathcal L}$ is the anisotropic Stokes operator defined in
\eqref{Stokes-new}.}

The next well-posedness result for the transmission problem \eqref{var-Brinkman-transmission-sl} plays a {major} role in the definition of the single layer potentials for the $L_{\infty }$ coefficient Stokes system in the Sobolev space $H^{-\frac{1}{2}}(\partial \Omega )^n$
(see also {\cite[Theorem 3.5, Definition 3.7, Lemma 3.8]{K-M-W-2} for the Stokes system with strongly elliptic coefficient tensor; \cite[Theorem 4.5]{K-L-W1}, \cite[Section 5]{Sa-Se}, \cite[Section 2]{B-H}  and \cite[Theorem 10.5.3]{M-W} for the isotropic case} \eqref{isotropic} with $\mu =1$ {and $\lambda=0$}).

\begin{thm}
\label{slp-var-apr-1-p}
Let conditions \eqref{Stokes-1}-\eqref{mu} hold in $\mathbb R^n$.
Then for any ${\boldsymbol\psi \in H^{-\frac{1}{2}}(\partial \Omega )^n}$ the transmission problem
\begin{eqnarray}
\label{var-Brinkman-transmission-sl}
\left\{
\begin{array}{ll}
{\boldsymbol{\mathcal L}({\bf u}_{\boldsymbol\psi },\pi _{\boldsymbol\psi })={\bf 0}}\,, \ \  
{\rm{div}}\, {\bf u}_{\boldsymbol\psi }=0 & \mbox{ in } {\mathbb R}^n\setminus \partial \Omega ,\
\\
\left[{\gamma }{\bf u}_{{\boldsymbol\psi}}\right]={\bf 0} & \mbox{ on } \partial \Omega ,\\
\left[{\bf t}({\bf u}_{{\boldsymbol\psi}},\pi _{{\boldsymbol\psi}})\right]
=\boldsymbol\psi & \mbox{ on } \partial \Omega ,
\end{array}\right.
\end{eqnarray}
has a unique solution $({\bf u}_{{\boldsymbol\psi }},\pi _{\boldsymbol\psi })\in {\mathcal H}^1({\mathbb R}^n\setminus \partial \Omega )^n\times L_{2}({\mathbb R}^n)$, and there exists a constant
$C=C(\partial \Omega ,c_{\mathbb A},n)>0$ such that
\begin{align}
\label{estimate-4-var}
\|{\bf u}_{\boldsymbol\psi}\|_{{\mathcal H}^1({\mathbb R}^n)^n}+\|\pi _{\boldsymbol\psi }\|_{L_{2}({\mathbb R}^n)}\leq C\|\boldsymbol\psi\|_{H^{-\frac{1}{2}}(\partial \Omega )^n}.
\end{align}
\end{thm}
\begin{proof}
Transmission problem \eqref{var-Brinkman-transmission-sl} has the following equivalent mixed variational formulation.

{\it Find $({\bf u}_{\boldsymbol\psi},\pi _{{\boldsymbol\psi }})\in {\mathcal H}^1({\mathbb R}^n)^n\times L_{2}({\mathbb R}^n)$ such that}
\begin{equation}
\label{single-layer-S-transmission}
\left\{\begin{array}{lll}
{\left\langle a_{ij}^{\alpha \beta }E_{j\beta }({\bf u}),E_{i\alpha }({\bf v})\right\rangle _{{\mathbb R}^n}}-\langle \pi _{\boldsymbol\psi},{\rm{div}}\, {\bf v}\rangle _{{{\mathbb R}^n}}
=\langle \boldsymbol\psi ,\gamma {\bf v}\rangle _{\partial \Omega }\quad \forall \, {\bf v}\in {\mathcal H}^{1}({\mathbb R}^n)^n,\\
\langle {\rm{div}}\, {\bf u}_{{\boldsymbol\psi}},q\rangle _{{\mathbb R}^n}=0\quad \forall \, q\in L_{2}({\mathbb R}^n)\,.
\end{array}\right.
\end{equation}
To show this equivalence, let us first assume that
$({\bf u}_{{\boldsymbol\psi }},\pi _{\boldsymbol\psi })\in
{\mathcal H}^1({\mathbb R}^n\setminus \partial \Omega )^n\times {L_{2}({\mathbb R}^n)}$
satisfy the transmission problem \eqref{var-Brinkman-transmission-sl}. Then the first transmission condition in \eqref{var-Brinkman-transmission-sl} implies the membership of ${\bf u}_{{\boldsymbol\psi }}$ in ${\mathcal H}^1({\mathbb R}^n)^n${, cf. Lemma~\ref{extention}(ii)}. Moreover, formula \eqref{jump-conormal-derivative-1} shows that the the pair $({\bf u}_{{\boldsymbol\psi }},\pi _{\boldsymbol\psi })$ satisfies also the first equation in \eqref{single-layer-S-transmission}. The second equation in \eqref{single-layer-S-transmission} follows from the second equation in \eqref{var-Brinkman-transmission-sl}.

Let us show the converse property. To this end, we assume that the pair $({\bf u}_{{\boldsymbol\psi }},\pi _{\boldsymbol\psi })\in {\mathcal H}^1({\mathbb R}^n)^n\times L_{2}({\mathbb R}^n)$ is a solution of the variational problem \eqref{single-layer-S-transmission}. By using the density of the space ${\mathcal D}({\mathbb R}^n)^n$ in ${\mathcal H}^{1}({\mathbb R}^n)^n$, and by considering in the first equation of \eqref{single-layer-S-transmission} any ${\bf v}\in C^{\infty }({\mathbb R}^n)^n$ with compact support in $\Omega _\pm $, we obtain the following variational equation
\begin{align*}
\left\langle\partial _\alpha\left(a_{ij}^{\alpha \beta }E_{j\beta }({\bf u}_{\boldsymbol\psi })\right)-\partial _i\pi _{\boldsymbol\psi },w_i\right\rangle _{\Omega _\pm}=0\quad \forall \, {\bf w}\in C_0^{\infty }(\Omega _\pm )^n\,,
\end{align*}
which leads to the first equation of the transmission problem \eqref{var-Brinkman-transmission-sl}. The second equation in \eqref{var-Brinkman-transmission-sl} is an immediate consequence of the second equation in \eqref{single-layer-S-transmission}.
On the other hand, the membership of ${\bf u}_{\boldsymbol\psi }$ in ${\mathcal H}^1({\mathbb R}^n)^n$ yields the first transmission condition in \eqref{var-Brinkman-transmission-sl}. In addition,  formula \eqref{jump-conormal-derivative-1} and the first equation in \eqref{single-layer-S-transmission} show that
\begin{align}
\label{j1-0}
\left\langle [{\bf t}({\bf u}_{\boldsymbol\psi },\pi _{\boldsymbol\psi })],\gamma {\bf v}\right\rangle _{\partial \Omega }=\langle \boldsymbol\psi ,\gamma {\bf v}\rangle _{\partial \Omega }\ \forall \, {\bf v}\in {\mathcal H}^{1}({\mathbb R}^n)^n\,.
\end{align}
Since {the trace operator $\gamma :{\mathcal H}^{1}({\mathbb R}^n)^n\to H^{\frac{1}{2}}(\partial \Omega )^n$ is surjective} (see Theorem \ref{trace-operator1}), equation \eqref{j1-0} can be written in the form $\left\langle [{\bf t}({\bf u}_{\boldsymbol\psi },\pi _{\boldsymbol\psi })]-\boldsymbol \psi ,\boldsymbol\Phi\right\rangle _{\partial \Omega }=0$ for any $\boldsymbol\Phi\in H^{\frac{1}{2}}(\partial \Omega )^n$, which implies the second transmission condition in \eqref{var-Brinkman-transmission-sl}. Thus, the transmission problem \eqref{var-Brinkman-transmission-sl} has the equivalent mixed variational formulation \eqref{single-layer-S-transmission}, which can be
written as
\begin{align}
\label{single-layer-transmission-sm-2v}
\left\{\begin{array}{ll}
a_{{\mathbb A};{\mathbb R}^n}({\bf u}_{\boldsymbol\psi },{\bf v})
+b_{{\mathbb R}^n}({\bf v},\pi _{\boldsymbol\psi })
={\langle{\bf F},{\bf v}\rangle_{{\mathbb R}^n}} & \forall \, {\bf v}\in {\mathcal H}^{1}({\mathbb R}^n)^n,\\
b_{{\mathbb R}^n}({\bf u}_{\boldsymbol\psi },q)=0 & \forall \, q\in L_{2}({\mathbb R}^n),
\end{array}
\right.
\end{align}
where $a_{{\mathbb A};{\mathbb R}^n}$ and $b_{{\mathbb R}^n}$ are the bounded bilinear forms given by \eqref{a-v} and \eqref{b-v}, and {${\bf F}\in{\mathcal H}^{-1}({\mathbb R}^n)^n$ is defined as}
\begin{align}
{\langle{\bf F},{\bf v}\rangle_{{\mathbb R}^n}}:=\langle \boldsymbol\psi ,\gamma {\bf v}\rangle _{\partial \Omega }=\langle \gamma ^*\boldsymbol\psi ,{\bf v}\rangle _{{\mathbb R}^n}\quad
\forall \, {\bf v}\in {\mathcal H^{1}}({\mathbb R}^n)^n\,,
\end{align}
where $\gamma ^*:H^{-\frac{1}{2}}(\partial \Omega )^n\to {\mathcal H}^{-1}({\mathbb R}^n)^n$ is the adjoint of the trace operator $\gamma :{\mathcal H}^{1}({\mathbb R}^n)^n\to H^{\frac{1}{2}}(\partial \Omega )^n$.
Then by Lemma \ref{lemma-a47-1-Stokes} 
the variational problem \eqref{single-layer-S-transmission} is well-posed.
Therefore, problem \eqref{var-Brinkman-transmission-sl} has a unique solution $({\bf u}_{{\boldsymbol\psi }},\pi _{\boldsymbol\psi })\in {\mathcal H}^1({\mathbb R}^n)^n\times L_2({\mathbb R}^n)$, which depends continuously on $\boldsymbol\psi$.
\end{proof}
Theorem \ref{slp-var-apr-1-p} allows to define the single-layer potentials for $L_{\infty }$ coefficient Stokes system and allows to obtain their continuity.

\begin{defn}
\label{s-l-S-variational-variable}
Let conditions \eqref{Stokes-1}-\eqref{mu} hold.
The single layer velocity and pressure potentials
\begin{align}
\label{s-l-S-v1-var}
&{\bf V}_{\partial\Omega}:H^{-\frac{1}{2}}(\partial \Omega )^n\to {\mathcal H}^1({\mathbb R}^n)^n{,\quad}
{\mathcal Q}_{\partial\Omega}^s:H^{-\frac{1}{2}}(\partial \Omega )^n\to L_2({\mathbb R}^n),
\end{align}
are defined as
\begin{align}
\label{slp-S-vp-var}
{\bf V}_{\partial\Omega}\boldsymbol\psi:={\bf u}_{{\boldsymbol\psi}},\quad
\mathcal Q^s_{\partial\Omega}\boldsymbol\psi:=\pi _{{\boldsymbol\psi}}
\quad \forall\ {\boldsymbol\psi \in H^{-\frac{1}{2}}(\partial \Omega )^n},
\end{align}
and the {boundary operators}
\begin{align}
\label{s-l-S-v2-var}
&\boldsymbol{\mathcal V}_{{\partial \Omega }}:H^{-\frac{1}{2}}(\partial \Omega )^n
\to H^{\frac{1}{2}}(\partial \Omega )^n{,\quad}
{\boldsymbol{\mathcal K}_{\partial\Omega}}:H^{-\frac{1}{2}}(\partial \Omega )^n\to H^{-\frac{1}{2}}(\partial \Omega )^n.
\end{align}
are defined as
\begin{align}
\label{slp-S-oper-var}
{\boldsymbol{\mathcal V}}_{\partial \Omega }{\boldsymbol\psi}:=\gamma{\bf u}_{\boldsymbol\psi},\quad
{\boldsymbol{\mathcal K}_{\partial\Omega}}\boldsymbol\psi:=
\frac{1}{2}\left(
{\bf t}^+({\bf u}_{\boldsymbol\psi},\pi_{\boldsymbol\psi})
+{\bf t}^-({\bf u}_{\boldsymbol\psi},\pi_{\boldsymbol\psi})\right)
\quad \forall\ \boldsymbol\psi \in H^{-\frac{1}{2}}(\partial \Omega )^n,
\end{align}
where $({\bf u}_{{\boldsymbol\psi}},\pi _{{\boldsymbol\psi}})$ is the unique solution of the transmission problem \eqref{var-Brinkman-transmission-sl} in ${\mathcal H}^1({\mathbb R}^n)^n\times L_2({\mathbb R}^n)$.
\end{defn}

\begin{lem}
\label{continuity-sl-S-h-var}
Operators \eqref{s-l-S-v1-var} and \eqref{s-l-S-v2-var} are linear and continuous and for any
$\boldsymbol\psi \in H^{-\frac{1}{2}}(\partial\Omega)^n$,
$$
\boldsymbol{\mathcal L}({\bf V}_{\partial\Omega}\boldsymbol\psi,
\mathcal Q^s_{\partial\Omega}\boldsymbol\psi)={\bf 0},\quad
{\rm{div}}\,{\bf V}_{\partial\Omega}\boldsymbol\psi=0 \quad \mbox{in } \Omega_\pm.
$$
\end{lem}

In addition, the following jump relations,
{that are similar} to the case of the Stokes system with constant coefficients (see also \cite[Lemma 3.8]{K-M-W-2}, \cite{M-W}, \cite[Propositions 5.2, 5.3]{Sa-Se}), {are implied by relations \eqref{slp-S-oper-var} and the transmission conditions in \eqref{var-Brinkman-transmission-sl}.}
\begin{lem}
\label{jump-s-l}
Let conditions \eqref{Stokes-1}-\eqref{mu} hold.
If $\boldsymbol\psi \in H^{-\frac{1}{2}}(\partial \Omega )^n$, then the following formulas hold
on $\partial \Omega $
\begin{align}
\label{sl-identities-var}
&\left[\gamma {\bf V}_{\partial\Omega }{\boldsymbol\psi }\right]={\bf 0},\\
\label{sl-identities-var-1}
&\left[{\bf t}\left({\bf V}_{\partial\Omega }{\boldsymbol\psi },{\mathcal Q}_{\partial\Omega }^s{\boldsymbol\psi }\right)\right]
=\boldsymbol\psi {,\quad}
{\bf t}_{{{\mathbb A} }}^{\pm }\left({\bf V}_{\partial\Omega }{\boldsymbol\psi },{\mathcal Q}_{\partial\Omega }^s{\boldsymbol\psi }\right)
={\pm \frac{1}{2}\boldsymbol\psi +{\boldsymbol{\mathcal K}_{\partial\Omega}}\boldsymbol\psi }.
\end{align}
\end{lem}

\subsubsection{\bf The single layer potential for the adjoint Stokes system}
{Recall that ${\boldsymbol{\mathbb L}^*}$ is the operator defined in \eqref{Stokes-tr}, and ${\bf t}^*$ is the conormal derivative operator for the adjoint Stokes system (see formula \eqref{Green-particular-p-adj}).}
The next well-posedness result follows with an argument similar to that for Theorem \ref{slp-var-apr-1-p}, and is based on the Green formula \eqref{Green-particular-p-adj}.
\begin{thm}
\label{slp-var-apr-1-p-adj}
Let conditions \eqref{Stokes-1}-\eqref{mu} hold in $\mathbb R^n$.
Then for any ${\boldsymbol\psi_* \in H^{-\frac{1}{2}}(\partial \Omega )^n}$ the transmission problem for the adjoint Stokes system
\begin{eqnarray}
\label{var-Brinkman-transmission-sl-adj}
\left\{
\begin{array}{ll}
{{\boldsymbol{\mathbb L}^*}{\bf v}_{\boldsymbol\psi_* }- \nabla q_{\boldsymbol\psi_* }={\bf 0}}\,, \ \
{\rm{div}}\, {\bf v}_{\boldsymbol\psi_* }=0 & \mbox{ in } {\mathbb R}^n\setminus \partial \Omega ,\
\\
\left[{\gamma }({\bf v}_{{\boldsymbol\psi_*}})\right]={\bf 0}& \mbox{ on } \partial \Omega ,\\
\left[{\bf t}^*({\bf v}_{{\boldsymbol\psi_* }},q_{{\boldsymbol\psi_* }})\right]
=\boldsymbol\psi_* & \mbox{ on } \partial \Omega ,
\end{array}\right.
\end{eqnarray}
has a unique solution
$({\bf v}_{\boldsymbol\psi_*},q_{\boldsymbol\psi_*})\in {\mathcal H}^{1}({\mathbb R}^n)^n\times L_{2}({\mathbb R}^n)$,
and there exists $C_*\!=\!C_*(\partial \Omega ,c_{\mathbb A},n)\!>\!0$ such that
\begin{align}
\label{estimate-4-var}
\|{\bf v}_{\boldsymbol\psi_*}\|_{{\mathcal H}^{1}({\mathbb R}^n)^n}+\|q_{\boldsymbol\psi_* }\|_{L_{2}({\mathbb R}^n)}\leq C_*\|\boldsymbol\psi_*\|_{H^{-\frac{1}{2}}(\partial \Omega )^n}.
\end{align}
\end{thm}
\begin{defn}
\label{s-l-S-variational-variable-adj}
Let conditions \eqref{Stokes-1}-\eqref{mu} hold.
{For the adjoint Stokes system
\eqref{Stokes-new-adjoint}, the {\it single layer velocity and pressure potential operators}
\begin{align}
\label{s-l-S-v1-var*}
&{\bf V}_{*\partial \Omega }:H^{-\frac{1}{2}}(\partial \Omega )^n\to {\mathcal H}^1({\mathbb R}^n)^n, \quad
{\mathcal Q}_{*\partial \Omega }^s:H^{-\frac{1}{2}}(\partial \Omega )^n\to L_2({\mathbb R}^n)
\end{align}
are defined as
\begin{align}
\label{slp-S-vp-var-adj}
{\bf V}_{*\partial\Omega}{\boldsymbol\psi_*}:={\bf v}_{\boldsymbol\psi_*},\quad
{\mathcal Q}^{s}_{*\partial\Omega}{\boldsymbol\psi_*}:=q_{\boldsymbol\psi_*}
\quad \forall\ \boldsymbol\psi_* \in H^{-\frac{1}{2}}(\partial \Omega )^n,
\end{align}
and the {boundary} operators
\begin{align}
\label{s-l-S-v2-var*}
&\boldsymbol{\mathcal V}_{*\partial \Omega }:H^{-\frac{1}{2}}(\partial \Omega )^n
\to H^{\frac{1}{2}}(\partial \Omega )^n{,\quad}
{\boldsymbol{\mathcal K}_{*\partial\Omega}}:H^{-\frac{1}{2}}(\partial \Omega )^n\to H^{-\frac{1}{2}}(\partial \Omega )^n.
\end{align}
are defined as}
\begin{align}
\label{slp-S-oper-var-adj}
{{\boldsymbol{\mathcal  V}}_{*\partial\Omega}{\boldsymbol\psi_* }}:=\gamma{\bf v}_{\boldsymbol\psi_* },\
{\boldsymbol{\mathcal K}_{*\partial\Omega}}{\boldsymbol\psi_* }:=
\frac{1}{2}\left(
{\bf t}^{*+}({\bf v}_{\boldsymbol\psi_* },q_{\boldsymbol\psi_* })
+{\bf t}^{*-}({\bf v}_{\boldsymbol\psi_* },q_{\boldsymbol\psi_* })\right)
\quad \forall\ \boldsymbol\psi_* \in H^{-\frac{1}{2}}(\partial \Omega )^n,
\end{align}
where $({\bf v}_{{\boldsymbol\psi_*}},q_{{\boldsymbol\psi_*}})$ is the unique solution of the transmission problem \eqref{var-Brinkman-transmission-sl-adj} in ${\mathcal H}^1({\mathbb R}^n\setminus \partial \Omega )^n\times L_2({\mathbb R}^n)$.
\end{defn}

\begin{lem}
\label{self-adj}
Let conditions \eqref{Stokes-1}-\eqref{mu} hold.
Then
\begin{align}
\label{sl-identities-var-1-adjoint}
&{\left[\gamma {\bf V}_{*\partial\Omega}{\boldsymbol\psi_* }\right]={\bf 0},\
{\bf t}^{*\pm }\left({\bf V}_{*\partial\Omega}{\boldsymbol\psi_*},{\mathcal Q}^{s}_{*\partial\Omega}{\boldsymbol\psi_* }\right)}
=\pm \frac{1}{2}\boldsymbol\psi_*
+{\boldsymbol{\mathcal K}_{*\partial\Omega}}\boldsymbol\psi_* \quad \forall \, \boldsymbol\psi_* \in H^{-\frac{1}{2}}(\partial \Omega )^n,\\
\label{sa-sl}
&{\left\langle \boldsymbol\psi ,{\mathcal V}_{*\partial\Omega}\boldsymbol\psi_* \right\rangle _{\partial \Omega }}=\left\langle \boldsymbol\psi_* ,{\mathcal V}_{\partial\Omega }\boldsymbol\psi \right\rangle _{\partial \Omega } \quad \forall \, \boldsymbol\psi ,\, \boldsymbol\psi_* \in H^{-\frac{1}{2}}(\partial \Omega )^n.
\end{align}
\end{lem}
\begin{proof}
First, formulas \eqref{sl-identities-var-1-adjoint}
 are implied by relations \eqref{slp-S-oper-var-adj} and the transmission conditions in \eqref{var-Brinkman-transmission-sl-adj}.

Now let let $\big({\bf V}_{\partial\Omega }{\boldsymbol\psi },{\mathcal Q}_{\partial\Omega }^s{\boldsymbol\psi }\big)$ be the unique solution in ${\mathcal H}^1({\mathbb R}^n)^n\times L_2({\mathbb R}^n)$ of transmission problem \eqref{var-Brinkman-transmission-sl} with the given datum $\boldsymbol\psi \in H^{-\frac{1}{2}}(\partial\Omega )^n$.
Also let {$\big({\bf V}_{*\partial\Omega}{\boldsymbol\psi_* },
{\mathcal Q}_{\partial\Omega }^{s}{*\boldsymbol\psi_* }\big)$ denote the unique solution in ${\mathcal H}^{1}({\mathbb R}^n)^n\!\times \!L_{2}({\mathbb R}^n)$ of transmission problem \eqref{var-Brinkman-transmission-sl-adj}} with the given datum $\boldsymbol\psi_* \!\in \!H^{-\frac{1}{2}}(\partial \Omega )^n$. Then the Green formulas \eqref{jump-conormal-derivative-1} and \eqref{jump-conormal-derivative-1-adj} imply that
\begin{align}
\label{jump-conormal-derivative-1-sl}
\left\langle \left[{\bf t}\left({\bf V}_{\partial\Omega }{\boldsymbol\psi },{\mathcal Q}_{\partial\Omega }^s{\boldsymbol\psi }\right)\right],{\mathcal V}_{*\partial\Omega}{\boldsymbol\psi_* }\right\rangle _{\partial \Omega }&
={\left\langle a_{ij}^{\alpha \beta }E_{j\beta }\left({\bf V}_{\partial\Omega }{\boldsymbol\psi }\right),E_{i\alpha }\left({\bf V}_{*\partial\Omega}{\boldsymbol\psi_* }\right)\right\rangle _{{\mathbb R}^n}}\\
\label{jump-conormal-derivative-1-sl-adj}
\left\langle \left[{\bf t}^*\left({\bf V}_{*\partial\Omega}{\boldsymbol\psi_* },{\mathcal Q}_{*\partial\Omega}^{s}{\boldsymbol\psi_* }\right)\right],{\mathcal V}_{\partial\Omega }{\boldsymbol\psi }\right\rangle _{\partial \Omega }
&={\left\langle a_{ij}^{\alpha \beta }E_{j\beta }\left({\bf V}_{\partial\Omega }{\boldsymbol\psi }\right),E_{i\alpha }\left({\bf V}_{*\partial\Omega}{\boldsymbol\psi_* }\right)\right\rangle _{{\mathbb R}^n}}.
\end{align}
Moreover, by the second formulas in \eqref{sl-identities-var-1} and \eqref{sl-identities-var-1-adjoint},
\begin{align}
\label{jsl}
\big[{\bf t}\big({\bf V}_{\partial\Omega }{\boldsymbol\psi },{\mathcal Q}_{\partial\Omega }^s{\boldsymbol\psi }\big)\big]=\boldsymbol\psi ,\quad
\left[{\bf t}^*\left({\bf V}_{*\partial\Omega}{\boldsymbol\psi_* },{\mathcal Q}_{*\partial\Omega }^{s}{\boldsymbol\psi_* }\right)\right]=\boldsymbol\psi_*.
\end{align}
Then equality \eqref{sa-sl} follows from \eqref{jump-conormal-derivative-1-sl}-\eqref{jsl} (cf. \cite[Proposition 5.4]{Sa-Se} {in the case \eqref{isotropic} with $\mu =1$} {and $\lambda=0$}).
\end{proof}

\begin{rem}
\label{isotropic-case-1}
{$(i)$ Formula \eqref{sa-sl} shows that the adjoint of the single layer operator ${\boldsymbol{\mathcal V}}_{{\partial \Omega }}:H^{-\frac{1}{2}}(\partial \Omega )^n\to H^{\frac{1}{2}}(\partial \Omega )^n$ corresponding to the Stokes system from \eqref{var-Brinkman-transmission-sl} is the operator ${\boldsymbol{\mathcal V}}_{*\partial\Omega}:H^{-\frac{1}{2}}(\partial \Omega )^n\to H^{\frac{1}{2}}(\partial \Omega )^n$ given by formula \eqref{slp-S-oper-var-adj} $($see Definition $\ref{s-l-S-variational-variable-adj}$$)$ and corresponding to the adjoint Stokes system from \eqref{var-Brinkman-transmission-sl-adj}.}

$(ii)$ In the {isotropic case \eqref{isotropic}}, Definition $\ref{s-l-S-variational-variable-adj}$ reduces to Definition $\ref{s-l-S-variational-variable}$, and the single layer operator {${\boldsymbol{\mathcal V}}_{{\partial \Omega }}:H^{-\frac{1}{2}}(\partial \Omega )^n\to H^{\frac{1}{2}}(\partial \Omega )^n$} is self adjoint. Thus, formula \eqref{sa-sl} becomes
\begin{align}
\label{sa-sl-adj}
{\left\langle \boldsymbol\psi ,\boldsymbol{\mathcal V}_{\partial\Omega }\boldsymbol\psi_* \right\rangle _{\partial \Omega }=\left\langle \boldsymbol\psi_* ,\boldsymbol{\mathcal V}_{\partial\Omega }\boldsymbol\psi \right\rangle _{\partial \Omega } \quad \forall \, \boldsymbol\psi ,\, \boldsymbol\psi_* \in H^{-\frac{1}{2}}(\partial \Omega )^n.}
\end{align}
\end{rem}
For a given operator $T:X\to Y$, the set ${\rm{Ker}}\left\{T:X\to Y\right\}:=\left\{x\in X:T(x)=0\right\}$ is the null space of $T$. Let $\boldsymbol \nu $ be the outward unit normal to $\Omega $, which exists a.e. on $\partial \Omega $, and let ${\rm span}\{\boldsymbol\nu\}:=\{c \boldsymbol \nu :c\in {\mathbb R}\}$. Let also
\begin{align}
&\chi _{_{\Omega _{+}}}:=
\left\{\begin{array}{ll}
1 & \mbox{ in }\ \Omega _+\\
0 & \mbox{ in }\ \Omega _-\,,
\end{array}\nonumber
\right.
\nonumber\\
\label{nu}
&H_{\boldsymbol \nu}^{\frac{1}{2}}(\partial \Omega )^n:=\big\{\boldsymbol\Phi \in H^{\frac{1}{2}}(\partial \Omega )^n:\langle \boldsymbol\Phi,\boldsymbol \nu \rangle _{\partial \Omega }=0\big\}.
\end{align}
Next, we mention the main properties of the single layer operator, similar to the ones provided in \cite[Lemma 3.12]{K-M-W-2} in the case of a strongly elliptic {viscosity coefficient tensor} (see also \cite[Lemma 4.9]{K-M-W-1}, \cite[Theorem 10.5.3]{M-W}, and \cite[Proposition 3.3(c)]{B-H} and \cite[Proposition 5.4]{Sa-Se} in the case { \eqref{isotropic} with $\mu =1$} {and $\lambda=0$}).
\begin{lem}
\label{slp-properties}
{Let conditions \eqref{Stokes-1}-\eqref{mu} hold.}
Then
\begin{align}
\label{sl-n}
&{\bf V}_{\partial \Omega }{\boldsymbol\nu }={\bf 0} \mbox{ in } {\mathbb R}^n,\ {\mathcal Q}^s_{\partial \Omega }{\boldsymbol\nu }=-\chi _{_{\Omega _{+}}},\\
\label{sl-n-p}
&{\boldsymbol{\mathcal V}_{\partial\Omega }\boldsymbol \nu ={\bf 0} \mbox{ on } \partial \Omega }\,,
\\
\label{range-sl-v}
&{\boldsymbol{\mathcal V}_{\partial\Omega }\boldsymbol\psi \in H_{\boldsymbol \nu }^{\frac{1}{2}}(\partial \Omega )^n\quad \forall \, \boldsymbol\psi \in H^{-\frac{1}{2}}(\partial \Omega )^n}\,.
\end{align}
\end{lem}
\begin{proof}
First, we note that the transmission problem \eqref{var-Brinkman-transmission-sl} with the given datum $\boldsymbol\psi =\boldsymbol \nu \in H^{-\frac{1}{2}}(\partial \Omega )^n$ is well-posed. We show that the pair
\begin{align}
\label{solution-nu}
\left({\bf u}_{\boldsymbol \nu },\pi _{\boldsymbol \nu }\right)=\big({\bf 0},-\chi _{_{\Omega _{+}}}\big)\in {\mathcal H}^1({\mathbb R}^n)^n\times L_2({\mathbb R}^n)
\end{align}
is the unique solution of this transmission problem. Indeed, $\left({\bf u}_{\boldsymbol \nu },\pi _{\boldsymbol \nu }\right)$ satisfies the equations and the first transmission condition in \eqref{var-Brinkman-transmission-sl}, and by formulas \eqref{conormal-derivative-var-Brinkman-3}, \eqref{jt0} and \eqref{solution-nu}, and by the divergence theorem we obtain that
\begin{align}
\langle [{\bf t}({\bf u}_{\boldsymbol\nu},\pi_{\boldsymbol\nu})],\boldsymbol\Phi \rangle _{\partial \Omega }
=-\left\langle \pi_{\boldsymbol\nu},{\rm{div}}(\gamma^{-1}_+{\boldsymbol\Phi})\right\rangle _{\Omega_+}
=\langle {\boldsymbol\nu},\boldsymbol\Phi\rangle_{\partial \Omega }\quad \forall \, \boldsymbol\Phi \in H^{\frac{1}{2}}(\partial \Omega )^n,
\end{align}
and hence $[{\bf t}({\bf u}_{\boldsymbol \nu },\pi_{\boldsymbol \nu })]=\boldsymbol \nu $. Consequently, the pair $\left({\bf u}_{\boldsymbol \nu },\pi _{\boldsymbol \nu }\right)$ given by \eqref{solution-nu} is the unique solution of the transmission problem \eqref{var-Brinkman-transmission-sl} with the given datum $\boldsymbol\psi =\boldsymbol \nu \in \!H^{-\frac{1}{2}}(\partial \Omega )^n$. Then relations \eqref{sl-n} and \eqref{sl-n-p} follow from Definition \ref{s-l-S-variational-variable}.
Thus, ${\rm span}\{\boldsymbol\nu\}\subseteq {\rm{Ker}}\big\{{\boldsymbol{\mathcal V}}_{\partial\Omega }:H^{-\frac{1}{2}}(\partial\Omega )^n\to H^{\frac{1}{2}}(\partial\Omega )^n\big\}$. Similarly, we obtain
\begin{align}
\label{sl-n-adj}
{{\bf V}_{*\partial \Omega }{\boldsymbol\nu }}={\bf 0} \mbox{ in } {\mathbb R}^n,\ \
{\boldsymbol{\mathcal V}_{*\partial\Omega}\boldsymbol \nu }={\bf 0} \mbox{ on } \partial \Omega .
\end{align}

Next, {we apply formula \eqref{sa-sl}} for the densities $\boldsymbol\psi \!\in \! H^{-\frac{1}{2}}(\partial\Omega )^n$ and $\boldsymbol\nu \!\in \! H^{-\frac{1}{2}}(\partial\Omega )^n$, and use the second relation in \eqref{sl-n-adj}. Then we obtain that
$\left\langle \boldsymbol{\mathcal V}_{ \partial\Omega }\boldsymbol\psi ,\boldsymbol\nu \right\rangle _{\partial \Omega }
=
\left\langle \boldsymbol\psi ,\boldsymbol{\mathcal V}_{*\partial\Omega}\boldsymbol\nu\right\rangle_{\partial\Omega }=0,$
and hence \eqref{range-sl-v} follows.
\end{proof}

\subsubsection{\bf Isomorphism property of the single layer operator}
Next we show the following invertibility property of the single layer potential operator (cf. \cite[Lemma 3.13]{K-M-W-2}, and \cite[Theorem 10.5.3]{M-W}, \cite[Proposition 3.3(d)]{B-H}, \cite[Proposition 5.5]{Sa-Se} in the case \eqref{isotropic} with $\mu =1$ {and $\lambda=0$}).
\begin{lem}
\label{isom-sl-v}
Let conditions \eqref{Stokes-1}-\eqref{mu} hold in $\mathbb R^n$.
Then
\begin{align}
\label{kernel-sl-v}
{{\rm{Ker}}\left\{\boldsymbol{\mathcal V}_{\partial\Omega }:H^{-\frac{1}{2}}(\partial\Omega )^n\to H^{\frac{1}{2}}(\partial\Omega )^n\right\}={\rm span}\{\boldsymbol\nu\} \,.}
\end{align}
and the following operator is an isomorphism,
\begin{align}
\label{sl-v-isom}
\boldsymbol{\mathcal V}_{\partial\Omega }:H^{-\frac{1}{2}}(\partial\Omega )^n/{\rm span}\{\boldsymbol\nu\} \to H_{\boldsymbol \nu }^{\frac{1}{2}}(\partial\Omega )^n\,.
\end{align}
\end{lem}
\begin{proof}
(i) Let $\boldsymbol \psi _0\in {\rm{Ker}}\left\{{\boldsymbol{\mathcal V}}_{\partial\Omega }:H^{-\frac{1}{2}}(\partial\Omega )^n\to H^{\frac{1}{2}}(\partial\Omega )^n\right\}$ and let
$({\bf u}_{\boldsymbol \psi _0},\pi_{\boldsymbol \psi _0})
=\big({\bf V}_{\partial \Omega }{\boldsymbol\psi _0},{\mathcal Q}^s_{\partial \Omega }{\boldsymbol\psi _0}\big)$ be the unique solution in ${\mathcal H}^1({\mathbb R}^n)^n\times L_2({\mathbb R}^n)$ of the transmission problem \eqref{var-Brinkman-transmission-sl} with the given datum $\boldsymbol \psi _0$.
In view of formula \eqref{jump-conormal-derivative-1} and since $\gamma {\bf u}_{\boldsymbol \psi _0}={\bf 0}$ on $\partial \Omega $, we obtain that
\begin{align}
a_{{\mathbb A};{\mathbb R}^n}\big({\bf u}_{\boldsymbol \psi _0},{\bf u}_{\boldsymbol \psi _0}\big)=\left\langle [{\bf t}({\bf u}_{\boldsymbol\psi_0},\pi_{\boldsymbol\psi_0})],\gamma {\bf u}_{\boldsymbol \psi _0}\right\rangle _{\partial \Omega }=0.
\end{align}
In addition, since ${\rm div}\, {\bf u}_{\boldsymbol\psi_0}=0$, we have $E_{ii}({\bf u}_{\boldsymbol\psi_0})=0$ and due to assumption \eqref{mu},
\begin{align}
a_{{\mathbb A};{\mathbb R}^n}\big({\bf u}_{\boldsymbol \psi _0},{\bf u}_{\boldsymbol \psi _0}\big)
\geq c_{\mathbb A}^{-1}\|{\mathbb E}({\bf u}_{\boldsymbol \psi _0})\|_{L_2({\mathbb R}^n)^n}^2\,,
\end{align}
which implies that ${\mathbb E}({\bf u}_{\boldsymbol \psi _0})=0$ and hence ${\bf u}_{\boldsymbol \psi _0}$ is a rigid body motion field (see, e.g., \cite[Lemma 3.1]{Med-AAM-11}).
Since ${\bf u}_{\boldsymbol \psi _0}$ belongs to the space ${\mathcal H}^1({\mathbb R}^n)^n$, which is embedded in $L_{\frac{2n}{n-2}}({\mathbb R}^n)^n$ (see \eqref{weight-Lp}), it follows that ${\bf u}_{\boldsymbol \psi _0}={\bf 0}$ in ${\mathbb R}^n$.
In addition, ${\bf u}_{\boldsymbol \psi _0}$ and $\pi _{\boldsymbol \psi _0}$ satisfy the Stokes equation in ${\mathbb R}^n\setminus \partial \Omega $ and $\pi _{\boldsymbol \psi _0}$ belongs to $L_2({\mathbb R}^n)$.
Thus, there exists $c_0\in {\mathbb R}$ such that $\pi _{\boldsymbol \psi _0}=c_0\chi _{_{\Omega _{+}}}$ in ${\mathbb R}^n$.
Then formulas \eqref{conormal-derivative-var-Brinkman-3}, \eqref{jt0} and the divergence theorem yield that
\begin{align*}
\langle [{\bf t}({\bf u}_{\boldsymbol\psi_0},\pi_{\boldsymbol\psi_0})],\boldsymbol\Phi \rangle _{\partial \Omega }
=-\left\langle \pi_{\boldsymbol\psi_0},{\rm{div}}(\gamma^{-1}_+{\boldsymbol\Phi})\right\rangle _{\Omega_+}
=-c_0\langle {\boldsymbol\nu},\boldsymbol\Phi\rangle_{\partial \Omega }\quad
\forall \, \boldsymbol\Phi \in H^{\frac{1}{2}}(\partial \Omega )^n,
\end{align*}
and accordingly that $\boldsymbol \psi _0=[{\bf t}({\bf u}_{\boldsymbol \psi _0 },\pi_{\boldsymbol \psi _0})]=-c_0\boldsymbol \nu $.
Taking into account \eqref{sl-n-p}, formula \eqref{kernel-sl-v} follows.

(ii) Next we use the notation {$\left[\!\left[\cdot \right]\!\right]$} for the classes of the space $H^{-\frac{1}{2}}(\partial\Omega )^n/{\rm span}\{\boldsymbol\nu\}$. Thus, $\left[\!\left[\boldsymbol \psi \right]\!\right]\!=\!\boldsymbol \psi \!+\!{\rm span}\{\boldsymbol\nu\}$, with $\boldsymbol \psi \in H^{-\frac{1}{2}}(\partial\Omega )^n$. We show that there exists a constant
$c=c(\partial \Omega ,c_{\mathbb A},n)>0$ such that the single layer potential operator satisfies the coercivity inequality
\begin{align}
\label{H-elliptic-sl-v}
\left\langle \left[\!\left[\boldsymbol \psi \right]\!\right],\boldsymbol{\mathcal V}_{\partial\Omega }\left[\!\left[\boldsymbol \psi \right]\!\right]\right\rangle _{\partial \Omega }\geq c\left\|\left[\!\left[\boldsymbol \psi \right]\!\right]\right\|_{H^{-\frac{1}{2}}(\partial\Omega )^n/{{\rm span}\{\boldsymbol\nu\}}}^2 \quad
\forall \, \left[\!\left[\boldsymbol \psi \right]\!\right]\in {H^{-\frac{1}{2}}(\partial\Omega )^n/{{\rm span}\{\boldsymbol\nu\}}}
\end{align}
(cf. \cite[Lemma 4.10]{K-M-W-1} and \cite[Proposition 5.5]{Sa-Se}).

Let $\left[\!\left[\boldsymbol \psi \right]\!\right]\!\in \!{H^{-\frac{1}{2}}(\partial\Omega )^n/{{\rm span}\{\boldsymbol\nu\}}}$.
In view of formula \eqref{jump-conormal-derivative-1}, Definition \ref{s-l-S-variational-variable}, relations \eqref{range-sl-v}, \eqref{kernel-sl-v}, and the Korn inequality, we obtain (cf. \eqref{a-1-v2-S}),
\begin{align}
\label{onto-div-2}
\left\langle \left[\!\left[\boldsymbol \psi \right]\!\right],\boldsymbol{\mathcal V}_{\partial\Omega }\left[\!\left[\boldsymbol \psi \right]\!\right]\right\rangle _{\partial \Omega }&=\left\langle \boldsymbol \psi ,\boldsymbol{\mathcal V}_{\partial\Omega }\boldsymbol \psi \right\rangle _{\partial \Omega }=\langle [{\bf t}({\bf u}_{\boldsymbol \psi},\pi_{\boldsymbol \psi})],\gamma {\bf u}_{\boldsymbol \psi}\rangle _{\partial \Omega }\nonumber\\
&=a_{{\mathbb A};{\mathbb R}^n}({\bf u}_{\boldsymbol \psi },{\bf u}_{\boldsymbol \psi })
\geq {c_{\mathbb A}^{-1}\|{\mathbb E}({\bf u}_{\boldsymbol \psi })\|_{L_2({\mathbb R}^n)^{n\times n}}^2
\geq 2^{-1}c_{\mathbb A}^{-1}c_1\|{\bf u}_{\boldsymbol \psi }\|_{\mathcal H^1({\mathbb R}^n)^n}^2},
\end{align}
where ${\bf u}_{\boldsymbol \psi }\!=\!{\bf V}_{\partial\Omega }\boldsymbol \psi $ and $\pi_{\boldsymbol\psi}\!=\!{\mathcal Q}^s_{{\partial \Omega }}{\boldsymbol\psi}$.
Moreover, the trace operator
$\gamma :{\mathcal H}_{{\rm div}}^1({\mathbb R}^n)^n\to H_{\boldsymbol \nu }^{\frac{1}{2}}(\partial\Omega )^n$
is surjective having a bounded right inverse $\gamma ^{-1}:H_{\boldsymbol \nu }^{\frac{1}{2}}(\partial\Omega )^n\to {\mathcal H}_{{\rm div}}^1({\mathbb R}^n)^n$ {\rd (cf., e.g., \cite[Proposition 4.4]{Sa-Se} in the case $n=3$.  
{\mn Arguments similar to} those for Proposition 4.4 of Sayas and Selgas \cite{Sa-Se} imply that the result remains valid also in the case $n\geq 3$)}.
{\rd Moreover, there} exists $c'\!=\!c'(\partial \Omega ,n)>0$ such that
\begin{align}
\label{estimate-norm}
|\langle \left[\!\left[\boldsymbol \psi \right]\!\right],\boldsymbol\Phi \rangle _{\partial \Omega }|&=|\langle \boldsymbol \psi ,\boldsymbol\Phi \rangle _{\partial \Omega }|
=|\langle [{\bf t}({\bf u}_{\boldsymbol \psi },\pi_{\boldsymbol\psi})],\boldsymbol\Phi\rangle_{\partial\Omega}|
=|a_{{\mathbb A};{\mathbb R}^n}({\bf u}_{\boldsymbol \psi },\gamma ^{-1}\boldsymbol\Phi)|\nonumber\\
&\leq {\|{\mathbb A}\|_{L_\infty ({\mathbb R}^n)}}c'\|{\bf u}_{\boldsymbol \psi }\|_{{\mathcal H}^1({\mathbb R}^n)^n}
\|\boldsymbol\Phi \|_{H^{\frac{1}{2}}(\partial\Omega )^n}\quad
\forall\ \boldsymbol\Phi \!\in \!H_{\boldsymbol \nu }^{\frac{1}{2}}(\partial\Omega )^n.
\end{align}
Inequality \eqref{estimate-norm} and the duality of the spaces $H_{\boldsymbol \nu }^{\frac{1}{2}}(\partial\Omega )^n$ and $H^{-\frac{1}{2}}(\partial\Omega )^n/{{\rm span}\{\boldsymbol\nu\}}$ show that
\begin{align}
\label{onto-div-1}
\|\left[\!\left[\boldsymbol \psi \right]\!\right]\|_{H^{-\frac{1}{2}}(\partial\Omega )^n/{\rm span}\{\boldsymbol\nu\}}
\leq {\|{\mathbb A}\|_{L_\infty ({\mathbb R}^n)}c'}\|{\bf u}_{\boldsymbol \psi }\|_{{\mathcal H}^{1}({\mathbb R}^n)^n}.
\end{align}
Then the coercivity inequality \eqref{H-elliptic-sl-v} follows from inequalities \eqref{onto-div-2} and \eqref{onto-div-1}, and the Lax-Milgram lemma yields that the single layer potential operator \eqref{sl-v-isom} is an isomorphism, as asserted.
\end{proof}

\subsection{The double layer potential operator for the Stokes system with $L_{\infty }$ viscosity coefficient tensor}
Note that if $\mathbf u\in L_{2,\rm loc}(\mathbb R^n)^n$ is such that $\mathbf u|_{\Omega_+}\in {H}^1(\Omega _+)^n$, $\mathbf u|_{\Omega_-}\in {\mathcal H}^1(\Omega _{-})^n$, then due to definition \eqref{Omega-pm}
$\mathbf u\in\mathcal H^1(\mathbb R^n\setminus\partial\Omega)^n$ and can be endowed with the norm
$\|{\bf u}\|^2_{\mathcal H^1(\mathbb R^n\setminus\partial\Omega)^n}:=\|{\bf u}\|^2_{H^1(\Omega_+)^n}+\|{\bf u}\|^2_{\mathcal H^1(\Omega_-)^n}$ that is equivalent to the norm \eqref{standard-weight-p}.

By following a similar approach to that used to define the Stokes single layer potentials, we now show the well-posedness of a transmission problem that allows us to define the $L_{\infty }$-coefficient Stokes double layer potentials with the densities in the space $H^{\frac{1}{2}}(\partial \Omega )^n$, $n\geq 3$ (cf. \cite[Theorem 3.14]{K-M-W-2} for the Stokes system with strongly elliptic coefficient tensor and \cite[Propositions 6.1, 7.1]{Sa-Se} for the isotropic case \eqref{isotropic} with $\mu =1$,  {$\lambda=0$}, and $n=2,3$).

\begin{thm}
\label{dlp-var-apr-var-p}
Let conditions \eqref{Stokes-1}-\eqref{mu} hold on $\mathbb R^n$.
Then for any $\boldsymbol\varphi\in H^{\frac{1}{2}}(\partial \Omega )^n$ the transmission problem
\begin{equation}
\label{transmission-B-dl-var}
\left\{
\begin{array}{lll}
{\boldsymbol{\mathcal L}({\bf u}_{\boldsymbol\varphi },\pi _{\boldsymbol\varphi })={\bf 0}}\,, \ \  
{\rm{div}}\, {\bf u}_{\boldsymbol\varphi }=0 & \mbox{ in } {\mathbb R}^n\setminus \partial \Omega \,,
\\
\left[\gamma {\bf u}_{\boldsymbol\varphi}\right]={-\boldsymbol\varphi } & \mbox{ on } \partial \Omega \,, \\ 
\left[{\bf t}({\bf u}_{\boldsymbol\varphi},\pi _{\boldsymbol\varphi})\right]
={{\bf 0}} & \mbox{ on } \partial \Omega \,,
\end{array}\right.
\end{equation}
has a unique solution $({\bf u}_{\boldsymbol\varphi},\pi _{{\boldsymbol\varphi}})\!\in \!{\mathcal H}^1({\mathbb R}^n\setminus \partial \Omega )^n\times L_2({\mathbb R}^n)$, and there exists
$C\!=\!C(\partial \Omega ,c_{\mathbb A},n)>0$ such that
\begin{align}
\label{estimate-4-D-var-p}
\|{\bf u}_{{\boldsymbol\varphi}}\|_{{\mathcal H}^1({\mathbb R}^n\setminus \partial \Omega )^n}+\|\pi _{\boldsymbol\varphi}\|_{L_2({\mathbb R}^n)}\leq C\|{\boldsymbol\varphi}\|_{H^{\frac{1}{2}}(\partial \Omega )^n}.
\end{align}
\end{thm}
\begin{proof}
First we show the uniqueness.
Let $({\bf u}_0,\pi _0)\in {\mathcal H}^1({\mathbb R}^n\setminus \partial \Omega )^n\times L_2({\mathbb R}^n)$ be a solution of the homogeneous version of problem \eqref{transmission-B-dl-var}.
Therefore, the couple $({\bf u}_0,\pi _0)$ is a solution of the homogeneous version of the transmission problem \eqref{var-Brinkman-transmission-sl}, which, in view of Theorem \ref{slp-var-apr-1-p}, has only the trivial solution. 

Next, we show that the transmission problem \eqref{transmission-B-dl-var} has the following equivalent variational formulation.

{\it Find $({\bf u}_{\boldsymbol\varphi},\pi _{{\boldsymbol\varphi}})\in {\mathcal H}^1({\mathbb R}^n\setminus \partial \Omega )^n\times L_2({\mathbb R}^n)$ such that}
\begin{equation}
\label{transmission-B-dl-var-equiv}
\hspace{-0.5em}\left\{\begin{array}{lll}
\left \langle a_{ij}^{\alpha \beta }E_{j\beta }({\bf u}_{\boldsymbol\varphi }),E_{i\alpha }({\bf v})\right \rangle _{{\Omega _+}}+\left \langle a_{ij}^{\alpha \beta }E_{j\beta }({\bf u}_{\boldsymbol\varphi }),E_{i\alpha }({\bf v})\right \rangle _{{{\Omega }_{-}}}
-\langle \pi _{{\boldsymbol\varphi}},{\rm{div}}\, {\bf v}\rangle _{{\mathbb R}^n}
=0 \quad \forall \, {\bf v}\in {\mathcal H}^{1}({\mathbb R}^n)^n,\\
\left\langle {\rm{div}}\, {\bf u}_{{\boldsymbol\varphi}},q\right\rangle _{{\mathbb R}^n\setminus \partial \Omega }=0 \quad \forall \, q\in L_{2}({\mathbb R}^n),\\
\left[\gamma {\bf u}_{{\boldsymbol\varphi}}\right]={-\boldsymbol\varphi } \mbox{ on } \partial \Omega .
\end{array}
\right.
\end{equation}

Indeed, if $({\bf u}_{{\boldsymbol\varphi}},\pi _{\boldsymbol\varphi})\in {{\mathcal H}^1({\mathbb R}^n\setminus \partial \Omega )^n}\times L_2({\mathbb R}^n)$ satisfies transmission problem \eqref{transmission-B-dl-var} then the Green formula \eqref{jump-conormal-derivative-1} yields the first equation of problem \eqref{transmission-B-dl-var-equiv}. The second equation of \eqref{transmission-B-dl-var-equiv} is the distributional form of the second equation of \eqref{transmission-B-dl-var}.
Conversely, assume that $({\bf u}_{{\boldsymbol\varphi}},\pi _{\boldsymbol\varphi})\in {{\mathcal H}^1({\mathbb R}^n\setminus \partial \Omega )^n}\times L_2({\mathbb R}^n)$ satisfies the variational problem \eqref{transmission-B-dl-var-equiv}. Then from the first equation of \eqref{transmission-B-dl-var-equiv} we deduce that
\begin{align}
\label{j1new}
\left\langle \left(\partial _\alpha\left(a_{ij}^{\alpha \beta }E_{j\beta }({\bf u}_{\boldsymbol\varphi })\right)-\partial _i\pi _{{\boldsymbol\varphi }}\right)\big|_{\Omega _\pm },v_i\right\rangle _{{\Omega _\pm }}=0 \quad \forall \, {\bf v}=(v_1,\ldots ,v_n)\in {\mathcal D}(\Omega _\pm )^n,
\end{align}
which is the distributional form of the first equation in \eqref{transmission-B-dl-var}.
The second equation of \eqref{transmission-B-dl-var} follows from the second equation of \eqref{transmission-B-dl-var-equiv}.
In addition,  the first equation of \eqref{transmission-B-dl-var-equiv} and the Green formula \eqref{jump-conormal-derivative-1} applied to the pair $({\bf u}_{{\boldsymbol\varphi}},\pi _{{\boldsymbol\varphi}})$ yield that
\begin{align}
\label{useful-3}
\left\langle [{\bf t}({\bf u}_{{\boldsymbol\varphi}},\pi _{{\boldsymbol\varphi}})],\gamma {\bf v}\right\rangle _{\partial \Omega}=0 \quad \forall \, {\bf v}\in {\mathcal H}^{1}({\mathbb R}^n)^n.
\end{align}
Moreover, the surjectivity property of the trace operator $\gamma :{\mathcal H}^{1}({\mathbb R}^n)^n\to H^{\frac{1}{2}}(\partial \Omega )^n$ shows that equation \eqref{useful-3} can be written in the equivalent form
\begin{align}
\label{useful-3-new}
\left\langle [{\bf t}({\bf u}_{{\boldsymbol\varphi}},\pi _{{\boldsymbol\varphi}})],\boldsymbol \Psi \right\rangle _{\partial \Omega}=0 \quad \forall \, \boldsymbol \Psi \in H^{\frac{1}{2}}(\partial \Omega )^n,
\end{align}
which yields the second transmission condition of \eqref{transmission-B-dl-var}. The first transmission condition in \eqref{transmission-B-dl-var} follows from the transmission condition in \eqref{transmission-B-dl-var-equiv}. Therefore, problems \eqref{transmission-B-dl-var} and \eqref{transmission-B-dl-var-equiv} are equivalent.

By using again the existence of a right inverse $\gamma^{-1}_\pm: H^{\frac{1}{2}}({\partial\Omega })\to\mathcal H^1(\Omega_\pm)$ of the trace operator $\gamma_\pm \!:\!{\mathcal H}^1(\Omega_\pm)\!\to \! H^{\frac{1}{2}}(\partial\Omega)$ (see Theorem \ref{trace-operator1}), we deduce that for $\boldsymbol\varphi\in H^{\frac{1}{2}}(\partial \Omega )^n$ given,
there exists ${\bf w}_{\boldsymbol\varphi}\in {\mathcal H}^1({\mathbb R}^n\setminus \partial \Omega )^n$ continuously depending on $\boldsymbol\varphi$ such that
$\left[\gamma {\bf w}_{\boldsymbol\varphi}\right]=-\boldsymbol\varphi \mbox{ on } \partial \Omega$.
For example, we can take ${\bf w}_{\boldsymbol\varphi}=0$ in $\Omega_-$ and
${\bf w}_{\boldsymbol\varphi}=-\gamma_+^{-1}\boldsymbol\varphi$ in $\Omega_+$.

Therefore, ${\bf v}_{{\boldsymbol\varphi}}:={\bf u}_{{\boldsymbol\varphi}}-{\bf w}_{{\boldsymbol\varphi}}$ satisfies the condition $[\gamma {\bf v}_{\boldsymbol\varphi}]={\bf 0}$, and hence by Lemma \ref{extention} can be extended to ${\bf v}_{{\boldsymbol\varphi}}\in {\mathcal H}^1({\mathbb R}^n)^n$.
In addition, \eqref{transmission-B-dl-var-equiv} reduces to the following variational problem
\begin{align}
\label{transmission-B-variational-dl-3-var}
\left\{\begin{array}{ll}
a_{{\mathbb A};{\mathbb R}^n}({\bf v}_{\boldsymbol\varphi},{\bf v})+b_{{\mathbb R}^n}({\bf v},\pi _{\boldsymbol\varphi})=\boldsymbol\xi _{{\boldsymbol\varphi}}({\bf v}) \quad \forall \, {\bf v}\in {\mathcal H}^{1}({\mathbb R}^n)^n\,,\\
b_{{\mathbb R}^n}({\bf v}_{\boldsymbol\varphi},q)=\zeta _{\boldsymbol\varphi} (q) \quad \forall \, q\in L_{2}({\mathbb R}^n)\,,
\end{array}
\right.
\end{align}
with the unknown $({\bf v}_{{\boldsymbol\varphi}},\pi _{{\boldsymbol\varphi}})\!\in \!{\mathcal H}^1({\mathbb R}^n)^n\!\times \!L_2({\mathbb R}^n)$, where $a_{{\mathbb A};{\mathbb R}^n}(\cdot ,\cdot ):{\mathcal H}^1({\mathbb R}^n)^n\times {\mathcal H}^{1}({\mathbb R}^n)^n\to {\mathbb R}$ and $b_{{\mathbb R}^n}:{\mathcal H}^1({\mathbb R}^n)^n\times L_{2}({\mathbb R}^n)\to {\mathbb R}$ are the bounded bilinear forms given by \eqref{a-v} and \eqref{b-v}, respectively. In addition, {conditions \eqref{Stokes-1}}
and the H\"{o}lder inequality 
show the boundedness of the linear forms
\begin{align}
&\boldsymbol\xi _{{\boldsymbol\varphi}}:{\mathcal H}^{1}({\mathbb R}^n)^n\to {\mathbb R},\ \boldsymbol\xi _{{\boldsymbol\varphi}}({\bf v})
:=-\left \langle a_{ij}^{\alpha \beta }E_{j\beta }({\bf w}_{\boldsymbol\varphi }),E_{i\alpha }({\bf v})\right \rangle _{\Omega _+}-\left \langle a_{ij}^{\alpha \beta }E_{j\beta }({\bf w}_{\boldsymbol\varphi }),E_{i\alpha }({\bf v})\right \rangle _{\Omega _-},\\
&\zeta _{\boldsymbol\varphi}:L_{2}({\mathbb R}^n)^n\to {\mathbb R},\ \zeta _{\boldsymbol\varphi}(q):=-\left(\langle {\rm{div}}\, {\bf w}_{{{\boldsymbol\varphi}}},q\rangle _{\Omega _+}+\langle {\rm{div}}\, {\bf w}_{{{\boldsymbol\varphi}}},q\rangle _{\Omega _{-}}\right) \quad \forall \, q\in L_{2}({\mathbb R}^n)\,.
\end{align}
Then Lemma \ref{lemma-a47-1-Stokes} implies that the variational problem \eqref{transmission-B-variational-dl-3-var} has a unique solution $({\bf v}_{{\boldsymbol\varphi}},\pi _{{\boldsymbol\varphi}})\in {\mathcal H}^1({\mathbb R}^n)^n\times L_{2}({\mathbb R}^n)$.
Hence, the pair
$({\bf u}_{{\boldsymbol\varphi}},\pi _{{\boldsymbol\varphi}})=({\bf w}_{{\boldsymbol\varphi}}+{\bf v}_{{\boldsymbol\varphi}},\pi _{{\boldsymbol\varphi}})$
is a solution of the variational problem \eqref{transmission-B-dl-var-equiv} in the space ${\mathcal H}^1({\mathbb R}^n\setminus \partial \Omega )^n\times L_2({\mathbb R}^n)$, and depends continuously on $\boldsymbol\varphi $.
The equivalence between problems \eqref{transmission-B-dl-var} and \eqref{transmission-B-dl-var-equiv} show that $({\bf u}_{{\boldsymbol\varphi}},\pi _{{\boldsymbol\varphi}})$ is the unique solution of the transmission problem \eqref{transmission-B-dl-var}. 
\end{proof}

Theorem \ref{dlp-var-apr-var-p}
suggests the following definition of the double layer potential operator for the anisotropic Stokes system \eqref{Stokes} in the case $n\geq 3$ (cf. \cite[p. 77]{Sa-Se} for the constant-coefficient Stokes system in ${\mathbb R}^3$, \cite[formula (4.5) and Lemma 4.6]{Barton} for general strongly elliptic differential operators, and \cite[Definition 3.15]{K-M-W-2} for the Stokes system with $L_\infty $ strongly elliptic viscosity coefficient).

\begin{defn}
\label{d-l-variational-var}
Let conditions \eqref{Stokes-1}-\eqref{mu} hold.
The double layer velocity and pressure potential operators
\begin{align}
\label{dl-var-mu-alpha}
&{\bf W}_{\partial\Omega }:H^{\frac{1}{2}}(\partial \Omega )^n\to {\mathcal H}^1({\mathbb R}^n\setminus \partial \Omega )^n,\quad
{\mathcal Q}_{\partial\Omega }^d:H^{\frac{1}{2}}(\partial \Omega )^n\to {L_2({\mathbb R}^n)}
\end{align}
are defined as
\begin{align}
\label{dlp-vp-var}
{\bf W}_{\partial\Omega }\boldsymbol\varphi:={\bf u}_{\boldsymbol\varphi },\quad
\mathcal Q^d_{\partial\Omega}\boldsymbol\varphi:=\pi _{\boldsymbol\varphi } \quad \forall\ \boldsymbol\varphi \in H^{\frac{1}{2}}(\partial \Omega )^n,
\end{align}
and the boundary operators
\begin{align}
\label{dl-var-mu-alpha-1}
&{\bf K}_{\partial\Omega }:H^{\frac{1}{2}}(\partial \Omega )^n\to H^{\frac{1}{2}}(\partial \Omega )^n,\quad
{\bf D}_{\partial\Omega }:H^{\frac{1}{2}}(\partial \Omega )^n\to H^{-\frac{1}{2}}(\partial \Omega )^n
\end{align}
are defined as
\begin{align}
\label{dlp-oper-var}
&{{\bf K}_{\partial \Omega }{\boldsymbol\varphi }:=\frac{1}{2}\left(\gamma _{+}{\bf u}_{\boldsymbol\varphi }+\gamma _{-}{\bf u}_{\boldsymbol\varphi }\right)}
\quad \forall\ \boldsymbol\varphi \in H^{\frac{1}{2}}(\partial \Omega )^n,
\\
\label{jump-dl-v-var-alpha-cd}
&{\bf D}_{\partial\Omega }\boldsymbol\varphi
:={\bf t}^{+}\left({\bf W}_{\partial\Omega }\boldsymbol\varphi ,{\mathcal Q}^d_{\partial\Omega }\boldsymbol\varphi \right)
={\bf t}^{-}\left({\bf W}_{\partial\Omega }\boldsymbol\varphi ,{\mathcal Q}^d_{\partial\Omega }\boldsymbol\varphi \right)
\quad \forall\ \boldsymbol\varphi \in H^{\frac{1}{2}}(\partial \Omega )^n,
\end{align}
where $({\bf u}_{\boldsymbol\varphi },\pi _{\boldsymbol\varphi })$ is the unique solution of the transmission problem \eqref{transmission-B-dl-var} in ${\mathcal H}^1({\mathbb R}^n\setminus \partial \Omega )^n\times L_2({\mathbb R}^n)$.
\end{defn}
Moreover, the well-posedness of the transmission problem \eqref{transmission-B-dl-var} and Definition \ref{d-l-variational-var} lead to the next result (see also \cite[(10.81), (10.82)]{M-W} and \cite[Propositions 6.2, 6.3]{Sa-Se} for the constant-coefficient Stokes system in ${\mathbb R}^3$, and {\cite[Lemma 5.8]{Barton} for strongly elliptic operators}).

\begin{lem}
\label{continuity-dl-h-var}
Let conditions \eqref{Stokes-1}-\eqref{mu} are satisfied.
Then the following assertions hold.
\begin{itemize}
\item[$(i)$]
Operators \eqref{dl-var-mu-alpha} and \eqref{dl-var-mu-alpha-1} are linear and continuous and for any
$\boldsymbol\varphi \in H^{-\frac{1}{2}}(\partial\Omega)^n$,
$$
\boldsymbol{\mathcal L}({\bf W}_{\partial\Omega}\boldsymbol\varphi,
\mathcal Q^d_{\partial\Omega}\boldsymbol\varphi)={\bf 0},\quad
{\rm{div}}\,{\bf W}_{\partial\Omega}\boldsymbol\varphi=0 \quad \mbox{in } \Omega_\pm.
$$
\item[$(ii)$]
For any $\boldsymbol\varphi \in H^{\frac{1}{2}}(\partial \Omega )^n$, the following jump formulas hold on $\partial \Omega $
\begin{align}
\label{jump-dl-v-var-alpha}
&\gamma _{\pm }\left({\bf W}_{\partial\Omega }\boldsymbol\varphi \right)={\mp }\frac{1}{2}\boldsymbol\varphi +{\bf K}_{\partial\Omega }\boldsymbol\varphi ,\quad
{\bf t}^{\pm }\left({\bf W}_{\partial\Omega }\boldsymbol\varphi ,{\mathcal Q}^d_{\partial\Omega }\boldsymbol\varphi \right)={\bf D}_{\partial\Omega }\boldsymbol\varphi \,.
\end{align}
\item[$(iii)$]
The operator $\boldsymbol{\mathcal K}_{*\partial \Omega }:H^{-\frac{1}{2}}(\partial \Omega )^n\to H^{-\frac{1}{2}}(\partial \Omega )^n$ defined in \eqref{slp-S-oper-var-adj}
is the transpose of the double layer operator ${\bf K}_{\partial \Omega }:H^{\frac{1}{2}}(\partial \Omega )^n\to H^{\frac{1}{2}}(\partial \Omega )^n$ defined in \eqref{dlp-oper-var}, i.e.,
\begin{align}
\label{transpose-dl-var-sm}
\left\langle \boldsymbol\psi_* ,{\bf K}_{\partial\Omega }\boldsymbol\varphi \right\rangle _{\partial \Omega}=\left\langle {{\boldsymbol{\mathcal K}_{*\partial\Omega}}\boldsymbol\psi_* },\boldsymbol\varphi \right\rangle _{\partial \Omega} \quad \forall \, \boldsymbol\varphi\in H^{\frac{1}{2}}(\partial \Omega )^n,\, {\boldsymbol\psi_* }\in H^{-\frac{1}{2}}(\partial \Omega )^n\,.
\end{align}
\end{itemize}
\end{lem}
\begin{proof}
The continuity of operators \eqref{dl-var-mu-alpha} and \eqref{dl-var-mu-alpha-1} follows from the well-posedness of transmission problem \eqref{transmission-B-dl-var} and Definition \ref{d-l-variational-var}. By invoking again Definition \ref{d-l-variational-var} and the transmission conditions in \eqref{transmission-B-dl-var} we obtain jump formulas \eqref{jump-dl-v-var-alpha}.

Next we show equality \eqref{transpose-dl-var-sm}, by using an argument similar to that in the proof of \cite[Proposition 6.7]{Sa-Se} for the constant-coefficient Stokes system and $p=2$.
Let $\boldsymbol\varphi\in H^{\frac{1}{2}}(\partial \Omega )^n$ be given, and let $({\bf u}_{\boldsymbol\varphi },\pi _{\boldsymbol\varphi })=\left({\bf W}_{\partial \Omega }\boldsymbol\varphi ,{\mathcal Q}_{\partial \Omega }^d\boldsymbol\varphi \right)\in {\mathcal H}^1({\mathbb R}^n\setminus \partial \Omega )^n\times L_2({\mathbb R}^n)$ be the unique solution of the problem \eqref{transmission-B-dl-var} with datum $\boldsymbol\varphi $.
Let also $\boldsymbol\psi_* \in H^{-\frac{1}{2}}(\partial \Omega )^n$ and {$({\bf v}_{\boldsymbol\psi_* },q_{\boldsymbol\psi_* })=({\bf V}_{*\partial \Omega }\boldsymbol\psi_* ,{\mathcal Q}_{*\partial \Omega }^{s}\boldsymbol\psi_* )\in {\mathcal H}^{1}({\mathbb R}^n)^n\times L_2({\mathbb R}^n)$ be the solution of the problem \eqref{var-Brinkman-transmission-sl-adj} with datum $\boldsymbol\psi_* $, i.e., the single layer velocity and pressure potentials with density $\boldsymbol\psi_* $ and corresponding the adjoint Stokes system (see Definition \ref{s-l-S-variational-variable-adj})}. Then by formulas \eqref{jump-conormal-derivative-1} and \eqref{jump-dl-v-var-alpha-cd}, 
\begin{align}
\label{dl-identity-var}
0=\langle [{\bf t}({\bf W}_{\partial \Omega }\boldsymbol\varphi ,{\mathcal Q}_{\partial \Omega }^d\boldsymbol\varphi )],\gamma {{\bf V}_{*\partial \Omega }\boldsymbol\psi_* }\rangle _{\partial \Omega }
=\left \langle a_{ij}^{\alpha \beta }E_{j\beta }\left({\bf W}_{\partial \Omega }\boldsymbol\varphi \right),
E_{i\alpha }\left({{\bf V}_{*\partial \Omega }\boldsymbol\psi_* }\right)\right \rangle _{{\mathbb R}^n\setminus \partial \Omega }.
\end{align}
Moreover, the Green formula {\eqref{Green-particular-p-adj} for the adjoint Stokes system} and equality \eqref{dl-identity-var} yield that
\begin{align}\label{sl-dl-var}
&{\left\langle {\bf t}^{*+}\left({\bf V}_{*\partial \Omega }\boldsymbol\psi_*,
{\mathcal Q}_{*\partial \Omega }^{s}\boldsymbol\psi_* \right),
\gamma _{+}({\bf W}_{\partial \Omega }\boldsymbol\varphi )\right\rangle _{\partial \Omega }
={\left \langle a_{ij}^{\alpha \beta }E_{i\alpha}\left({\bf V}_{*\partial \Omega }\boldsymbol\psi_* \right),
E_{j\beta}\left({\bf W}_{\partial \Omega }\boldsymbol\varphi \right)\right \rangle _{\Omega _+}}}\nonumber\\
&=-\left \langle a_{ij}^{\alpha \beta }E_{i\alpha}\left({\bf V}_{\partial \Omega }^*\boldsymbol\psi_* \right),
E_{j\beta}\left({\bf W}_{\partial \Omega }\boldsymbol\varphi \right)\right \rangle _{\Omega _-}
={\left\langle {\bf t}^{*-}\left({\bf V}_{*\partial \Omega }\boldsymbol\psi ,{\mathcal Q}^{s}_{*\partial \Omega }\boldsymbol\psi_* \right),\gamma _{-}\left({\bf W}_{\partial \Omega }\boldsymbol\varphi \right)\right\rangle _{\partial \Omega }}\,.
\end{align}
{\rd Therefore, we obtain the equality}
\begin{align}
\label{sl-dl-var-new1}
{\rd \left\langle {\bf t}^{*+}\left({\bf V}_{*\partial \Omega }\boldsymbol\psi_*,
{\mathcal Q}_{*\partial \Omega }^{s}\boldsymbol\psi_* \right),
\gamma _{+}({\bf W}_{\partial \Omega }\boldsymbol\varphi )\right\rangle _{\partial \Omega }
={\left\langle {\bf t}^{*-}\left({\bf V}_{*\partial \Omega }\boldsymbol\psi ,{\mathcal Q}^{s}_{*\partial \Omega }\boldsymbol\psi_* \right),\gamma _{-}\left({\bf W}_{\partial \Omega }\boldsymbol\varphi \right)\right\rangle _{\partial \Omega }}\,.}
\end{align}
Then the {second formula \eqref{sl-identities-var-1-adjoint}},
the first formula \eqref{jump-dl-v-var-alpha} and formula \eqref{sl-dl-var-new1} lead to the equality
\begin{align}
\label{sl-dl-var-new2}
{\rd \left\langle \frac{1}{2}\boldsymbol\psi_*
+{\boldsymbol{\mathcal K}_{*\partial\Omega}}\boldsymbol\psi_* ,
-\frac{1}{2}\boldsymbol\varphi +{\bf K}_{\partial\Omega }\boldsymbol\varphi
\right\rangle _{\partial \Omega }
=\left\langle -\frac{1}{2}\boldsymbol\psi_*
+{\boldsymbol{\mathcal K}_{*\partial\Omega}}\boldsymbol\psi_* ,\frac{1}{2}\boldsymbol\varphi +{\bf K}_{\partial\Omega }\boldsymbol\varphi
\right\rangle _{\partial \Omega }}\,,
\end{align}
and hence to equality \eqref{transpose-dl-var-sm}, as asserted.
\end{proof}
\begin{rem}
{If the operator ${\mathbb L}$ is self-adjoint, i.e., $A^{*\alpha \beta}=A^{\beta \alpha}$,
$a^{\beta\alpha}_{ji}=a^{\alpha\beta}_{ij}$, $\alpha,\beta,i,j =1,\ldots ,n$, see \eqref{2.45},
and particularly in the isotropic case \eqref{isotropic}, then Definition $\ref{s-l-S-variational-variable-adj}$ reduces to Definition $\ref{s-l-S-variational-variable}$ and the operator $\boldsymbol{\mathcal K}_{*\partial \Omega }:H^{-\frac{1}{2}}(\partial \Omega )^n\to H^{-\frac{1}{2}}(\partial \Omega )^n$ given by \eqref{slp-S-oper-var-adj} {coincides with
$\boldsymbol{\mathcal K}_{\partial\Omega}:H^{-\frac{1}{2}}(\partial \Omega )^n\!\to \!H^{-\frac{1}{2}}(\partial \Omega )^n$} given by \eqref{slp-S-oper-var}.}
\end{rem}
\subsubsection{\bf Invertibility of the operator ${\bf D}_{\partial\Omega }$}
Let $\boldsymbol{\mathcal R}$ be the set of rigid body motion fields in ${\mathbb R}^n$, 
\begin{align}
\boldsymbol{\mathcal R}:=\left\{{\bf b}+{\bf B}{\bf x}: {\bf b}\in {\mathbb R}^n \mbox{ and } {\bf B}\in {\mathbb R}^{n\times n} \mbox{ such that } {\bf B}=-{\bf B}^\top\right\},\ \boldsymbol{\mathcal R}_{\partial \Omega }:=\gamma \boldsymbol{\mathcal R},
\label{E3.74}
\\
\label{R-perp}
\boldsymbol{\mathcal R}_{\partial \Omega }^\perp :=\left\{\boldsymbol\Psi \in H^{-\frac{1}{2}}(\partial\Omega )^n: \langle \boldsymbol\Psi ,{\bf r}\rangle _{\partial \Omega }=0 \quad \forall \, {\bf r}\in \boldsymbol{\mathcal R}_{\partial \Omega }\right\}.
\end{align}
Also let $H_{\boldsymbol{\mathcal R}}^{\frac{1}{2}}(\partial\Omega )^n$ be the closed subspace of $H^{\frac{1}{2}}(\partial\Omega )^n$ defined by
\begin{align}
\label{HR}
H_{\boldsymbol{\mathcal R}}^{\frac{1}{2}}(\partial\Omega )^n:=\left\{\boldsymbol\varphi \in H^{\frac{1}{2}}(\partial\Omega )^n :\int_{\partial \Omega } \boldsymbol\varphi \cdot {\bf r}d\sigma =0 \quad \forall \, {\bf r}\in \boldsymbol{\mathcal R}_{\partial \Omega }\right\}\,.
\end{align}
It is easy to see that
\begin{align}\label{Er}
{\mathbb E}({\bf r})=0,\quad {\rm{div}}\,{\bf r}=0\quad\forall\ {\bf r}\in\boldsymbol{\mathcal R}.
\end{align}

Next we show the isomorphism property of the operator ${\bf D}_{\partial\Omega }$ defined in \eqref{jump-dl-v-var-alpha-cd} (cf. \cite[Lemma 3.17]{K-M-W-2} for a different structure of the kernel and range of the similar operator when ${\mathbb A}$ is an $L_\infty $ strongly elliptic {viscosity coefficient tensor}, and \cite[Propositions 6.4 and 6.5]{Sa-Se} for the Stokes system with constant coefficients).
\begin{lem}
\label{isom-dlc}
Let conditions \eqref{Stokes-1}-\eqref{mu} hold.
Then
\begin{align}
\label{kernel-range-dl}
&{\rm{Ker }}\left\{{\bf D}_{\partial\Omega }:H^{\frac{1}{2}}(\partial\Omega )^n\to H^{-\frac{1}{2}}(\partial\Omega )^n\right\}=\boldsymbol{\mathcal R}_{\partial \Omega }\,,\\
\label{range-dl}
&{\bf D}_{\partial\Omega }\boldsymbol\varphi \in \boldsymbol{\mathcal R}_{\partial \Omega }^\perp  \quad \forall\, \boldsymbol\varphi \in H^{\frac{1}{2}}(\partial\Omega )^n \,,
\end{align}
and the following operator is an isomorphism,
\begin{align}
\label{sl-v-isom}
{\bf D}_{\partial\Omega }:H_{\boldsymbol{\mathcal R}}^{\frac{1}{2}}(\partial\Omega )^n\to \boldsymbol{\mathcal R}_{\partial \Omega }^\perp \,.
\end{align}
\end{lem}
\begin{proof}
$(i)$
First, we show formula \eqref{kernel-range-dl}.
Let us assume that $\boldsymbol\varphi \in H^{\frac{1}{2}}(\partial\Omega )^n$ satisfies the equation ${\bf D}_{\partial\Omega }\boldsymbol\varphi ={\bf 0}$ on $\partial \Omega $.
Let ${\bf u}_{\boldsymbol\varphi} :={\bf W}_{\partial\Omega }\boldsymbol\varphi $ and $\pi _{\boldsymbol\varphi} :={\mathcal Q}_{\partial\Omega }^d\boldsymbol\varphi $.
Since ${\rm div}\, {\bf u}_{\boldsymbol\varphi}=0$ in $\Omega_\pm$, we have $E_{ii}({\bf u}_{\boldsymbol\varphi})=0$ implying that assumption \eqref{mu} is applicable for $E_{i\alpha}({\bf u}_{\boldsymbol\varphi})$.
According to  Lemma \ref{lem-add1}, the jump relations \eqref{jump-dl-v-var-alpha} and \eqref{jump-dl-v-var-alpha-cd}, and assumption \eqref{mu} we obtain that
\begin{align}
\label{kernel-dlc-D}
0=\langle -{\bf D}_{\partial\Omega }\boldsymbol\varphi ,\boldsymbol\varphi \rangle _{\partial \Omega }
&=\left \langle a_{ij}^{\alpha \beta }E_{j\beta }\left({\bf u}_{\boldsymbol\varphi} \right),
E_{i\alpha }\left({\bf u}_{\boldsymbol\varphi} \right)\right \rangle _{\Omega _+}
+\left \langle a_{ij}^{\alpha \beta }E_{j\beta }\left({\bf u}_{\boldsymbol\varphi} \right),
E_{i\alpha }\left({\bf u}_{\boldsymbol\varphi} \right)\right \rangle _{\Omega _-}\nonumber\\
&\ge c_{\mathbb A}^{-1}\left(\|{\mathbb E}({\bf u}_{\boldsymbol\varphi} )\|_{L_2(\Omega _+)^{n\times n}}^2+\|{\mathbb E}({\bf u}_{\boldsymbol\varphi} )\|_{L_2(\Omega _-)^{n\times n}}^2\right)
\end{align}
and accordingly ${\mathbb E}({\bf u}_{\boldsymbol\varphi} )=0$ in $\Omega _\pm$. Then by, e.g., \cite[Lemma 3.1]{Med-AAM-11} there exist a constant ${\bf b}\in {\mathbb R}^n$ and an antisymmetric matrix ${\bf B}\in {\mathbb R}^{n\times n}$ such that ${\bf u}_{\boldsymbol\varphi} ={\bf b}+{\bf B}{\bf x}$ in $\Omega _+$, while the condition ${\mathbb E}({\bf u}_{\boldsymbol\varphi} )=0$ in $\Omega _{-}$ yields that
${\bf u}_{\boldsymbol\varphi}|_{\Omega _-} \in {\mathcal R}|_{\Omega _-}$, and the membership
${\bf u}_{\boldsymbol\varphi} \in\mathcal H^1(\Omega _{-})^n$ implies that ${\bf u}_{\boldsymbol\varphi} ={\bf 0}$ in $\Omega _{-}$.
Then by using again the jump relations \eqref{jump-dl-v-var-alpha} we obtain that $\boldsymbol\varphi =-({\bf b}+{\bf B}{\bf x})|_{\partial \Omega }$. This relation shows that
\begin{align}
\label{kernel-dl-1}
{\rm{Ker }}\, {\bf D}_{\partial\Omega }\subseteq \boldsymbol{\mathcal R}_{\partial \Omega }\,.
\end{align}
Now let ${\bf r}\in \boldsymbol{\mathcal R}$ and let ${\bf u}_{\bf r}$ and $\pi _{\bf r}$ be the fields given by
\begin{align}\label{E3.79}
{\bf u}_{\bf r}:=
\left\{\begin{array}{ll}
-{\bf r} & \mbox{ in } \Omega _+\\
\, {\bf 0} & \mbox{ in } \Omega _- ,
\end{array}
\right. \mbox{ and }\
\pi _{\bf r}=0 \mbox{ in } {\mathbb R}^n\,.
\end{align}
By \eqref{Er}, ${\mathbb E}({\bf u}_{\bf r})=0$ and ${\rm{div}}\, {\bf u}_{\bf r}=0$ in ${\mathbb R}^n\setminus \partial \Omega $, and hence, in view of Lemma \ref{lem-add1},
\begin{align}
{\pm \langle {\bf t}^\pm ({\bf u}_{\bf r},\pi _{\bf r}),\gamma _\pm {\bf v}_\pm \rangle _{\partial \Omega }=0 \quad \forall \, {\bf v}_\pm \in {\mathcal H}^1(\Omega _\pm )^n}\,,
\end{align}
which show that ${\bf t}^\pm ({\bf u}_{\bf r},\pi _{\bf r})={\bf 0}$, and accordingly that $[{\bf t}({\bf u}_{\bf r},\pi _{\bf r})]={\bf 0}$ on $\partial \Omega $. Moreover, we have that $[\gamma {\bf u}_{\bf r}]={-{\bf r}|_{\partial \Omega }}$ on $\partial \Omega $. Consequently, the pair $({\bf u}_{\bf r},\pi _{\bf r})$ belongs to ${\mathcal H}^1({\mathbb R}^n\setminus \partial \Omega )^n\times L_2({\mathbb R}^n)$ and satisfies the transmission problem \eqref{transmission-B-dl-var} with given boundary datum ${\bf r}|_{\partial \Omega }\in H^{\frac{1}{2}}(\partial \Omega )^n$. Then Definition \ref{d-l-variational-var} yields that ${\bf W}_{\partial \Omega }({\bf r}|_{\partial \Omega })={\bf u}_{\bf r}$ and ${\mathcal Q}^d_{\partial \Omega }({\bf r}|_{\partial \Omega })=0$ in ${\mathbb R}^n\setminus \partial \Omega $, and by formula \eqref{jump-dl-v-var-alpha-cd} we obtain that
${\bf D}_{\partial\Omega }({\bf r}|_{\partial \Omega })={\bf 0}$ on $\partial \Omega $. Therefore,
\begin{align}
\label{kernel-dl-2}
\boldsymbol{\mathcal R}_{\partial \Omega }\subseteq {\rm{Ker }}\, {\bf D}_{{\mathbb A} ;\partial\Omega }\,.
\end{align}
Relations \eqref{kernel-dl-1} and \eqref{kernel-dl-2} imply \eqref{kernel-range-dl}.

Now let $\boldsymbol\varphi \in H^{\frac{1}{2}}(\partial\Omega )^n$. By applying the Green formula \eqref{Green-particular-p} to the pair $({\bf W}_{\partial \Omega }\boldsymbol\varphi ,{\mathcal Q}_{\partial\Omega }^d\boldsymbol\varphi )$ and by using relation \eqref{jump-dl-v-var-alpha-cd}  along with \eqref{Er},
we obtain the formula
\begin{align}
{\langle {\bf D}_{\partial \Omega }\boldsymbol\varphi ,\gamma _+{\bf r}\rangle _{\partial \Omega }=\left\langle a_{ij}^{\alpha \beta }E_{j\beta }({\bf W}_{\partial \Omega }\boldsymbol\varphi ),E_{i\alpha }({\bf r})\right\rangle _{\Omega _+}-\left\langle {\mathcal Q}_{\partial\Omega }^d\boldsymbol\varphi ,{\rm{div}}\, {\bf r}\right\rangle _{\Omega _+}}=0 \quad \forall \, {\bf r}\in \boldsymbol{\mathcal R}\,,
\end{align}
implying formula \eqref{range-dl}.

$(ii)$
To prove that operator \eqref{sl-v-isom} is an isomorphism,
we show that there exists a constant ${\mathcal C}={\mathcal C}(\partial \Omega ,c_{\mathbb A},n)>0$ such that
\begin{align}
\label{coercive-dlc}
\langle -{\bf D}_{\partial \Omega }\boldsymbol\varphi ,\boldsymbol\varphi \rangle _{\partial \Omega}\geq {\mathcal C}\|\boldsymbol\varphi \|_{H^{\frac{1}{2}}(\partial \Omega )^n}^2 \quad \forall \, \boldsymbol\varphi \in H_{\boldsymbol{\mathcal R}}^{\frac{1}{2}}(\partial\Omega )^n
\end{align}
(cf. \cite[Proposition 6.5]{Sa-Se} in the constant-coefficient Stokes system).
Indeed, by applying Lemma \ref{lem-add1} to the pair $({\bf u}_{\boldsymbol\varphi} ,\pi _{\boldsymbol\varphi} ):=({\bf W}_{\partial \Omega }\boldsymbol\varphi ,{\mathcal Q}_{\partial\Omega }^d\boldsymbol\varphi )$ with $\boldsymbol\varphi \in H_{\boldsymbol{\mathcal R}}^{\frac{1}{2}}(\partial\Omega )^n$, and using the jump relations \eqref{jump-dl-v-var-alpha} and {condition} \eqref{mu}, 
we obtain the inequality
\begin{align}
\label{coercive-dlc-5}
\langle -{\bf D}_{\partial\Omega }\boldsymbol\varphi ,\boldsymbol\varphi \rangle _{\partial \Omega }
\geq c_{\mathbb A}^{-1}\left(\|{\mathbb E}({\bf u}_{\boldsymbol\varphi})\|_{L_2(\Omega _+)^{n\times n}}^2+\|{\mathbb E}({\bf u}_{\boldsymbol\varphi})\|_{L_2(\Omega _-)^{n\times n}}^2\right)\,.
\end{align}
In addition, the continuity of the trace operators $\gamma _\pm :{\mathcal H}^1(\Omega _\pm )^n\to H^{\frac{1}{2}}(\partial \Omega )^n$ and the jump formulas \eqref{jump-dl-v-var-alpha} imply that there exists a constant ${\mathcal C}_1={\mathcal C}_1(\partial \Omega ,c_{\mathbb A},n)>0$ such that
\begin{align}
\label{coercive-dlc-1}
\|\boldsymbol\varphi \|_{H^{\frac{1}{2}}(\partial \Omega )^n}^2=\left\|[\gamma {\bf u}_{\boldsymbol\varphi} ]\right\|_{H^{\frac{1}{2}}(\partial \Omega )^n}^2
\leq {\mathcal C}_1\left(\|{\bf u}_{\boldsymbol\varphi}\|_{{\mathcal H}^1(\Omega _{+})^n}^2+\|{\bf u}_{\boldsymbol\varphi}\|_{{\mathcal H}^1(\Omega _{-})^n}^2\right)={\mathcal C}_1\|{\bf u}_{\boldsymbol\varphi}\|_{{\mathcal H}^1({\mathbb R}^n\setminus \partial \Omega )^n}^2\,.
\end{align}
Now let $\left\{{\bf r}_j:j=1,\ldots ,{n(n+1)}/{2}\right\}$ be a basis of the $n(n+1)/2$-dimensional space $\boldsymbol{\mathcal R}$. Then the formula
\begin{align}
\label{coercive-dlc-2}
\|{\bf w}\|_{1;\rho ;{\mathbb R}^n\setminus \partial \Omega }^2:=\|{\mathbb E}({\bf w})\|_{L_2({\mathbb R}^n\setminus \partial \Omega )^{n\times n}}^2+\sum_{j=1}^{n(n+1)/2}\left|\int_{\partial \Omega }[\gamma {\bf w}]\cdot {\bf r}_jd\sigma \right |^2 \quad \forall \, {\bf w}\in {\mathcal H}^1({\mathbb R}^n\setminus \partial \Omega )^n
\end{align}
defines a norm on the space ${\mathcal H}^1({\mathbb R}^n\setminus \partial \Omega )^n$, which is equivalent to the norm $\|\cdot \|_{{\mathcal H}^1({\mathbb R}^n\setminus \partial \Omega )^n}$ (see Lemma \ref{equiv-norm-Sobolev}, cf. also \cite[p.78]{Sa-Se} for $n=3$). Therefore, there exists a constant ${\mathcal C}_2>0$ such that
\begin{align}
\label{coercive-dlc-3}
\|{\bf w}\|_{{\mathcal H}^1({\mathbb R}^n\setminus \partial \Omega )^n}\leq {\mathcal C}_2\|{\bf w}\|_{1;\rho ;{\mathbb R}^n\setminus \partial \Omega } \quad \forall \, {\bf w}\in {\mathcal H}^1({\mathbb R}^n\setminus \partial \Omega )^n\,.
\end{align}

Now, by considering ${\bf w}={\bf u}_{\boldsymbol\varphi} $ in \eqref{coercive-dlc-2} and by using the jump formulas \eqref{jump-dl-v-var-alpha}, and the assumption that $\boldsymbol\varphi \in H_{\boldsymbol{\mathcal R}}^{\frac{1}{2}}(\partial\Omega )^n$, as well as inequality \eqref{coercive-dlc-3}, we obtain that
\begin{align}
\label{coercive-dlc-4}
\|{\mathbb E}({\bf u}_{\boldsymbol\varphi})\|_{L_2(\Omega _+)^{n\times n}}^2+\|{\mathbb E}({\bf u}_{\boldsymbol\varphi})\|_{L_2(\Omega _-)^{n\times n}}^2=\|{\bf u}_{\boldsymbol\varphi} \|_{1;\rho ;{\mathbb R}^n\setminus \partial \Omega }^2
\geq {{\mathcal C}_2^{-2}}\|{\bf u}_{\boldsymbol\varphi} \|_{{\mathcal H}^1({\mathbb R}^n\setminus \partial \Omega )^n}^2\,.
\end{align}

Finally, by exploiting inequalities \eqref{coercive-dlc-5}, \eqref{coercive-dlc-1} and \eqref{coercive-dlc-4} we obtain the coercivity inequality \eqref{coercive-dlc} with the constant ${\mathcal C}=c_{\mathbb A}^{-1}{\mathcal C}_1^{-1}{\mathcal C}_2^{-2}$. Then the Lax-Milgram Lemma and the isomorphic identification of the dual of the space $H_{\boldsymbol{\mathcal R}}^{\frac{1}{2}}(\partial\Omega )^n$ with $\boldsymbol{\mathcal R}_{\partial \Omega }^\perp$, 
imply that operator \eqref{sl-v-isom} is an isomorphism, as asserted.
\end{proof}

\subsection{Poisson problems of transmission type for anisotropic Stokes system in $\mathbb R^n$. 
Potential approach.}

For given data $\tilde{\bf f}_\pm,\tilde{\bf f}_-,g_+,g_-,\boldsymbol\varphi,\boldsymbol\psi$,
we consider the following Poisson problem of transmission type, 
\begin{equation}
\label{Dirichlet-var-Stokes-tran}
\left\{
\begin{array}{ll}
{\boldsymbol{\mathcal L}({\bf u}_\pm ,\pi _\pm)={\tilde{\bf f}_{\pm}}|_{\Omega _{\pm }}\,,} \ \
{\rm{div}} \, {\bf u}_\pm = g_\pm & \mbox{ in } \Omega _{\pm }\,,\
\\
{\gamma }_{+}{\bf u}_+-{\gamma }_{-}{\bf u}_{-}=\boldsymbol\varphi &  \mbox{ on } \partial \Omega ,\\
{\bf t}^{+}({\bf u}_+,\pi _+;\tilde{\bf f}_+)-{\bf t}^{-}({\bf u}_-,\pi _- ;\tilde{\bf f}_-)=\boldsymbol\psi &  \mbox{ on } \partial \Omega \,.
\end{array}\right.
\end{equation}
where $\boldsymbol{\mathcal L}$ denotes the Stokes operator defined in \eqref{Stokes-new}.
The left-hand side in the last transmission condition in \eqref{Dirichlet-var-Stokes-tran} is understood in the sense of Definition~\ref{conormal-derivative-var-Brinkman}.

\begin{thm}
\label{T-2-tran}
Let conditions \eqref{Stokes-1}-\eqref{mu} hold.
Then for all given data $(\tilde{\bf f}_+,\tilde{\bf f}_-,g_+,g_-,\boldsymbol\varphi,\boldsymbol\psi)$ in the space
$\widetilde{H}^{-1}(\Omega _+)^n\times \widetilde{\mathcal H}^{-1}(\Omega_-)^n\times
L_2(\Omega _+)\times L_2(\Omega _-) \times H^{\frac{1}{2}}(\partial\Omega)^n\times H^{-\frac{1}{2}}(\partial\Omega)^n$,
transmission problem \eqref{Dirichlet-var-Stokes-tran} has a unique solution
$({\bf u}_+,\pi _+,{\bf u}_-,\pi _-)\in
{H}^1(\Omega _+)^n\times L_2(\Omega _+)\times {\mathcal H}^1(\Omega _-)^n\times L_2(\Omega _-)$.
Moreover, there exists a constant $C=C(\partial \Omega , c ,n)>0$ such that
\begin{multline}
\label{solution-transm-tran}
\|{\bf u}_+\|_{{H}^1(\Omega _+)^n}+\|\pi _+\|_{L_2(\Omega_+)}+\|{\bf u}_-\|_{{\mathcal H}^1(\Omega _-)^n}+\|\pi _-\|_{L_2(\Omega _-)}\leq
C\Big(\|\tilde{\bf f}_+\|_{\widetilde{H}^{-1}(\Omega _{+})^n}
+\|\tilde{\bf f}_-\|_{\widetilde{\mathcal H}^{-1}(\Omega _{-})^n}\\
+\|g_+\|_{L_2(\Omega_+)}+\|g_-\|_{L_2(\Omega_-)}
+\|\boldsymbol\varphi\|_{H^{\frac{1}{2}}(\partial \Omega )^n}
+\|\boldsymbol\psi\|_{H^{-\frac{1}{2}}(\partial \Omega )^n}\Big).
\end{multline}
\end{thm}
\begin{proof}
Theorem \ref{slp-var-apr-1-p} yields uniqueness. Now we show existence, by considering the potentials
\begin{align}
\label{existence-linear-a}
&{\bf u}_\pm =\big(\boldsymbol{\mathcal N}_{\mathbb R^n}\tilde{\bf f}_\pm \big)|_{\Omega_\pm}
+\big(\boldsymbol{\mathcal G}_{\mathbb R^n}\mathring E_\pm g_\pm \big)|_{\Omega_\pm}
+{\bf V}_{\partial\Omega}\boldsymbol\psi^0-{\bf W}_{\partial\Omega}\boldsymbol\varphi^0
\quad \mbox{in } \Omega _\pm ,\\
\label{existence-linear-b}
&\pi _\pm =\big({\mathcal Q}_{\mathbb R^n}\tilde{\bf f}_\pm \big)|_{\Omega _\pm }
+\big({\mathcal G}^0_{\mathbb R^n}\mathring E_\pm g_\pm \big)|_{\Omega_\pm}
+{\mathcal Q}^s_{\partial\Omega}\boldsymbol\psi^0
-{\mathcal Q}^d_{\partial\Omega}\boldsymbol\varphi^0\quad \mbox{in } \Omega _\pm ,
\end{align}
where
\begin{align*}
\boldsymbol\varphi^0:=\boldsymbol\varphi
&-\gamma _+\boldsymbol{\mathcal N}_{\mathbb R^n}\tilde{\bf f}_+
+\gamma _-\boldsymbol{\mathcal N}_{\mathbb R^n}\tilde{\bf f}_-
-\gamma _+\boldsymbol{\mathcal G}_{\mathbb R^n}\mathring E_+ g_+
+\gamma _-\boldsymbol{\mathcal G}_{\mathbb R^n}\mathring E_- g_-,\\
\boldsymbol\psi^0:=\boldsymbol\psi
&-{\bf t}^+\big(\boldsymbol{\mathcal N}_{\mathbb R^n}\tilde{\bf f}_+,
{\mathcal Q}_{\mathbb R^n}\tilde{\bf f}_+;\tilde{\bf f}_+\big)
+{\bf t}^{-}\big(\boldsymbol{\mathcal N}_{\mathbb R^n}\tilde{\bf f}_-,
{\mathcal Q}_{\mathbb R^n}\tilde{\bf f}_-;\tilde{\bf f}_-\big)\\
&-{\bf t}^+\big(\boldsymbol{\mathcal G}_{\mathbb R^n}\mathring E_+ g_+,
{\mathcal G}^0_{\mathbb R^n}\mathring E_+ g_+\big)
+{\bf t}^{-}\big(\boldsymbol{\mathcal G}_{\mathbb R^n}\mathring E_- g_-,
{\mathcal G}^0_{\mathbb R^n}\mathring E_- g_-\big).
\end{align*}
Note that $\boldsymbol\varphi^0\in H^{\frac{1}{2}}(\partial \Omega )^n$ and $\boldsymbol\psi^0\in H^{-\frac{1}{2}}(\partial \Omega )^n$.
From Lemmas \ref{Newtonian-B-var-1}, \ref{continuity-sl-S-h-var} and \ref{continuity-dl-h-var}, we deduce that $({\bf u}_\pm,\pi_\pm)$ given in \eqref{existence-linear-a}, \eqref{existence-linear-b} provide a solution of the transmission problem \eqref{Dirichlet-var-Stokes-tran} in the space $({H}^1(\Omega _+)^n\times L_2(\Omega _+))\times ({\mathcal H}^1(\Omega _-)^n\times L_2(\Omega _-))$ satisfying inequality \eqref{solution-transm-tran}.
\end{proof}

\subsection{ \bf The third Green identities for anisotropic Stokes system}
Next we prove { the representation formulas (the third Green identities) for solutions of} the anisotropic Stokes system with $L_{\infty }$ coefficient tensor ({cf.} \cite[Proposition 6.8]{Sa-Se} for the homogeneous Stokes system {in case \eqref{isotropic} with $\mu =1$,  {$\lambda=0$}, and $n=3$,
and \cite[Theorem 6.10]{Lean} for the strongly elliptic systems with smooth coefficients}).
{They can be employed, e.g., for reduction of the boundary and transmission problems to {\em direct} boundary (integral) equations.}

\begin{thm}
\label{int-repres-formula}
Let conditions \eqref{Stokes-1}-\eqref{mu} hold and let $\boldsymbol{\mathcal L}$ denote the Stokes operator defined in \eqref{Stokes-new}.
Let ${\bf u}_+\in {H}^1(\Omega_+)^n$, ${\bf u}_-\in {\mathcal H}^1(\Omega_-)^n$ and $\pi_\pm\in L_2(\Omega_\pm)$ satisfy 
the Stokes system
\begin{align}
\label{int-repres-f1}
\boldsymbol{\mathcal L}({\bf u}_\pm,\pi_\pm)=\tilde{\mathbf f}_\pm|_{\Omega_\pm},\quad
{\rm{div}}\, {\bf u}_\pm= g_\pm &\quad \mbox{in } \Omega_\pm
\end{align}
for some
$\tilde{\mathbf f}_+\in \widetilde H^1(\Omega_+)^n$, $\tilde{\mathbf f}_-\in \widetilde {\mathcal H}^1(\Omega_-)^n$, $g_+\in L_2(\Omega _+)$, $g_-\in L_2(\Omega _-)$.
Let $\tilde{\bf f}:=\tilde{\mathbf f}_+ + \tilde{\mathbf f}_-$, $g:=\mathring E_+g_+ + \mathring E_-g_-$.
Then the following representations in terms of jumps hold,
\begin{align}
\label{int-repres-f2u}
&{\bf u}_\pm={-}{\bf W}_{\partial \Omega }[\gamma {\bf u}]
{+}{\bf V}_{\partial \Omega }\big[{\bf t}({\bf u},\pi;\tilde{\bf f})\big]
+\boldsymbol{\mathcal N}_{{\mathbb R}^n}\tilde{\bf f}
+\boldsymbol{\mathcal G}_{{\mathbb R}^n}g\,
\mbox{ in }\, \Omega_\pm \,,\\
\label{int-repres-f2p}
&\pi_\pm ={-}{\mathcal Q}^d_{\partial\Omega}[\gamma {\bf u}]
{+}{\mathcal Q}^s_{\partial \Omega }\big[{\bf t}({\bf u},\pi;\tilde{\bf f})\big]
+\mathcal Q_{{\mathbb R}^n}\tilde{\bf f}
+{\mathcal G}^0_{{\mathbb R}^n}g\,
\mbox{ in }\, \Omega_\pm \,.
\end{align}

Moreover, the following single-side representations also hold,
\begin{align}
\label{int-repres-f2us}
&{\bf u}_\pm=\mp{\bf W}_{\partial \Omega }\gamma_\pm {\bf u}_\pm
\pm{\bf V}_{\partial \Omega }{\bf t}^\pm({\bf u}_\pm,\pi_\pm;\tilde{\bf f}_\pm)
+\boldsymbol{\mathcal N}_{{\mathbb R}^n}\tilde{\bf f}_\pm
+\boldsymbol{\mathcal G}_{{\mathbb R}^n}\mathring E_\pm g_\pm\,
\mbox{ in }\, \Omega_\pm \,,\\
\label{int-repres-f2ps}
&\pi_\pm =\mp{\mathcal Q}^d_{\partial \Omega }\gamma_\pm{\bf u}_\pm
\pm{\mathcal Q}^s_{\partial \Omega }{\bf t}^\pm({\bf u}_\pm,\pi_\pm;\tilde{\bf f}_\pm)
+\mathcal Q_{{\mathbb R}^n}\tilde{\bf f}_\pm
+{\mathcal G}^0_{{\mathbb R}^n}\mathring E_\pm g_\pm\,
\mbox{ in }\, \Omega_\pm \,.
\end{align}
\end{thm}
\begin{proof}
Due to the assumptions on ${\bf u}_\pm$, $\pi_\pm$ and $\tilde{\mathbf f}_\pm$, we have the inclusions $\boldsymbol{\boldsymbol\varphi} :={[\gamma {\bf u}]}\in H^{\frac{1}{2}}(\partial \Omega )^n$ and
$\boldsymbol\psi :=\big[{\bf t}({\bf u},\pi;\tilde{\mathbf f})\big]\in H^{-\frac{1}{2}}(\partial \Omega )^n$. Let
\begin{align}
\label{int-repres-f3}
{\bf v}:=-{\bf W}_{\partial \Omega }\boldsymbol{\boldsymbol\varphi}
+{\bf V}_{\partial \Omega }\boldsymbol\psi
+\boldsymbol{\mathcal N}_{{\mathbb R}^n}\tilde{\bf f}
+\boldsymbol{\mathcal G}_{{\mathbb R}^n}g,\
q:=-{\mathcal Q}^d_{\partial \Omega }\boldsymbol{\boldsymbol\varphi}
+{\mathcal Q}^s_{\partial \Omega }\boldsymbol\psi
+{\mathcal Q}_{{\mathbb R}^n}\tilde{\bf f}
+{\mathcal G}^0_{{\mathbb R}^n}g
\quad \mbox{in } {\mathbb R}^n\setminus \partial\Omega.
\end{align}
By definitions of the potentials and according to Lemmas \ref{Newtonian-B-var-1}, \ref{continuity-sl-S-h-var} and \ref{continuity-dl-h-var}(i), the pair $({\bf v},q)$ belongs to the space
${\mathcal H}^1({\mathbb R}^n\setminus \partial \Omega )^n\times L_2({\mathbb R}^n)$
and satisfies the Stokes system
\begin{align}
\label{int-repres-f4}
\boldsymbol{\mathcal L}({\bf v}_\pm,q_\pm)=\tilde{\bf f},\quad
{\rm{div}}\, {\bf v}_\pm=g \quad & \mbox{in } \Omega_\pm.
\end{align}
Due to Lemma \ref{Newtonian-B-var-1},
$\boldsymbol{\mathcal N}_{{\mathbb R}^n}\tilde{\bf f},\boldsymbol{\mathcal G}_{{\mathbb R}^n}g\in
{\mathcal H}^1({\mathbb R}^n)^n$
implying
$[\gamma\boldsymbol{\mathcal N}_{{\mathbb R}^n}\tilde{\bf f}]=\bf 0$,
$[\gamma\boldsymbol{\mathcal G}_{{\mathbb R}^n}g]=\bf 0$.
Then by formulas \eqref{sl-identities-var} and  \eqref{jump-dl-v-var-alpha},
\begin{align}
\label{int-repres-f4u}
[\gamma {\bf v}]=\boldsymbol{\boldsymbol\varphi} & \mbox{ on } \partial \Omega .
\end{align}

Let $r_{\Omega_\pm}$ be restriction operators to $\Omega _\pm $, i.e., $r_{\Omega_\pm}g:=g|_{\Omega_\pm}$.
By Definition~\ref{conormal-derivative-var-Brinkman}, the generalized conormal derivative is linear with respect to the triple of its arguments, implying that
\begin{multline}\label{t-strong1}
\mathbf t^\pm({\bf v}|_{\Omega_\pm},q|_{\Omega_\pm};\tilde{\bf f}_\pm)
=-\mathbf t^\pm(r_{\Omega_\pm}{\bf W}_{\partial\Omega}\boldsymbol{\boldsymbol\varphi},
r_{\Omega_\pm}{\mathcal Q}^d_{\partial \Omega }\boldsymbol{\boldsymbol\varphi};\mathbf 0)
+\mathbf t^\pm(r_{\Omega_\pm}{\bf V}_{\partial\Omega}\boldsymbol{\psi},
r_{\Omega_\pm}{\mathcal Q}^s_{\partial \Omega }\boldsymbol{\psi};\mathbf 0)
\\
+\mathbf t^\pm(r_{\Omega_\pm}(\boldsymbol{\mathcal N}_{{\mathbb R}^n}\tilde{\bf f}
+\boldsymbol{\mathcal G}_{{\mathbb R}^n}g),
r_{\Omega_\pm}({\mathcal Q}_{\mathbb R^n}\tilde{\bf f}+{\mathcal G}^0_{{\mathbb R}^n}g);\tilde{\bf f}_\pm).
\end{multline}
By formulas \eqref{sl-identities-var-1} and \eqref{jump-dl-v-var-alpha-cd}, we obtain
\begin{align}\label{tVW}
[\mathbf t({\bf V}_{\partial\Omega}\boldsymbol{\psi},
{\mathcal Q}^s_{\partial \Omega }\boldsymbol{\psi};\mathbf 0)]=\boldsymbol\psi,\quad
[\mathbf t({\bf W}_{\partial\Omega}\boldsymbol{\boldsymbol\varphi},
{\mathcal Q}^d_{\partial \Omega }\boldsymbol{\boldsymbol\varphi};{\bf 0})]={\bf 0}.
\end{align}
On the other hand, from \eqref{conormal-derivative-var-Brinkman-3} we have for any $\mathbf w\!\in\!H^{\frac{1}{2}}(\partial\Omega )^n$,
\begin{multline}
\label{conormal-derivative-var-Brinkman-4}
\left\langle [{\bf t}(\boldsymbol{\mathcal N}_{\mathbb R^n}\tilde{\bf f}+\boldsymbol{\mathcal G}_{{\mathbb R}^n}g,
{\mathcal Q}_{\mathbb R^n}\tilde{\bf f}+{\mathcal G}^0_{{\mathbb R}^n}g;\tilde{\bf f})],
\mathbf w\right\rangle _{_{\!\partial\Omega}}
\\
=\left\langle A^{\alpha\beta}\partial_\beta r_{\Omega_+}(\boldsymbol{\mathcal N}_{\mathbb R^n}\tilde{\bf f}
+\boldsymbol{\mathcal G}_{{\mathbb R}^n}g),
\partial_\alpha (\gamma^{-1}\mathbf w) \right\rangle_{\Omega_+}
+\left\langle A^{\alpha\beta}\partial_\beta r_{\Omega_-}(\boldsymbol{\mathcal N}_{\mathbb R^n}\tilde{\bf f}
+\boldsymbol{\mathcal G}_{{\mathbb R}^n}g),
\partial_\alpha (\gamma^{-1}\mathbf w) \right\rangle_{\Omega_-}
\\
-\left\langle r_{\Omega_+}({\mathcal Q}_{\mathbb R^n}\tilde{\bf f}+{\mathcal G}^0_{{\mathbb R}^n}g),
{\rm{div}}(\gamma^{-1}\mathbf w)\right\rangle _{\Omega_+}
-\left\langle r_{\Omega_-}({\mathcal Q}_{\mathbb R^n}\tilde{\bf f}+{\mathcal G}^0_{{\mathbb R}^n}g),
{\rm{div}}(\gamma^{-1}\mathbf w)\right\rangle _{\Omega_-}
+\left\langle {\tilde{\bf f}}_+,\gamma^{-1}\mathbf w\right\rangle_{\Omega_+}
+\left\langle {\tilde{\bf f}}_-,\gamma^{-1}\mathbf w\right\rangle_{\Omega_-}
\\
=\left\langle A^{\alpha \beta }\partial_\beta (\boldsymbol{\mathcal N}_{\mathbb R^n}\tilde{\bf f}
+\boldsymbol{\mathcal G}_{{\mathbb R}^n}g),
\partial_\alpha (\gamma^{-1}\mathbf w) \right\rangle_{\mathbb R^n}
-\left\langle {\mathcal Q}_{\mathbb R^n}\tilde{\bf f}+{\mathcal G}^0_{{\mathbb R}^n}g,
{\rm{div}}(\gamma^{-1}\mathbf w)\right\rangle_{\mathbb R^n}
+\left\langle \tilde{\bf f},\gamma^{-1}\mathbf w\right\rangle_{\mathbb R^n}
\\=\left\langle -\boldsymbol{\mathcal L}(\boldsymbol{\mathcal N}_{{\mathbb R}^n}\tilde{\bf f}
+\boldsymbol{\mathcal G}_{{\mathbb R}^n}g,
{\mathcal Q}_{\mathbb R^n}\tilde{\bf f}+{\mathcal G}^0_{{\mathbb R}^n}g)+\tilde{\bf f},
\gamma^{-1}\mathbf w\right\rangle_{\mathbb R^n}=\mathbf 0,
\end{multline}
where $\gamma^{-1}:H^{\frac{1}{2}}(\partial\Omega )^n\to {\mathcal H}^{1}(\mathbb R^n)^n$ is a $($non-unique$)$ bounded right inverse of the trace operator $\gamma:{\mathcal H}^{1}(\mathbb R^n)^n\to H^{\frac{1}{2}}(\partial\Omega )^n$.
The last equality in \eqref{conormal-derivative-var-Brinkman-4} follows since
$\boldsymbol{\mathcal L}(\boldsymbol{\mathcal N}_{{\mathbb R}^n}\tilde{\bf f}
+\boldsymbol{\mathcal G}_{{\mathbb R}^n}g,
{\mathcal Q}_{\mathbb R^n}\tilde{\bf f}+{\mathcal G}^0_{{\mathbb R}^n}g)=\tilde{\bf f}$ in $\mathbb R^n$.
Combining \eqref{t-strong1}-\eqref{conormal-derivative-var-Brinkman-4}, we obtain that the couple $({\bf v},q)$
satisfies the transmission condition
\begin{align}
\label{int-repres-f4t}
\big[{\bf t}({\bf v},q;\tilde{\bf f})\big]=\boldsymbol\psi  & \mbox{ on } \partial \Omega
\end{align}
and thus the transmission problem \eqref{int-repres-f4}, \eqref{int-repres-f4u}, \eqref{int-repres-f4t}.
The pair $({\bf u},\pi )$ satisfies the same transmission problem, which due to Theorem \ref{T-2-tran} has at most one solution in ${\mathcal H}^1({\mathbb R}^n\setminus \partial \Omega )^n\times L_2({\mathbb R}^n)$.
Consequently, ${\bf u}\!=\!{\bf v}$ and $\pi \!=\!q$, and then formulas \eqref{int-repres-f3} yield the representation formulas \eqref{int-repres-f2u}-\eqref{int-repres-f2p}.

To obtain formulas \eqref{int-repres-f2us}, \eqref{int-repres-f2ps} for $({\bf u}_+,\pi_+)$, we can employ representations \eqref{int-repres-f2u}, \eqref{int-repres-f2p} with
${\bf u}_-=\bf 0$, $\pi_-=0$, $\tilde{\bf f}_-=\bf 0$ and $g_-=0$. Formulas \eqref{int-repres-f2us}, \eqref{int-repres-f2ps} for $({\bf u}_-,\pi_-)$ can be obtained in a similar way.
\end{proof}

\section{Boundary value problems for the anisotropic Stokes system. Variational and potential approaches in weighted Sobolev spaces}

Girault and Sequeira in \cite[Theorem 3.4]{Gi-Se} used a variational approach to show the well-posedness in  ${\mathcal H}^{1}(\Omega ')^n\times L_2(\Omega ')$ for the exterior Dirichlet problem for the constant coefficient Stokes system in an exterior Lipschitz domain $\Omega '$ of ${\mathbb R}^n$, $n=2,3$. Dindo\u{s} and Mitrea \cite[Theorems 5.1, 5.6, 7.1, 7.3]{D-M} used a boundary integral approach and properties of Calder\'{o}n-Zygmund type singular integral operators to show well-posedness results in Sobolev and Besov spaces for Poisson problems of Dirichlet type for the Stokes and Navier-Stokes systems with smooth coefficients in Lipschitz domains on compact Riemannian manifolds (see also \cite[Theorem 7.1]{M-T}, and \cite[Proposition 4.5]{B-H} for an evolutionary exterior Stokes problem).

Recall that $\boldsymbol{\mathcal L}$ is the Stokes operator defined in \eqref{Stokes-new}, such that the corresponding {viscosity coefficient tensor} ${\mathbb A}=\left(A^{\alpha \beta }\right)_{1\leq \alpha ,\beta \leq n}$ satisfies conditions \eqref{Stokes-1}-\eqref{mu}.

\subsection{Exterior Dirichlet problem for the compressible Stokes system}
Let us consider the following
Dirichlet problem for the {anisotropic Stokes system with $L_{\infty}$ coefficients} 
\begin{equation}
\label{Dirichlet-var-Stokes-a}
\left\{
\begin{array}{ll}
{\boldsymbol{\mathcal L}({\bf u},\pi )={\bf f}}, \ \
{\rm{div}} \ {\bf u}=g & \mbox{ in } \Omega _{-},\
\\
{\gamma }_{-}{\bf u}={\boldsymbol\varphi} &  \mbox{ on } \partial \Omega \,,
\end{array}\right.
\end{equation}
where $\boldsymbol{\mathcal L}$ is the Stokes operator defined in \eqref{Stokes-new}
and the given data $({\bf f},{g,}{\boldsymbol\varphi})$ belong to
${\mathcal H}^{-1}(\Omega _{-})^n\times L_2(\Omega_-)\times H^{\frac{1}{2}}(\partial \Omega )^n$.
We show the well-posedness of this problem in the space ${\mathcal H}^1(\Omega _-)^n\times L_2(\Omega _-)$, and
for $g=0$ also
express its
solution in terms of the Newtonian and single layer potentials defined in Section \ref{N-S-D}.

\subsubsection{\bf $L_2$-based weighted Sobolev spaces with vanishing traces}
Recall that $\mathring{\mathcal H}^1(\Omega _-)^n$ is the closure of the space ${\mathcal D}(\Omega _{-})^n$ in ${\mathcal H}^1(\Omega _-)^n$. Equivalently,
\begin{align}
\label{null-trace-forms}
\mathring{\mathcal H}^1({\Omega _-})^n
&=\left\{{\bf v}\in {\mathcal H}^1(\Omega _{-})^n:\gamma_{-}{\bf v}={0} \mbox{ on } \partial \Omega \right\}\nonumber\\
&=\left\{{\bf v}|_{\Omega _{-}}:{\bf v}\in {\mathcal H}^1({\mathbb R}^n)^n \mbox{ and } \gamma {\bf v}={0} \mbox{ a.e. on } \partial \Omega \right\}\,.
\end{align}
The space $\mathring{\mathcal H}^1(\Omega _-)^n$ can be isomorphically identified with the space of vector fields in ${\mathcal H}^1(\Omega _-)^n$ whose extensions by zero belong to the space ${\mathcal H}^1({\mathbb R}^n)^n$, and also with the space
\begin{align}
\label{2.5}
&\widetilde{\mathcal H}^1(\Omega _-)^n:=\left\{\widetilde{f}\in {\mathcal H}^1({\mathbb R}^n)^n:{\rm{supp}}\,\widetilde{f}\subseteq \overline{\Omega }_-\right\}
\end{align}
(cf., e.g., \cite[Lemma 5.10, (5.71)]{B-M-M-M}, \cite[Proposition 3.3]{J-K1}, \cite[Theorem 3.33]{Lean}).
We also need the spaces
\begin{align}
\label{0-1-div}
&\mathring{\mathcal H}_{\rm{div}}^1(\Omega _-)^n:=\left\{{\bf v}\in \mathring{\mathcal H}^1(\Omega _-)^n:{\rm{div}}\, {\bf v}=0 \mbox{ in } \Omega _-\right\},\\
\label{X-p}
&{\mathcal X}(\Omega _-):=\mathring{\mathcal H}^1(\Omega _-)^n\times L_2(\Omega _-),
\end{align}
and the space ${\mathcal X}'(\Omega _-)$ which is the dual of ${\mathcal X}(\Omega _-)$, i.e.,
\begin{align}
\label{0-1-div-prime}
{\mathcal X}'(\Omega _-)={\mathcal H}^{-1}(\Omega _-)^n\times L_{2}(\Omega _-).
\end{align}

The next result plays a key role in the analysis of the Poisson problem with inhomogeneous Dirichlet condition for the Stokes system with $L_{\infty}$ elliptic coefficient tensor in an exterior Lipschitz domain, and follows from 
Theorem 2 of Bogovski\u{\i} \cite{Bogovskii} (see also \cite[Theorem 3.1]{Am-Ciarlet}, \cite[Theorem 3.3]{Borchers-Shor}).
\begin{lem}\label{Bog}
The divergence operator
$
{\rm{div}}:\mathring{\mathcal H}^1(\Omega _{-})^n\to L_2(\Omega _-)
$
is surjective and there exists a linear continuous operator
$R_{\Omega _-}:L_2(\Omega _-)\to \mathring{\mathcal H}^1(\Omega _{-})^n$
such that
$
{\rm{div}}(R_{\Omega _-}g)=g
$
for any $g\in L_2(\Omega _- )$.
\end{lem}

\subsubsection{\bf Variational approach to the Poisson problem for the compressible Stokes system with homogeneous Dirichlet condition}
Let us consider the following problem with $L_{\infty }$ viscosity coefficients tensor,
\begin{equation}
\label{Dirichlet-var-p}
\left\{
\begin{array}{ll}
\boldsymbol{\mathcal L}({\bf u},\pi )={\mathbf f},\ \
{{\rm{div}}}\, {\bf u}=g & \mbox{ in } \Omega _-,\
\\
{\gamma }_{-}{\bf u}={\bf 0} &  \mbox{ on } \partial \Omega ,
\end{array}\right.
\end{equation}
We show that problem \eqref{Dirichlet-var-p} has a unique solution $({\bf u},\pi )$ in ${\mathcal X}(\Omega _-)$ whenever ${\mathbf f}\in{\mathcal H}^{-1}(\Omega_-)^n$, $g\in L_{2}(\Omega_-)$ (see also \cite[Theorem 7.4]{B-M-M-M} for well-posedness of the higher-order inhomogeneous Dirichlet problem, Lemma 3.2 and Corollary 5.3 in \cite{Choi-Lee} devoted to the Stokes system with VMO coefficients in Lipschitz domains with small Lipschitz constant in ${\mathbb R}^n$, \cite[Theorem 5.2]{Maz'ya-Rossmann}, and \cite[Theorems 5.1, 5.6]{D-M} for the Stokes system with smooth coefficients in compact Riemannian manifolds, \cite[Proposition 4.5]{B-H}, \cite[Theorem 6.7]{Barton}).

Let $a_{{{\mathbb A};\Omega _-}}:{\mathring{\mathcal H}^1(\Omega _-)^n\times \mathring{\mathcal H}^{1}(\Omega _-)^n}\to {\mathbb R}$ and $b_{_{\Omega _-}}:\mathring{\mathcal H}^1(\Omega _-)^n\times L_{2}(\Omega _-)\to {\mathbb R}$ be the bilinear forms
\begin{align}
\label{a-Omega-p}
&a_{{{\mathbb A};\Omega _-}}({\bf v},{\bf w}):=\left \langle A^{\alpha \beta }\partial _\beta {\bf v},\partial _\alpha {\bf w}\right \rangle _{\Omega _-}=\left \langle a_{ij}^{\alpha \beta }E_{j\beta }({\bf v}),E_{i\alpha }({\bf w})\right \rangle _{\Omega _-} \quad
\forall \, {\bf v},{\bf w}\in {\mathring{\mathcal H}^{1}(\Omega _-)^n},\\
\label{b-Omega-p}
&b_{_{\Omega_-}}({\bf v},q):=-\langle {\rm{div}}\, {\bf v},q\rangle _{\Omega _-} \quad \forall \, {\bf v}\in  \mathring{\mathcal H}^1(\Omega _-)^n,\ q\in L_{2}(\Omega _-).
\end{align}
\begin{defn}
Let $\boldsymbol{\mathbf f}\in {\mathcal H}^{-1}(\Omega _-)^n$ and $g\in L_2(\Omega_-)$. Then the pair $({\bf u},\pi )\in {\mathcal X}(\Omega _-)$ is a {\it weak solution} of the Poisson problem with homogeneous Dirichlet condition \eqref{Dirichlet-var-p} if
\begin{equation}
\label{variational-Dirichlet-p}
\left\{\begin{array}{ll}
a_{{{\mathbb A};\Omega _-}}({\bf u},{\bf w})+b_{\Omega _-}({\bf w},\pi)
= -\langle \boldsymbol{\mathbf f},{\bf w}\rangle_{\Omega_-}
 &\forall \, {\bf w}\in \mathring{\mathcal H}^{1}(\Omega _-)^n,\\
b_{_{\Omega_-}}({\bf u},q)= -\langle g,q\rangle_{\Omega_-}  &\forall \, q\in L_2(\Omega _-).
\end{array}\right.
\end{equation}
\end{defn}
We note that the dense embedding of the space ${\mathcal D}(\Omega _{-})^n$ in $\mathring{\mathcal H}^1(\Omega _{-})^n$ implies that the exterior Dirichlet problem \eqref{Dirichlet-var-p} has the equivalent variational formulation \eqref{variational-Dirichlet-p}.
Therefore, in order to show the well-posedness of the Poisson problem of Dirichlet type \eqref{Dirichlet-var-p} we
show that the mixed variational formulation \eqref{variational-Dirichlet-p} has a unique solution $({\bf u},\pi )$ in the space $\mathcal X$. To do so, we define the operator
\begin{align}
\label{T}
{\mathcal T}_{\Omega _-}:{\mathcal X}(\Omega _-)\to {\mathcal X}'(\Omega _-),\ {\mathcal T}_{\Omega _{-}}({\bf u},\pi )=\left({\mathcal T}_{1;\Omega _-}({\bf u},\pi ),{\mathcal T}_{2;\Omega _-}({\bf u},\pi )\right)\in {\mathcal X}'(\Omega _-)\,, 
\end{align}
by setting for every $({\bf u},\pi )\in {\mathcal X}(\Omega _-)$ and $({\bf w},q)\in {\mathcal X}(\Omega _-)$,
\begin{equation}
\label{variational-Dirichlet-p-1-1}
\left\{\begin{array}{ll}
\langle {\mathcal T}_{1;\Omega _-}({\bf u},\pi ),{\bf w}\rangle _{\Omega _-}
=a_{{{\mathbb A};\Omega _-}}({\bf u},{\bf w})+b_{\Omega _-}({\bf w},\pi ),\\
\langle {\mathcal T}_{2;\Omega _-}({\bf u},\pi ),q\rangle _{\Omega _-}=b_{\Omega _{-}}({\bf u},q)\,.
\end{array}\right.
\end{equation}
The H\"{o}lder inequality and the boundedness of the bilinear forms in \eqref{variational-Dirichlet-p-1-1} imply that the operator ${\mathcal T}_{\Omega _-}\!:\!{\mathcal X}(\Omega _-)\!\to \!{\mathcal X}'(\Omega _-)$ is linear and bounded. Moreover, we prove the following result (see also \cite[Theorem 7.5]{B-M-M-M} for the well-posedness of the inhomogeneous Poisson problem for higher-order elliptic operators, \cite[Theorem 5.1]{Maz'ya-Rossmann} and \cite[Theorem 3.1]{Ott-Kim-Brown} for the Stokes mixed problem in polyhedral domains or in bounded Lipschitz domains of ${\mathbb R}^2$).
\begin{thm}
\label{T-p}
Let conditions \eqref{Stokes-1}-\eqref{mu} be satisfied in $\Omega_-$.
Then the following assertions hold.
\begin{itemize}
\item[$(i)$]
The operator ${\mathcal T}_{\Omega _-}:{\mathcal X}(\Omega _-)\to {\mathcal X}'(\Omega _-)$ given by \eqref{T} and \eqref{variational-Dirichlet-p-1-1} is an isomorphism.
\item[$(ii)$]
For any $\boldsymbol{\mathbf f}\in {\mathcal H}^{-1}(\Omega _-)^n$
and $g\in L_{2}(\Omega_-)$,
problem
\eqref{Dirichlet-var-p} has a unique solution $({\bf u},\pi )\in {\mathcal X}(\Omega _-)$, and there exists a constant $C\!=\!C(\Omega _-, c ,n)\!>\!0$ such that
\begin{align}
\|({\bf u},\pi )\|_{{\mathcal X}(\Omega _-)}\leq C\left(\|\boldsymbol{\mathbf f}\|_{\mathcal H^{-1}(\Omega_-)^n}
+\|g\|_{L_{2}(\Omega_-)}\right)\,.\nonumber
\end{align}
\end{itemize}
\end{thm}
\begin{proof}
$(i)$ Let $(\mathbf F ,\eta )\in {\mathcal X}'(\Omega _-)$. We show that there exists a unique pair $({\bf u},\pi )\in {\mathcal X}(\Omega _-)$ such that ${\mathfrak T}_{\Omega _-}({\bf u},\pi )=(\mathbf F ,\eta )$, i.e.,
there exists a unique solution $({\bf u},\pi )\in \mathring{\mathcal H}^{1}(\Omega _{-})^n\times L_2(\Omega _{-})$ of the variational problem
\begin{equation}
\label{variational-Dirichlet-2-Stokes}
\left\{\begin{array}{lll}
a_{{{\mathbb A};\Omega _{-}}}({\bf u},{\bf w})+b_{\Omega _{-}}({\bf w},\pi )=
\langle \mathbf F ,{\bf w}\rangle _{\Omega _-} \quad \forall \, {\bf w}\in \mathring{\mathcal H}^{1}(\Omega _{-})^n,\\
b_{_{\Omega _{-}}}({\bf u},q)=\langle \eta ,q\rangle _{\Omega _-} \quad \forall \, q\in L_2(\Omega _{-}),
\end{array}\right.
\end{equation}
where $a_{{{\mathbb A};\Omega _{-}}}$ and $b_{_{\Omega _{-}}}$ are the bilinear forms given by \eqref{a-Omega-p} and \eqref{b-Omega-p}, respectively.

Note that in order to have a well defined pairing of dual spaces in the right hand side of the first equation in \eqref{variational-Dirichlet-2-Stokes} we use the property that the spaces  $\mathring{\mathcal H}^{1}(\Omega _{-})^n$ and $\widetilde{\mathcal H}^{1}(\Omega _{-})^n$ can be identified through the isomorphism $\mathring E_{-}:\mathring{\mathcal H}^{1}(\Omega _{-})^n\to\widetilde{\mathcal H}^{1}(\Omega _{-})^n$ (see \cite[Theorem 5.11]{B-M-M-M}).

Now, in view of condition \eqref{Stokes-1}
and the H\"{o}lder inequality, as well as the definition of the norm of the space $\mathring{\mathcal H}^{1}(\Omega _{-})$
(cf. \eqref{weight-2p}, with $\Omega _-$ in place of ${\mathbb R}^n$), we obtain that
\begin{align}
\label{a-1-v-d-Stokes}
|a_{{{\mathbb A};\Omega _{-}}}({\bf v}_1,{\bf v}_2)|&\leq {\|{\mathbb A}\|_{L_{\infty }(\Omega _-)}}\|\nabla {\bf v}_1\|_{L_2(\Omega _{-})^{n\times n}}\|\nabla {\bf v}_2\|_{L_2(\Omega _{-})^{n\times n}}\nonumber\\
&\leq {\|{\mathbb A}\|_{L_{\infty }(\Omega _-)}}\|{\bf v}_1\|_{{\mathcal H}^1(\Omega _{-})^n}\|{\bf v}_2\|_{{\mathcal H}^1(\Omega _{-})^n} \quad \forall\, {\bf v}_1,\, {\bf v}_2\in \mathring{\mathcal H}^1(\Omega _{-})^n.
\end{align}
Due to the equivalence of the semi-norm \eqref{seminorm} to the norm \eqref{weight-2p} (with $\Omega _{-}$ in place of ${\mathbb R}^n$) on $\mathring{\mathcal H}^1(\Omega_-)^n$ and  assumption \eqref{mu}, there exists a constant $C=C(\Omega _{-},n)>0$ such that
\begin{align}
\label{estimate-Korn-2a}
\|{\bf u}\|_{{\mathcal H}^1(\Omega _{-})^n}^2&\leq C\|\nabla {\bf u}\|_{L_2(\Omega_-)^{n\times n}}^2
\leq 2C\|{\mathbb E}({\bf u})\|_{L_2(\Omega _{-})^{n\times n}}^2
\leq 2C{c_{\mathbb A}}a_{\mathbb A;\Omega_-}({\bf u},{\bf u}) \quad \forall \, {\bf u}\in \mathring{\mathcal H}^1_{\rm{div}}(\Omega _{-})^n.
\end{align}
Note that the second inequality in \eqref{estimate-Korn-2a} is a version of the Korn type inequality \eqref{Korn3-R3} with $\Omega _{-}$ in place of ${\mathbb R}^n$, and follows with similar arguments to those for inequality (2.2) in \cite{Sa-Se}.
Inequalities \eqref{a-1-v-d-Stokes} and \eqref{estimate-Korn-2a} show that
$a_{{{\mathbb A};\Omega _{-}}}(\cdot ,\cdot ):\mathring{\mathcal H}^1(\Omega _{-})^n
\times \mathring{\mathcal H}^1(\Omega _{-})^n\to {\mathbb R}$
is bounded, while
$a_{{{\mathbb A};\Omega _{-}}}(\cdot ,\cdot ):\mathring{\mathcal H}^1_{\rm{div}}(\Omega_{-})^n
\times \mathring{\mathcal H}^1_{\rm{div}}(\Omega _{-})^n\to {\mathbb R}$
is coercive.

Moreover, arguments similar to those for inequality \eqref{a-1-v-d-Stokes} imply that the bilinear form
$b_{_{\Omega _{-}}}(\cdot ,\cdot):\mathring{\mathcal H}^1(\Omega _{-})^n\times L_2(\Omega _{-})\to {\mathbb R}$ given by \eqref{b-Omega-p} is also bounded.
Lemma \ref{Bog} implies that there exists a constant $C_{\Omega_-} >0$ such that for any $q\in L_2(\Omega _{-})$ there exists ${\bf v}\in \mathring{\mathcal H}^1(\Omega _{-})^n$ satisfying the equation $-{\rm{div}}\, {\bf v}=q$ and the inequality
$\|{\bf v}\|_{\mathring{\mathcal H}^1(\Omega _{-})^n}\leq C_{\Omega_-}\|q\|_{L_2(\Omega _{-})}$, and hence
\begin{align*}
b_{_{\Omega _{-}}}({\bf v},q)=-\left\langle {\rm{div}}\, {\bf v},q\right\rangle _{\Omega _{-}}=\langle q,q\rangle _{\Omega _{-}}=\|q\|_{L_2(\Omega _{-})}^2\geq C_{\Omega_-}^{-1} \|{\bf v}\|_{\mathring{\mathcal H}^1(\Omega _{-})^n}\|q\|_{L_2(\Omega _{-})}.
\end{align*}
Consequently, the bilinear form $b_{_{\Omega _{-}}}(\cdot ,\cdot):\mathring{\mathcal H}^1(\Omega _{-})^n\times L_2(\Omega _{-})\to {\mathbb R}$ satisfies the inf-sup condition
\begin{align}
\label{inf-sup-Omega-Stokes}
\inf _{q\in L_2(\Omega _{-})\setminus \{0\}}\sup _{{\bf w}\in \mathring{\mathcal H}^1(\Omega _{-})^n\setminus \{\bf 0\}}\frac{b_{_{\Omega _{-}}}({\bf w},q)}{\|{\bf w}\|_{\mathring{\mathcal H}^1(\Omega _{-})^n}\|q\|_{L_2(\Omega _{-})}}
\geq\inf _{q\in L_2(\Omega _{-})\setminus \{0\}}\frac{b_{_{\Omega _{-}}}({\bf v},q)}{\|{\bf v}\|_{\mathring{\mathcal H}^1(\Omega _{-})^n}\|q\|_{L_2(\Omega _{-})}}
\geq C_{\Omega_-}^{-1}
\end{align}
(cf. \cite[{Theorem 3.3}]{Gi-Se}).
In addition, we use the following characterization of the subspace $\mathring{\mathcal H}_{{\rm{div}}}^1(\Omega_{-})^n$ of $\mathring{\mathcal H}^1(\Omega_{-})^n$ consisting of divergence-free vector fields in $\Omega _{-}$,
\begin{align*}
\mathring{\mathcal H}_{{\rm{div}}}^1(\Omega_{-})^n
&=\left\{{\bf w}\in \mathring{\mathcal H}^1(\Omega_{-})^n: b_{_{\Omega _{-}}}({\bf w},q)=0 \quad \forall \, q\in L_2(\Omega _{-})\right\}\,.
\end{align*}
Then Theorem \ref{B-B} (with $X=\mathring{\mathcal H}^1(\Omega _{-})^n$, $V=\mathring{\mathcal H}_{{\rm{div}}}^1(\Omega_{-})^n$ and $M=L_2(\Omega _{-})$) implies that for any $(\mathbf F ,\eta )\in {\mathcal H}^{-1}(\Omega _-)^n\times L_2(\Omega _-)$, there exists a unique solution $({\bf u},\pi )\in \mathring{\mathcal H}^1(\Omega _{-})^n\times L_2(\Omega _{-})$ of the variational problem \eqref{variational-Dirichlet-2-Stokes}, which depends continuously on the given data $\mathbf F $ and $\eta $. Equivalently, the pair $({\bf u},\pi )\in {\mathcal X}(\Omega _-)$ is the unique solution of the equation ${\mathcal T}_{\Omega _{-}}({\bf u},\pi )=(\mathbf F ,\eta )$, and accordingly, the operator ${\mathcal T}_{\Omega _{-}}:{\mathcal X}(\Omega _-)\to \mathcal X'(\Omega _-)$ is an isomorphism, as asserted.

$(ii)$
The existence and uniqueness of a solution $({\bf u},\pi )\in {\mathcal X}(\Omega _-)$ of the Poisson problem \eqref{Dirichlet-var-p}, depending continuously on $\boldsymbol{\mathbf f}$ and $g$, follows from item $(i)$.
\end{proof}

\subsubsection{\bf Variational approach to the Poisson problem for the compressible Stokes system  with inhomogeneous Dirichlet condition}

Next, we extend the well-posedness result in Theorem \ref{T-p}
to a Poisson problem with inhomogeneous Dirichlet condition (cf. \cite[Theorem 5.1]{K-M-W-1} for the Stokes system in the isotropic case \eqref{isotropic}, \cite[Theorem 3.4]{Gi-Se} and \cite[Theorem 9.2.2]{M-W} for the Stokes system \eqref{Dirichlet-var-p} with constant coefficients,
\cite[Proposition 3.1]{Kozono-Shor}, and \cite[Theorem 7.5]{B-M-M-M} for a inhomogeneous higher-order Poisson problem on bounded $(\varepsilon ,\delta )$-domains in ${\mathbb R}^n$).
\begin{thm}
\label{T-p-complete}
Let conditions \eqref{Stokes-1}-\eqref{mu} hold in $\Omega_-$.
for all $({\bf f},{g,}{\boldsymbol\varphi})\in
{\mathcal H}^{-1}(\Omega _{-})^n\times L_2(\Omega_-)
\times H^{\frac{1}{2}}(\partial \Omega )^n$,
problem \eqref{Dirichlet-var-Stokes-a} has a unique solution $({\bf u},\pi )$ in ${\mathcal H}^1(\Omega _{-})^n\!\times \!L_2(\Omega _{-})$ and there exists a constant $C=C(\partial \Omega ,c_{\mathbb A},n)>0$ such that
\begin{align}
\label{estimate-Dirichlet-var-ext}
\|{\bf u}\|_{{\mathcal H}^1(\Omega _{-})^n}
+\|\pi \|_{L_2(\Omega _{-})}\leq
C\left(\|{\bf f}\|_{{{\mathcal H}^{-1}(\Omega _{-})^n}}
+\|g\|_{L_{2}(\Omega_-)}
+\|\boldsymbol\varphi \|_{H^{\frac{1}{2}}(\partial \Omega )^n}\right).
\end{align}
\end{thm}
\begin{proof}
Let $({\bf f},g,{\boldsymbol\varphi})\in
{\mathcal H}^{-1}(\Omega _{-})^n\times L_2(\Omega_-)
\times H^{\frac{1}{2}}(\partial \Omega )^n$ be given.
First, let us introduce ${\bf u}^0:=\gamma^{-1}_-\boldsymbol\varphi\in {\mathcal H}^1(\Omega _{-})^n$, where $\gamma^{-1}_-:H^{\frac{1}{2}}({\partial\Omega })^n\!\to \! \mathcal{H}^1({\Omega }_-)^n$ is a continuous linear right inverse to the trace operator $\gamma_{-}$ and hence $\gamma_{-}{\bf u}^0=\boldsymbol\varphi$ on $\partial \Omega$.

By considering the new function $\mathring{\bf u}:={\bf u}-{\bf u}^0\in \mathring{\mathcal H}^1(\Omega _{-})^n$, the Poisson problem with inhomogeneous Dirichlet condition \eqref{Dirichlet-var-Stokes-a} reduces to the following Poisson problem with homogeneous Dirichlet condition 
\begin{equation}
\label{Dirichlet-var-p-complete}
\left\{
\begin{array}{ll}
{\boldsymbol{\mathcal L}(\mathring{\bf u},\pi )={\bf f}^0,}\ \
{\rm{div}}\, \mathring{\bf u}=g^0 & \mbox{ in } \Omega _-,\
\\
{\gamma }_{-}\mathring{\bf u}={\bf 0} &  \mbox{ on } \partial \Omega ,
\end{array}\right.
\end{equation}
where
${\bf f}^0\!:=\!{\bf f}-\partial _\alpha\left(A^{\alpha \beta }\partial _\beta ({\bf u}^0)\right)\in {\mathcal H}^{-1}(\Omega _-)^n$,
$g^0=g-{\rm{div}}\,{\bf u}^0\in L_2(\Omega _{-})$.

In view of the well-posedness result in Theorem \ref{T-p} (ii), problem \eqref{Dirichlet-var-p-complete} has a unique solution $(\mathring{\bf u},\pi )\in {\mathcal X}(\Omega _-)=\mathring{\mathcal H}^1(\Omega )^n\times L_2(\Omega _-)$ which depends continuously on ${\bf f}^0$ and $g^0$.
Then the pair
$({\bf u},\pi )=(\mathring{\bf u}+{\bf u}^0,\pi )$
belongs to the space ${\mathcal H}^1(\Omega _-)^n\times L_2(\Omega _-)$ and solves the Poisson problem with inhomogeneous Dirichlet condition \eqref{Dirichlet-var-Stokes-a}.
Moreover, in view of Theorem \ref{T-p}(ii) and continuity of the operator $\gamma^{-1}_-:H^{\frac{1}{2}}({\partial\Omega })^n\!\to \! \mathcal{H}^1({\Omega }_-)^n$, this solution depends continuously on the given data ${\bf f}$, $g$ and $\boldsymbol\varphi $.
Finally, uniqueness of the solution of the problem \eqref{Dirichlet-var-Stokes-a} in the space ${\mathcal H}^1(\Omega _-)^n\times L_2(\Omega _-)$ follows from the uniqueness result in Theorem \ref{T-p} (ii).
\end{proof}

\subsubsection{\bf Potential approach to the exterior Dirichlet problem for the Stokes system}
Theorem \ref{T-p-complete} shows the well-posedness of the exterior Dirichlet problem for the $L_{\infty}$ coefficient Stokes system \eqref{Dirichlet-var-Stokes-a} in the space ${\mathcal H}^1(\Omega _{-})^n\times L_2(\Omega _{-})$. 
When the given datum $({\bf f},{\boldsymbol\varphi})$ belongs to ${\mathcal H}^{-1}(\Omega _{-})^n\times H_{\boldsymbol \nu }^{\frac{1}{2}}(\partial \Omega )^n$ and $g=0$, the solution can be
expressed in terms of volume and single layer potentials (cf. \cite[Theorem 5.2]{K-M-W-1} in the isotropic case \eqref{isotropic}, \cite[Theorem 3.4]{Gi-Se} for the constant coefficient Stokes system, \cite[Theorem 10.1]{Fa-Me-Mi}, \cite[Theorem 5.1]{Lang-Mendez} for the Laplace operator).
\begin{thm}
\label{Dirichlet-ext-var-Stokes-new}
Let conditions \eqref{Stokes-1}-\eqref{mu} hold in $\Omega_-$.
If ${\bf f}\in {\mathcal H}^{-1}(\Omega_-)^n$, $g\in L_2(\Omega_-)$ and ${\boldsymbol\varphi}\in H_{\boldsymbol \nu }^{\frac{1}{2}}(\partial \Omega )^n$ then
problem \eqref{Dirichlet-var-Stokes-a}
has a unique solution $({\bf u},\pi )\in {\mathcal H}^{1}(\Omega _{-})^n\times L_2(\Omega _{-})$, given by
\begin{align}
\label{solution-Dirichlet-ext-0}
&{\bf u}\!=\!\boldsymbol{\mathcal N}_{{{\mathbb R}^n}}\tilde{\bf f}
+\boldsymbol{\mathcal G}_{\mathbb R^n}\mathring E_- g
+\!{\bf V}_{\partial \Omega }\boldsymbol{\mathcal V}_{\partial \Omega }^{-1}\big({\boldsymbol\varphi}
-\gamma _{-}\boldsymbol{\mathcal N}_{\mathbb R^n}\tilde{\bf f}
-\gamma_-\boldsymbol{\mathcal G}_{\mathbb R^n}\mathring E_- g\big)\quad \mbox{in } \Omega_{-},\\
\label{solution-Dirichlet-ext-p-0}
&\pi \!=\!{\mathcal Q}_{\mathbb R^n}\tilde{\bf f}
+{\mathcal G}^0_{\mathbb R^n}\mathring E_- g
+{\mathcal Q}_{\partial \Omega }^s\boldsymbol{\mathcal V}_{\partial \Omega }^{-1}\big({\boldsymbol\varphi}
-\gamma _{-}\boldsymbol{\mathcal N}_{\mathbb R^n}\tilde{\bf f}
-\gamma_-\boldsymbol{\mathcal G}_{\mathbb R^n}\mathring E_- g\big)\quad \mbox{in } \Omega_{-},
\end{align}
where $\tilde{\bf f}$ is an extension of ${\bf f}$ to an element of $\widetilde{\mathcal H}^{-1}(\Omega_-^n)\subset {\mathcal H}^{-1}({\mathbb R}^n)^n$.
\end{thm}
\begin{proof}
Let ${\bf f}\in {\mathcal H}^{-1}(\Omega _{-})^n$. {Since ${\mathcal H}^{-1}(\Omega _{-})^n=(\mathring{\mathcal H}^{1}(\Omega _{-})^n)'$ and the space $\mathring{\mathcal H}^{1}(\Omega _{-})^n$ can be identified, via the operator of extension by zero outside $\Omega _{-}$, with the
space $\widetilde{\mathcal H}^{1}(\Omega _{-})^n$, see \eqref{2.5},
the Hahn-Banach extension Theorem implies that there exists $\tilde{\bf f}\in {\mathcal H}^{-1}({\mathbb R}^n)^n$ such that $\tilde{\bf f}|_{\Omega _-}={\bf f}$ and $\|\tilde{\bf f}\|_{{\mathcal H}^{-1}({\mathbb R}^n)^n}\leq c\|{\bf f}\|_{{\mathcal H}^{-1}(\Omega _{-})^n}$, with some constant $c>0$ (see also the proof of \cite[Theorem 10.1]{Fa-Me-Mi}).}
Then Lemmas \ref{Newtonian-B-var-1},
\ref{continuity-sl-S-h-var},
\ref{jump-s-l}, and
\ref{isom-sl-v}
imply that $({\bf u},\pi)$
represented by \eqref{solution-Dirichlet-ext-0}, \eqref{solution-Dirichlet-ext-p-0}
solve problem \eqref{Dirichlet-var-Stokes-a}, while Theorem \ref{T-p-complete} yields uniqueness.
\end{proof}

\subsection{Potential approach for the exterior Neumann problem of the anisotropic Stokes system} 
The Neumann problem for the constant coefficient Stokes system in an exterior Lipschitz domain in ${\mathbb R}^n$, with boundary datum in $L_p$ spaces, has been studied in \cite[Theorem 9.2.6]{M-W} by a potential approach (see also \cite[Theorem 10.6.4]{M-W} for the Neumann problem for the same system in a bounded Lipschitz domain).
Next we consider the following exterior Neumann problem for the $L_{\infty}$ coefficient Stokes system 
\begin{equation}
\label{Neumann-var-Stokes}
\left\{
\begin{array}{ll}
{\boldsymbol{\mathcal L}({\bf u},\pi )={\bf 0}},\ \ 
{\rm{div}} \ {\bf u}=0 & \mbox{ in } \Omega _{-},\
\\
{\bf t}^{-}({\bf u},\pi )=\boldsymbol\psi \in \boldsymbol{\mathcal R}_{\partial \Omega }^\perp &  \mbox{ on } \partial \Omega ,
\end{array}\right.
\end{equation}
Recall that $\boldsymbol{\mathcal R}_{\partial \Omega }^\perp $ is defined in \eqref{R-perp},
${\bf D}_{\partial\Omega }:H_{\boldsymbol{\mathcal R}}^{\frac{1}{2}}(\partial\Omega )^n\!\to \!\boldsymbol{\mathcal R}_{\partial \Omega }^\perp $ is given by \eqref{jump-dl-v-var-alpha}
and
${\bf D}_{\partial\Omega}^{-1}:\boldsymbol{\mathcal R}_{\partial \Omega }^\perp\to
H_{\boldsymbol{\mathcal R}}^{\frac{1}{2}}(\partial\Omega )^n$
is a continuous operator due to Lemma~\ref{isom-dlc}.
\begin{thm}
\label{Neumann-ext-var-Stokes}
Let conditions \eqref{Stokes-1}-\eqref{mu} hold in $\Omega_-$.
If {${\boldsymbol\psi}\in \boldsymbol{\mathcal R}_{\partial \Omega }^\perp $} then problem \eqref{Neumann-var-Stokes} has a unique solution $({\bf u},\pi )\in {\mathcal H}^{1}(\Omega _{-})^n\times L_2(\Omega _{-})$, given by
\begin{align}
\label{solution-Neumann-ext}
{{\bf u}={\bf W}_{\partial \Omega }\big({\bf D}_{\partial \Omega }^{-1}{\boldsymbol\psi}\big),\
\pi ={\mathcal Q}^d_{\partial\Omega}\big({\bf D}_{\partial\Omega}^{-1}{\boldsymbol\psi}\big)
\mbox{ in } \Omega _{-}\,.}
\end{align}
Moreover, there exists a constant $C=C(\Omega _-,c_{\mathbb A},n)>0$ such that
\begin{align}
\label{solution-Neumann-ext-C}
\|{\bf u}\|_{{\mathcal H}^{1}(\Omega _{-})^n}+\|\pi \|_{L_2(\Omega _{-})}\leq C\|\boldsymbol\psi \|_{\boldsymbol{\mathcal R}_{\partial \Omega }^\perp }\,.
\end{align}
\end{thm}
\begin{proof}
Lemmas \ref{continuity-dl-h-var} and \ref{isom-dlc} imply that $({\bf u},\pi)$ represented by \eqref{solution-Neumann-ext} solve problem \eqref{Neumann-var-Stokes}
and the operators involved in \eqref{solution-Neumann-ext} are continuous, which implies inequality \eqref{solution-Neumann-ext-C}.

To show uniqueness let us assume that a pair $({\bf u}_0,\pi _0)\in {\mathcal H}^{1}_{\rm{div}}(\Omega _{-})^n\times L_2(\Omega _{-})$ satisfies the homogeneous version of the exterior Neumann problem \eqref{Neumann-var-Stokes}. Then Lemma \ref{lem-add1} and assumption \eqref{mu} imply that
\begin{align}
0\!=\!-\langle {\bf t}^{-}({\bf u}_0,\pi _0),\gamma _{-}({\bf u}_0)\rangle _{\partial \Omega }\!
=\!\big\langle a_{ij}^{\alpha \beta }E_{j\beta }({\bf u}_0),E_{i\alpha }({\bf u}_0)\big\rangle _{\Omega _{-}}\nonumber
\!\!\geq \!c_{\mathbb A}^{-1}\|{\mathbb E}({\bf u}_0)\|_{L_2(\Omega _-)^{n\times n}}^2
\end{align}
and hence ${\mathbb E}({\bf u}_0)=0$ in $\Omega _{-}$. Then ${\bf u}_0\in \boldsymbol{\mathcal R}|_{\Omega _{-}}$,
see \eqref{Neumann-var-Stokes}, i.e.,
there exist a constant ${\bf b}_0\in {\mathbb R}^n$ and an antisymmetric matrix ${\bf B}_0\in {\mathbb R}^{n\times n}$ such that ${\bf u}_0={\bf b}_0+{\bf B}_0{\bf x}$ in $\Omega _{-}$.
However, the membership of ${\bf u}_0$ in the weighted Sobolev space ${\mathcal H}^{1}(\Omega _{-})^n$ and the embedding ${\mathcal H}^{1}(\Omega _{-})^n\hookrightarrow L_{\frac{2n}{n-2}}(\Omega _-)^n$ imply that ${\bf u}_0={\bf 0}$ in $\Omega_{-}$. Moreover, the Stokes equation in \eqref{Neumann-var-Stokes} shows that $\pi _0$ reduces to a constant $c_0\in {\mathbb R}$, but the membership of $\pi _0$ in $L_2(\Omega _{-})$ yields that $c_0=0$, and accordingly that $\pi _0=0$ in $\Omega _{-}$.
\end{proof}

\subsection{Variational approach of the exterior Poisson problem for the anisotropic Stokes system with mixed  Dirichlet-Neumann boundary conditions} 

As before, we assume that $\Omega_+ \subset {\mathbb R}^n$
is a bounded Lipschitz domain with connected boundary $\partial \Omega $ and let $\Omega _{-}:={\mathbb R}^n\setminus \overline\Omega_+ $.
Let additionally $D$ and $N$ be relatively open subsets of $\partial \Omega $, such that $D$ has {positive $(n-1)$-Hausdorff measure}, $D\cap N=\emptyset $, $\overline{D}\cup \overline{N}=\partial \Omega $, and $\overline{D}\cap \overline{N}=\Sigma $, where $\Sigma $ is an $(n-2)$-dimensional closed Lipschitz submanifold of $\partial \Omega $.

Let ${\mathcal H}^1_{D}({\mathbb R}^n)$ be the closure of the space ${\mathcal D}({\mathbb R}^n\setminus {D})$ (of smooth compactly supported functions in ${\mathbb R}^n\setminus {D}$) in ${\mathcal H}^1({\mathbb R}^n)$. Let $C_D^{\infty }(\Omega _{-}):=\{\phi |_{\Omega _-}:\phi \in {\mathcal D}({\mathbb R}^n) \mbox{ with } {\rm{supp}}(\phi )\cap {D}=\emptyset\}$. Then
${\mathcal H}^1_{D}(\Omega _-)$ denotes the closure of $C_D^{\infty }(\Omega _-)$ in $\left({\mathcal H}^1(\Omega _-),\|\cdot \|_{{\mathcal H}^1(\Omega _{-})}\right)$ (see also \cite[Definition 3.1]{B-M-M-M}, \cite[Definition 2.3]{H-J-K-R}). The spaces ${\mathcal H}^1_{D}({\mathbb R}^n)^n$ and ${\mathcal H}^1_{D}(\Omega _-)^n$ can be similarly defined. We have the following characterization 
\begin{align}
\label{H-D-0}
{\mathcal H}^1_{D}(\Omega _-)^n
=&\left\{{\bf v}|_{_{\Omega _-}}:{\bf v}\in {\mathcal H}^1_{D}({\mathbb R}^n)^n\right\}.
\end{align}
In addition (see also \cite[Corollary 3.11]{B-M-M-M}, \cite[Definition 2.7]{H-J-K-R}),
\begin{align}
\label{H-D}
{\mathcal H}^1_{D}(\Omega _-)^n\!=\!\left\{{\bf v}\in {\mathcal H}^1(\Omega _-)^n:
\left(\gamma _-{\bf v}\right)|_{_{D}}\!=\!{\bf 0}\ \mbox{ on } D\right\}.
\end{align}
Let us also define
$$
{\mathcal H}_{D;\rm{div}}^1(\Omega_-)^n:=\left\{{\bf v}\in {\mathcal H}^1_D(\Omega_-)^n:{\rm{div}}\, {\bf v}=0 \mbox{ in } \Omega _-\right\}
$$

Let $\Gamma$ be {a relatively} open $(n-1)$-dimensional subset of $\partial\Omega$, e.g., $D$ or $N$. Let $r_{_{\Gamma}}$ denote the operator of restriction of distributions from $\partial\Omega$ to  $\Gamma$.
We will also employ the boundary spaces (cf. e.g., \cite{Lean}, \cite[Definition 4.8, Theorem 5.1]{B-M-M-M})
\begin{align}
\label{D-N-spaces}
&H^{\pm\frac{1}{2}}(\Gamma)^n:=\left\{r_{_\Gamma}\boldsymbol\phi: \boldsymbol\phi\in H^{\pm\frac{1}{2}}(\partial\Omega )^n\right\},\
\widetilde{H}^{\pm\frac{1}{2}}(\Gamma)^n\!:=\!\left\{\widetilde{\boldsymbol\phi }\in H^{\pm\frac{1}{2}}(\partial \Omega )^n\!:
\widetilde{\boldsymbol\phi }\!=\!{\bf 0}\!\ \mbox{ on } \partial\Omega\setminus\bar\Gamma\right\},
\end{align}
where
$H^{\pm\frac{1}{2}}(\Gamma)^n$ can be identified with $\Big(\widetilde{H}^{\mp\frac{1}{2}}(\Gamma)^n\Big)'$.

The existence of a right inverse of the exterior trace operator in \eqref{ext-trace} implies the following result.
\begin{lem}
\label{mixt-trace-D}
The trace operator $\gamma _-:{\mathcal H}^1_{D}(\Omega _-)^n\to \widetilde{H}^{\frac{1}{2}}(N)^n$ is linear, bounded and surjective, having a $($non-unique$)$ linear, bounded right inverse
$\gamma _-^{-1}:\widetilde{H}^{\frac{1}{2}}(N)^n\to {\mathcal H}^1_{D}(\Omega _-)^n.$
\end{lem}
On the other hand, in view of the surjectivity of the operator ${\rm{div}}:\mathring{\mathcal H}^1(\Omega_{-})^n\to L_2(\Omega _{-})$ (see Lemma \ref{Bog})
and of the inclusion $\mathring{\mathcal H}^1(\Omega _{-})^n\subset {\mathcal H}_D^{1}(\Omega _-)^n$, the operator
\begin{align}
\label{div-tilde-Stokes-D}
{\rm{div}}:{\mathcal H}_D^1(\Omega_{-})^n\to L_2(\Omega _{-})
\end{align}
is surjective and has a bounded right inverse $R_{\Omega _-}:L_2(\Omega _-)\to {\mathcal H}_D^{1}(\Omega _{-})^n$
by Lemma \ref{Bog},
i.e., there exists a constant $C_D=C_D(D,n)>0$ such that
\begin{align}
\label{Borchers-Shor-div-2}
{\rm{div}}(R_{\Omega _-}g)=g,\ \|R_{\Omega _-}g\|_{{\mathcal H}_D^{1}(\Omega _{-})^n}\leq C_D\|g\|_{L_2(\Omega _-)} \quad \forall \, g\in L_2(\Omega _- ).
\end{align}

Let us consider the spaces
\begin{align}
\label{interp-D}
{\mathcal X}_{D}(\Omega_-):={\mathcal H}^1_{D}(\Omega_-)^n\times L_2(\Omega_-),\ {\mathcal X}_{D}'(\Omega_-)
:=\left({\mathcal X}_{D}(\Omega _-)\right)'
=\left({\mathcal H}^1_{D}(\Omega_-)^n\right)'\times L_2(\Omega_-)
\end{align}

Next, we define a conormal derivative operator related to the Poisson problem with mixed conditions for the Stokes system with {$L_{\infty }$ coefficient tensor}.
To this end, consider the space
\begin{align}
\label{conormal-derivative-var-Brinkman-1-DN}
{\pmb{\mathcal H}}^1_{D}(\Omega _-,\boldsymbol{\mathcal L})
:=\Big\{\big({\bf u},\pi ,{\tilde{\bf f}}\big)&\in {\mathcal H}^1_{D}(\Omega _-)^n\times L_2(\Omega _-)\times \big({\mathcal H}^{1}_{D}({\Omega _-})^n\big)':
\boldsymbol{\mathcal L}({\bf u},\pi )={\tilde{\bf f}}|_{\Omega _-}
 \mbox{ in } \Omega _-\Big\},
\end{align}
where $\boldsymbol{\mathcal L}$ is the operator defined in \eqref{Stokes}.

For arbitrary $\big({\bf u},\pi \big)\in {\mathcal H}^1_{D}(\Omega _-)^n\times L_2(\Omega _-)$,
$\tilde{\bf f}\in \big({\mathcal H}^{1}_{D}({\Omega _-})^n\big)'$, and particularly for
$\big({\bf u},\pi ,{\tilde{\bf f}}\big)\in {\pmb{\mathcal H}}^1_{D}(\Omega _-,\boldsymbol{\mathcal L})$,
we can not directly employ Definition~\ref{conormal-derivative-var-Brinkman} of the formal or generalised conormal derivative
${\bf t}^-({\bf u},\pi,{\tilde{\bf f}})$ on $\partial\Omega_-$ since the last term in \eqref{conormal-derivative-var-Brinkman-3} is then not well defined for arbitrary $\boldsymbol\Phi\!\in\!H^{\frac{1}{2}}(\partial\Omega )^n$.
To overcome this, we could (non-uniquely) extend the distribution $\tilde{\bf f}$ from
$\big({\mathcal H}^{1}_{D}({\Omega _-})^n\big)'$ to
$\big({\mathcal H}^{1}({\Omega _-})^n\big)'=\widetilde{\mathcal H}^{-1}({\Omega _-})^n$, e.g. appropriately generalising the procedure used in the proof of Theorem 2.16 in \cite{Mikh}.
Then the conormal derivative on $D$ will depend on the chosen extension of $\tilde{\bf f}$.
However, the restriction of the conormal derivative on $N$, $r_{_N}{\bf t}^{-}({\bf u},\pi ,\tilde{\bf f})$ is independent of the possible extension of $\tilde{\bf f}$ and can be still directly defined as follows (cf Definition \ref{conormal-derivative-var-Brinkman}).

\begin{defn}
\label{conormal-derivative-var-BrinkmanN}
Let conditions {\eqref{Stokes-1} and \eqref{Stokes-sym}} hold.
For any $\big({\bf u},\pi ,{\tilde{\bf f}}\big)\in {\mathcal H}^1_{D}(\Omega _-)^n\times L_2(\Omega _-)\times \big({\mathcal H}^{1}_{D}({\Omega _-})^n\big)'$,
the  formal conormal derivatives restriction $r_{_N}{\bf t}^{-}({\bf u},\pi ,\tilde{\bf f})\in H^{-\frac{1}{2}}(N)^n$ is defined in the weak form by
\begin{align}
\label{conormal-derivative-var-Brinkman-3N}
-\left\langle r_{_N}{\bf t}^-({\bf u},\pi;{\tilde{\bf f}}),\widetilde{\boldsymbol\Phi}\right\rangle_{N}\!\!
:=\!\left\langle a_{ij}^{\alpha \beta }E_{j\beta }({\bf u}),
E_{i\alpha }(\gamma _-^{-1}\widetilde{\boldsymbol\Phi})\right\rangle _{\Omega_{-}}\!
-\!\left\langle {\pi},{\rm{div}}(\gamma^{-1}_{-}\widetilde{\boldsymbol\Phi})\right\rangle _{\Omega_{-}}\!
+\!\left\langle {\tilde{\bf f}},\gamma^{-1}_{-}\widetilde{\boldsymbol\Phi}\right\rangle _{{\Omega_{-}}},
\end{align}
for any $\widetilde{\boldsymbol\Phi}\!\in\widetilde H^{\frac{1}{2}}(N)^n$, where $\gamma^{-1}_{-}:H^{\frac{1}{2}}(\partial\Omega )^n\to {\mathcal H}^{1}({\Omega_{-}})^n$ is a $($non-unique$)$ bounded right inverse to the trace operator
$\gamma_{-}:{\mathcal H}^{1}({\Omega_{-}})^n\to H^{\frac{1}{2}}(\partial\Omega)^n$.
Moreover, if
$({\bf u}_{-},\pi_{-} ,{\tilde{\bf f}}_{-})\!\in
\!\boldsymbol{\mathcal H}_D^1({\Omega_{-}},{\boldsymbol{\mathcal L}})$,
then relations \eqref{conormal-derivative-var-Brinkman-3N} define the generalized conormal derivative restriction
$r_{_N}{\bf t}^{-}({\bf u}_{-},\pi_{-} ;{\tilde{\bf f}}_{-})\in H^{-\frac{1}{2}}(N)^n$.
\end{defn}

 By using similar arguments to those for Lemma \ref{lem-add1}, we obtain its counterpart for the mixed problem (cf. \cite[Definition 7.1]{B-M-M-M} for strongly elliptic higher-order systems in divergence form).
\begin{lem}
\label{conormal-derivative-generalized-mixed}

Let conditions {\eqref{Stokes-1} and \eqref{Stokes-sym}} hold.
\begin{itemize}
\item[$(i)$] The formal conormal derivative operator restriction
$r_{_N}{\bf t}^{-}:{\mathcal H}^1_D({\Omega_{-}})^n\times L_2({\Omega_{-}})
\times \big({\mathcal H}^{1}_{D}({\Omega _-})^n\big)'\to H^{-\frac{1}{2}}(N)^n$
is linear and bounded.
\item[$(ii)$] The generalised conormal derivative operator restriction
$r_{_N}{\bf t}^{-}:{\pmb{\mathcal H}}^1_{D}(\Omega _-,\boldsymbol{\mathcal L})\to H^{-\frac{1}{2}}(N)^n$, is linear and bounded, and definition \eqref{conormal-derivative-var-Brinkman-3N} is invariant with respect to the choice of the right inverse operator $\gamma _{-}^{-1}$. Moreover, the Green formula
\begin{align}
\label{conormal-derivative-generalized-2-mixed}
-\left\langle r_{_N}{\bf t}^{-}({\bf u},\pi ,{\tilde{\bf f}}),\gamma _{-}{\bf w}\right\rangle_{{N}}=\left\langle a_{ij}^{\alpha \beta }E_{j\beta }({\bf u}),E_{i\alpha }({\bf w})\right\rangle _{\Omega _-}-\left\langle \pi ,{\rm{div}}\, {\bf w}\right\rangle _{\Omega _-}+\left\langle \tilde{\bf f},{\bf w}\right\rangle _{\Omega _-}
\end{align}
holds for all $\big({\bf u},\pi ,{\tilde{\bf f}}\big)\in {\pmb{\mathcal H}}^1_{D}({\Omega _- },{\boldsymbol {\mathcal L}})$ and any ${\bf w}\in {\mathcal H}^{1}_{D}(\Omega _-)^n$.
\end{itemize}
\end{lem}

Then the following well-posedness result holds for the Poisson problem of mixed type for the anisotropic Stokes system (see also \cite{Ch-Mi-Na-3} for the exterior mixed problem for a strongly elliptic operator).
\begin{thm}
\label{mixed-B}
Let conditions \eqref{Stokes-1}-\eqref{mu} hold in $\Omega_-$ and let $\boldsymbol{\mathcal L}$ denote the Stokes operator defined in \eqref{Stokes-new}.
Then
for all $\tilde{\bf f}\in \left({\mathcal H}^{1}_{D}(\Omega _-)^n\right)'$,
$g\in L_2(\Omega_-)$
and $\boldsymbol\psi\in H^{-\frac{1}{2}}(N)^n$, the problem
\begin{align}
\label{mixed-B-manifold}
\left\{\begin{array}{llll}
{\boldsymbol{\mathcal L}({\bf u},\pi )=\tilde{\bf f}}|_{\Omega _-},\ \
{{\rm{div}}}\, {\bf u}= g & \mbox{ in } \Omega _-\,,\\
r_{_D}\gamma _-{\bf u}={\bf 0} & \mbox{ on } D,\\
r_{_N}{\bf t}^{-}({\bf u},\pi ;\tilde{\bf f})=\boldsymbol\psi & \mbox{ on } N
\end{array}
\right.
\end{align}
has a unique solution $({\bf u},\pi )\!\in \!{\mathcal H}^1_{D}(\Omega _-)^n\!\times \! L_2(\Omega _-)$, and there exists a constant $C_{DN}\!=\!C_{DN}(\partial \Omega ,c_{\mathbb A},n)\!>\!0$ such that
$\|{\bf u}\|_{{\mathcal H}^1(\Omega _-)^n}+\|\pi \|_{{L_2(\Omega _-)}}\leq
C_{DN}\left(\|\tilde{\bf f}\|_{({\mathcal H}^1_{D}(\Omega _-))'}
+\|g\|_{L_{2}(\Omega_-)}
+\|\boldsymbol\psi  \|_{H^{-\frac{1}{2}}(N)^n}\right).$
\end{thm}
\begin{proof}
First, we show that the mixed problem \eqref{mixed-B-manifold} has the following equivalent variational formulation.

{\it Find $({\bf u},\pi )\in {\mathcal H}^1_{D}(\Omega _-)^n\times L_2(\Omega _-)$ such that
\begin{equation}
\label{variational-mixed-2a}
\left\{\begin{array}{lll}
a_{{\mathbb A};\Omega _-}({\bf u},{\bf w})+b_{\Omega _-}({\bf w},\pi )=
\mathbf F({\bf w}) & \forall \, {\bf w}\in {\mathcal H}^{1}_{D}(\Omega _-)^n,\\
b_{_{\Omega _-}}({\bf u},q)=-\langle g,q\rangle_{\Omega_-} & \forall \, q\in L_{2}(\Omega _-),
\end{array}\right.
\end{equation}
where $a_{{\mathbb A};\Omega _-}:{\mathcal H}^1_{D}(\Omega _-)^n\times {\mathcal H}^{1}_{D}(\Omega _-)\to {\mathbb R}$ and
$b_{\Omega _-}:{\mathcal H}^1_{D}(\Omega _-)^n\times L_2(\Omega _-)\to {\mathbb R}$ are the bilinear forms given by \eqref{a-Omega-p} and \eqref{b-Omega-p}, respectively, and $\mathbf F:{\mathcal H}^{1}_{D}(\Omega _-)^n\to {\mathbb R}$ is the linear form defined by
}
\begin{align}
\label{l-u-0-mixed}
\mathbf F({\bf w}):=-\left\langle \boldsymbol\psi,\gamma _-{\bf w}\right\rangle _N -\big\langle \tilde{\bf f},{\bf w}\big\rangle _{\Omega _-} \quad \forall \, {\bf w}\in {\mathcal H}^{1}_{D}(\Omega _-)^n.
\end{align}

Indeed, if $({\bf u},\pi )\in {\mathcal H}^1_{D}(\Omega _-)^n\times L_p(\Omega _-)$ satisfies the mixed boundary value problem \eqref{mixed-B-manifold}, then the Green formula \eqref{conormal-derivative-generalized-2-mixed} holds and implies the first equation in \eqref{variational-mixed-2a}.
The second equation in \eqref{variational-mixed-2a} follows from the second equation in \eqref{mixed-B-manifold} and a duality argument.

Conversely, let us assume that $({\bf u},\pi )\in {\mathcal H}^1_{D}(\Omega _-)^n\times L_2(\Omega _-)$ satisfies the mixed variational formulation \eqref{variational-mixed-2a}.
Then the first equation in \eqref{mixed-B-manifold}, in the sense of distributions, follows from
the first equation in \eqref{variational-mixed-2a}
written for any ${\bf w}\in {\mathcal D}(\Omega _-)^n\subset {\mathcal D}(\Omega _-)^n$.
The divergence equation in \eqref{mixed-B-manifold} follows from the second variational equation in \eqref{variational-mixed-2a}. Therefore, the pair $({\bf u},\pi )\in {\mathcal H}^1_{D}(\Omega _-)^n\times L_2(\Omega _-)$ satisfies the
Stokes system in \eqref{mixed-B-manifold} and the boundary condition $\left(\gamma _-{\bf u}\right)|_{D}={\bf 0}$, due to the definition of the space ${\mathcal H}^1_{D}(\Omega _-)^n$. Thus, $({\bf u},\pi )$ satisfies the Green formula \eqref{conormal-derivative-generalized-2-mixed}, where
the right hand side is equal to $-\left\langle \boldsymbol\psi,\gamma _-{\bf w}\right\rangle _N$, as follows from the first equation in \eqref{variational-mixed-2a}. Therefore, we obtain the equation
\begin{align}
\label{conormal-derivative-particular-3}
\left\langle r_{_N}{\bf t}^{-}({\bf u},\pi ,{\tilde{\bf f}}),\gamma _{-}{\bf w}\right\rangle_{{N}}=\left\langle \boldsymbol\psi,\gamma _-{\bf w}\right\rangle _N \quad \forall \, {\bf w}\in {\mathcal H}^{1}_{D}(\Omega _-)^n.
\end{align}
which, in view of the surjectivity of the trace map $\gamma _-:{\mathcal H}^{1}_{D}(\Omega _-)^n\to \widetilde{H}^{\frac{1}{2}}(N)^n$ (see Lemma \ref{mixt-trace-D}), implies
$r_{_N}{\bf t}^{-}({\bf u},\pi ,{\tilde{\bf f}})=\boldsymbol\psi$ on $N$.
Consequently, the mixed boundary value problem \eqref{mixed-B-manifold} has indeed the equivalent mixed variational formulation \eqref{variational-mixed-2a}, as asserted.

Next, we define the following operator related to the mixed variational formulation \eqref{variational-mixed-2a}
\begin{align}
\label{T-mixed}
{\mathcal T}_{D;\Omega _-}:{\mathcal X}_{D}(\Omega _-)\to {\mathcal X}_{D}'(\Omega _-),\ {\mathcal T}_{D;\Omega _{-}}({\bf u},\pi )=\left({\mathcal T}_{1;D;\Omega _-}({\bf u},\pi ),{\mathcal T}_{2;D;\Omega _-}({\bf u},\pi )\right)\in {\mathcal X}'(\Omega _-)\,,
\end{align}
by setting
\begin{equation}
\label{variational-mixed-D-N}
\left\{\begin{array}{ll}
\langle {\mathcal T}_{1;D;\Omega _-}({\bf u},\pi ),{\bf v}\rangle _{\Omega }
=a_{{{\mathbb A};\Omega _-}}({\bf u},{\bf v})+b_{\Omega _-}({\bf v},\pi ),\\
\langle {\mathcal T}_{2;D;\Omega _-}({\bf u},\pi ),q\rangle _{\Omega _-}=b_{\Omega _{-}}({\bf u},q),
\end{array}\right.
\end{equation}
for every $({\bf u},\pi ),({\bf v},q)\in {\mathcal X}_{D}(\Omega _-)$.
The H\"{o}lder inequality and continuity of the bilinear forms in \eqref{variational-mixed-D-N} imply that ${\mathcal T}_{D;\Omega _-}:{\mathcal X}_{D}(\Omega _-)\to {\mathcal X}_{D}'(\Omega _-)$ is a continuous linear operator.

By using condition \eqref{Stokes-1} and again the H\"{o}lder inequality, we deduce that
\begin{align}
\label{a-1-v-d-Stokes-mixed}
|a_{{{\mathbb A};\Omega _{-}}}({\bf v},{\bf w})|
\leq {\|{\mathbb A}\|_{L_{\infty }(\Omega _{-})}}\|{\bf v}\|_{{\mathcal H}^1(\Omega _{-})^n}\|{\bf w}\|_{{\mathcal H}^1(\Omega _{-})^n} \quad \forall\, {\bf v},\, {\bf w}\in {\mathcal H}_D^1(\Omega _{-})^n.
\end{align}
On the other hand, there exists a constant $c=c(\Omega _{-},n)>0$ such that the following inequalities hold
\begin{align}
\label{estimate-Korn-2}
\!\!\!{\|{\bf u}\|_{{\mathcal H}^1(\Omega _{-})^n}^2\!\leq c\|{\mathbb E}({\bf u})\|_{L_2(\Omega_-)^{n\times n}}^2}
\!
&\leq cc_{\mathbb A}\left\langle a_{ij}^{\alpha \beta }E_{j\beta }({\bf u}),E_{i\alpha }({\bf u})\right\rangle _{\Omega _{-}}\nonumber
\\
&=\!cc_{\mathbb A}a_{{{\mathbb A};\Omega _{-}}}({\bf u},{\bf u})\quad \forall \, {\bf u}\in {\mathcal H}^1_{D;\rm{div}}(\Omega_{-})^n,
\end{align}
where the first inequality follows from Theorem \ref{Korn-exterior}, while the second inequality is a consequence of assumption \eqref{mu}.
Inequalities \eqref{a-1-v-d-Stokes-mixed} and \eqref{estimate-Korn-2} show that
$a_{{{\mathbb A};\Omega _{-}}}(\cdot ,\cdot ):{\mathcal H}_D^1(\Omega _{-})^n
\times {\mathcal H}_D^1(\Omega _{-})^n\to {\mathbb R}$
is bounded, while
$a_{{{\mathbb A};\Omega _{-}}}(\cdot ,\cdot ):{\mathcal H}^1_{D;\rm{div}}(\Omega _{-})^n
\times {\mathcal H}^1_{D;\rm{div}}(\Omega _{-})^n\to {\mathbb R}$
is coercive.
The bilinear form $b_{_{\Omega _{-}}}(\cdot ,\cdot):{\mathcal H}_D^1(\Omega _{-})^n\times L_2(\Omega _{-})\to {\mathbb R}$ given by \eqref{b-Omega-p} is bounded, as well.

In addition, since the divergence operator in \eqref{div-tilde-Stokes-D}
is bounded and surjective and there is a bounded right inverse $R_{\Omega _-}$ satisfying \eqref{Borchers-Shor-div-2}  with a constant $C_D>0$, we obtain
 that for any $q\in L_2(\Omega _{-})$ the equation $-{\rm{div}}\, {\bf v}=q$ has a solution ${\bf v}\in {\mathcal H}_D^1(\Omega _{-})^n$, which satisfies the inequality
$\|{\bf v}\|_{{\mathcal H}_D^1(\Omega _{-})^n}\leq C_D\|q\|_{L_2(\Omega _{-})}$. Therefore, we obtain the inequality
\begin{align*}
b_{_{\Omega _{-}}}({\bf v},q)=-\left\langle {\rm{div}}\, {\bf v},q\right\rangle _{\Omega _{-}}=\langle q,q\rangle _{\Omega _{-}}=\|q\|_{L_2(\Omega _{-})}^2\geq C_D^{-1}\|{\bf v}\|_{{\mathcal H}_D^1(\Omega _{-})^n}\|q\|_{L_2(\Omega _{-})},
\end{align*}
which shows that the bounded bilinear form $b_{_{\Omega _{-}}}(\cdot ,\cdot):{\mathcal H}_D^1(\Omega _{-})^n\times L_2(\Omega _{-})\to {\mathbb R}$ satisfies the condition
\begin{align}
\label{inf-sup-Omega-Stokes-mixed}
\inf _{q\in L_2(\Omega _{-})\setminus \{0\}}\sup _{{\bf v}\in {\mathcal H}_D^1(\Omega _{-})^n\setminus \{\bf 0\}}\frac{b_{_{\Omega _{-}}}({\bf v},q)}{\|{\bf v}\|_{{\mathcal H}_D^1(\Omega _{-})^n}\|q\|_{L_2(\Omega _{-})}}\geq C_D^{-1}.
\end{align}
Then Theorem \ref{B-B} with $X={\mathcal H}_D^1(\Omega _{-})^n$,
$V={\mathcal H}_{D;{\rm{div}}}^1(\Omega_{-})^n$
and $\mathcal M=L_2(\Omega _{-})$ implies that for any $(\mathbf F ,\eta )\in \left({\mathcal H}_D^{1}(\Omega _-)^n\right)'\times L_2(\Omega _-)$, there exists a unique solution $({\bf u},\pi )\in {\mathcal H}_D^1(\Omega _{-})^n\times L_2(\Omega _{-})$ of the variational problem
\begin{equation}
\label{variational-mixed-2a-DN}
\left\{\begin{array}{lll}
a_{{\mathbb A};\Omega _-}({\bf u},{\bf w})+b_{\Omega _-}({\bf w},\pi )=
\mathbf F ({\bf w}) & \forall \ {\bf w}\in {\mathcal H}_D^{1}(\Omega _-)^n,\\
b_{_{\Omega _-}}({\bf u},q)=\eta (q) & \forall \ q\in L_2(\Omega _-),
\end{array}\right.
\end{equation}
which depends continuously on the given data $\mathbf F $ and $\eta $. In view of definition \eqref{variational-mixed-D-N} of the operator ${\mathcal T}_{D;\Omega _{-}}:{\mathcal X}_{D;2}(\Omega _-)\to {\mathcal X}_{D;2}'(\Omega _-)$, we deduce that $({\bf u},\pi )\in {\mathcal X}_{D;2}(\Omega _-)$ is the unique solution of the equation ${\mathcal T}_{D:\Omega _{-}}({\bf u},\pi )=(\mathbf F ,\eta )$, and accordingly that this operator is an isomorphism.

Consequently, for all given data $\tilde{\bf f}\in \left({\mathcal H}^{1}_{D}(\Omega _-)^n\right)'$, $g\in L_2(\Omega_-)$ and $\boldsymbol\psi\in H^{-\frac{1}{2}}(N)^n$, the variational problem \eqref{variational-mixed-2a} has a unique solution $({\bf u},\pi )\in {\mathcal H}^1_{D}(\Omega _-)^n\times L_2(\Omega _-)$. The equivalence of the Poisson problem of mixed type \eqref{mixed-B-manifold} with the variational problem \eqref{variational-mixed-2a} shows that the pair $({\bf u},\pi )$ is also the unique solution in the space ${\mathcal H}^1_{D}(\Omega _-)^n\times L_2(\Omega _-)$ of the mixed problem \eqref{mixed-B-manifold}, and depends continuously on $\tilde{\bf f}$, $g$ and $\boldsymbol\psi$.
\end{proof}

Theorem \ref{mixed-B} implies the following well-posedness result for the inhomogeneous Dirichlet condition (cf. \cite[Theorem 8.5 ]{Ch-Mi-Na-3} for the exterior non-homogeneous mixed problem for a strongly elliptic operator).
\begin{thm}
\label{mixed-B-1}
Let conditions \eqref{Stokes-1}-\eqref{mu} hold in $\Omega_-$.
Then for all
$\tilde{\bf f}\in \left({\mathcal H}^{1}_{D}(\Omega _-)^n\right)'$, $g\in L_2(\Omega_-)$, $\boldsymbol\varphi\in H^{\frac{1}{2}}(D)^n$ and $\boldsymbol\psi\in H^{-\frac{1}{2}}(N)^n$, the problem
\begin{align}
\label{mixed-B-inhomogeneous}
\left\{\begin{array}{llll}
{\boldsymbol{\mathcal L}({\bf u},\pi )=\tilde{\bf f}}|_{\Omega _-},\ \
{{\rm{div}}}\, {\bf u}=g & \mbox{ in } \Omega _-\,,\\
r_{_D}\gamma _-{\bf u}=\boldsymbol\varphi & \mbox{ on } D,\\
r_{_N}{\bf t}^{-}({\bf u},\pi ;\tilde{\bf f})=\boldsymbol\psi & \mbox{ on } N
\end{array}
\right.
\end{align}
has a unique solution $({\bf u},\pi )\!\in \!{\mathcal H}^1(\Omega _-)^n\!\times \! L_2(\Omega _-)$, and there exists a constant ${\mathcal C}_{DN}\!=\!{\mathcal C}_{DN}(\partial \Omega ,c_{\mathbb A},n)\!>\!0$ such that
\begin{align*}
\|{\bf u}\|_{{\mathcal H}^1(\Omega _-)^n}+\|\pi \|_{{L_2(\Omega _-)}}\leq
{\mathcal C}_{DN}\left(\|\tilde{\bf f}\|_{({\mathcal H}^1_{D}(\Omega _-))'}
+\|g\|_{L_{2}(\Omega_-)}
+\|\boldsymbol\varphi\|_{ H^{\frac{1}{2}}(D)^n}
+\|\boldsymbol\psi\|_{H^{-\frac{1}{2}}(N)^n}\right).
\end{align*}
\end{thm}
\begin{proof}
Let ${\mathfrak E}_{_D}$
be a linear continuous extension operator from $H^{\frac{1}{2}}(D)^n $ to $H^{\frac{1}{2}}(\partial \Omega )^n$. Then for given datum $\boldsymbol\varphi\in H^{\frac{1}{2}}(D)^n$, we consider its extension $\breve{\boldsymbol\varphi}:={\mathfrak E}_{_D}\boldsymbol\varphi$ to $H^{\frac{1}{2}}(\partial \Omega )^n$ and the Dirichlet problem
\begin{align}
\label{Dirichlet-B-inhomogeneous-3}
\left\{\begin{array}{llll}
{\boldsymbol{\mathcal L}(\breve{\bf u},\breve \pi )={\bf 0}},\ \
{{\rm{div}}}\, \breve{\bf u}=0 & \mbox{ in } \Omega _-\,,\\
\gamma _-\breve {\bf u}=\breve{\boldsymbol\varphi} & \mbox{ on } \partial \Omega \,
\end{array}
\right.
\end{align}
for $(\breve{\bf u},\breve\pi )\!\in \!{\mathcal H}^1(\Omega _-)^n\!\times \! L_2(\Omega _-)$.
By Theorem~\ref{T-p-complete}, problem \eqref{Dirichlet-B-inhomogeneous-3} is uniquely solvable and the map
$\breve{\boldsymbol\varphi}\mapsto \left(\breve{\bf u},\breve{\pi }\right)$ is linear and continuous. Hence the map $\boldsymbol\varphi\mapsto \left(\breve{\bf u},\breve{\pi }\right)$ is linear and continuous as well.

Let  $\boldsymbol\psi^0:=\boldsymbol\psi-r_{_N}{\bf t}^{-}(\breve{\bf u},\tilde{\pi })\in H^{-\frac{1}{2}}(N)^n$ and the pair $({\bf u}^0,\pi ^0)\in {\mathcal H}^1_{D}(\Omega _-)^n\times L_2(\Omega _-)$ be unique solution of the problem
\begin{align}
\label{mixed-B-inhomogeneous-add}
\left\{\begin{array}{llll}
{\boldsymbol{\mathcal L}({\bf u}^0,\pi ^0)=\tilde{\bf f}},\ \
{{\rm{div}}}\, {\bf u}^0= g & \mbox{ in } \Omega _-\,,\\
r_{_D}\gamma _-{\bf u}^0={\bf 0} & \mbox{ on } D,\\
r_{_N}{\bf t}^{-}({\bf u}^0,\pi ^0;\tilde{\bf f})=\boldsymbol\psi^0 & \mbox{ on } N\,.
\end{array}
\right.
\end{align}
The well-posedness of problem \eqref{mixed-B-inhomogeneous-add} is implied by Theorem \ref{mixed-B}.

By \eqref{Dirichlet-B-inhomogeneous-3} and  \eqref{mixed-B-inhomogeneous-add}, we deduce that the pair
$({\bf u},\pi ):=\left(\breve{\bf u}+{\bf u}^0,\breve{\pi }+\pi ^0\right)$ belongs to the space
${\mathcal H}^1(\Omega _-)^n\times L_2(\Omega _-)$, is a solution of the exterior mixed problem \eqref{mixed-B-inhomogeneous} and depends continuously on the given data
$\tilde{\bf f}$, $g$, $\boldsymbol\varphi$ and $\boldsymbol\psi$.
The uniqueness result in Theorem \ref{mixed-B} shows that this solution is unique.
\end{proof}

\subsection{Variational approach to a transmission problem for the Stokes system in a bounded composite domain}

To facilitate considering in Section \ref{NSTP} the Navier-Stokes problems in bounded and unbounded domains, we show in this section the well-posedness of a linear Dirichlet-transmission problem for the anisotropic Stokes system in a  bounded Lipschitz composite domain of ${\mathbb R}^n$ ($n\geq 3$), consisting of two adjacent bounded Lipschitz domains such that the boundary of one of them is their common interface (see Theorem \ref{T-2}).
\subsubsection{\bf Notations and auxiliary results}
Let $\Omega _0$ be a bounded Lipschitz domain in ${\mathbb R}^n$. 
\comment{
Next we consider the {\it homogeneous Sobolev space} $\dot{H}^1 (\Omega _0)^n$
defined as the closure of the space ${\mathcal D}(\Omega _0)^n$ with respect to the norm
\begin{align}
\label{seminorm-Omega}
\|{\bf u}\|_{\dot{H}^1(\Omega _0)^n}:=|{\bf u}|_{{H}^1(\Omega _0)^n}=\|\nabla {\bf u}\|_{L_2(\Omega _0)^{n\times n}}\,.
\end{align}
}
{\mn Let us denote by $\mathring{H}^1(\Omega _0)^n$ the closed subspace of the space $H^1(\Omega _0)$ consisting of functions with zero traces on $\partial \Omega _0$.
The semi-norm 
\begin{align}
\label{seminorm-Omega}
|{\bf u}|_{\mathring{H}^1(\Omega _0)^n}=|{\bf u}|_{{H}^1(\Omega _0)^n}=\|\nabla {\bf u}\|_{L_2(\Omega _0)^{n\times n}}\,.
\end{align}
 is also a norm in} $\mathring{H}^1(\Omega _0)$  equivalent to the norm
\begin{align}
\label{norm-H1}
\|{\bf u}\|_{H^1(\Omega _0)^n}=\|{\bf u}\|_{L_2(\Omega _0)^n}+\|\nabla {\bf u}\|_{L_2(\Omega _0 )^{n\times n}}
\end{align}
(cf., e.g., Theorem II.5.1 and Remark II.6.2 {\mn in \cite{Galdi}).}
Let also ${H}^{-1}(\Omega _0)^n=\left(\mathring{H}^1(\Omega _0)^n\right)'$ and let
$|\!|\!|\cdot |\!|\!|_{H^{-1}(\Omega _0)^n}$ denote the corresponding norm on ${H}^{-1}(\Omega _0)^n$ generated by the {\mn seminorm $|\cdot|_{\mathring{H}^1(\Omega _0)^n}$} in \eqref{seminorm-Omega}, i.e.,
\begin{align}
\label{norm-3e}
|\!|\!|{\bf g}|\!|\!|_{{H}^{-1}(\Omega _0)^n}:=\sup_{{\bf v}\in \mathring{H}^1(\Omega _0)^n,\ \|\nabla {\bf v}\|_{L_2(\Omega _0)^{n\times n}}=1}|\langle {\bf g},{\bf v}\rangle _{\Omega _0}| \quad \forall \ {\bf g}\in {H}^{-1}(\Omega _0)^n.
\end{align}

{Let $\Omega$ be a bounded or exterior Lipschitz domain and let us introduce the spaces
\begin{align}
\label{data-1}
&{\mathcal D}_{\rm{div}}(\Omega)^n:=\{{\bf w}\in {\mathcal D}(\Omega)^n:{\rm{div}}\, {\bf w}=0 \mbox{ in } \Omega\},\\
\label{data-div}
&\mathring{\mathcal H}_{\rm{div}}^1(\Omega)^n:=\{{\bf w}\in \mathring{\mathcal H}^1(\Omega)^n:{\rm{div}}\, {\bf w}=0 \mbox{ in } \Omega\},
\\
\label{data-divH}
&\mathring{H}_{\rm{div}}^1(\Omega)^n:=\{{\bf w}\in \mathring{H}^1(\Omega)^n:{\rm{div}}\, {\bf w}=0 \mbox{ in } \Omega\},\\
&(\mathring H^{1}_{\rm div}(\Omega)^n)^\bot=\{\boldsymbol\Phi\in H^{-1}(\Omega)^n :
\langle\boldsymbol\Phi,\mathbf v\rangle_{\Omega}= 0\quad \forall\,  \mathbf v\in \mathring H^{1}_{\rm div}(\Omega)^n\}.
\end{align}
Thus $(\mathring H^{1}_{\rm div}(\Omega)^n)^\bot$ is the subspace of $H^{-1}(\Omega)^n$ orthogonal to  $\mathring H^{1}_{\rm div}(\Omega)^n$, and {note that} its dual, $[(\mathring H^{1}_{\rm div}(\Omega)^n)^\bot]'$, can be identified with the space
$\mathring H^{1}(\Omega)^n/\mathring H^{1}_{\rm div}(\Omega)^n.$}
For $\mn\mathbf v\in \mathring{H}^1(\Omega)^n/\mathring{H}_{\rm{div}}^1(\Omega)^n$ 
we also introduce the following norm in this quotient space,
\begin{align*}
\mn|\mathbf v|_{\mathring{H}^1(\Omega)^n/\mathring{H}_{\rm{div}}^1(\Omega)^n}
:={\rm inf}_{\boldsymbol\phi\in\mathring{H}_{\rm{div}}^1(\Omega)^n}
\|\nabla ({\bf v}-\boldsymbol\phi)\|_{L_2(\Omega _0)^{n\times n}}.
\end{align*}

We will further need the following analogue of Lemma \ref{Bog} for bounded domains, cf.\cite[Lemma 2.2]{LS1976}.
\begin{lem}\label{La-So}
 Let $\Omega$ be a bounded Lipschitz domain.
The divergence operator
$
{\rm{div}}:\mathring{H}^1(\Omega)^n\to L_{2;0}(\Omega)
$
is surjective and there exists a linear operator
$R_{\Omega}:L_{2;0}(\Omega)\to \mathring{H}^1(\Omega)^n$
such that
$
{\rm{div}}(R_{\Omega}g)=g
$
for any $g\in L_{2;0}(\Omega)$ and
${\mn|R_{\Omega}g|_{\mathring{H}^1(\Omega)^n}}\le C_\Omega\|g\|_{L_{2}(\Omega)}$
with a positive constant $C_\Omega$ that does not dependent on $g$.
\end{lem}
\begin{rem}\label{La-SoC}
 By simple scaling it is easy to show (cf. \cite[Corollary 2.1]{LS1976}) that the operator $R_{\Omega}$ in Lemma~\ref{La-So} can be chosen in such a way that the constant $C_\Omega$ may depend on the shape of $\Omega$ but not on {the domain scaling}. Particularly, if $\Omega$ is a ball $B$, then {$C_\Omega$ will not depend on its diameter.}
\end{rem}

Let us provide a theorem about isomorphism of divergence and gradient operators (cf., e.g., Proposition 1.1 and Remark 1.4 in \cite[Chapter 1]{Temam}).
\begin{thm}\label{La-SoT}
{Let $\Omega$ be a bounded Lipschitz domain {in ${\mathbb R}^n$, $n\geq 3$.
Then the} operators}
\begin{align}
\label{div-is}
&{\rm{div}}:\mathring{H}^1(\Omega)^n/\mathring{H}_{\rm{div}}^1(\Omega)^n\to L_{2;0}(\Omega),\\
\label{grad-is}
&{\rm{grad}}:L_{2}(\Omega)/\mathbb R\to (\mathring H^{1}_{\rm div}(\Omega)^n)^\bot
\end{align}
are isomorphisms {and the following estimates hold}
\begin{align}
\label{div-est}
&{\mn|\mathbf v|_{\mathring{H}^1(\Omega)^n/\mathring{H}_{\rm{div}}^1(\Omega)^n}}
\le C_\Omega\|{\rm{div}}\, \mathbf v\|_{L_{2}(\Omega)}\quad
\forall\, \mathbf v\in \mathring{H}^1(\Omega)^n/\mathring{H}_{\rm{div}}^1(\Omega)^n,\\
\label{grad-est}
&\|q\|_{L_{2}(\Omega)/\mathbb R}
\le C_\Omega|\!|\!|{\rm{grad}}\, q|\!|\!|_{ H^{-1}(\Omega)^n}\quad
\forall\, q\in L_{2}(\Omega)/\mathbb R,
\end{align}
with a positive constant $C_\Omega$ that may depend on the shape of $\Omega$ but not on the domain scaling. Particularly, if $\Omega$ is a ball $B$, then $C_\Omega$ will not depend on its diameter.
\end{thm}
\begin{proof}
In view of Lemma \ref{La-So}, the operator ${\rm{div}}:\mathring{H}^1(\Omega)^n/\mathring{H}_{\rm{div}}^1(\Omega)^n\to L_{2;0}(\Omega)$ is an isomorphism, and then, its adjoint, ${\rm{grad}}:L_{2}(\Omega)/\mathbb R\to (\mathring H^{1}_{\rm div}(\Omega)^n)^\bot$ is an isomorphism as well.
In addition, the isomorphism property of the operator in \eqref{div-is} yields that
\begin{align*}
{\mn|\mathbf v|_{\mathring{H}^1(\Omega)^n/\mathring{H}_{\rm{div}}^1(\Omega)^n}}
\leq \|{\rm{div}}^{-1}\|\, \|{\rm{div}}\, \mathbf v\|_{L_{2}(\Omega)}
\leq C_\Omega \|{\rm{div}}\, \mathbf v\|_{L_{2}(\Omega)}
\quad
\forall\, \mathbf v\in \mathring{H}^1(\Omega)^n/\mathring{H}_{\rm{div}}^1(\Omega)^n,
\end{align*}
where 
$$
\mn \|\mathrm{div}^{-1}\|:=\sup_{g\in  L_{2;0}(\Omega),\ \|g\|_{L_2(\Omega)}=1}|\mathrm{div}^{-1}g|_{\mathring{H}^1(\Omega)^n/\mathring{H}_{\rm{div}}^1(\Omega)^n}
$$
and we have used the inequality $\|{\rm{div}}^{-1}\|\leq C_\Omega $ provided by Lemma \ref{La-So}. Therefore, estimate \eqref{div-est} follows. Estimate \eqref{grad-est} can be obtained with a similar argument.
{Finally, the claim that $C_\Omega$ may depend on the shape of $\Omega$ but not on its scale follows from Remark~\ref{La-SoC}}.
\end{proof}

\subsubsection{\bf The Poisson problem of Dirichlet-transmission type for the Stokes system in a bounded composite domain}\label{S5.3-new}

Let $n\geq 3$. 
Let $\Omega ^0\subset {\mathbb R}^n$ be a bounded {\it composite} domain consisting of two Lipschitz domains $\Omega _+$ and $\Omega _-^0$ such that $\overline{\Omega }_+\subset \Omega _0$ and $\Omega _-^0:=\Omega^0\setminus \overline\Omega _{+}$. Let $\partial\Omega^0 $ be the boundary of $\Omega ^0$. Thus, $\Omega ^0=\overline \Omega _+\cup \Omega ^0_{-}$, and the boundary $\partial \Omega $ of $\Omega _+$ is also the interface between $\Omega _+$ and $\Omega _-^0$.

Let us introduce the following spaces,
\begin{align}
& L_{2;0}(\Omega):=\left\{f\in L_2(\Omega):\langle f,1\rangle _{\Omega}=0\right\},\\
\label{5.8}
&{H}_{\partial\Omega^0} ^1(\Omega ^0_{-})^n:=\left\{{\bf v}\in H^1(\Omega ^0_{-})^n:\gamma _+{\bf v}={\bf 0} \mbox{ on } \partial\Omega^0 \right\},\\
\label{sol-dt}
&{\mathfrak X}_{\Omega _+,\Omega ^0_{-}}
:= {H}^1(\Omega _+)^n
\times (L_{2}(\Omega^0)/\mathbb R)\big|_{\Omega_+}
\times {H}^1(\Omega^0_{-})^n
\times (L_{2}(\Omega^0)/\mathbb R)\big|_{\Omega^0_-}\,,\\
\label{data-dt}
&{\mathfrak Y}_{\Omega _+,\Omega^0_-}
:=\widetilde{H}^{-1}(\Omega _{+})^n
\times\widetilde H^{-1}_{\partial\Omega}(\Omega^0_-)^n
\times L_{2;0}(\Omega^0)
\times H^{-\frac{1}{2}}(\partial \Omega )^n\,.
\end{align}
Here $\widetilde H^{-1}_{\partial\Omega}(\Omega^0_-)^n$ is defined as a space of distributions from $H^{-1}(\Omega^0)^n$ that vanish on $\Omega _+$, i.e., $\widetilde H^{-1}_{\partial\Omega}(\Omega^0_-)^n:=\left\{\boldsymbol\varphi \in H^{-1}(\Omega^0)^n: \boldsymbol\varphi ={\bf 0} \mbox{ on } \Omega _+\right\}.$
This space can be identified with the dual to ${H}_{\partial\Omega^0}^1(\Omega ^0_{-})^n$, cf., e.g., arguments of Theorems 3.29 and 3.30 in \cite{Lean}.

Let $\boldsymbol{\mathcal L}$ be the Stokes operator defined in \eqref{Stokes-new}.
Next, we consider the Poisson problem of Dirichlet-transmission type for the anisotropic Stokes system 
\begin{equation}
\label{Dirichlet-var-Stokes}
\left\{
\begin{array}{ll}
{\boldsymbol{\mathcal L}({\bf u}_+,\pi _+)={\tilde{\bf f}_{+}}|_{\Omega _{+}}}\,, \ \
{\rm{div}} \, {\bf u}_+= g|_{\Omega_+} & \mbox{ in } \Omega _{+}\,,\\
{\boldsymbol{\mathcal L}({\bf u}_-,\pi _-)={\tilde{\bf f}_{-}}|_{\Omega _{-}^0}}\,, \ \
{\rm{div}} \, {\bf u}_-= g|_{\Omega^0_-} & \mbox{ in } \Omega ^0_{-},\\
{\gamma }_{+}{\bf u}_+-{\gamma }_{-}{\bf u}_{-}={\bf 0} &  \mbox{ on } \partial \Omega ,\\
{\bf t}^{+}({\bf u}_+,\pi _+;\tilde{\bf f}_+)-{\bf t}^{-}({\bf u}_-,\pi _- ;\tilde{\bf f}_-)=\boldsymbol\psi &  \mbox{ on } \partial \Omega ,\\
{\gamma }_{+}{\bf u}_-={\bf 0} &  \mbox{ on } \partial\Omega^0 \,,
\end{array}\right.
\end{equation}
with given data $(\tilde{\bf f}_+,\tilde{\bf f}_-, g,\boldsymbol\psi)\in {\mathfrak Y}_{\Omega _+,\Omega^0_-}$
and unknown
$({\bf u}_+,\pi _+,{\bf u}_-,\pi _-)\in {\mathfrak X}_{\Omega _+,\Omega^0_-}$.

Let us consider the following mixed variational formulation.

{\it Given $\mathbf F\in {H}^{-1}(\Omega ^0)^n$,
find
$({\bf u},\pi )\in \mathring{H}^1(\Omega ^0)^n\times L_{2}(\Omega^0)/\mathbb R$ such that
\begin{equation}
\label{transmission-linear-Stokes-a}
\left\{\begin{array}{lll}
a_{{\mathbb A};\Omega ^0}({\bf u},{\bf v})+b_{\Omega ^0}({\bf v},\pi )=\langle \mathbf F,{\bf v}\rangle _{\Omega ^0} & \forall \, {\bf v}\in
\mathring{H}^{1}(\Omega ^0)^n,\\
b_{\Omega ^0}({\bf u},q)= -\langle g,q\rangle_{\Omega ^0} & \forall \, q\in L_{2}(\Omega ^0)/\mathbb R,
\end{array}
\right.
\end{equation}
where
$a_{{\mathbb A};\Omega ^0}{:\mathring{H}^1(\Omega ^0)^n\!\times \!\mathring{H}^1(\Omega ^0)^n\!\to \! {\mathbb R}}$ and $b_{\Omega ^0}:{\mathring{H}^1(\Omega ^0)^n\!\times \!{L_{2}(\Omega ^0)/{\mathbb R}}\!\to \!{\mathbb R}}$ are the bounded bilinear forms}
\begin{align}
\label{a-v-Omega0}
&a_{{\mathbb A};\Omega ^0}({\bf u},{\bf v}):
=\left\langle a_{ij}^{\alpha \beta }E_{j\beta }({\bf u}),E_{i\alpha }({\bf v})\right\rangle _{\Omega ^0} \quad \forall \, {\bf u},{\bf v}\in \mathring{H}^1(\Omega ^0)^n\,,\\
\label{b-v-Omega0}
&b_{\Omega ^0}({\bf v},q):=-\langle {\rm{div}}\, {\bf v},q\rangle _{\Omega ^0} \quad \forall \,
{\bf v}\in \mathring{H}^1(\Omega ^0)^n\ \ \forall \, q\in {L_{2}(\Omega ^0)/{\mathbb R}}\,.
\end{align}

The variational problem \eqref{transmission-linear-Stokes-a} is equivalent to the Dirichlet problem
\begin{equation}
\label{int-NS-D-00S}
\hspace{-1em}\left\{
\begin{array}{ll}
{\boldsymbol{\mathcal L}({\bf u},\pi)}=-\mathbf F\,, \ \
{\rm{div}} \, {\bf u}=0
& \mbox{ in } \Omega^0,\\
{\gamma }_{+}{\bf u}={\bf 0} &  \mbox{ on } \partial\Omega^0,
\end{array}\right.
\end{equation}
{with the unknown} $({\bf u},\pi )\in {H}^1(\Omega ^0)^n\times L_{2}(\Omega^0)/\mathbb R$.

An argument similar to that for problem \eqref{var-Brinkman-transmission-sl} implies that the Dirichlet-transmission problem \eqref{Dirichlet-var-Stokes}
{is also equivalent to the mixed variational formulation \eqref{transmission-linear-Stokes-a}-\eqref{F}
if $\mathbf F\in {H}^{-1}(\Omega ^0)^n$ is taken as
\begin{align*}
\langle \mathbf F,{\bf v}\rangle _{\Omega ^0}
&=-\big\langle \tilde{\bf f}_+,{\bf v}{|_{\Omega_+}}\big\rangle_{\Omega_+}
-\big\langle \tilde{\bf f}_-,{\bf v}{|_{\Omega^0_-}}\big\rangle_{\Omega^0_-}+
\langle \boldsymbol\psi,\gamma {\bf v}\rangle _{\partial \Omega }\nonumber\\
&=-\big\langle \tilde{\bf f}_+,{\bf v}\big\rangle _{\Omega ^0}-\big\langle \tilde{\bf f}_-,{\bf v}\big\rangle _{\Omega ^0}+\langle \gamma ^*\boldsymbol\psi,{\bf v}\rangle _{\Omega ^0}
 \quad \forall \, {\bf v}\in \mathring{H}^{1}(\Omega ^0)^n\,,
\end{align*}
i.e.,
\begin{align}
\label{F}
\mathbf F=-(\tilde{\bf f}_++\tilde{\bf f}_-)+\gamma ^*\boldsymbol\psi.
\end{align}
Note that $\gamma ^*:H^{-\frac{1}{2}}(\partial \Omega )^n\to {H}^{-1}(\mathbb R^n)^n$ is the adjoint of the trace operator $\gamma :{H}^{1}(\mathbb R^n)^n\to H^{\frac{1}{2}}(\partial \Omega )^n$, and the support of $\gamma ^*\boldsymbol\psi$ is a subset of $\partial \Omega $.
The solutions of \eqref{Dirichlet-var-Stokes} and \eqref{transmission-linear-Stokes-a} are related as}
\begin{align}\label{u}
{\bf u}_+={\bf u}|_{\Omega _+},\ {\bf u}_-={\bf u}|_{\Omega_-^0},\
\pi_+=\pi|_{\Omega _+},\ \pi_-=\pi|_{\Omega _-^0}.
\end{align}

\begin{thm}
\label{T-2}
Let conditions \eqref{Stokes-1}-\eqref{mu} hold in $\Omega^0$. {Then} for any $\mathbf F\in {H}^{-1}(\Omega ^0)^n$, the linear variational problem \eqref{transmission-linear-Stokes-a}-\eqref{b-v-Omega0} {and the Dirichlet problem \eqref{int-NS-D-00S}} have a unique solution $({\bf u},\pi )\in \mathring{H}^1(\Omega^0)^n\times L_2(\Omega^0)/\mathbb R$ and the following estimates hold
\begin{align}
\label{mixed-Cu0}
&{\mn\|\nabla{\bf u}\|_{L_2(\Omega^0)^{n\times n}}}
\leq {2}c_{\mathbb A}|\!|\!|\mathbf F|\!|\!|_{H^{-1}(\Omega ^0)^n}
+C'_{\Omega^0}\|g\|_{L_{2}(\Omega^0)},\\
\label{mixed-Cp0}
&\|\pi\|_{L_2(\Omega^0)/\mathbb R}
\leq C'_{\Omega^0}
|\!|\!|\mathbf F|\!|\!|_{H^{-1}(\Omega ^0)^n}
+C^*_{\Omega^0}\|g\|_{L_{2}(\Omega^0)},
\end{align}
where the {ellipticity} constant $c_{\mathbb A}$ is defined in \eqref{mu}, the constant $C_{\Omega^0}$ in {Lemma} $\ref{La-So}$, while
$$
C'_{\Omega^0}:=C_{\Omega^0}(1+{2}c_{\mathbb A}n^4\|{\mathbb A}\|_{L_\infty ({\Omega^0})^n}),\quad
C^*_{\Omega^0}:=n^4\|{\mathbb A}\|_{L_\infty ({\Omega^0})^n}C_{\Omega^0}C'_{\Omega^0}.
$$

In addition, {given $(\tilde{\bf f}_+,\tilde{\bf f}_-, g,\boldsymbol\psi)\in {\mathfrak Y}_{\Omega _+,\Omega^0_-}$}, relations \eqref{u} provide the unique solution $({\bf u}_+,\pi _+,{\bf u}_-,\pi _-)$ in ${\mathfrak X}_{\Omega _+,\Omega^0_-}$ of the linear Dirichlet-transmission problem \eqref{Dirichlet-var-Stokes} whenever $\mathbf F\in {H}^{-1}(\Omega ^0)^n$ is taken as in \eqref{F}.
\end{thm}
\begin{proof}
The bilinear form
$a_{{\mathbb A};\Omega ^0}:\mathring{H}^1(\Omega )^n\times \mathring{H}^1(\Omega ^0)^n\to {\mathbb R}$ is bounded on the Hilbert space $\mathring{H}^1(\Omega ^0)^n$
and by \eqref{A-norm}, {we have the following estimate {\mn for} the bilinear form $a_{{\mathbb A};\Omega ^0}$ on $\mn\mathring{H}^1(\Omega )^n\times \mathring{H}^1(\Omega ^0)^n$,}
\begin{align}\label{Anorm0}
\mn|a_{{\mathbb A};\Omega ^0}({\bf v},{\bf w})|
\le n^4\|{\mathbb A}\|_{L_\infty ({\Omega^0})^n}
\|\nabla{\bf v}\|_{L_2(\Omega^0)^{n\times n}}\|\nabla{\bf w}\|_{L_2(\Omega^0)^{n\times n}}\quad 
\forall\, {\bf v},{\bf w}\in \mathring{H}^1(\Omega ^0)^n.
\end{align}
Moreover, condition \eqref{mu} and the first Korn inequality (see, e.g., Theorem 10.1 in \cite{Lean}) show that
\begin{align*}
\frac{1}{2}c_{\mathbb A}^{-1}\|\nabla {\bf v}\|_{L_2(\Omega ^0)^{n\times n}}^2\leq c_{\mathbb A}^{-1}\|{\mathbb E}({\bf v})\|_{L_2(\Omega ^0)^{n\times n}}^2
\leq \left\langle a_{ij}^{\alpha \beta }E_{j\beta }({\bf v}),E_{i\alpha }({\bf v})\right\rangle _{\Omega ^0} \quad \forall \, {\bf v}\in \mathring{H}_{\rm{div}}^1(\Omega ^0)^n\,,
\end{align*}
implying that $a_{{\mathbb A};\Omega ^0}$ is also coercive on $\mathring{H}_{\rm{div}}^1(\Omega ^0)^n$ and
\begin{align}\label{a-coers}
a_{{\mathbb A};\Omega ^0}({\bf v},{\bf v})\ge \frac{1}{2c_{\mathbb A}}{\mn\|\nabla{\bf v}\|_{L_2(\Omega^0)^{n\times n}}^2}
\quad \forall\, {\bf v}\in \mathring{H}_{\rm{div}}^1(\Omega ^0)^n.
\end{align}
{By Theorem~\ref{La-SoT} and Lemma~\ref{surj-inj-inf-sup}}, we obtain that the bilinear form
$b_{_{\Omega ^0}}(\cdot ,\cdot):\mathring{H}^1(\Omega ^0)^n\times L_2(\Omega ^0)\to {\mathbb R}$
satisfies the inf-sup condition
\begin{align}
\label{inf-sup-Omega-Stokes0}
\inf _{\left[\!\left[q \right]\!\right]\in L_2(\Omega ^0)/\mathbb R\setminus \{0\}}
\sup _{{\bf v}\in \mathring{H}^1(\Omega ^0)^n\setminus \{\bf 0\}}
\frac{b_{_{\Omega ^0}}({\bf v},\left[\!\left[q \right]\!\right])}
{{\mn\|\nabla{\bf v}\|_{L_2(\Omega^0)^{n\times n}}}\left[\!\left[q \right]\!\right]\|_{L_2(\Omega ^0)/\mathbb R}}
\geq C_{\Omega^0}^{-1}.
\end{align}

Then due to estimates \eqref{Anorm0}, \eqref{a-coers}, \eqref{inf-sup-Omega-Stokes0}, Theorem \ref{B-B} with $X=\mathring{H}^1(\Omega^0)^n$,
$V=\mathring{H}_{\rm{div}}^1(\Omega^0)^n$
and $\mathcal M=L_2(\Omega^0)/\mathbb R$ implies that for any
$(\mathbf F ,g )\in {H}^{-1}(\Omega^0)^n\times L_{2;0}(\Omega^0)$, there exists a unique solution
$({\bf u},\pi )\in \mathring{H}^1(\Omega^0)^n\times L_2(\Omega^0)/\mathbb R$ of the variational problem \eqref{transmission-linear-Stokes-a} and inequalities \eqref{mixed-Cu0}, \eqref{mixed-Cp0} hold.
\end{proof}

\section{Transmission problems for the anisotropic Navier-Stokes systems with $L_\infty $ coefficients and general data in $L_2$-based Sobolev spaces}\label{NSTP}

In this section we show the existence of a weak solution of a transmission problem in ${\mathbb R}^3$ for the anisotropic Navier-Stokes system
with general data in $L_2$-based weighted Sobolev spaces (see problem \eqref{Stokes-NS}).
We combine a well-posedness result for the linear Stokes system with the Leray-Schauder fixed point Theorem, and show existence of
solutions for the nonlinear {Dirichlet and} Dirichlet-transmission problems for the anisotropic Navier-Stokes system in a bounded Lipschitz composite domain (see Theorem \ref{int-D-NS-var}). Further, we construct a sequence of weak solutions $\{{\bf u}_k\}_{k\in {\mathbb N}}$ for Dirichlet-transmission problems for the Navier-Stokes system in a sequence of increasing bounded Lipschitz domains which approximate the exterior Lipschitz domain, such that each ${\bf u}_k$ satisfies {an estimate uniformly in $k$}.
Finally, we take the limit and obtain a
solution of the nonlinear transmission problem
in the sense of distributions (see Theorem \ref{existence-N-S-variable-ext}).

\subsection{Dirichlet and Dirichlet-transmission problems for the incompressible Navier-Stokes system in a bounded Lipschitz composite domain}\label{S5.3}

Next we need the following form of the Leray-Schauder fixed point theorem (see, e.g., \cite[Theorem 11.3]{Gilbarg-Trudinger}).
\begin{thm}
\label{L-S-fixed}
Let $X$ be a Banach space. Let $T:X\to X$ be a continuous and compact operator. If the set
$\left\{x\in X : x=\chi Tx \mbox{ for some } \chi\in[0,1]\right\}$
is bounded, then the operator $T$ has a fixed point.
\end{thm}

\subsubsection{\bf Existence of solution}

Let us consider the following mixed nonlinear variational problem in a bounded Lipschitz domain $\Omega^0$.

{\it Given $\mathbf F\in {H}^{-1}(\Omega ^0)^n$,
find $({\bf u},\pi )\in \mathring{H}^1(\Omega ^0)^n\times  L_{2}(\Omega^0)/\mathbb R$
such that}
\begin{equation}
\label{transmission-linear-NS-a}
\left\{\begin{array}{lll}
a_{{\mathbb A};\Omega ^0}({\bf u},{\bf v})+b_{\Omega ^0}({\bf v},\pi )=\langle \mathbf F,{\bf v}\rangle _{\Omega ^0}
-\!\left\langle({\bf u}\cdot \nabla ){\bf u},{\bf v}\right\rangle_{\Omega ^0}&
\forall \, {\bf v}\in \mathring{H}^{1}(\Omega ^0)^n,\\[1ex]
b_{\Omega ^0}({\bf u},q)=0
&
\forall \, q\in L_{2}(\Omega ^0)/{\mathbb R}\,,
\end{array}
\right.
\end{equation}
{\it where the bilinear forms
$a_{{\mathbb A};\Omega ^0}:\mathring{H}^1(\Omega )^n\times \mathring{H}^1(\Omega ^0)^n\to {\mathbb R}$
and
$b_{\Omega^0}: \mathring{H}^1(\Omega^0)^n\times L_2(\Omega^0)/\mathbb R\to \mathbb R$
are defined in \eqref{a-v-Omega0}, \eqref{b-v-Omega0}.
}

The {variational problem \eqref{transmission-linear-NS-a} is equivalent to the Dirichlet problem
\begin{equation}
\label{int-NS-D-00}
\hspace{-1em}\left\{
\begin{array}{ll}
{\boldsymbol{\mathcal L}({\bf u},\pi)}=-\mathbf F+({\bf u}\cdot\nabla){\bf u}\,, \ \
{\rm{div}} \, {\bf u}=0
& \mbox{ in } \Omega^0,\\
{\gamma }_{+}{\bf u}={\bf 0} &  \mbox{ on } \partial\Omega^0,
\end{array}\right.
\end{equation}
{with the unknown} $({\bf u},\pi )\in {H}^1(\Omega ^0)^n\times  L_{2}(\Omega^0)/\mathbb R$.

Now we can prove the existence result for variational problem \eqref{transmission-linear-NS-a} and hence for the Dirichlet problem \eqref{int-NS-D-00}, for the anisotropic $L_\infty -$coefficient Navier-Stokes equation in a bounded domain.}
\begin{thm}
\label{int-D-NS-var}
Let $n=3$. 
Let conditions \eqref{Stokes-1}-\eqref{mu} hold in $\Omega^0$. 
Then for any $\mathbf F\in {H}^{-1}(\Omega ^0)^n$ there exists a pair 
$({\bf u},\pi )\in { \mathring{H}^1(\Omega^0)^n}\times L_2(\Omega^0)/\mathbb R$ which satisfies the nonlinear variational problem \eqref{transmission-linear-NS-a}, as well as the nonlinear Dirichlet problem \eqref{int-NS-D-00}. Moreover, the following estimates hold
\begin{align}
\label{mixed-Cu0N}
&{\mn\|\nabla{\bf u}\|_{L_2(\Omega^0)^{n\times n}}}
\leq {2}c_{\mathbb A}|\!|\!|\mathbf F|\!|\!|_{{H^{-1}(\Omega ^0)^n}}
\\
\label{mixed-Cp0N}
&\|\pi\|_{L_2(\Omega^0)/\mathbb R}
\leq  C'_{\Omega^0}|\!|\!|\mathbf F|\!|\!|_{{H^{-1}(\Omega ^0)^n}}
+C''_{\Omega^0}|\Omega^0|^{1/6}|\!|\!|\mathbf F|\!|\!|^2_{{H^{-1}(\Omega ^0)^n}},
\end{align}
where the constant $c_{\mathbb A}$ is defined in \eqref{mu},
\begin{align}\label{E5.5}
C'_{\Omega^0}:=C_{\Omega^0}(1+{2}c_{\mathbb A}n^4\|{\mathbb A}\|_{L_\infty ({\Omega^0})^n}),\quad
C''_{\Omega^0}:=\frac{16}{3}{C_{\Omega^0}}c^2_{\mathbb A},
\end{align}
while $C_{\Omega^0}$ is as in {Lemma} $\ref{La-So}$ and does not depend on the diameter of $\Omega^0$.
\end{thm}
\begin{proof}
Let us first assume that $n\!\in \!\{3,4\}$. 
For any $\mathbf F\in {H}^{-1}(\Omega ^0)^3$
and an arbitrary, fixed element ${\bf w}\in {\mathring{H}^1(\Omega ^0)^n}$, we define the functional
\begin{align}
\label{P}
&\mathbf F_{\mathbf w}:\mathring{H}^1(\Omega ^0)^n\to {\mathbb R},\
\langle\mathbf F_{\mathbf w},{\bf v}\rangle _{\Omega ^0}:= \langle\mathbf F,{\bf v}\rangle _{\Omega ^0}-\left\langle({\bf w}\cdot \nabla ){\bf w},{\bf v}\right\rangle _{\Omega ^0} \quad
\forall \, {\bf v}\in {\mathring{H}^1(\Omega ^0)^n}.
\end{align}
and inequality \eqref{P-01} implies that it is well-defined, linear and continuous (cf. \cite[Proposition 1.1]{Seregin}),
and hence $\mathbf F_{\mathbf w}\in {H}^{-1}(\Omega ^0)^n$.

By Theorem~\ref{T-2}, there exists a unique solution $({\bf u}_{\mathbf w},\pi_{\mathbf w} )\in \mathring{H}^1(\Omega^0)^n\times L_2(\Omega^0)/\mathbb R$ to the linear variational problem \eqref{transmission-linear-Stokes-a},
\begin{equation}
\label{P-2}
\left\{\begin{array}{lll}
a_{{\mathbb A};\Omega ^0}({\bf u}_{\mathbf w},{\bf v})+b_{\Omega ^0}({\bf v},\pi_{\mathbf w})=
\langle\mathbf F_{\mathbf w},{\bf v}\rangle _{\Omega ^0}&
\forall \, {\bf v}\in \mathring{H}^{1}(\Omega ^0)^n,\\[1ex]
b_{\Omega ^0}({\bf u}_{\mathbf w},q)=0
&
\forall \, q\in L_{2}(\Omega ^0)/{\mathbb R}\,.
\end{array}
\right.
\end{equation}

Consequently, we can define the operators
\begin{align}
\label{N}
&{\mathbf U}:\mathring{H}^1(\Omega ^0)^n
\to \mathring{H}^1(\Omega ^0)^n,\
\boldsymbol{\bf U}(\mathbf w):={\bf u}_{\mathbf w},\\
\label{P}
&{\rm P}:\mathring{H}^1(\Omega ^0)^n
\to L_{2}(\Omega ^0)/{\mathbb R},\
{\rm P}(\mathbf w):=\pi_{\mathbf w},
\end{align}
which associate to $\mathbf w\in \mathring{H}^1(\Omega ^0)^n$ the unique solution
$({\bf u}_{\mathbf w},\pi_{\mathbf w})\in \mathring{H}^1(\Omega^0)^n\times L_2(\Omega^0)/\mathbb R$ of the equation \eqref{P-2}.
In addition, we also define the operator
\begin{align}\label{Ufr}
&\mathfrak U:\mathring{H}^1(\Omega ^0)^n\times L_{2}(\Omega ^0)/{\mathbb R}\to \mathring{H}^1(\Omega ^0)^n
\times L_{2}(\Omega ^0)/{\mathbb R},\
{\mathfrak U}(\mathbf w,p)=({\mathbf U}(\mathbf w),{\rm P}(\mathbf w)),
\end{align}
where $p\in L_{2}(\Omega ^0)/{\mathbb R}$ is a fictitious parameter.

Next we show that the nonlinear operator $\mathfrak U$ has a fixed point
$(\mathbf u,\pi)\in \mathring{H}^1(\Omega ^0)^n\times L_2(\Omega^0)/\mathbb R$ , i.e., $\mathfrak U(\mathbf u,\pi)=(\mathbf u,\pi)$, which will be a solution of the nonlinear variational problem \eqref{transmission-linear-NS-a}.
We intend to apply the variant of the Leray-Schauder fixed point Theorem stated in Theorem \ref{L-S-fixed}.

First, we show that $\mathfrak U$ is continuous.
Let ${\bf w},{\bf w}_0\in \mathring{H}^1(\Omega ^0)^n$.
Then by \eqref{P-2}, \eqref{mixed-Cu0} and \eqref{P-01},
\begin{align}\label{L5.33}
{\mn\|\nabla({\boldsymbol{\bf U}(\mathbf w)-\boldsymbol{\bf U}(\mathbf w_0))}\|_{L_2(\Omega^0)^{n\times n}}}
&\le c_{\mathbb A}|\!|\!|\mathbf F_{\mathbf w}-\mathbf F_{\mathbf w_0}|\!|\!|_{H^{-1}(\Omega ^0)^n}
\le c_{\mathbb A}|\!|\!|({\bf w}\cdot \nabla ){\bf w}
-({\bf w}_0\cdot \nabla ){\bf w}_0|\!|\!|_{H^{-1}(\Omega ^0)^n}\nonumber\\
&\le c_4c_{\mathbb A}\|(({\bf w}-{\bf w}_0)\cdot \nabla ){\bf w}
+({\bf w}_0\cdot \nabla )({\bf w}-{\bf w}_0)\|_{H^{-1}(\Omega ^0)^n}\nonumber\\
&\leq 2c_4'c_{\mathbb A}\|{\bf w}-{\bf w}_0\|_{H^1(\Omega ^0)^n}\left(\|{\bf w}\|_{H^1(\Omega ^0)^n}
+\|{\bf w}_0\|_{H^1(\Omega ^0 )^n}\right).
\end{align}
for some constants $c_4,c_4'>0$. Similarly, by \eqref{P-2}, \eqref{mixed-Cp0} and \eqref{P-01},
\begin{align}\label{L5.33P}
\|{\mathrm P(\mathbf w)-\mathrm P(\mathbf w_0)}\|_{L_2(\Omega^0)/\mathbb R}
&\le C'_{\Omega^0}
|\!|\!|\mathbf F_{\mathbf w}-\mathbf F_{\mathbf w_0}|\!|\!|_{H^{-1}(\Omega ^0)^n}
\nonumber\\
&\leq c_4'C_{\Omega^0}'\|{\bf w}-{\bf w}_0\|_{H^1(\Omega ^0)^n}\left(\|{\bf w}\|_{H^1(\Omega ^0)^n}
+\|{\bf w}_0\|_{H^1(\Omega ^0 )^n}\right).
\end{align}

Therefore, the operator $\mathfrak U$ defined by \eqref{Ufr} is continuous.

Let us now assume that $n=3$ and show that the operator $\mathfrak U$ is compact.
To this end, assume that $\{({\bf w}_k,p_k)\}_{k\in {\mathbb N}}$ is a bounded sequence in
$\mathring{H}^1(\Omega ^0 )^3\times L_2(\Omega^0)/\mathbb R$.
Thus, there exists a constant $M>0$ such that 
$\mn\|\nabla{\bf w}_k\|_{L_2(\Omega^0)^{3\times 3}}\leq M$, 
for any $k\in {\mathbb N}$. We claim that the sequence $\{\boldsymbol{\bf U}({\bf w}_k)\}_{k\in {\mathbb N}}$ contains a convergent subsequence in $\mathring{H}^1(\Omega ^0)^3$.

Indeed, the continuous embedding of the space $\mathring{H}^1(\Omega ^0)^3$ into the space $L_{6}(\Omega^0)^3$, by the Rellich theorem implies compactness of the embedding of the space $\mathring{H}^1(\Omega ^0)^3$ into the space
$L_3(\Omega )^3$.
This, in turn, implies that there exists a subsequence of $\{{\bf w}_k\}_{k\in {\mathbb N}}$, labeled as the sequence, which converges in $L_3(\Omega ^0)^3$, and, thus, is a Cauchy sequence in $L_3(\Omega ^0)^3$.
Then we show that the corresponding subsequence $\{{\bf u}_k\}_{k\in {\mathbb N}}$, ${\bf u}_k=\boldsymbol{\bf U}({\bf w}_k)$, is a Cauchy sequence in $\mathring{H}^1(\Omega ^0)^3$.

As in \eqref{L5.33}, we obtain by \eqref{P-2}, \eqref{mixed-Cu0}, \eqref{P-04} and \eqref{P-06},
\begin{align}\label{L5.33k}
\mn\|\nabla(\boldsymbol{\bf U}(\mathbf w_k)-\boldsymbol{\bf U}(\mathbf w_\ell))\|_{L_2(\Omega^0)^{3\times 3}}
&\le {2}c_{\mathbb A}|\!|\!|\mathbf F_{\mathbf w_k}-\mathbf F_{\mathbf w_\ell}|\!|\!|_{H^{-1}(\Omega ^0)^3}
\le {2}c_{\mathbb A}|\!|\!|({\bf w}_k\cdot \nabla ){\bf w}_k
-({\bf w}_\ell\cdot \nabla ){\bf w}_\ell|\!|\!|_{H^{-1}(\Omega ^0)^3}
\nonumber\\
&={2}c_{\mathbb A}|\!|\!|(({\bf w}_k-{\bf w}_\ell)\cdot \nabla ){\bf w}_k
+({\bf w}_\ell\cdot \nabla )({\bf w}_k-{\bf w}_\ell)|\!|\!|_{H^{-1}(\Omega ^0)^3}.
\nonumber\\
&\hspace{-8em}\leq {2}c_{\mathbb A}\|{\bf w}_k-{\bf w}_\ell\|_{L_3(\Omega ^0)^3}
\left(c_7\|{\bf w}_k\|_{H^1(\Omega ^0)^3}+c_8\|{\bf w}_\ell\|_{H^1(\Omega ^0 )^3}\right)
\leq C_*\|{\bf w}_k-{\bf w}_\ell \|_{L_3(\Omega ^0)^3},
\end{align}
where $C_*={2}c_{\mathbb A}(c_7+c_8)M>0$.
Similarly, by \eqref{P-2}, \eqref{mixed-Cp0}, \eqref{P-04} and \eqref{P-06},
\begin{align}\label{L5.33kP}
\|\mathrm P({\bf w}_k)-\mathrm P({\bf w}_\ell)\|_{_{L_2(\Omega^0)/\mathbb R}}
&\le C'_{\Omega^0}|\!|\!|\mathbf F_{\mathbf w_k}-\mathbf F_{\mathbf w_\ell}|\!|\!|_{H^{-1}(\Omega ^0)^3}
\leq C_{**}\|{\bf w}_k-{\bf w}_\ell \|_{L_3(\Omega ^0)^3},
\end{align}
where $C_{**}=C'_{\Omega^0}(c_7+c_8)M>0$.

Therefore, the convergence of $\{{\bf w}_k\}_{k\in {\mathbb N}}$  in $L_3(\Omega )^3$ implies that
$\{\mathfrak U({\bf w}_k,p_k)\}_{k\in {\mathbb N}}$ is a Cauchy sequence and thus converges in the space $\mathring{H}^1(\Omega ^0)^3\times L_2(\Omega^0)/\mathbb R$.
Consequently, the operator $\mathfrak U$ defined by \eqref{Ufr}, with $n=3$, is compact.

In order to apply the statement of Theorem \ref{L-S-fixed} to the operator $\mathfrak U$, it remains to show that the following set is bounded
\begin{align}
\label{A}
{\mathcal A}:=\left\{({\bf w},p)\in \mathring{H}^1(\Omega ^0)^3\times L_2(\Omega^0)/\mathbb R :
({\bf w},p)=\lambda \mathfrak U({\bf w},p),\ 0\leq \lambda \leq 1\right\}\,.
\end{align}

Let $\lambda \in (0,1]$ and $({\bf w},p)\in \mathring{H}^1(\Omega ^0)^3\times L_2(\Omega^0)/\mathbb R$ be such that
$({\bf w},p)=\lambda \mathfrak U({\bf w},p)$,
Thus, $\frac{1}{\lambda}{\bf w}\!=\!\boldsymbol{\bf U}({\bf w})$ and $\frac{1}{\lambda}p={\rm P}(\mathbf w)$.
By \eqref{P-2} and \eqref{N}-\eqref{P}, this implies
\begin{align}
\label{P-7}
\left\langle a_{ij}^{\alpha \beta }E_{j\beta }({\bf w}),E_{i\alpha }({\bf v})\right\rangle _{\Omega ^0}
-\langle {\rm{div}}\,{\bf v},p\rangle _{\Omega ^0}
&=\lambda\langle \mathbf F_{\bf w},{\bf v}\rangle _{\Omega ^0 }\nonumber\\
&=\lambda\langle \mathbf F,{\bf v}\rangle _{\Omega ^0 }
-\lambda\left\langle({\bf w}\cdot \nabla ){\bf w},{\bf v}\right\rangle _{\Omega ^0} \quad
\forall \, {\bf v}\in \mathring{H}^1(\Omega ^0)^3\,,\\
\label{P-7div}
-\langle {\rm{div}}\,{\bf w},q\rangle _{\Omega ^0} =0 \quad \forall \, q\in {L_{2}(\Omega ^0)/{\mathbb R}}\,.&
\end{align}
Equation \eqref{P-7div} implies that ${\rm{div}}\,{\bf w}=0$, i.e.,
${\bf w}\in \mathring{H}_{\rm{div}}^1(\Omega ^0)^3$.
Let us take ${\bf v}={\bf w}$ in \eqref{P-7}.
Then by relation \eqref{P-5a} we obtain
\begin{align}
\label{P-8}
\left\langle a_{ij}^{\alpha \beta }E_{j\beta }({\bf w}),E_{i\alpha }({\bf w})\right\rangle_{\Omega ^0}
=\lambda \langle \mathbf F,{\bf w}\rangle _{\Omega ^0}\,.
\end{align}
Finally, according to the first Korn inequality, condition \eqref{mu},  definition \eqref{norm-3e} of the norm $|\!|\!|\cdot |\!|\!|_{{H}^{-1}(\Omega ^0)^3}$ of
the space ${H}^{-1}(\Omega )^3$, and by equality \eqref{P-8}, we obtain that
\begin{align*}
\frac{1}{2}c_{\mathbb A}^{-1}\|\nabla {\bf w}\|_{L_2(\Omega ^0)^{3\times 3}}^2
\leq c_{\mathbb A}^{-1}\|{\mathbb E}({\bf w})\|_{L_2(\Omega ^0)^{3\times 3}}^2
\leq \left\langle a_{ij}^{\alpha \beta }E_{j\beta }({\bf w}),E_{i\alpha }({\bf w})\right\rangle _{\Omega ^0}\leq \lambda |\!|\!|\mathbf F|\!|\!|_{H^{-1}(\Omega ^0)^3}\|\nabla {\bf w}\|_{L_2(\Omega ^0)^{3\times 3}}.
\end{align*}
Therefore, for any $\lambda \in [0,1]$, we have the inequality
\begin{align}
\label{P-12}
\mn\|\nabla {\bf w}\|_{L_2(\Omega ^0)^{3\times 3}}&
\leq 2{\lambda }c_{\mathbb A}|\!|\!|\mathbf F|\!|\!|_{H^{-1}(\Omega ^0)^3}
\leq {2}c_{\mathbb A}|\!|\!|\mathbf F|\!|\!|_{H^{-1}(\Omega ^0)^3}\,.
\end{align}

On the other hand, Theorem~\ref{T-2} implies that if $p$ satisfies the variational problem \eqref{P-7}-\eqref{P-7div}, then for any $\lambda \in [0,1]$,
\begin{align*}
\|p\|_{L_2(\Omega^0)/\mathbb R}\leq\lambda C'_{\Omega^0}|\!|\!|\mathbf F_{\bf w}|\!|\!|_{H^{-1}(\Omega ^0)^3}
\leq C'_{\Omega^0}|\!|\!|\mathbf F_{\bf w}|\!|\!|_{H^{-1}(\Omega ^0)^3}
\leq C'_{\Omega^0}\left(|\!|\!|\mathbf F|\!|\!|_{H^{-1}(\Omega ^0)^3}
+|\!|\!|({\bf w}\cdot \nabla ){\bf w}|\!|\!|_{H^{-1}(\Omega ^0)^3}\right).
\end{align*}
Since $n=3$, from \eqref{P-061} and \eqref{P-12} we have,
\begin{align}\label{E5.20}
|\!|\!|({\bf w}\cdot \nabla ){\bf w}|\!|\!|_{H^{-1}(\Omega ^0)^3}
\le  \frac{4}{3}|\Omega^0|^{1/6}\mn\|\nabla {\bf w}\|_{L_2(\Omega ^0)^{3\times 3}}^2&
\le \frac{16}{3}|\Omega^0|^{1/6}c^2_{\mathbb A} |\!|\!|\mathbf F|\!|\!|^2_{H^{-1}(\Omega ^0)^3}\,,
\end{align}
and hence
\begin{align}\label{E5.41-new}
\|p\|_{L_2(\Omega^0)/\mathbb R}
\le C'_{\Omega^0}|\!|\!|\mathbf F|\!|\!|_{H^{-1}(\Omega ^0)^3}
+C'_{\Omega^0}\frac{16}{3}|\Omega^0|^{1/6}c^2_{\mathbb A} |\!|\!|\mathbf F|\!|\!|^2_{H^{-1}(\Omega ^0)^3}\,.
\end{align}
Inequalities \eqref{P-12} and \eqref{E5.41-new} show that the set ${\mathcal A}$ given by \eqref{A} is bounded
in $\mathring{H}^1(\Omega ^0)^3\times L_2(\Omega^0)/\mathbb R$.

Consequently, the operator $\mathfrak U$ in \eqref{Ufr} satisfies the assumption of Theorem \ref{L-S-fixed}, and then there exist a couple $(\mathbf u,\pi)\in \mathring{H}^1(\Omega ^0)^3\times L_2(\Omega^0)/\mathbb R$ such that $\mathfrak U(\mathbf u,\pi)=(\mathbf u,\pi)$, or, equivalently, $(\mathbf u,\pi)$ satisfies the variational problem \eqref{transmission-linear-NS-a} {and hence the Dirichlet problem \eqref{int-NS-D-00},} and estimates \eqref{P-12}, \eqref{E5.41-new} hold with $\mathbf w=\mathbf u$ and $p=\pi$.

{Estimate \eqref{P-12} gives the desired estimate \eqref{mixed-Cu0N}, but estimate \eqref{E5.41-new} for $\pi$ can be further improved as follows.
The first equation in \eqref{int-NS-D-00} implies that
\begin{align}\label{E5.21}
\nabla\pi={\mathbf F}+{\rm{div}}\left({\mathbb A}{\mathbb E}({\bf u})\right)
-{({\bf u}\cdot \nabla ){\bf u}} \mbox{ in } \Omega^0.
\end{align}
Then by inequality \eqref{grad-est} in Theorem~\ref{La-SoT} along with \eqref{E5.21},  we obtain that
\begin{align}
\label{weak-Dpi-e1b}
\|\pi\|_{L_2(\Omega^0)/\mathbb R}
&\leq C_{\Omega^0}\big(|\!|\!|{\bf F}|\!|\!|_{H^{-1}(\Omega^0)^n}
+|\!|\!|{\rm{div}}\left({\mathbb A}{\mathbb E}({\bf u})\right)|\!|\!|_{H^{-1}(\Omega^0)^n}
+|\!|\!|({\bf u}\cdot \nabla){\bf u}|\!|\!|_{H^{-1}(\Omega^0)^n}\big),
\end{align}
By \eqref{A-norm}, and \eqref{mixed-Cu0N},
\begin{align}\label{E5.59b}
|\!|\!|{\rm{div}}\left({\mathbb A}{\mathbb E}({\bf u})\right)|\!|\!|_{H^{-1}(\Omega^0)^n}&=
{\sup_{\boldsymbol \Psi \in \mathring{H}^1(\Omega ^0)^n,\,
\|\nabla \boldsymbol \Psi \|_{L_2(\Omega^0)^n}=1}}\Big|\left\langle{\rm{div}}\left({\mathbb A}{\mathbb E}({\bf u})\right),\boldsymbol \Psi\right\rangle _{\Omega^0}\Big|\nonumber\\
&={\sup_{\boldsymbol \Psi \in \mathring{H}^1(\Omega ^0)^n,\,
\|\nabla \boldsymbol \Psi \|_{L_2(\Omega^0)^n}=1}}\Big|\left\langle{\mathbb A}{\mathbb E}({\bf u}),\nabla \boldsymbol \Psi\right\rangle _{\Omega^0}\Big|\nonumber\\
&\leq \|{\mathbb A}{\mathbb E}({\bf u})\|_{L_2(\Omega ^0)^{n\times n}}\leq n^4\|{\mathbb A}\|_{L_\infty (\Omega ^0)}\|\nabla \, {\bf u}\|_{L_2(\Omega ^0)^{n\times n}}\nonumber\\
&
\leq 2c_{{\mathbb A}}n^4\|{\mathbb A}\|_{L_\infty (\Omega ^0)}|\!|\!|{\bf F}|\!|\!|_{H^{-1}(\Omega^0)^n},
\end{align}
Hence, substituting  in \eqref{weak-Dpi-e1b} estimate \eqref{E5.59b} and estimate \eqref{E5.20} for $\mathbf w=\mathbf u$, we obtain  \eqref{mixed-Cp0N}.
}
\end{proof}

Let us now consider the following  Dirichlet-transmission problem for the Navier-Stokes system in $\Omega^0$.
\begin{equation}
\label{int-NS-D-0}
\hspace{-1em}\left\{
\begin{array}{ll}
{\boldsymbol{\mathcal L}({\bf u}_+,\pi_+)}={\tilde{\bf f}_+}|_{\Omega _+}+({\bf u}_+\cdot\nabla){\bf u}_+\,, \ \
{\rm{div}} \, {\bf u}_+=0
& \mbox{ in } \Omega _{+},\\
{\boldsymbol{\mathcal L}({\bf u}_-,\pi _-)}={\tilde{\bf f}_{-}}|_{\Omega _{-}^0}
+({\bf u}_-\cdot \nabla ){\bf u}_- \,, \ \
{\rm{div}} \, {\bf u}_-=0
& \mbox{ in } \Omega ^0_{-},\\
{\gamma }_{+}{\bf u}_+-{\gamma }_{-}{\bf u}_{-}={\bf 0} &  \mbox{ on } \partial \Omega ,\\
{\bf t}^{+}\left({\bf u}_+,\pi _+;\tilde{\bf f}_++\mathring{E}_{\Omega _+}\left(({\bf u}_+\cdot \nabla ){\bf u}_+\right)\right)-{\bf t}^{-}\left({\bf u}_-,\pi _- ;\tilde{\bf f}_-+\mathring{E}_{\Omega ^0_-}\left(({\bf u}_-\cdot \nabla ){\bf u}_-\right)\right)=\boldsymbol\psi &  \mbox{ on } \partial \Omega ,\\
{\gamma }_{+}{\bf u}_-={\bf 0} &  \mbox{ on } \partial\Omega^0
\end{array}\right.
\end{equation}
{with given data $\big({\tilde{\bf f}}_{+},{\tilde{\bf f}}_{-},\boldsymbol\psi\big)\in
{\widetilde{H}^{-1}(\Omega _{+})^n
\times \widetilde H^{-1}_{\partial\Omega}(\Omega^0_-)^n
\times H^{-\frac{1}{2}}(\partial \Omega )^n\,}$
and unknown
$({\bf u}_+,\pi _+,{\bf u}_-,\pi _-)\in {\mathfrak X}_{\Omega _+,\Omega^0_-}$.
}
Here
$\Omega ^0$, $\Omega _+$ and $\Omega ^0_{-}$, $\partial\Omega$ and $\partial\Omega^0$ are the domains and boundaries introduced at the beginning of Section~\ref{S5.3-new},
{$\boldsymbol{\mathcal L}$ is the Stokes operator defined in \eqref{Stokes-new}}, and $\mathring{E}_{\Omega _+}$ and $\mathring{E}_{\Omega ^0_-}$ are the operators of extension by zero outside $\Omega _+$ and $\Omega _-^0$, respectively.

Then an argument similar to that for problem \eqref{var-Brinkman-transmission-sl}
implies that the nonlinear problem \eqref{int-NS-D-0} is equivalent to the mixed variational formulation \eqref{transmission-linear-NS-a} in the following sense.
\begin{thm}
\label{int-D-NS-var-Dtr}
Let $n=3$ and {let conditions} \eqref{Stokes-1}-\eqref{mu} hold in $\Omega^0$.
{For given data
$\big({\tilde{\bf f}}_{+},{\tilde{\bf f}}_{-},\boldsymbol\psi\big)\in
{{\widetilde{H}^{-1}(\Omega _{+})^n}
\times {\left({H}_{\partial\Omega^0} ^1(\Omega ^0_{-})^n\right)'}
\times H^{-\frac{1}{2}}(\partial \Omega )^n}$} let $({\bf u},\pi )\in {\mathring{H}^1(\Omega^0)^n}\times L_2(\Omega^0)/\mathbb R$ be the solution of variational problem \eqref{transmission-linear-NS-a} provided by Theorem $\ref{int-D-NS-var}$ {whenever}
$\mathbf F=-(\tilde{\bf f}_++\tilde{\bf f}_-)+\gamma ^*\boldsymbol\psi.$
Then there exists a solution
$({\bf u}_+,\pi _+,{\bf u}_-,\pi _-)\in {\mathfrak X}_{\Omega _+,\Omega^0_-}$ of the nonlinear
Dirichlet-transmission problem \eqref{int-NS-D-0} given by the relations
${\bf u}_+={\bf u}|_{\Omega _+},\ {\bf u}_-={\bf u}|_{\Omega_-^0},\ \pi_+=\pi|_{\Omega _+},\ \pi_-=\pi|_{\Omega _-^0},$
and estimates \eqref{mixed-Cu0N}, \eqref{mixed-Cp0N} hold.
\end{thm}

\subsubsection{\bf Uniqueness of the weak solution in a bounded domain}
In the case $n=3$, the space $\mathring{H}^1(\Omega ^0)^n$ is continuously embedded in $L_4(\Omega ^0)^n$. Moreover, the semi-norm $\|{\nabla \bf v} \|_{L^2(\Omega ^0)^{n\times n}}$ is a norm on $\mathring{H}^1(\Omega ^0)^n$, which is equivalent to the norm $\|{\bf v} \|_{{H}^1(\Omega ^0)^n}$
given by \eqref{Sobolev}.
Then there exists a constant $c=c(n,\Omega ^0)>0$ such that
\begin{align}
\label{L4}
\|{\bf v}\|_{L_4(\Omega ^0)^n}\leq c\|\nabla {\bf v}\|_{L^2(\Omega ^0)^{n\times n}} \quad \forall \ {\bf v}\in \mathring{H}^1(\Omega ^0)^n.
\end{align}
This inequality and an additional constraint to the given data of the nonlinear Dirichlet-transmission problem \eqref{int-NS-D-0} imply the following uniqueness result (see also \cite[Lemma 3.1]{Seregin} and \cite[Corollary 1]{Ot-Sa} {in the isotropic case \eqref{isotropic} with {$\mu =1$, $\lambda =0$}}
and homogeneous Dirichlet condition, and \cite[Theorem 4.2]{K-M-W-2} for a pseudostress approach).
\begin{thm}
\label{well-posed-N-S-Stokes-small}
Let $n=3$ and conditions \eqref{Stokes-1}-\eqref{mu} hold on $\Omega^0$.
Let $\mathbf F\in {H}^{-1}(\Omega ^0)^n$ satisfy the condition
\begin{align}
\label{uniqueness}
4c_{\mathbb A}^2c^2\|\mathbf F\|_{{H}^{-1}(\Omega ^0)^n}<1,
\end{align}
where $c_{\mathbb A}$ is the constant in \eqref{mu}, $c$ is the constant in inequality \eqref{L4}.
Then the Dirichlet problem
\eqref{int-NS-D-00} has a unique solution
$({\bf u},\pi )\in {H}^1(\Omega^0)^n\times L_2(\Omega^0)/\mathbb R$.

If, moreover,  $\mathbf F=-(\tilde{\bf f}_+ +\tilde{\bf f}_-)+\gamma ^*\boldsymbol\psi$ and
$\big({\tilde{\bf f}}_{+},{\tilde{\bf f}}_{-},\boldsymbol\psi\big)\in
{\widetilde{H}^{-1}(\Omega _{+})^n
\times \left({H}_{\partial\Omega^0} ^1(\Omega ^0_{-})^n\right)'
\times H^{-\frac{1}{2}}(\partial \Omega )^n}$
then the Dirichlet-transmission
{problem} \eqref{int-NS-D-0} has a unique {solution}
$({\bf u}_+,\pi _+,{\bf u}_-,\pi _-)\in {\mathfrak X}_{\Omega _+,\Omega^0_-}$.
\end{thm}
\begin{proof}
Assume that problem \eqref{int-NS-D-00} has two solutions $({\bf u}_1,\pi_1)$ and $({\bf u}_2,\pi_2)$. Therefore, each of them belongs to the space $\mathring{H}^1(\Omega^0)^n\times L_2(\Omega^0)/\mathbb R$ and satisfies variational problem \eqref{transmission-linear-NS-a}, 
as well as inequality \eqref{mixed-Cu0N}.
Then we obtain that for any ${\bf v}\in {\mathcal D}(\Omega ^0)^n$,
\begin{equation}
\label{NS-var-eq-int-unique}
\left\langle a_{ij}^{\alpha \beta }E_{j\beta }({\bf u}_1-{\bf u}_2),E_{i\alpha }({\bf v })\right\rangle _{\Omega ^0}+\left\langle({\bf u}_1\cdot \nabla ){\bf u}_1-({\bf u}_2\cdot \nabla ){\bf u}_2,{\bf v}\right\rangle _{\Omega ^0}
-\langle {\rm{div}}\, {\bf v},\pi_1-\pi_2\rangle _{\Omega ^0} =0.
\end{equation}
Moreover, the dense embedding of the space ${\mathcal D}_{\rm{div}}(\Omega ^0)^n$ in $\mathring{H}_{\rm{div}}^1({\Omega }^0)^n$ (see, e.g., \cite[p.143]{Amrouche-3}) shows that relation \eqref{NS-var-eq-int-unique} is satisfied also for any ${\bf v}\in \mathring{H}_{\rm{div}}^1({\Omega }^0)^n$. Then by choosing ${\bf v}={\bf u}_1-{\bf u}_2$ in \eqref{NS-var-eq-int-unique}, we obtain
\begin{align}
\label{NS-var-eq-int-unique-0}
\left\langle a_{ij}^{\alpha \beta }E_{j\beta }({\bf u}_1-{\bf u}_2),E_{i\alpha }({\bf u}_1-{\bf u}_2)\right\rangle _{\Omega ^0}=&-\left\langle\left(({\bf u}_1-{\bf u}_2)\cdot \nabla )\right){\bf u}_1,{\bf u}_1-{\bf u}_2\right\rangle _{\Omega ^0}\nonumber\\
&-\left\langle({\bf u}_2\cdot \nabla )({\bf u}_1-{\bf u}_2),{\bf u}_1-{\bf u}_2\right\rangle _{\Omega ^0}\,.
\end{align}
Due to the membership of ${\bf u}_1$ and ${\bf u}_2$ { in} the space $\mathring{H}_{\rm{div}}^1({\Omega }^0)^n$, relation \eqref{P-5a} implies that
\begin{align}
\left\langle({\bf u}_2\cdot \nabla )({\bf u}_1-{\bf u}_2),{\bf u}_1-{\bf u}_2\right\rangle _{\Omega ^0}=0\,,
\end{align}
and accordingly that equation \eqref{NS-var-eq-int-unique-0} reduces to
\begin{align}
\label{NS-var-eq-int-uniqe-1}
\left\langle a_{ij}^{\alpha \beta }E_{j\beta }({\bf u}_1-{\bf u}_2),E_{i\alpha }({\bf u}_1-{\bf u}_2)\right\rangle _{\Omega ^0}=&-\left\langle\left(({\bf u}_1-{\bf u}_2)\cdot \nabla )\right){\bf u}_1,{\bf u}_1-{\bf u}_2\right\rangle _{\Omega ^0}\,.
\end{align}
On the other hand, in view of {condition \eqref{mu} and the first Korn inequality}, we deduce that
\begin{align}
\label{P-11-unique}
\|\nabla ({\bf u}_1-{\bf u}_2)\|_{L_2(\Omega ^0)^{n\times n}}^2\, \leq
2c_{\mathbb A}\left\langle a_{ij}^{\alpha \beta }E_{j\beta }({\bf u}_1-{\bf u}_2),E_{i\alpha }({\bf u}_1-{\bf u}_2)\right\rangle _{\Omega ^0}\,,
\end{align}
and by the H\"{o}lder inequality and inequalities \eqref{L4} and \eqref{mixed-Cu0N},
we obtain that
\begin{align}
\label{NS-var-eq-int-uniqe-2}
\left\langle\left(({\bf u}_1-{\bf u}_2)\cdot \nabla )\right){\bf u}_1,{\bf u}_1-{\bf u}_2\right\rangle _{\Omega ^0}&\leq \|{\bf u}_1-{\bf u}_2\|_{L_4(\Omega ^0)^n}^2\|\nabla {\bf u}_1\|_{L_2(\Omega ^0)^{n\times n}}\nonumber\\
&\leq c^2\|\nabla ({\bf u}_1-{\bf u}_2)\|_{L^2(\Omega ^0)^{n\times n}}^2\|\nabla {\bf u}_1\|_{L_2(\Omega ^0)^{n\times n}}\nonumber\\
&\leq 2c_{{\mathbb A}}c^2|\!|\!|\mathbf F|\!|\!|_{{H}^{-1}(\Omega ^0)^n}\|\nabla ({\bf u}_1-{\bf u}_2)\|_{L^2(\Omega ^0)^{n\times n}}^2.
\end{align}
Then equality \eqref{NS-var-eq-int-uniqe-1} and inequalities \eqref{P-11-unique} and \eqref{NS-var-eq-int-uniqe-2} imply that
\begin{align}
\label{uniqueness-1}
\|\nabla ({\bf u}_1-{\bf u}_2)\|_{L_2(\Omega ^0)^{n\times n}}^2\leq 4c_{{\mathbb A}}^2c^2|\!|\!|\mathbf F|\!|\!|_{{H}^{-1}(\Omega ^0)^n}\|\nabla ({\bf u}_1-{\bf u}_2)\|_{L_2(\Omega ^0)^{n\times n}}^2.
\end{align}
However, assumption \eqref{uniqueness} shows that estimate \eqref{uniqueness-1} is possible only if ${\bf u}_1-{\bf u}_2={\bf 0}$.

Hence, \eqref{NS-var-eq-int-unique} reduces to
$\langle {\rm{div}}\, {\bf v},\pi_1-\pi_2\rangle _{\Omega ^0} =0$
for any ${\bf v}\in {\mathcal D}(\Omega ^0)^n,$
i.e., $\nabla(\pi_1-\pi_2)=0$ in $\Omega ^0$, which implies that $\pi_1-\pi_2=const.$ and thus $\pi_1=\pi_2$ in $L_2(\Omega^0)/\mathbb R$.

Finally, let us mention that the equivalence of the nonlinear Dirichlet-transmission problem \eqref{int-NS-D-0} with the mixed variational formulation \eqref{transmission-linear-NS-a}, and hence with the nonlinear Dirichlet problem \eqref{int-NS-D-00}, implies that the Dirichlet-transmission problem \eqref{int-NS-D-0} has also a unique solution $({\bf u}_+,\pi _+,{\bf u}_-,\pi _-)\in {\mathfrak X}_{\Omega _+,\Omega^0_-}$.
\end{proof}

\subsection{Existence result for the anisotropic Navier-Stokes system in $\mathbb R^n$ and for the corresponding transmission problem}

We {assume in this section that $n=3$} and prove, first, the existence of a solution
$({\bf u},\pi)\in {\mathcal H}_{{\rm{div}}}^1({\mathbb R}^n)^n\times L_{2,\rm loc}(\Omega)/\mathbb R$,
of the Navier-Stokes equation
\begin{align}\label{E5.63}
{\boldsymbol{\mathcal L}({\bf u},\pi )}=-\mathbf F+({\bf u}\cdot \nabla ){\bf u}\, 
\mbox{ in } \, \mathbb R^n
\end{align}
in the sense of distributions, implying that the couple $({\bf u},\pi)$ satisfies the variational equation
\begin{equation}
\label{transmission-nonlinear-NS}
\left\langle a_{ij}^{\alpha \beta }E_{j\beta }({\bf u}),E_{i\alpha }({\bf v})\right\rangle _{{\mathbb R}^n}
+\left\langle({\bf u}\cdot \nabla ){\bf u},{\bf v}\right\rangle _{{\mathbb R}^n}
-\langle {\rm{div}}\, {\bf v},\pi \rangle _{{\mathbb R}^n}=\langle \mathbf F,{\bf v}\rangle _{{\mathbb R}^n} \quad \forall \, {\bf v}\in { {\mathcal D}({\mathbb R}^n)^n}\,.
\end{equation}
In particular, we will prove some estimates of pressure norm growth, which seem to be new even in the simpler isotropic case.

Next, we will apply these results to prove the {existence of a solution} for the following Poisson problem of transmission type for the anisotropic Navier-Stokes system in ${\mathbb R}^n$:
\begin{equation}
\label{Stokes-NS}
\left\{
\begin{array}{ll}
{\boldsymbol{\mathcal L}({\bf u}_+,\pi _+)}={\tilde{\bf f}_{+}}|_{\Omega _{+}}+({\bf u}_+\cdot \nabla ){\bf u}_+\,, \ \ {\rm{div}} \, {\bf u}_+=0 & \mbox{ in } \Omega _{+},
\\
{\boldsymbol{\mathcal L}({\bf u}_-,\pi _-)}={\tilde{\bf f}_{-}}|_{\Omega _{-}}+({\bf u}_-\cdot \nabla ){\bf u}_-\,, \ \  {\rm{div}} \, {\bf u}_-=0 & \mbox{ in } \Omega _{-},\ \\
{\gamma }_{+}{\bf u}_+-{\gamma }_{-}{\bf u}_{-}={\bf 0} &  \mbox{ on } \partial \Omega ,\\
{\bf t}^{+}\left({\bf u}_+,\pi _+;\tilde{\bf f}_+ +\mathring{E}_+\left(({\bf u}_+\cdot \nabla ){\bf u}_+\right)\right)-{\bf t}^{-}\left({\bf u}_-,\pi _- ;\tilde{\bf f}_- +\mathring{E}_{-}\left(({\bf u}_-\cdot \nabla ){\bf u}_-\right)\right)=\boldsymbol{\psi}  &  \mbox{ on } \partial \Omega ,
\end{array}\right.
\end{equation}
with general given data $(\tilde{\bf f}_+,\tilde{\bf f}_-,\boldsymbol{\psi})\in {\mathfrak Y}$ and unknown $\left({\bf u}_{+},\pi _{+},{\bf u}_{-},\pi _{-}\right)\in {\mathfrak X}$.
Here, as previously, $\Omega _+$ is a bounded Lipschitz domain in ${\mathbb R}^n$ with connected boundary {denoted by $\partial \Omega $}, and $\Omega _{-}:={\mathbb R}^n\setminus \overline{\Omega_+}$, while
\begin{align}
\label{sol}
&{\mathfrak X}:=
{H}^1(\Omega _+)^n\times
(L_{2,\rm loc}(\mathbb R^n)/\mathbb R)|_{\Omega_+}\times
{\mathcal H}^1(\Omega _-)^n\times
(L_{2,\rm loc}(\mathbb R^n)/\mathbb R)|_{\Omega_-}\,,\\
\label{data}
&{\mathfrak Y}:=
\widetilde{H}^{-1}(\Omega _{+})^n\times \widetilde{\mathcal H}^{-1}(\Omega _{-})^n
\times H^{-\frac{1}{2}}(\partial \Omega )^n\,.
\end{align}

Let us recall that the semi-norm
\begin{align*}
|{\bf v}|_{{\mathcal H}^1({\mathbb R}^n)^n}:=\|\nabla {\bf v}\|_{L_2({\mathbb R}^n)^{n\times n}} \quad \forall \, {\bf v}\in {\mathcal H}^1({\mathbb R}^n)^n,
\end{align*}
is a norm on the space ${\mathcal H}^1({\mathbb R}^n)^n$, equivalent to the norm $\|\cdot \|_{{\mathcal H}^1({\mathbb R}^n)^n}$ given by \eqref{weight-2p}.
The closure of the space ${\mathcal D}({\mathbb R}^n)^n$ with respect to the semi-norm $|\cdot |_{{\mathcal H}^1({\mathbb R}^n)^n}$ coincides with ${\mathcal H}^1({\mathbb R^n})^n$.
Due to this, the space ${\mathcal H}^{-1}({\mathbb R}^n)^n=\left({\mathcal H}^1({\mathbb R}^n)^n\right)'$ can be endowed with the norm
\begin{align}
\label{norm-Rn}
|\!|\!|{\bf g}|\!|\!|_{{\mathcal H}^{-1}({\mathbb R}^n)^n}:=\sup_{{\bf v}\in {\mathcal H}^1({\mathbb R}^n)^n,\ \|\nabla {\bf v}\|_{L_2({\mathbb R}^n)^{n\times n}}=1}|\langle {\bf g},{\bf v}\rangle _{{\mathbb R}^n}| \quad \forall \ {\bf g}\in {\mathcal H}^{-1}({\mathbb R}^n)^n.
\end{align}

If $\Omega _-\subseteq {\mathbb R}^n$ is an exterior Lipschitz domain, then the semi-norm
\begin{align}
\label{seminorm^n}
|{\bf v}|_{{\mathcal H}^1(\Omega _{-})^n}:=\|\nabla {\bf v}\|_{L_2(\Omega _{-})^{n\times n}}
\end{align}
is a norm on the space $\mathring{\mathcal H}^1(\Omega _-)^n$ of all functions in ${\mathcal H}^1(\Omega _-)^n$ with null traces on $\partial \Omega$, and is equivalent to the norm $\|\cdot \|_{{\mathcal H}^1(\Omega _-)^n}$ given by \eqref{weight-2p} (with $\Omega _-$ in place of ${\mathbb R}^n$){, cf. \eqref{seminorm}.
{\rd The space $\mathring{\mathcal H}^1(\Omega _-)^n$ can be equivalently defined as} the closure of the space
${\mathcal D}(\Omega _-)^n$ with respect to the {\it semi-norm} $|\cdot |_{{\mathcal H}^1(\Omega _-)^n}$.}
Let ${\mathcal H}^{-1}(\Omega _-)^n$ be the dual of $\mathring{\mathcal H}^1(\Omega _-)^n$, endowed with the norm $|\!|\!|\cdot |\!|\!|_{\mathcal H^{-1}(\Omega _-)^n}$ generated by the semi-norm \eqref{seminorm^n}
i.e.,
\begin{align}
\label{norm-3f}
|\!|\!|{\bf g}|\!|\!|_{{\mathcal H}^{-1}(\Omega _-)^n}
:=\sup_{{\bf v}\in \mathring{\mathcal H}^1(\Omega _-)^n,\ \|\nabla {\bf v}\|_{L_2(\Omega _{-})^{n\times n}}=1}
|\langle {\bf g},{\bf v}\rangle _{\Omega _-}| \quad
\forall \ {\bf g}\in \mathcal H^{-1}(\Omega _-)^n.
\end{align}

\begin{thm}
\label{existence-N-S-variable-ext}
Let $n=3$.
Let conditions \eqref{Stokes-1}-\eqref{mu} hold on $\mathbb R^n$ and let $\boldsymbol{\mathcal L}$ denote the Stokes operator defined in \eqref{Stokes-new}.
Let
$\mathbf F\in {\mathcal H}^{-1}({\mathbb R}^n)^n$.
Then there exists a {solution
$({\bf u},\pi)\in {\mathcal H}_{{\rm{div}}}^1({\mathbb R}^n)^n\times L_{2,\rm loc}(\mathbb R^n)/\mathbb R$,}
of the Navier-Stokes equation \eqref{E5.63}
{in} the sense of distributions, {{which means that} the couple $({\bf u},\pi)$ satisfies {the} variational equation \eqref{transmission-nonlinear-NS}.}
In addition, {the following estimates hold}
\begin{align}
\label{estimate-NS}
&\|\nabla {\bf u}\|_{L_2({\mathbb R}^n)^n}\leq  2c_{\mathbb A}
|\!|\!|\mathbf F|\!|\!|_{\mathcal H^{-1}(\mathbb R^n)^n},\\
\label{estimate-NS-pi}
&\|\pi\|_{L_2(\Omega^0)/\mathbb R}
\leq  C'_{\mathbb R^n}|\!|\!|\mathbf F|\!|\!|_{\mathcal H^{-1}(\mathbb R^n)^n}
+C''_{\mathbb R^n}|\Omega^0|^{1/6}|\!|\!|\mathbf F|\!|\!|^2_{\mathcal H^{-1}(\mathbb R^n)^n}
\end{align}
for any ball $\Omega^0$ {in ${\mathbb R}^n$}.
Here $c_{\mathbb A}$ is {the ellipticity constant introduced} in \eqref{mu},
\begin{align}\label{E5.40}
C'_{\mathbb R^n}:=C_{\Omega^0}(1+2c_{\mathbb A}n^4\|{\mathbb A}\|_{L_\infty ({\mathbb R^n})^n}),\quad
C''_{\mathbb R^n}:=\frac{16}{3}{C_{\Omega^0}}c^2_{\mathbb A},
\end{align}
while {the constant} $C_{\Omega^0}$ is as in Lemma $\ref{La-So}$ and does not depend on the radius of the ball $\Omega^0$.

Moreover, if $\mathbf F$ is defined {in terms of the given data $(\tilde{\bf f}_+,\tilde{\bf f}_-,\boldsymbol{\psi})\in {\mathfrak Y}$ by formula} \eqref{E5.56},
then the pair $({\bf u},\pi )$ solves the nonlinear transmission problem \eqref{Stokes-NS} {with the given data $(\tilde{\bf f}_+,\tilde{\bf f}_-,\boldsymbol{\psi})\in {\mathfrak Y}$}, in the sense of distributions.
\end{thm}
\begin{proof}
We follow {first} an argument similar to that in the proof of \cite[Theorem 1.3]{A-A}. {To this end,} we consider an increasing sequence of real numbers $\{R_k\}_{k\geq 0}$ with $R_0>0$, such that $\overline\Omega_+ \subset B_{R_0}$, and $R_k\to \infty $ as $k\to \infty $, where $B_{R_k}$ is the ball in ${\mathbb R}^n$ of radius $R_k$ and center $0$ (assumed to be a point in $\Omega_+$).
Let $\Omega_{k-}:=\Omega _{-}\cap B_{R_k}$.
Then $\Omega_k:=\overline{\Omega_+}\cup \Omega_{k-}=B_{R_k}$.

{If $\mathbf F\!\in \!{\mathcal H}^{-1}({\mathbb R}^n)^n$,
then $\mathbf F|_{\Omega_k}\in {{H}^{-1}(\Omega_k)^n}$, and
by Theorem \ref{int-D-NS-var}, there exist
$({\bf u}_k, \pi_k)\in {\mathring{H}_{\rm{div}}^1(\Omega_k)^n}\times L_2(\Omega_k)/\mathbb R$ which satisfy
the Navier-Stokes system
\begin{equation}
\label{int-NS-D-0k}
{\boldsymbol{\mathcal L}({\bf u}_k,\pi_k)}=-\mathbf F|_{\Omega_k}+({\bf u}_k\cdot\nabla){\bf u}_k\,, \ \
{\rm{div}} \, {\bf u}_k=0
 \mbox{ in } \Omega_k,
\end{equation}
and the inequality
\begin{align}
\label{weak-D-k}
&\|\nabla {\bf u}_k\|_{L_2(\Omega_k)^{n\times n}}
\leq 2c_{\mathbb A}|\!|\!|\mathbf F|_{\Omega_k}|\!|\!|_{H^{-1}(\Omega_k)^n}\,.
\end{align}
Relations \eqref{int-NS-D-0k} imply that}
\begin{equation}
\label{NS-var-eq1}
\left\langle a_{ij}^{\alpha \beta }E_{j\beta }({\bf u}_k),E_{i\alpha }({\bf v})\right\rangle _{\Omega_k}+\langle ({\bf u}_k\cdot \nabla ){\bf u}_k,{\bf v}\rangle _{\Omega_k}=\langle \mathbf F|_{\Omega_k},{\bf v}\rangle _{\Omega_k} \quad \forall \ {\bf v}\in {\mathcal D}_{\rm{div}}(\Omega_k)^n\,.
\end{equation}

Note that $\mathbf F\in {{\mathcal H}^{-1}({\mathbb R}^n)^n}$ satisfies the relations
\begin{align}\label{E5.44}
{|\!|\!|\mathbf F|_{\Omega_k}|\!|\!|_{H^{-1}(\Omega_k)^n}}
&\!=\!{\sup_{\boldsymbol \Psi \in \mathring{H}^1(\Omega_k)^n,\,
\|\nabla \boldsymbol \Psi \|_{L_2(\Omega_k)^n}=1}}
\Big|\left\langle \mathbf F|_{\Omega_k},\boldsymbol \Psi \right\rangle _{\Omega_k}\Big|
\!=\!{\sup_{\boldsymbol \Psi \in \mathring{H}^1(\Omega_k)^n,\,
\|\nabla\mathring{E}_{\Omega_k}{\boldsymbol \Psi }\|_{L_2({\mathbb R}^n)^n}=1}}
\big|\big\langle \mathbf F,\mathring{E}_{\Omega_k}\boldsymbol \Psi \big\rangle _{{\mathbb R}^n}\big|\nonumber\\
&\leq {\sup_{\boldsymbol \phi \in {\mathcal H}^1({\mathbb R}^n)^n,\ \|\nabla \boldsymbol \phi \|_{L_2({\mathbb R}^n)^{n\times n}}=1}}|\left\langle \mathbf F,\boldsymbol \phi \right\rangle _{{\mathbb R}^n}|
={|\!|\!|\mathbf F|\!|\!|_{{\mathcal H}^{-1}({\mathbb R}^n)^n}}\,,
\end{align}
and hence the inequality
\begin{align}
\label{f-ext}
{|\!|\!|\mathbf F|_{\Omega_k}|\!|\!|_{H^{-1}(\Omega_k)^n}\leq |\!|\!|\mathbf F|\!|\!|_{{\mathcal H}^{-1}({\mathbb R}^n)^n}}.
\end{align}

Further, define
$\mathring{\bf u}_k:=\mathring{E}_{\Omega_k}{\bf u}_k$,
i.e., $\mathring{\bf u}_k$ is the extension of ${\bf u}_k$
by zero  in ${\mathbb R}^n\setminus \overline{\Omega_k}$.
Then $\mathring{\bf u}_k\in {{\mathcal H}_{\rm{div}}^1({\mathbb R}^n)^n}$ (since $\mathring{\bf u}_k$ does not have jump across $\partial \Omega_k$, and then ${\rm{div}}\, \mathring{\bf u}_k=0$ in ${\mathbb R}^n$) and  $\mathring{\pi}_k\in L_2(\mathbb R^3)/\mathbb R$.
Moreover, by inequalities \eqref{weak-D-k} and \eqref{f-ext},
\begin{align}
\label{weak-D-k-c}
&\|\nabla \mathring{\bf u}_k\|_{L^2({\mathbb R}^n)^{n\times n}}
{=\|\nabla {\bf u}_k\|_{L_2(\Omega_k)^{n\times n}}}
\leq 2c_{\mathbb A}{ |\!|\!|\mathbf F|\!|\!|_{\mathcal H^{-1}({\mathbb R}^n)^n}}.
\end{align}

By inequality \eqref{weak-D-k-c}, the sequence $\{\mathring{\bf u}_k\}_{k\in {\mathbb N}}$ is bounded in the Hilbert space
${\mathcal H}^1({\mathbb R}^n)^n$ and also in its closed subspace ${\mathcal H}_{\rm{div}}^1({\mathbb R}^n)^n$.
Hence $\{\mathring{\bf u}_k\}_{k\in {\mathbb N}}$ contains a {subsequence (still labeled as the sequence) weakly convergent} to an element ${\bf u}\in {\mathcal H}_{\rm{div}}^1({\mathbb R}^n)^n$.
This particularly implies that
\begin{multline}
\label{convergence-1}
\left\langle a_{ij}^{\alpha \beta }E_{j\beta }(\mathring{\bf u}_k),
E_{i\alpha }(\boldsymbol \phi )\right\rangle _{{\mathbb R}^n}
=\left\langle E_{j\beta }(\mathring{\bf u}_k),
a_{ij}^{\alpha \beta }E_{i\alpha }(\boldsymbol \phi )\right\rangle _{{\mathbb R}^n}
\to
\left\langle E_{j\beta }({\bf u}),
a_{ij}^{\alpha \beta }E_{i\alpha }(\boldsymbol \phi )\right\rangle _{{\mathbb R}^n}\\
=\left\langle a_{ij}^{\alpha \beta }E_{j\beta }({\bf u}),
E_{i\alpha }(\boldsymbol \phi )\right\rangle _{{\mathbb R}^n}
\mbox{ as } k\to \infty  \quad \forall \, \boldsymbol \phi \in {\mathcal H}^1({\mathbb R}^n)^n.
\end{multline}
Since $\|\nabla {\bf u}\|_{L_2(\Omega_k)^{n\times n}}$ is a norm in ${\mathcal H}^1({\mathbb R}^n)^n$ then by, e.g., \cite[Theorem II.1.3(i)]{Galdi} and \eqref{weak-D-k-c} we obtain
\begin{align}
\label{weak-D-k-2a}
{\|\nabla {\bf u}\|_{L_2({\mathbb R}^n)^{n\times n}}}   
\leq {\liminf}_{k\to \infty }{\|\nabla \mathring{\bf u}_k\|_{L_2({\mathbb R}^n)^{n\times n}}}
\leq 2c_{\mathbb A}|\!|\!|\mathbf F|\!|\!|_{{\mathcal H}^{-1}({\mathbb R}^n)^n}
\end{align}
i.e., ${\bf u}$ satisfies estimate \eqref{estimate-NS}, as asserted.

Next we show that ${\bf u}$ satisfies equation

\begin{equation}
\label{NS-var-eq}
\left\langle a_{ij}^{\alpha \beta }E_{j\beta }({\bf u}),E_{i\alpha }({\bf v})\right\rangle _{{\mathbb R}^n}+\langle ({\bf u}\cdot \nabla ){\bf u},{\bf v}\rangle _{{\mathbb R}^n}=\big\langle \mathbf F,{\bf v}\big\rangle _{{\mathbb R}^n} \quad\forall \ {{\bf v}\in {\mathcal D}_{\rm{div}}({\mathbb R}^n)^n}.
\end{equation}

To this end, let $\boldsymbol \phi \in {\mathcal D}_{\rm{div}}({\mathbb R}^n)^n$ and let $k_0\in {\mathbb N}$ be such that
${\rm{supp}}\, \boldsymbol \phi \!\subset \!\Omega_{k_0}\!\subseteq \!\Omega_k$ for any $k\geq k_0$. Then $\boldsymbol \phi \!\in \!{\mathcal D}_{\rm{div}}(\Omega_k)^n$ for any $k\geq k_0$ and by \eqref{NS-var-eq1},
\begin{equation}
\label{NS-var-eq-k}
\left\langle a_{ij}^{\alpha \beta }E_{j\beta }(\mathring{\bf u}_k),E_{i\alpha }(\boldsymbol \phi )\right\rangle _{\Omega_k}
+\langle (\mathring{\bf u}_k\cdot \nabla )\mathring{\bf u}_k,\boldsymbol \phi \rangle _{{\mathbb R}^n}=\langle \mathbf F,\boldsymbol \phi \rangle _{{\mathbb R}^n} \quad \forall \ k\geq k_0.
\end{equation}
Moreover, the compactness of embedding $H^1(\Omega ^0_{k_0})^n\!\hookrightarrow \!L_2(\Omega ^0_{k_0})^n$ yields that there exists a subsequence of the sequence $\{{\bf u}_k\}_{k\in {\mathbb N}}$, labeled again as the sequence, such that $\{{\bf u}_k\}_{k\in {\mathbb N}}$ converges strongly to ${\bf u}$ in $L_2(\Omega ^0_{k_0})^n$.
Let us prove that
\begin{align}
\label{convergence-3}
\langle (\mathring{\bf u}_k\cdot \nabla )\mathring{\bf u}_k,\boldsymbol \phi \rangle _{{\mathbb R}^n}\to
\langle ({\bf u}\cdot \nabla ){\bf u},\boldsymbol \phi \rangle _{{\mathbb R}^n} \mbox{ as } k\to \infty .
\end{align}
Indeed, the H\"{o}lder inequality, \eqref{weak-D-k-c} and the limiting relation $\|\mathring{\bf u}_k-{\bf u}\|_{L_2(\Omega ^0_{k_0})}\to 0 \mbox{ as } k\to \infty $ yield that
\begin{align}
\label{convergence-4}
\Big|\int_{{\mathbb R}^n}\left(\left((\mathring{\bf u}_k-{\bf u})\cdot \nabla \right)\mathring{\bf u}_k\right)\cdot \boldsymbol \phi dx\Big|&=
\Big|\int_{\Omega ^0_{k_0}}\left(\left((\mathring{\bf u}_k-{\bf u})\cdot \nabla \right)\mathring{\bf u}_k\right)\cdot \boldsymbol \phi dx\Big|\\
&\leq \|\mathring{\bf u}_k-{\bf u}\|_{L_2(\Omega ^0_{k_0})^n}\|\nabla \mathring{\bf u}_k\|_{L_2({\mathbb R}^n)^{n\times n}}\|\boldsymbol \phi \|_{L_{\infty }(\Omega ^0_{k_0})^n}\nonumber\\
&\leq 2c_{\mathbb A}{|\!|\!|\mathbf F|\!|\!|_{{\mathcal H}^{-1}({\mathbb R}^n)^n}}\|\boldsymbol \phi \|_{L_{\infty }(\Omega ^0_{k_0})^n}\|\mathring{\bf u}_k-{\bf u}\|_{L_2(\Omega ^0_{k_0})^n}\to 0
\mbox{ as } k\to \infty .\nonumber
\end{align}
In addition, by using the assumption that ${\rm{supp}}\, \boldsymbol \phi \!\subset \!\Omega ^0_{k_0}$, identity \eqref{antisym}, and again the strong convergence property of $\{\mathring{\bf u}_k\}_{k\in {\mathbb N}}$ to ${\bf u}$ in $ L_2(\Omega ^0_{k_0})^n$, we obtain that
\begin{align}
\label{convergence-5}
\Big|\int_{{\mathbb R}^n}\left(({\bf u}\cdot \nabla )(\mathring{\bf u}_k-{\bf u})\right)\cdot \boldsymbol \phi\, dx\Big|
&=\Big|\int_{\Omega ^0_{k_0}}\left(({\bf u}\cdot \nabla )\boldsymbol \phi \right)\cdot (\mathring{\bf u}_k-{\bf u})dx\Big|\nonumber\\
&\leq \|\nabla \boldsymbol \phi \|_{L_{\infty }(\Omega ^0_{k_{0}})}\|{\bf u}\|_{L_2(\Omega ^0_{k_{0}})}\|\mathring{\bf u}_k-{\bf u}\|_{L_2(\Omega ^0_{k_0})}
\to 0 \mbox{ as } k\to \infty .
\end{align}
Then relations \eqref{convergence-4} and \eqref{convergence-5} lead to relation \eqref{convergence-3}.
Finally, passing to the limit in formula \eqref{NS-var-eq-k} and using relations \eqref{convergence-1} and \eqref{convergence-3}, 
we conclude that {${\bf u}$ satisfies equation} \eqref{NS-var-eq}, and accordingly that ${\bf u}$ is a weak solution of the Navier-Stokes equation (in the Leray sense).

Note that
${\rm{div}}\left({\mathbb A}{\mathbb E}({\bf u})\right)\in {\mathcal H}^{-1}({\mathbb R}^n)^n \hookrightarrow {\mathcal D}'({\mathbb R}^n)^n$ ($\partial _\alpha $ { continuously maps $L_2({\mathbb R}^n)$} to ${\mathcal H}^{-1}({\mathbb R}^n)$).
In addition, the embedding ${\mathcal H}^1({\mathbb R}^n)\hookrightarrow L_{\frac{2n}{n-2}}({\mathbb R}^n)$, cf. \eqref{weight-Lp}, and the H\"{o}lder inequality imply for ${\bf u}\in {\mathcal H}^1({\mathbb R}^n)^n$ that
{$({\bf u}\cdot \nabla ){\bf u}\!\in \!{L_{\frac{n}{n-1}}({\mathbb R}^n)^n
\!\hookrightarrow \!L_{\frac{n}{n-1};{\rm{loc}}}({\mathbb R}^n)^n
\!\hookrightarrow \!H_{{\rm loc}}^{-1}({\mathbb R}^n)^n}
\!\hookrightarrow \!{\mathcal D}'({\mathbb R}^n)^n$.}
Thus, for given $\mathbf F\in {{\mathcal H}^{-1}({\mathbb R}^n)^n\hookrightarrow {\mathcal D}'({\mathbb R}^n)^n}$, we have
\begin{align}\label{E5.70}
\widetilde{\boldsymbol{\mathcal F}}:={\mathbf F}+{\rm{div}}\left({\mathbb A}{\mathbb E}({\bf u})\right)
-{({\bf u}\cdot \nabla ){\bf u}}\in {\mathcal D}'({\mathbb R}^n)^n,
\end{align}
and, by \eqref{NS-var-eq},
\begin{align}\label{E.5.75}
\langle \widetilde{\boldsymbol{\mathcal F}},\boldsymbol \phi \rangle _{{\mathbb R}^n}=0 \quad \forall \, \boldsymbol \phi \in {\mathcal D}_{\rm div}({\mathbb R}^n)^n.
\end{align}
Then due to the De Rham Theorem (cf., e.g., Proposition 1.1 in \cite[Chapter 1]{Temam}),
there exists $\pi \in {\mathcal D}'({\mathbb R}^n)$, unique up to an additive constant, such that $\nabla \pi \!=\! \widetilde{\boldsymbol{\mathcal F}}$ in ${\mathcal D}'({\mathbb R}^n)^n$, i.e.,
\begin{align}
\label{NS-Rn}
\langle\nabla \pi, \mathbf v\rangle_{\mathbb R^n}
=\langle{\mathbf F}+{\rm{div}}\left({\mathbb A}{\mathbb E}({\bf u})\right)
-({\bf u}\cdot \nabla){\bf u}, \mathbf v\rangle_{\mathbb R^n} \quad \forall \ {{\bf v}\in {\mathcal D}({\mathbb R}^n)^n},
\end{align}
and hence, equation \eqref{transmission-nonlinear-NS} holds.
Moreover, since $\widetilde{\boldsymbol{\mathcal F}}$ defined by \eqref{E5.70} belongs locally to
$H^{-1}({\mathbb R}^n)^n$, equation $\nabla \pi \!=\! \widetilde{\boldsymbol{\mathcal F}}$
implies that $\pi \in L_{2;\rm loc}({\mathbb R}^n)$
(see, e.g., Proposition 1.2 (ii) and Remark 1.4 in \cite[Chapter 1]{Temam}, \cite[Lemma 9]{Tartar1}, \cite[Lemma X.1.1]{Galdi}).

{Moreover, by inequality \eqref{grad-est} in Theorem~\ref{La-SoT} along with \eqref{E5.70},  we obtain that for any ball $\Omega^0=B_R$ of radius $R$,
\begin{align}
\label{weak-Dpi-e1}
\|\pi\|_{L_2(\Omega^0)/\mathbb R}
&\leq C_{\Omega^0}|\!|\!|\widetilde{\boldsymbol{\mathcal F}}|\!|\!|_{H^{-1}(\Omega^0)^n}\nonumber\\
&\leq C_{\Omega^0}\big(|\!|\!|{\bf F}|\!|\!|_{H^{-1}(\Omega^0)^n}
+|\!|\!|{\rm{div}}\left({\mathbb A}{\mathbb E}({\bf u})\right)|\!|\!|_{H^{-1}(\Omega^0)^n}
+|\!|\!|({\bf u}\cdot \nabla){\bf u}|\!|\!|_{H^{-1}(\Omega^0)^n}\big),
\end{align}
Similar to \eqref{E5.44}, we have
\begin{align}\label{E5.58}
|\!|\!|\mathbf F|\!|\!|_{H^{-1}(\Omega^0)^n}\le |\!|\!|\mathbf F|\!|\!|_{\mathcal H^{-1}(\mathbb R^n)^n}
\end{align}
By \eqref{A-norm}, and \eqref{estimate-NS},
\begin{align}\label{E5.59}
|\!|\!|{\rm{div}}\left({\mathbb A}{\mathbb E}({\bf u})\right)|\!|\!|_{H^{-1}(\Omega^0)^n}&=
{\sup_{\boldsymbol \Psi \in \mathring{H}^1(\Omega ^0)^n,\,
\|\nabla \boldsymbol \Psi \|_{L_2(\Omega^0)^n}=1}}\Big|\left\langle{\rm{div}}\left({\mathbb A}{\mathbb E}({\bf u})\right),\boldsymbol \Psi\right\rangle _{\Omega^0}\Big|\nonumber\\
&={\sup_{\boldsymbol \Psi \in \mathring{H}^1(\Omega ^0)^n,\,
\|\nabla \boldsymbol \Psi \|_{L_2(\Omega^0)^n}=1}}\Big|\left\langle{\mathbb A}{\mathbb E}({\bf u}),\nabla \boldsymbol \Psi\right\rangle _{\Omega^0}\Big|\nonumber\\
&\leq \|{\mathbb A}{\mathbb E}({\bf u})\|_{L_2(\Omega ^0)^{n\times n}}\leq n^4\|{\mathbb A}\|_{L_\infty (\Omega ^0)}\|\nabla \, {\bf u}\|_{L_2(\Omega ^0)^{n\times n}}\nonumber\\
&\leq n^4\|{\mathbb A}\|_{L_\infty (\Omega ^0)}\|\nabla \, {\bf u}\|_{L_2({\mathbb R}^n)^{n\times n}}
\leq 2c_{{\mathbb A}}n^4\|{\mathbb A}\|_{L_\infty (\Omega ^0)}\|{\bf F}\|_{{\mathcal H}^{-1}({\mathbb R}^n)^n},
\end{align}
}

{Since $n=3$, from \eqref{P-061cal} and \eqref{estimate-NS} we have,
\begin{align}\label{E5.60}
|\!|\!|({\bf u}\cdot \nabla ){\bf u}|\!|\!|_{H^{-1}(\Omega^0)^n}
\le  \frac{4}{3}|\Omega^0|^{1/6}\|{\nabla\bf u}\|^2_{L_2(\mathbb R^n)^{n\times n}}
\le\frac{4}{3}|\Omega^0|^{1/6}4c^2_{\mathbb A} |\!|\!|\mathbf F|\!|\!|^2_{\mathcal H^{-1}(\mathbb R^n)^n}.
\end{align}
{Hence, substituting \eqref{E5.58}, \eqref{E5.59}, and \eqref{E5.60} in \eqref{weak-Dpi-e1}, we obtain  \eqref{estimate-NS-pi}.}}
\end{proof}

Next we show the existence of a solution of the nonlinear problem \eqref{Stokes-NS} with general given data. 
In case \eqref{isotropic} with $\mu=1$ and $\lambda=0$,
we refer to \cite[Theorem 1.3]{A-A} and \cite[Theorem 1.3]{Al-Am} for the existence of a weak solution of the exterior Dirichlet problem for Navier-Stokes system, \cite{Ot-Sa} for the Dirichlet problem for the Navier-Stokes system in a bounded Lipschitz domain in ${\mathbb R}^2$, under singular sources, and \cite[Theorem 5.2]{K-L-M-W}. Existence in the case of the anisotropic tensor $\mathbb A$ satisfying a more restrictive ellipticity condition than in \eqref{mu} was analyzed in \cite[Theorem 4.2]{K-M-W-2} in a pseudostress setting, assuming small given data.
\begin{thm}
\label{int-NS-var-Dtr}
Let n=3
and conditions \eqref{Stokes-1}-\eqref{mu} hold in $\mathbb R^n$.
Let
$\big({\tilde{\bf f}}_{+},{\tilde{\bf f}}_{-},\boldsymbol{\psi}\big)\in {\mathfrak Y}$
and
$({\bf u},\pi)\in {\mathcal H}_{{\rm{div}}}^1({\mathbb R}^n)^n\times L_{2,\rm loc}(\Omega)/\mathbb R$
be the solution of equation \eqref{E5.63}
provided by Theorem~\ref{int-D-NS-var} for
\begin{align}\label{E5.56}
\mathbf F=-(\tilde{\bf f}_++\tilde{\bf f}_-)+\gamma ^*\boldsymbol{\psi}.
\end{align}
Then there exists  a solution
$({\bf u}_+,\pi _+,{\bf u}_-,\pi _-)\in {\mathfrak X}$ of the nonlinear
Dirichlet-transmission problem \eqref{int-NS-D-0} given by relations
$
{\bf u}_+={\bf u}|_{\Omega _+},\ {\bf u}_-={\bf u}|_{\Omega_-},\
\pi_+=\pi|_{\Omega _+},\ \pi_-=\pi|_{\Omega_-},
$
and estimates \eqref{mixed-Cu0N}, \eqref{mixed-Cp0N} hold.
\end{thm}
\begin{proof}
We have to show that
$({\bf u},\pi)\in {\mathcal H}_{{\rm{div}}}^1({\mathbb R}^n)^n\times L_{2,\rm loc}(\Omega)/\mathbb R$ solving equation \eqref{E5.63}
solves also the transmission problem \eqref{Stokes-NS} in the sense of distributions.
Since
${\bf u}_\pm:={\bf u}|_{\Omega_\pm}$ and $\pi_\pm:=\pi|_{\Omega _\pm}$,
we have ${\bf u}_+\in {H}_{{\rm{div}}}^1(\Omega_+)^n$, ${\bf u}_-\in \mathcal{H}_{{\rm{div}}}^1(\Omega_-)^n$,
$\pi_+\in (L_{2,\rm loc}(\mathbb R^n)/\mathbb R)|_{\Omega _+}$, $\pi_-\in(L_{2,\rm loc}(\mathbb R^n)/\mathbb R)|_{\Omega _-}$.
Due to \eqref{E5.56}, $\mathbf F\in {\mathcal H}^{-1}({\mathbb R}^n)^n$ and equation  \eqref{E5.63} implies
that $({\bf u}_\pm,\pi_\pm )$ satisfy the Navier-Stokes equations
\begin{align}
\label{eq-int}
\partial _\alpha\left(A^{\alpha \beta }\partial _\beta ({\bf u}_\pm)\right)-\nabla \pi_\pm
={\tilde{\bf f}_{\pm}}|_{\Omega _{\pm}}+({\bf u}_\pm\cdot \nabla ){\bf u}_\pm\
\mbox{ in } \Omega_\pm.
\end{align}
By choosing any ${\bf v}\in {\mathcal D}(\mathbb R^n)^n$ in \eqref{transmission-nonlinear-NS} and again taking into account \eqref{E5.56}, we obtain
\begin{align}
\label{eq-int-ext}
&\left\langle a_{ij}^{\alpha \beta }E_{j\beta }({\bf u}_+),E_{i\alpha }({\bf v})\right\rangle _{\Omega _+}+\langle ({\bf u}_+\cdot \nabla ){\bf u}_+,{\bf v}\rangle _{\Omega _+}-\langle \pi _+,{\rm{div}}\, {\bf v}\rangle_{\Omega _+}+\langle \tilde{\bf f}_+,{\bf v}\rangle _{\Omega _+}\nonumber\\
+&\left\langle a_{ij}^{\alpha \beta }E_{j\beta }({\bf u}_-),E_{i\alpha }({\bf v})\right\rangle _{\Omega_-}+\langle ({\bf u}_-\cdot \nabla ){\bf u}_-,{\bf v}\rangle _{\Omega_-}-\langle \pi _-,{\rm{div}}\, {\bf v}\rangle_{\Omega_-}+\langle \tilde{\bf f}_-,{\bf v}\rangle _{\Omega_-}=\langle \boldsymbol{\psi},\gamma {\bf v}\rangle _{\partial \Omega }\,.
\end{align}

Now let $\Omega_0$ be a bounded Lipschitz domain (e.g., a ball) such that $\overline\Omega_+ \subset \Omega_0$ and let $\Omega_{0-}:=\Omega _-\cap \Omega_0$.
By choosing ${\bf v}\in \mathcal D(\Omega_0)^n$, the domain $\Omega_-$ in \eqref{eq-int-ext} can be replaced by $\Omega_{0-}$.
Then the Green formula \eqref{Green-particular-p} shows that equation \eqref{eq-int-ext} reduces to
\begin{align}
\left\langle{\bf t}^{+}\left({\bf u}_+,\pi _+;\tilde{\bf f}_+ +({\bf u}_+\cdot \nabla ){\bf u}_+\right)-{\bf t}^{-}\left({\bf u}_-,\pi _- ;\tilde{\bf f}_- +({\bf u}_-\cdot \nabla ){\bf u}_-\right),\gamma{\bf v}\right\rangle _{\partial \Omega }=\left\langle \boldsymbol{\psi},\gamma {\bf v}\right\rangle _{\partial \Omega } \quad \forall \, {\bf v}\in {\mathcal D}(\Omega_0)^n,\nonumber
\end{align}
or, equivalently,
\begin{align}
\left\langle{\bf t}^{+}\left({\bf u}_+,\pi _+;\tilde{\bf f}_+ +({\bf u}_+\cdot \nabla ){\bf u}_+\right)-{\bf t}^{-}\left({\bf u}_-,\pi _- ;\tilde{\bf f}_- +({\bf u}_-\cdot \nabla ){\bf u}_-\right),\boldsymbol\phi \right\rangle _{\partial \Omega }=\left\langle \boldsymbol{\psi},\boldsymbol\phi \right\rangle _{\partial \Omega } \quad \forall \, \boldsymbol\phi \in H^{\frac{1}{2}}(\partial\Omega)^n,\nonumber
\end{align}
due to the dense embedding of the space ${\mathcal D}(\Omega_0)^n$ in $\mathring{H}^1(\Omega_0)^n$ and the surjectivity of the trace operator $\gamma $ from
$\mathring{H}^1(\Omega_0)^n$ to $H^{\frac{1}{2}}(\partial \Omega )^n$.
Therefore, the second transmission condition in \eqref{Stokes-NS} follows, as asserted. The first transmission condition is obviously satisfied since ${\bf u}\in {\mathcal H}^1({\mathbb R}^n)^n$.
\end{proof}

\subsection{Existence result for the exterior Dirichlet problem for the anisotropic Navier-Stokes system in exterior Lipschitz domains of ${\mathbb R}^n$ with general data}
Let $n=3$, $\Omega _+$ be a bounded Lipschitz domain in ${\mathbb R}^n$ with connected boundary $\partial \Omega $ and let $\Omega _{-}:={\mathbb R}^n\setminus \overline{\Omega_+}$.
Let us introduce the following norm in the space ${\mathcal H}^{-1}(\Omega _{-})^n$:
\begin{align}
\label{norm-3e}
{|\!|\!|{\bf g}|\!|\!|_{{\mathcal H}^{-1}(\Omega )^n}:=\sup_{{\bf v}\in \mathring{\mathcal H}^1(\Omega _{-})^n,\ \|\nabla {\bf v}\|_{L^2(\Omega _{-})^n}=1}|\langle {\bf g},{\bf v}\rangle _{\Omega _{-}}|,\ \ \forall \ {\bf g}\in {\mathcal H}^{-1}(\Omega _{-})^n.}
\end{align}

Next we consider the exterior Dirichlet problem for the anisotropic Navier-Stokes system 
\begin{eqnarray}
\label{ext-NS-D}
\left\{
\begin{array}{ll}
{\boldsymbol{\mathcal L}({\bf u}_-,\pi _-)}=-{\mathbf F}+({\bf u}_-\cdot \nabla ){\bf u}_-\,,\ \  
{\rm{div}}\, {\bf u}_{-}=0 & \mbox{ in } \Omega _{-},\\
\gamma _{-}({\bf u}_-)={\bf 0} &  \mbox{ on } \partial \Omega \,,
\end{array}\right.
\end{eqnarray}
with the given datum {${\bf F}\in {\mathcal H}^{-1}(\Omega _-)^n$}.

{By using similar arguments to those in the proof of Theorem \ref{existence-N-S-variable-ext} we obtain the following result (cf. \cite[Theorem 1.3]{Amrouche-1} for the exterior Dirichlet problem for {the Navier-Stokes system with constant coefficients})}.
\begin{thm}
\label{existence-N-S-variable-ext-Dirichlet}
Let $n=3$.
Let $\Omega_+ \subset {\mathbb R}^n$ be a bounded Lipschitz domain with connected boundary and let
${\Omega }_{-}:={\mathbb R}^n\setminus \overline{\Omega_+}$.
Let conditions \eqref{Stokes-1}-\eqref{mu} hold on $\Omega_-$.
Then for any ${\bf F}\in {\mathcal H}^{-1}(\Omega_-)^n$, there exists a solution
$({\bf u},\pi)\in {\mathcal H}_{{\rm{div}}}^1(\Omega_-)^n\times L_{2,\rm loc}(\Omega_-)/\mathbb R$ of the exterior Dirichlet problem \eqref{ext-NS-D} in the sense of distributions, which means that {the couple $({\bf u},\pi)$} satisfies the variational equation
\begin{equation}
\label{NS-var-eq-D}
\left\langle a_{ij}^{\alpha \beta }E_{j\beta }({\bf u}_{-}),E_{i\alpha }({\bf v})\right\rangle _{\Omega_-}
+\left\langle({\bf u}_{-}\cdot \nabla ){\bf u}_{-},{\bf v}\right\rangle _{\Omega _-}
{-\langle {\rm{div}}\, {\bf v},\pi \rangle_{\Omega _{-}}}
=-\langle {\bf F},{\bf v}\rangle _{\Omega _{-}}\ \forall \ {\bf v}\in {\mathcal D}(\Omega _{-})^n.
\end{equation}
In addition, 
\begin{align}
\label{estimate-NS-D}
&{\|\nabla {\bf u}_-\|_{L^2(\Omega _{-})^n}\leq {2}c_{\mathbb A}{|\!|\!|{\bf F}|\!|\!|_{{\mathcal H}^{-1}(\Omega _{-})^n}},}\\
\label{estimate-NS-D-pi}
&\|\pi\|_{L_2(\Omega^0)/\mathbb R}
\leq  C'_{\Omega_-}|\!|\!|\mathbf F|\!|\!|_{\mathcal H^{-1}(\Omega_-)^n}
+C''_{\Omega_-}|\Omega^0|^{1/6}|\!|\!|\mathbf F|\!|\!|^2_{\mathcal H^{-1}(\Omega_-)^n}
\end{align}
for any ball $\Omega^0$  such that $\partial\Omega\subset\Omega^0$.
Here $c_{\mathbb A}$ is defined in \eqref{mu},
\begin{align}\label{E5.40-}
C'_{\Omega_-}:=C_{\Omega^0}(1+n^4\|{\mathbb A}\|_{L_\infty ({\Omega_-})^n}2c_{\mathbb A}),\quad
C''_{\Omega_-}:=\frac{16}{3}{C_{\Omega^0}}c^2_{\mathbb A},
\end{align}
while $C_{\Omega^0}$ is as in Lemma $\ref{La-So}$ and does not depend on the radius of $\Omega^0$.
\end{thm}

\begin{rem}
The well-posedness results obtained in this paper 
can be extended, similar to \cite{K-M-W-2} and \cite{K-W1}, to the setting of $L_p$-based Sobolev spaces with $p$ in an open interval containing $2$.
\end{rem}

\appendix
{

\section{The Agmon-Douglis-Nirenberg ellipticity of the anisotropic Stokes system}
\label{ADN}

Let $\sigma ({\bf x},\boldsymbol\xi )$ denote the principal symbol matrix of the anisotropic Stokes system \eqref{Stokes-0-0}, \eqref{Stokes}. Thus,
\begin{align}
\label{ADN-elliptic}
\sigma ({\bf x},\boldsymbol\xi )=
\left(\begin{array}{cc}
\xi _\alpha a_{\ell j}^{\alpha \beta }({\bf x})\xi _\beta  & -i\xi _\ell \\
-i\xi _j & 0
\end{array}
\right)_{\ell ,j=1,\ldots ,n}\,. 
\end{align}
Here $i^2=-1$, and $\boldsymbol\xi =(\xi _1,\ldots,\xi _n)$.

The Stokes system is {\it elliptic in the sense of Agmon-Douglis-Nirenberg} at ${\bf x}\in {\mathbb R}^n$ if $\sigma ({\bf x},\boldsymbol\xi )$ is defined and non-singular for any $\boldsymbol\xi \in {\mathbb R}^n\setminus \{{\bf 0}\}$ (see, e.g., \cite[Definition 6.2.3]{H-W}).
This property is well known in the isotropic case \eqref{isotropic} with $\mu =1$ and $\lambda =0$  (cf., e.g., \cite[p.329]{H-W}).
Next, we show that this ellipticity property remains valid even in the more general anisotropic case.
\begin{lem}
\label{ADN-system}
Let conditions \eqref{Stokes-1}-\eqref{mu} hold on ${\mathbb R}^n$. Then the anisotropic Stokes system
defined by \eqref{Stokes-0-0} and \eqref{Stokes} is elliptic in the sense of Agmon-Douglis-Nirenberg at almost any ${\bf x}\in {\mathbb R}^n$.
\end{lem}
\begin{proof}
First, we observe that the symbol matrix $\sigma ({\bf x},\boldsymbol\xi )$ given by \eqref{ADN-elliptic}
is non-singular
if and only if the modified symbol matrix
\begin{align}
\label{ADN-elliptic-1}
\widetilde\sigma ({\bf x},\boldsymbol\xi )=
\left(\begin{array}{cc}
\xi _\alpha a_{\ell j}^{\alpha \beta }({\bf x})\xi _\beta  & -\xi _\ell \\
-\xi _j & 0
\end{array}
\right)_{\ell ,j=1,\ldots ,n}
\end{align}
is non-singular as well.
Let ${\bf x}\in {\mathbb R}^n$ be such that the coefficients $a_{\ell j}^{\alpha \beta }({\bf x})$ are well defined and finite and the ellipticity condition \eqref{mu} holds. In order to show that $\widetilde\sigma ({\bf x},\boldsymbol\xi )$ is non-singular for any $\boldsymbol\xi \in {\mathbb R}^n\setminus \{0\}$, we
use Theorem \ref{B-B}.
To this end, for a fixed $\boldsymbol\xi \in {\mathbb R}^n\setminus \{0\}$, we consider the bilinear forms $a_0:{\mathbb R}^n\times {\mathbb R}^n\to {\mathbb R}$ and $b_0:{\mathbb R}^n\times {\mathbb R}\to {\mathbb R}$,
\begin{align}
\label{a-0}
&a_0(\hat{\bf u},\hat{\bf v}):=\hat{u}_\ell \xi_\alpha a_{\ell j}^{\alpha \beta }({\bf x})\xi _\beta \hat{v}_j\quad \ \forall \, \hat{\bf u},\hat{\bf v}\in {\mathbb R}^n\,,\\
\label{b-0}
&b_0(\hat{\bf v},\hat{q}):=-\xi_j\hat{v}_j\hat{q}\hspace{5.5em} \forall \, \hat{\bf v}\in {\mathbb R}^n,\, \hat{q}\in {\mathbb R},
\end{align}
as well as the closed subspace $V_{\boldsymbol\xi}$ of ${\mathbb R}^n$ given by
\begin{align}
\label{V-0}
V_{\boldsymbol\xi}:=\left\{\hat{\bf v}\in {\mathbb R}^n:b_0(\hat{\bf v},\hat{q})=0,\ \forall \, \hat{q}\in {\mathbb R}\right\}=\left\{\hat{\bf v}\in {\mathbb R}^n:\xi_j\hat{v}_j=0\right\}.
\end{align}
It is immediate that these bilinear forms are bounded, as they satisfy the estimates:
\begin{align*}
|a_0(\hat{\bf u},\hat{\bf v})|\leq \|{\mathbb A}\|_{L_\infty ({\mathbb R}^n)}
|\boldsymbol{\xi }|^2|\hat{\bf u}|\, |\hat{\bf v}|,\ \ |b_0(\hat{\bf v},\hat{q})|\leq |\boldsymbol{\xi }|\, |\hat{\bf v}|\, |\hat{q}| \quad \forall \, \hat{\bf u},\hat{\bf v}\in {\mathbb R}^n,\ \forall \, \hat{q}\in {\mathbb R}.
\end{align*}

The symmetry conditions \eqref{Stokes-sym} allow us to write the bilinear form $a_0$ as
\begin{align}
\label{a-0-1}
a_0(\hat{\bf u},\hat{\bf v}) 
=a_{\ell j}^{\alpha \beta }({\bf x})(\hat{\bf u}\otimes \boldsymbol\xi )_{\ell \alpha }^s(\hat{\bf v}\otimes \boldsymbol\xi )^s_{\beta j}\,,
\end{align}
where $(\hat{\bf u}\otimes \boldsymbol\xi )^s$ is the symmetric part of the matrix $\hat{\bf u}\otimes \boldsymbol\xi$, i.e.,
\begin{align}
(\hat{\bf u}\otimes \boldsymbol\xi )^s_{\ell \alpha}
:=\frac{1}{2}\left(\hat{u}_\ell \xi_\alpha +\hat{u}_\alpha \xi_\ell\right),\ \ \ell ,\alpha =1,\ldots ,n\,.
\end{align}
According to \eqref{a-0-1} and the ellipticity condition \eqref{mu} we obtain that $a_0$ satisfies the estimate
\begin{align}
a_0(\hat{\bf v},\hat{\bf v})\geq c_{{\mathbb A}}^{-1}|(\hat{\bf v}\otimes {\boldsymbol\xi })^s|^2\quad  \forall \, \hat{\bf v}\in {\mathbb R}^n \mbox{ such that } \hat{\bf v}\cdot \boldsymbol\xi =0\,,
\end{align}
where $\hat{\bf v}\cdot \boldsymbol\xi =\sum _{\ell =1}^n(\hat{\bf v}\otimes \boldsymbol\xi )_{\ell \ell }^s$ is the trace of the symmetric matrix $(\hat{\bf v}\otimes \boldsymbol\xi )^s$. Therefore, the bounded bilinear form $a_0:V_{\boldsymbol \xi}\times V_{\boldsymbol \xi}\to {\mathbb R}$ is coercive.

In addition, an elementary computation shows that
\begin{align}
\label{inf-sup-b0}
\inf_{\hat{q}\in {\mathbb R}\setminus \{0\}}\sup _{\hat{\bf v}\in {\mathbb R}^n\setminus \{{\bf 0}\}}\frac{b_0(\hat{\bf v},\hat{q})}{|\hat{\bf v}|\, |\hat{q}|}\geq |\boldsymbol\xi |\,,
\end{align}
and accordingly that the bilinear form $b_0$ satisfies the inf-sup condition with the inf-sup constant $|\boldsymbol\xi |$.

By applying Theorem \ref{B-B}, we conclude that the modified symbol matrix $\widetilde\sigma ({\bf x},\boldsymbol\xi )$ given by \eqref{ADN-elliptic-1} is invertible for any $\boldsymbol\xi \neq {\bf 0}$, and hence that the symbol matrix $\sigma ({\bf x},\boldsymbol\xi )$ given by \eqref{ADN-elliptic} has the same property. Thus, the anisotropic Stokes system is elliptic in the sense of Agmon-Douglis-Nirenberg, as asserted.
\end{proof}

\section{Extension result in weighted Sobolev spaces}
Theorem 5.13 in \cite{B-M-M-M} asserts that gluing two Sobolev functions defined inside and outside of a two-sided $(\epsilon ,\delta )$-domain (e.g., a Lipschitz domain), preserves Sobolev smoothness whenever the functions have equal traces on the boundary of the domain. Next we show a variant of this result in the case of a Lipschitz domain and for {\it weighted Sobolev spaces}.

\begin{lemma}
\label{extention}
Let $\Omega _+ \subset {\mathbb R}^n$ be a bounded Lipschitz domain with connected boundary and
${\Omega}_{-}:={\mathbb R}^n\setminus \overline{\Omega_+}$.
\begin{enumerate}
\item[(i)]
Let $q_+\in L_2(\Omega _+)$ and $q_-\in L_2(\Omega _-)$.
Then there exists a unique function $q\in L_2(\mathbb R^n)$ such that
$
q|_{\Omega_\pm}=q_\pm.
$
Moreover, 
${\|q\|_{L_2({\mathbb R}^n)}^2=\|q_+\|_{L_2(\Omega _+)}^2+\|q_-\|_{L_2(\Omega _-)}^2}.$
\item[(ii)]
Let $u_+\in H^1(\Omega _+)$ and $u_-\in {\mathcal H}^1(\Omega _-)$ be such that
$\gamma _+ u_+=\gamma _- u_- $ on  $\partial \Omega$.
Then there exists a unique function $u\in \mathcal H^1(\mathbb R^n)$ such that
$
{u|_{\Omega_\pm}=u_\pm}.
$
Moreover, there exists a constant $C>0$ depending on $n$ and $\Omega _\pm $, such that
\begin{align}
\label{E-u-4}
\|u\|_{{\mathcal H}^1({\mathbb R}^n)}\leq C\left(\|u_+\|_{H^1(\Omega _+)}+\|u_-\|_{{\mathcal H}^1(\Omega _-)}\right).
\end{align}
\item[(iii)]
{If $u\in \mathcal H^1(\mathbb R^n)$ then $[\gamma u]=0$, where $[\gamma u]=\gamma _+(u|_{\Omega _+})-\gamma _-(u|_{\Omega _-})$.}
\end{enumerate}
\end{lemma}
\begin{proof}
(i) We can take
$q \!=\!\mathring E_{\Omega _+}q_+ \!+\!\mathring E_{\Omega _-}q_-\!\in \!L_2(\mathbb R)$,
where $\mathring E_{\Omega\pm}$ are the operators of extension by zero are defined in \eqref{ringE}.
Then evidently $q|_{\Omega_\pm}=q_\pm.$
To prove the uniqueness  let us assume that there are two such functions, $q_1$ and $q_2$. Then $q_0:=q_1-q_2$ belongs to $L_2(\mathbb R^n)$ and  $q_0|_{\Omega_\pm}=0$.
Hence $q_0=0$ in $\mathbb R^n$ in the sense of Lebesgue classes.

(ii) We follow similar arguments to those for Theorem 5.13 in \cite{B-M-M-M}.
Let $\mathcal E_{\Omega _+}$ be a bounded linear extension operator from $H^1(\Omega _+)$ to $H^1({\mathbb R}^n)$ (see, e.g., \cite[Theorem 2.4.1]{M-W}).
Let us take
\begin{align}
\label{E-u}
u^*_-:=(\mathcal E_{\Omega _+}u_+)|_{\Omega _-} \mbox{ in } \Omega _-.
\end{align}
Then $u^*_-\in H^1(\Omega _-)\subset {\mathcal H}^1(\Omega _-)$. Moreover, there exists a constant $c>0$ depending on $n$ and $\Omega _\pm $, such that
$\|u^*_-\|_{{\mathcal H}^1(\Omega _-)}\leq c\|u_+\|_{H^1(\Omega _+)}.$

In addition, in view of \eqref{E-u} we have $\gamma _- u^*_-=\gamma _- (\mathcal E_{\Omega _+}u_+)=\gamma _+u_+=\gamma _-u_-,$ and hence $u_--u^*_-$ belongs to $\mathring{\mathcal H}^1(\Omega _-)$. Thus, $\mathring{E}_{\Omega_-}(u_--u^*_-)$ belong to ${\mathcal H}^1({\mathbb R}^n)$ and there exists a constant $c_1=c_1(n,\Omega _\pm )$, such that
\begin{align}
\label{E-u-2}
\|\mathring{E}_{\Omega_-}(u_--u^*_-)\|_{{\mathcal H}^1({\mathbb R}^n)}\leq c_1\left(\|u_+\|_{H^1(\Omega _+)}+\|u_-\|_{{\mathcal H}^1(\Omega _-)}\right).
\end{align}
Let us now define the function
\begin{align}
\label{E-u-3}
u:=\mathring{E}_{\Omega_-}(u_--u^*_-)+\mathcal E_{\Omega _+}u_+ .
\end{align}
It belongs to
${\mathcal H}^1({\mathbb R}^n)$, and there exists a constant $C_1>0$ depending on $n$ and $\Omega _\pm $, such that the inequality \eqref{E-u-4} holds.
According to \eqref{E-u} and \eqref{E-u-3} we have also the following relations
\begin{align*}
&u|_{\Omega _+}=0+(\mathcal E_{\Omega _+}u_+)|_{\Omega _+}=u_+ \mbox{ a.e. in } \Omega _+,\\
&u|_{\Omega _-}=u_--u^*_-+(\mathcal E_{\Omega _+}u_+)|_{\Omega _-}=u_--u^*_-+u^*_-=u_- \mbox{ a.e. in } \Omega _-,
\end{align*}
and thus the existence of a function $u$ is proved.

To prove that the function $u$ is unique, let us assume that there are two such functions, $u_1$ and $u_2$. Then $u_0:=u_1-u_2$ belongs to $\mathcal H^1(\mathbb R^n)$ and  $u_0|_{\Omega_\pm}=0$. Thus, $u_0\in H^1(\mathbb R^n)\subset L_2(\mathbb R^n)$ and its support is a subset of $\partial\Omega$. Hence $u_0=0$ in $\mathbb R^n$ in the sense of Lebesgue classes (cf. also Theorem 2.10(i) in \cite{Mikh}).

{(iii) Let $u\in {\mathcal H}^1({\mathbb R}^n)$. Consequently, $u\in H_{\rm{loc}}^1({\mathbb R}^n)$, and then $\gamma _+u=\gamma _-u$, i.e., $[\gamma u]=0$.}
\end{proof}

\section{Equivalent norms in the weighted Sobolev space ${\mathcal H}^1({\mathbb R}^n\setminus \partial \Omega )^n$}

The next result plays a main role in establishing the equivalence of norms on Banach spaces, in particular, on some Sobolev spaces that appear in our arguments (cf. \cite[Lemma 11.1]{Tartar}).
\begin{lem}
\label{Tartar-lemma}
Let $(X,\|\cdot \|_X)$ be a Banach space, and let $(Y,\|\cdot \|_Y)$, $(Z,\|\cdot \|_Z)$, $(\Upsilon ,\|\cdot \|_\Upsilon )$ be normed spaces. Let ${\mathcal P}:X\to Y$, ${\mathfrak C}:X\to Z$ and ${\mathcal T}:X\to \Upsilon $ be linear and continuous operators, such that  \begin{itemize}
\item[$(i)$]
The operator ${\mathfrak C}:X\to Z$ is compact.
\item[$(ii)$]
$\|P(\cdot )\|_{Y}+\|{\mathfrak C}(\cdot )\|_{Z}$ is a norm on $X$ equivalent to the norm $\|\cdot \|_X$.
\item[$(iii)$]
The operator ${\mathcal T}:X\to \Upsilon $ satisfies the condition ${\mathcal T}(u)\neq 0$ whenever $P(u)=0$ and $u\neq 0$.
\end{itemize}
Then $\|u\|:=\|P(u)\|_Y+\|{\mathcal T}(u)\|_{\Upsilon },\ u\in X,$
is a norm on $X$ equivalent to the given norm $\|\cdot \|_X$.
\end{lem}

The following  result
for $n=3$ is implied by Proposition 2.7(i) in \cite{Sa-Se}, and its proof is based on the Korn first inequality (see \cite[Theorem 10.1]{Lean}), a compactness argument, and Lemma \ref{Tartar-lemma}.
The result for $n> 3$ follows with the same arguments.
\begin{thm}
\label{Korn-exterior}
Let $n\geq 3$. Let $\Omega \subset {\mathbb R}^n$ be a bounded Lipschitz domain with connected boundary $\partial \Omega $, and $\Omega _-:={\mathbb R}^n\setminus \overline{\Omega }$. Then $\|{\mathbb E}(\cdot )\|_{L_2(\Omega _-)^{n\times n}}$ is a norm in the weighted Sobolev space ${\mathcal H}^1(\Omega _-)^n$, which is equivalent to the norm $\|\cdot \|_{{\mathcal H}^1(\Omega _-)^n}$ given by \eqref{weight-2p} with $\Omega _-$ in place of ${\mathbb R}^n$. Therefore, there exists a constant $C=C(\Omega _-,n)>0$ such that
\begin{align}
\label{Korn-ineq-exterior}
C\|{\bf u}\|_{{\mathcal H}^1(\Omega _{-})^n}\leq \|{\mathbb E}({\bf u})\|_{L_2(\Omega_-)^{n\times n}}\leq \|{\bf u}\|_{{\mathcal H}^1(\Omega _{-})^n}\ \forall \, {\bf u}\in {\mathcal H}^1(\Omega _{-})^n.
\end{align}
\end{thm}

Recall that $\rho $ is the weight function given by \eqref{rho},
${\mathcal H}^1({\mathbb R}^n\setminus \partial \Omega )$ is the space defined in \eqref{Omega-pm}-\eqref{standard-weight-p},
$\boldsymbol{\mathcal R}$ and $\boldsymbol{\mathcal R}_{\partial\Omega}$ are the space of rigid body motion fields in ${\mathbb R}^n$ and its trace defined in \eqref{E3.74}.
Note that ${\rm{dim}}\, \boldsymbol{\mathcal R}=n(n+1)/2$ (cf., e.g., \cite[p. 287]{Med-AAM-11}) and let $\left\{{\bf r}_j:j=1,\ldots ,{n(n+1)}/{2}\right\}$ be a basis of $\boldsymbol{\mathcal R}$.
\begin{lem}
\label{equiv-norm-Sobolev}
Let $\Omega \subset {\mathbb R}^n$, $n\geq 3$, be a bounded Lipschitz domain with connected boundary $\partial \Omega $.
Then the formula
\begin{align}
\label{coercive-dlc-apend}
\|{\bf w}\|_{1;\rho ;{\mathbb R}^n\setminus \partial \Omega }^2:=\|{\mathbb E}({\bf w})\|_{L_2({\mathbb R}^n\setminus \partial \Omega )^{n\times n}}^2
+\sum_{j=1}^{n(n+1)/2}\left|\int_{\partial \Omega }[\gamma {\bf w}]\cdot \gamma {\bf r}_jd\sigma \right|^2 \quad \forall \, {\bf w}\in {\mathcal H}^1({\mathbb R}^n\setminus \partial \Omega )^n
\end{align}
defines a norm in the weighted Sobolev space ${\mathcal H}^1({\mathbb R}^n\setminus \partial \Omega )^n$, which is equivalent to the norm
\begin{align}
\label{standard-weight}
\|{\bf w}\|_{{\mathcal H}^1({\mathbb R}^n\setminus \partial \Omega )^n}^2=\|\rho ^{-1}{\bf w}\|_{L_2({\mathbb R}^n\setminus \partial \Omega )^n}^2+\|\nabla {\bf w}\|_{L_2({\mathbb R}^n\setminus \partial \Omega )^n}^2\,.
\end{align}
\end{lem}
\begin{proof}
First, we note that by Theorem \ref{Korn-exterior}, $\|{\mathbb E}(\cdot )\|_{L_2(\Omega _{-})^{n\times n}}$ is a norm in ${\mathcal H}^1(\Omega _{-})^n$, which is equivalent to the norm $\|\cdot \|_{{\mathcal H}^1(\Omega _{-})^n}$, defined as in \eqref{standard-weight} with $\Omega_{-}$ in place of ${\mathbb R}^n\setminus \partial \Omega $.
Moreover, in view of the Korn inequality (see, e.g., \cite[Theorem 10.2]{Lean}, \cite[Proposition 11.4.2]{M-W}), $\|{\mathbb E}(\cdot )\|_{L_2(\Omega _+)^{n\times n}}+\|\cdot \|_{L_2(\Omega _+)^{n}}$ is an equivalent norm in the space ${H}^1(\Omega _{+})^n$.
Therefore,
\begin{align}
\label{standard-weight-1}
\|{\mathbb E}({\bf w})\|_{L_2(\Omega _{-})^{n\times n}}+\|{\mathbb E}({\bf w})\|_{L_2(\Omega _+)^{n\times n}}
+\|{\bf w}\|_{L_2(\Omega _+)^{n}}
=\|{\mathbb E}({\bf w})\|_{L_2({\mathbb R}^n\setminus \partial \Omega )^{n\times n}}+\|{\bf w}\|_{L_2(\Omega _+)^{n}} 
\end{align}
is a norm in the space ${\mathcal H}^1({\mathbb R}^n\setminus \partial \Omega )^n$, equivalent to the norm \eqref{standard-weight} of this space.

Now we consider the operators
\begin{align}
&P:{\mathcal H}^1({\mathbb R}^n\setminus \partial \Omega )^n\to L_2({\mathbb R}^n\setminus \partial \Omega )^{n\times n},\ P({\bf w}):={\mathbb E}({\bf w})\,,\\
&{\mathfrak C}:{\mathcal H}^1({\mathbb R}^n\setminus \partial \Omega )^n\to L_2(\Omega _+)^{n},\ {\mathfrak C}({\bf w})={\bf w}|_{\Omega _+},\\
&{\mathcal T}:{\mathcal H}^1({\mathbb R}^n\setminus \partial \Omega )^n\to {\mathbb R}^{{n(n+1)}/{2}},\ {\mathcal T}({\bf w})
:=\left(\int_{\partial \Omega }[\gamma {\bf w}]\cdot {\gamma}{\bf r}_1d\sigma ,\ldots ,\int_{\partial \Omega }[\gamma {\bf w}]\cdot {\gamma}{\bf r}_{n(n+1)/2}d\sigma \right),
\end{align}
which are linear and continuous. 
Moreover, the operator ${\mathfrak C}$ is compact due to the compact embedding of the space $H^1(\Omega _+)^{n}$ in $L_2(\Omega _+)^{n}$. In terms of these operators, the norm in \eqref{standard-weight-1} becomes
\begin{align}
\label{standard-weight-2}
\|{\mathbb E}({\bf w})\|_{L_2({\mathbb R}^n\setminus \partial \Omega )^{n\times n}}+\|{\bf w}\|_{L_2(\Omega _+)^{n}}=\|P({\bf w})\|_{L_2({\mathbb R}^n\setminus \partial \Omega )^{n\times n}}+\|{\mathfrak C}({\bf w})\|_{L_2(\Omega _+)^{n}}\,.
\end{align}
In addition, the operator ${\mathcal T}$ satisfies the condition ${\mathcal T}({\bf w})\neq 0$ whenever $P({\bf w})=0$ and ${\bf w}\neq {\bf 0}$.
Indeed, the condition
$P({\bf w})=0$
is equivalent to ${\bf w}\in \boldsymbol{\mathcal R}|_{\Omega _\pm }$ (cf., e.g., \cite[Lemma 3.1]{Med-AAM-11}).
Assume that ${\mathcal T}({\bf w})=0$ and $P({\bf w})=0$.
Then $\gamma {\bf w}\in\boldsymbol{\mathcal R}_{\partial\Omega}$ and
\begin{align}
\label{standard-weight-3}
\int_{\partial \Omega }[\gamma {\bf w}]\cdot {\gamma}{\bf r}_jd\sigma =0,\ j=1,\ldots ,n(n+1)/2\,.
\end{align}
Since $\boldsymbol{\mathcal R}_{\partial \Omega }={\rm{span}}\left\{{\gamma}{\bf r}_j:j=1,\ldots ,{n(n+1)}/{2}\right\}$, \eqref{standard-weight-3} yields that $[\gamma {\bf w}]=0$ on $\partial \Omega $, and accordingly that ${\bf w}\in {\mathcal H}^1({\mathbb R}^n)^n$ (cf. {Lemma \ref{extention}}).
On the other hand, the conditions ${\bf w}\in \boldsymbol{\mathcal R}|_{\Omega _\pm }$ and $[\gamma {\bf w}]={\bf 0}$ imply that there exist ${\bf a}\in {\mathbb R}^n$, ${\bf B}\in {\mathbb R}^{n\times n}$, ${\bf B}=-{\bf B}^\top $, such that ${\bf w}={\bf a}+{\bf B}{\bf x}$ in ${\mathbb R}^n$. However, the membership ${\bf w}\in {\mathcal H}^1({\mathbb R}^n)^n$ and the embedding ${\mathcal H}^1({\mathbb R}^n)^n \hookrightarrow L_{\frac{2n}{n-2}}({\mathbb R}^n)^n$ show that ${\bf w}={\bf 0}$, which contradicts the assumption ${\bf w}\neq {\bf 0}$. Thus, ${\mathcal T}({\bf w})\neq 0$ whenever $P({\bf w})=0$ and ${\bf w}\neq {\bf 0}$, as asserted.

Consequently, the conditions {of Lemma \ref{Tartar-lemma}
with $X:={\mathcal H}^1({\mathbb R}^n\setminus \partial \Omega )^n$, $Y=L_2({\mathbb R}^n\setminus \partial \Omega )^{n\times n}$, $Z=L_2(\Omega _+)^{n}$ and
{$\Upsilon :={\mathbb R}^{{n(n+1)}/{2}}$}}
are satisfied, and hence
\begin{align}
\label{coercive-dlc-apend-3}
\|P({\bf w})\|_{Y}+\|{\mathcal T}({\bf w})\|_{\Upsilon } =\|{\mathbb E}({\bf w})\|_{L_2({\mathbb R}^n\setminus \partial \Omega )^{n\times n}}+{ \sum_{j=1}^{n(n+1)/2}\left|\int_{\partial \Omega }[\gamma {\bf w}]\cdot {\gamma}{\bf r}_jd\sigma \right|}
\end{align}
is a norm on ${\mathcal H}^1({\mathbb R}^n\setminus \partial \Omega )^n$ equivalent to norm
\eqref{standard-weight}.
This result and the equivalence of the norms \eqref{coercive-dlc-apend} and \eqref{coercive-dlc-apend-3}
show that \eqref{coercive-dlc-apend} is also a norm in ${\mathcal H}^1({\mathbb R}^n\setminus \partial \Omega )^n$ equivalent to norm \eqref{standard-weight}.
\end{proof}

\section{Some norm estimates}
Let $n=3$ or $n=4$ and $\Omega$ be a bounded domain.
Then by the Sobolev embedding theorem (see, e.g., \cite[Section 2.2.4, Corollary 2]{Runst-Sickel}), the space $H^1(\Omega)^n$ is continuously embedded in $L_{\frac{2n}{n-2}}(\Omega)^n$
and hence in $L_n(\Omega)^n$.
Thus by the H\"older inequality there exists a constant $c_1>0$ such that
\begin{align*}
\|({\bf v}_1\cdot \nabla ){\bf v}_2\|_{L_{\frac{n}{n-1}}(\Omega)^n}&\le \|{\bf v}_1\|_{L_{\frac{2n}{n-2}}(\Omega)^n}\|\nabla {\bf v}_2\|_{L_{2}(\Omega)^{n\times n}}\leq c_1\|{\bf v}_1\|_{H^1(\Omega)^n}\|{\bf v}_2\|_{H^1(\Omega)^n} \quad \forall \, {\bf v}_1,{\bf v}_2\in H^1(\Omega)^n.
\end{align*}
Consequently,  $({\bf v}_1\cdot \nabla ){\bf v}_2\in L_{\frac{n}{n-1}}(\Omega)^n$ for any ${\bf v}_1,{\bf v}_2\in H^1(\Omega)^n$, and, thus,
$({\bf v}\cdot \nabla ){\bf v}\in L_{\frac{n}{n-1}}(\Omega)^n$ for any ${\bf v}\in {H}^1(\Omega)^n$.
Then, using again the H\"{o}lder inequality, we have for any ${\bf v}_1,{\bf v}_2, {\bf v}_3\in {H}^1(\Omega)^n$,
\begin{multline}
\label{P-0}
\left|\left\langle({\bf v}_1\cdot \nabla ){\bf v}_2,{\bf v}_3\right\rangle _{\Omega}\right|
=\left|\int_{\Omega}\left(({\bf v}_1\cdot \nabla ){\bf v}_2\right)\cdot {\bf v}_3d{\bf x}\right|
\leq \|({\bf v}_1\cdot \nabla ){\bf v}_2\|_{L_{\frac{n}{n-1}}(\Omega)^n}\|{\bf v}_3\|_{L_n(\Omega )^n}\\
\leq \|{\bf v}_1\|_{L_{\frac{2n}{n-2}}(\Omega)^n}\|\nabla {\bf v}_2\|_{L_{2}(\Omega)^{n\times n}}\|{\bf v}_3\|_{L_{n}(\Omega)^{n}}
\leq c_2\|{\bf v}_1\|_{L_{\frac{2n}{n-2}}(\Omega)^n}\|\nabla {\bf v}_2\|_{L_{2}(\Omega )^{n\times n}}\|{\bf v}_3\|_{H^1(\Omega)^{n}},
\end{multline}
with some constant $c_2>0$.
Taking ${\bf v}_3\in \mathring{H}^1(\Omega)^n$, \eqref{P-0} implies that term $({\bf v}_1\cdot \nabla ){\bf v}_2$ belongs to the dual of space $\mathring{H}^1(\Omega)^n$, i.e., to the space ${H}^{-1}(\Omega)^n$, and
for any ${\bf v}_1,{\bf v}_2\in {H}^1(\Omega)^n$,
\begin{align}
\label{P-01}
\left\|({\bf v}_1\cdot \nabla ){\bf v}_2\right\|_{{H}^{-1}(\Omega)^n}
\le c_2\|{\bf v}_1\|_{L_{\frac{2n}{n-2}}(\Omega)^n}\|\nabla {\bf v}_2\|_{L_{2}(\Omega )^{n\times n}}
\le c_3\|{\bf v}_1\|_{{H}^{1}(\Omega)^n}\|{\bf v}_2\|_{{H}^{1}(\Omega)^n}.
\end{align}
Similar to \eqref{P-0}, we have for any ${\bf v}_1,{\bf v}_2, {\bf v}_3\in {H}^1(\Omega)^n$,
\begin{multline}
\label{P-03}
\left|\left\langle({\bf v}_1\cdot \nabla ){\bf v}_2,{\bf v}_3\right\rangle _{\Omega}\right|
=\left|\int_{\Omega}{v}_{1,i} (\partial_i {v}_{2,j}){v}_{3,j}d{\bf x}\right|
\leq \|({v}_{3,j} \nabla {v}_{2,j}\|_{L_{\frac{n}{n-1}}(\Omega)^n}\|{\bf v}_1\|_{L_n(\Omega )^n}\\
\leq \|{\bf v}_3\|_{L_{\frac{2n}{n-2}}(\Omega)^n}\|\nabla {\bf v}_2\|_{L_{2}(\Omega)^{n\times n}}
\|{\bf v}_1\|_{L_{n}(\Omega)^{n}}
\leq c_6\|{\bf v}_1\|_{L_{n}(\Omega)^n}\|\nabla {\bf v}_2\|_{L_{2}(\Omega )^{n\times n}}\|{\bf v}_3\|_{H^1(\Omega)^{n}},
\end{multline}
Taking ${\bf v}_3\in \mathring{H}^1(\Omega)^n$, \eqref{P-03} implies that for any ${\bf v}_1,{\bf v}_2\in {H}^1(\Omega)^n$
\begin{align}
\label{P-04}
\left\|({\bf v}_1\cdot \nabla ){\bf v}_2\right\|_{{H}^{-1}(\Omega)^3}
\le c_6\|{\bf v}_1\|_{L_{n}(\Omega)^n}\|\nabla {\bf v}_2\|_{L_{2}(\Omega )^{n\times n}}
\le c_7\|{\bf v}_1\|_{L_{n}(\Omega)^n}\|{\bf v}_2\|_{H^1(\Omega )^{n}}.
\end{align}

Let now ${\bf v}_1,{\bf v}_2\in {H}^1(\Omega)^n$, ${\bf v}_3\in \mathring{H}^1(\Omega)^n$.
The density of $\mathcal D((\Omega)^n$ in $\mathring{H}^1(\Omega)^n$ along with the divergence theorem and estimate \eqref{P-01} lead to the following identity
\begin{align}
\label{P-051}
\left\langle({\bf v}_1\cdot \nabla ){\bf v}_2,{\bf v}_3\right\rangle _{\Omega}
&=\int_{\Omega}\nabla\cdot\left({\bf v}_1({\bf v}_2\cdot {\bf v}_3)\right)d{\bf x}
-\left\langle(\nabla \cdot{\bf v}_1){\bf v}_3
+({\bf v}_1\cdot \nabla ){\bf v}_3,{\bf v}_2\right\rangle _{\Omega}
\nonumber\\
&=-\left\langle(\nabla \cdot{\bf v}_1){\bf v}_3
+({\bf v}_1\cdot \nabla ){\bf v}_3,{\bf v}_2\right\rangle _{\Omega}
\quad \forall\ {\bf v}_1,{\bf v}_2\in {H}^1(\Omega)^n,\ {\bf v}_3\in \mathring{H}^1(\Omega)^n.
\end{align}
Then we obtain similar to \eqref{P-0},
\begin{multline}
\label{P-05}
\left|\left\langle({\bf v}_1\cdot \nabla ){\bf v}_2,{\bf v}_3\right\rangle _{\Omega}\right|
\leq \|(\nabla \cdot{\bf v}_1){\bf v}_3
+({\bf v}_1\cdot \nabla ){\bf v}_3\|_{L_{\frac{n}{n-1}}(\Omega)^n}\|{\bf v}_2\|_{L_n(\Omega )^n}\\
\leq \left(\|{\bf v}_3\|_{L_{\frac{2n}{n-2}}(\Omega)^n}\|\nabla\cdot {\bf v}_1\|_{L_{2}(\Omega)}
+\|{\bf v}_1\|_{L_{\frac{2n}{n-2}}(\Omega)^n}\|\nabla {\bf v}_3\|_{L_{2}(\Omega)^{n\times n}}\right)
\|{\bf v}_2\|_{L_{n}(\Omega)^{n}}\\
\leq c_8\|{\bf v}_1\|_{H^1(\Omega)^{n}}\|{\bf v}_2\|_{L_{n}(\Omega)^{n}}\|{\bf v}_3\|_{H^1(\Omega)^{n}},
\end{multline}
which implies that
\begin{align}
\label{P-06}
\left\|({\bf v}_1\cdot \nabla ){\bf v}_2\right\|_{{H}^{-1}(\Omega)^3}
\le c_8\|{\bf v}_1\|_{H^1(\Omega)^{n}}\|{\bf v}_2\|_{L_{n}(\Omega)^{n}}
\quad \forall\ {\bf v}_1,{\bf v}_2\in {H}^1(\Omega)^n.
\end{align}

From \eqref{P-051} we also have
\begin{align}
\label{antisym}
\left\langle({\bf v}_1\cdot \nabla ){\bf v}_2,{\bf v}_3\right\rangle _{\Omega}
&=-\left\langle({\bf v}_1\cdot \nabla ){\bf v}_3,{\bf v}_2\right\rangle _{\Omega}
\quad \forall\ {\bf v}_1\in {H}_{\rm{div}}^1(\Omega)^n,\
{\bf v}_2\in {H}^1(\Omega)^n,\ {\bf v}_3\in \mathring{H}^1(\Omega)^n,
\end{align}
implying the well known formula
\begin{equation}
\label{P-5a}
\left\langle ({\bf v}_1\cdot \nabla ){\bf v}_2,{\bf v}_2\right\rangle _{\Omega}=0 \quad \forall \
{\bf v}_1\in {H}_{\rm{div}}^1(\Omega)^n,\ {\bf v}_2\in \mathring{H}^1(\Omega)^n.
\end{equation}

Identity \eqref{P-051} also implies for ${\bf v}_1\in {H}_{\rm{div}}^1(\Omega)^n$ and ${\bf v}_2\in {H}^1(\Omega)^n$,
\begin{align*}
{|\!|\!|({\bf v}_1\cdot \nabla ){\bf v}_2|\!|\!|_{H^{-1}(\Omega)^n}}&={\sup_{{\bf v}_3\in \mathring{H}^1(\Omega )^n,\, \|\nabla {\bf v}_3\|_{L_2(\Omega )^{n\times n}}=1}}\Big|\left\langle ({\bf v}_1\cdot \nabla ){\bf v}_2,{\bf v}_3\right\rangle _{\Omega}\Big|\\
&={\sup_{{\bf v}_3\in \mathring{H}^1(\Omega )^n,\, \|\nabla {\bf v}_3\|_{L_2(\Omega )^{n\times n}}=1}}\Big|\left\langle ({\bf v}_1\cdot \nabla ){\bf v}_3,{\bf v}_2\right\rangle _{\Omega}\Big|\nonumber\\
&\leq {\sup_{{\bf v}_3\in \mathring{H}^1(\Omega )^n,\, \|\nabla {\bf v}_3\|_{L_2(\Omega )^{n\times n}}=1}}\|{\bf v}_1\otimes {\bf v}_2\|_{L_2(\Omega )^{n\times n}}\|\nabla {\bf v}_3\|_{L_2(\Omega )^{n\times n}},
\end{align*}
and hence
\begin{align}
\label{P-0512}
|\!|\!|({\bf v}_1\cdot \nabla ){\bf v}_2|\!|\!|_{H^{-1}(\Omega)^n}
\le \|{\bf v}_1\otimes{\bf v}_2\|_{L_{2}(\Omega)^{n\times n}}
\quad \forall\ {\bf v}_1\in {H}_{\rm{div}}^1(\Omega)^n,\
{\bf v}_2\in {H}^1(\Omega)^n.
\end{align}

By the H\"older inequality,
$$
\|{\bf v}_1\|^4_{L_{4}(\Omega)^{n}}\le \|{\bf v}_1\|^4_{L_{6}(\Omega)^{n}}|\Omega|^{1/3},\quad
|\Omega|:=\int_{\Omega}dx,
$$
and the Sobolev inequality (see, e.g., Eq. (II.3.7) in \cite{Galdi}) gives
\begin{align}\label{SoEn}
\|{v}\|_{L_{6}(\mathbb R^n)}
\le \frac{(n-1)}{\sqrt{n}(n-2)}\|\nabla {v}\|_{L^2(\mathbb R^n)} \quad \forall\ {v}\in \mathcal D(\mathbb R^n).
\end{align}
Then, applying the density argument in \eqref{SoEn}, relation \eqref{P-0512} implies {for $n=3$,}
\begin{multline}
\label{P-061}
|\!|\!|({\bf v}_1\cdot \nabla ){\bf v}_1|\!|\!|_{H^{-1}(\Omega)^3}
\le \|{\bf v}_1\otimes{\bf v}_1\|_{L_{2}(\Omega)^{3\times 3}}
\le \|{\bf v}_1\|^2_{L_{4}(\Omega)^3}
\le |\Omega|^{1/6}\|{\bf v}_1\|^2_{L_{6}(\Omega)^3}\\
\le \frac{4}{3}|\Omega|^{1/6}{\mn\|\nabla {\bf v}_1\|_{L_2(\Omega ^0)^{3\times 3}}^2}
\quad \forall\ {\bf v}_1\in \mathring{H}_{\rm{div}}^1(\Omega)^3,
\end{multline}
cf. Lemma IX.1.1 in \cite{Galdi}, where a similar estimate was obtained with the coefficient $\frac{2\sqrt{2}}{3}$ instead of $\frac{4}{3}$.

Similarly, for functions from $\mathcal{H}_{\rm{div}}^1(\mathbb R^3)^3$,
\begin{align}
\label{P-061cal}
|\!|\!|({\bf v}_1\cdot \nabla ){\bf v}_1|\!|\!|_{H^{-1}(\Omega)^3}
\le |\Omega|^{1/6}\|{\bf v}_1\|^2_{L_{6}(\Omega)^3}
&\le |\Omega|^{1/6}\|{\bf v}_1\|^2_{L_{6}(\mathbb R^3)^3}\nonumber\\
&\le \frac{4}{3}|\Omega|^{1/6}\|\nabla{\bf v}_1\|^2_{L_2(\mathbb R^3)^{3\times 3}}
\quad \forall\ {\bf v}_1\in \mathcal{H}_{\rm{div}}^1(\mathbb R^3)^3.
\end{align}
}

\section*{\bf Acknowledgements}
The research has been supported by the grant EP/M013545/1: "Mathematical Analysis of Boundary-Domain Integral Equations for Nonlinear PDEs" from the EPSRC, UK.
M. Kohr has been also partially supported by the Babe\c{s}-Bolyai University research grant AGC35124/31.10.2018. W.L. Wendland has been partially supported by "Deutsche Forschungsgemeinschaft (DFG, German Research Foundation) under Germany's Excellence Strategy--EXC 2075--390740016".
Part of this work was done in May 2019, when M. Kohr visited the Department of Mathematics of the University of Toronto, and in August 2019, during her research visit to the Department of Mathematics of Brunel University London.
She expresses her gratitude to the members of these departments
for their hospitality during those visits.

\end{document}